\documentclass{amsart}

\usepackage{bm}
\usepackage{wrapfig}                  
\usepackage{graphicx}        
\usepackage{multicol}        
\usepackage[bottom]{footmisc}
\usepackage{epsfig}
\usepackage{amsbsy,amsmath}
\usepackage{rotating}
\usepackage{graphicx}
\usepackage{psfrag}
\usepackage{amscd}

\usepackage{mathptmx}
\usepackage{amsmath}
\usepackage{amscd}
\usepackage{amssymb}

\usepackage{mathrsfs}
\usepackage{amsfonts}
\usepackage{amsmath,amsfonts,amssymb,amsthm}
\usepackage{mathdots}
\usepackage[all]{xy}

\usepackage{imakeidx}
\makeindex[title=Subject index] 
\makeindex[name=authors,title=Index of  authors] 
\newcommand{\si}{\index}
\newcommand{\ai}{\index[authors]}

\usepackage{nohyperref}
\usepackage[pageref]{backref}
\renewcommand*{\backref}[1]{}
\renewcommand*{\backrefalt}[4]{%
    \ifcase #1 (Not cited.)%
    \or       $\ \ \ $(Cited on p.~#2.)%
    \else      $\ \ \ $(Cited on pp.~#2.)%
    \fi}

 \newtheorem{thm}[equation]{Theorem}
 \newtheorem{prop}[equation]{Proposition}
 \newtheorem{lem}[equation]{Lemma}
 \newtheorem{coro}[equation]{Corollary}
 \newtheorem{conj}[equation]{Conjecture}

\theoremstyle{definition}
 \newtheorem{de}[equation]{Definition}
 \newtheorem{rem}[equation]{Remark}
  \newtheorem{nota}[equation]{Notation}
 \newtheorem{ex}[equation]{Example}
 
  \newtheorem{exer}[equation]{Exercise}
 \newtheorem{prob}[equation]{Problem}

\theoremstyle{definition}

\usepackage{tikz}
\usetikzlibrary{positioning,quotes,arrows}

\usepackage{enumerate}

\usepackage{ulem}

\newcommand{\apr}{(Ap. \ref}
\newcommand{\dist}{\textnormal{dist}}
\newcommand{\Z}{\mathbb Z}
\newcommand{\Q}{\mathbb Q}
\newcommand{\R}{\mathbb R}
\newcommand{\C}{\mathbb C}

\newcommand{\N}{\mathbb N}
\newcommand{\ZG}{\mathbb Z G}
\newcommand{\GL}{\textnormal{GL}}
\newcommand{\cok}{\textnormal{cok}}

\newcommand{\overd}{\overset{\bullet}}
\newcommand{\tr}{\textnormal{trace}}

\usepackage{stackengine}
\usepackage{scalerel}
\usepackage{caption}

\newcommand{\calR}{\mathcal{R}}
\newcommand{\calS}{\mathcal{S}}

\newcommand\reallywidesim[1]{\mathrel{\ThisStyle{%
  \setbox0=\hbox{$\SavedStyle\mkern3mu_{\text{#1}}\mkern3mu$}%
  \mkern1mu\stackengine{1\LMpt}{%
    \stretchto{\scaleto{\SavedStyle\mkern-1mu\sim\mkern-1mu}{.28\wd0}}{1\ht0}%
  }{$\SavedStyle_{\text{#1}}$}{O}{c}{F}{T}{S}\mkern1mu%
}}}
\newcommand{\aut}{\textnormal{Aut}}
\newcommand{\esse}[2]{{#1} \begin{tiny} \reallywidesim{\hspace{.03in}  esse-$\mathcal{R}$ } #2 \end{tiny}}
\newcommand{\sse}[2]{{#1} \begin{tiny} \reallywidesim{\hspace{.03in}  sse-$\mathcal{R}$ } #2 \end{tiny}}
\newcommand{\se}[2]{{#1} \begin{tiny} \reallywidesim{\hspace{.03in}  se-$\mathcal{R}$ } #2 \end{tiny}}

\newcommand{\eleq}[2]{{#1} \begin{tiny} \reallywidesim{\hspace{.03in}  El-$\mathcal{R}[t]$ } #2 \end{tiny}}
\newcommand{\sym}{\textnormal{Sym}}
\newcommand{\simp}[1]{\textnormal{Simp}(\sigma_{#1})}
\newcommand{\inert}[1]{\textnormal{Inert}(\sigma_{#1})}
\newcommand{\autinf}[1]{\textnormal{Aut}^{(\infty)}(\sigma_{#1})}
\newcommand{\inertinf}[1]{{\textnormal{Inert}}^{(\infty)}(\sigma_{#1})}

\newcommand*\circled[1]{\tikz[baseline=(char.base)]{
    \node[shape=circle, draw, inner sep=1pt,
        minimum height=12pt] (char) {#1};}}
\newcommand{\deter}{\textnormal{det}}
\newcommand{\sta}{\textnormal{St}}
\newcommand{\sgn}{\textnormal{sign}}
\newcommand{\image}{\textnormal{Image}}

\numberwithin{equation}{subsection}

\begin{document}

\title[Symbolic dynamics and stable algebra]{Symbolic dynamics and the stable algebra of
  matrices}


\author[M. Boyle]{Mike Boyle}


\address{Department of Mathematics, University of Maryland, College
Park, MD 20742-4015, USA
}
\email{mmb@math.umd.edu}

\author[S. Schmieding]{Scott Schmieding}
\address{Department of Mathematics, 
Pennsylvania State University, 
State College, PA 16802, USA}

\email{sks7247@psu.edu}


\subjclass[2020]{Primary 37B10; Secondary 19-01, 15B48, 15-B36, 06-01}






\begin{abstract} We give an introduction to a topic in
  the ``stable algebra\si{stable algebra} of matrices'',  as related to
  certain problems in symbolic dynamics. We introduce enough
  symbolic dynamics to explain these connections, but,
    the algebra
  is of independent interest and can be followed with little attention
  to the symbolic dynamics. This
``stable algebra of matrices''\si{stable algebra}
involves the study of
  properties and relations of square matrices over a semiring $\mathcal S$
  which are invariant under two fundamental equivalence relations:
  shift equivalence and strong shift equivalence.
  When $\mathcal S$ is a field, these relations are the same, and
  matrices over $\mathcal S$ are shift equivalent if and only if
  the nonnilpotent parts of their canonical forms are
  similar. We give a detailed account of these relations over
  other rings and semirings, especially $\Z$, $\Z_+ $ and $\R_+$.
  When $\mathcal S$ is a ring,
  this involves  module theory and algebraic $K$-theory.
  %
    We discuss in detail and contrast the problems
    of characterizing the possible spectra, and the possible nonzero
    spectra, of nonnegative real matrices.
  We also review key features of the automorphism group of a shift
  of finite type;  the recently introduced stabilized automorphism
  group;  and the work of Kim\ai{Kim, K.H.}, Roush\ai{Roush, F.W.} and Wagoner\ai{Wagoner, J.B.} giving
  counterexamples to Williams' Shift Equivalence\si{shift equivalence} Conjecture.\ai{Williams, R.F.}
\end{abstract}

\maketitle




\tableofcontents

\addtocounter{section}{-1}
\section{Introduction}



The bulk of this work is devoted to
 an exposition  of  algebraic properties
and relations of square
matrices which are invariant under
similarity and certain
``stabilizing'' relations,
with  connections to symbolic dynamics, linear algebra and
algebraic K-theory.
For matrices over a field, this amounts simply to neglecting
the nilpotent part of its action; for matrices over other rings,
even $\Z$,
the analogous stabilization is  more subtle. For matrices over
$\Z_+$, the analogous stabilization
become very subtle indeed.
The  algebra of these relations
arises naturally in problems of
symbolic dynamics.
However, quite apart from symbolic dynamics, we believe
this algebra is an important topic in the theory of matrices.
We aim to give a presentation useful for both students and
experts.

The slice of the
``stable algebra of matrices\si{stable algebra}''
we study involves
matrix relations and
properties invariant under two fundamental
equivalence relations:
shift equivalence  and  strong shift
equivalence.
We begin with   definitions of these
relations.
Let $\mathcal C$ be a category,
with $\mathcal M$ its class of morphisms
and  $\mathcal E$ its class of
   endomorphisms\si{endomorphism} (morphisms with domain = codomain).
   Elements  $A,B$ of  $\mathcal E$ are
   {\it elementary strong shift
     equivalent\si{elementary strong shift equivalent}} if there are morphisms
   $U,V$    in  $\mathcal M$
    such that
    $A=UV$ and $ B=VU$. Here, $UV$ is the
    composition\begin{footnote}{Unless otherwise specified,
      we adopt the convention
      that    morphisms\si{morphism, domain and codomain}
      are arrows from a domain object on the left to
   a codomain object on the right. So, for the composition $UV$
   of morphisms $U,V$ to be defined, the codomain of $U$
   must equal the domain of $V$.}\end{footnote}\begin{footnote}{Note,
          when the morphisms $U,V$ are functions,
     our domain/codomain convention means that
     we are  reading composition  from
     left to right: for an input $x$, the output  $(UV)(x)$ is
     $V(U(x))$. The chosen convention is irrelevant (and safely
     ignored) until Section 7,
     as explained  there.}\end{footnote}
    of
    the morphisms $U,V$. In general, $U$ and $V$ do {\it not}
    have to be endomorphisms, although the equations force
    $A$ and $B$ to be endomorphisms\begin{footnote}{E.g.,
    $\text{domain}(A) =\text{domain}(U) =
    \text{codomain}(V) =\text{codomain}(A) $.}\end{footnote}.
    We refer to $(U,V)$ as an elementary
    strong shift equivalence  from $A$ to $B$, or between $A$ and $B$.
   We use ESSE\si{ESSE} as a multipurpose abbreviation whose meaning
   should be clear in context.
    The relation    ESSE is only defined for endomorphisms, but depends on
    $\mathcal M$.
     We may refer to ESSE  in $\mathcal C$, or
     in $\mathcal M$.

     We say endomorphisms $A,B$ are
     {\it similar}\si{similarity of morphisms} (isomorphic as endomorphisms)
    if there is an invertible morphism\begin{footnote}{A morphism $A$
      from object $D$ to object $C$ is
      invertible if there are morphisms $U,V$ such that
      $UA= 1_{\text{codom}(A)}$
      and $AV=1_{\text{dom}(A)}$. In this case, $U$ must equal $V$, as
      $U=U(AV)=(UA)V=V$; $U$ is the inverse of $A$, denoted
      $A^{-1}$.}\end{footnote}
    $U$ such that
     $B=U^{-1}AU$.
Now suppose  that $A$ and $B$ are automorphisms
   (invertible endomorphisms) in a category $\mathcal C$,
   with an ESSE $A=UV$ and $B=VU$. Then
   $U$ and $V$ must be
   invertible\begin{footnote}{E.g.,
     $U$ is invertible because $U(VA^{-1})=1_{\text{dom}(U)}$ and
     $(B^{-1}V)U=1_{\text{codom}(U)}$.}\end{footnote},
     with
   $B=U^{-1}AU$. So,  for automorphisms ESSE is similarity.
   We see ESSE for endomorphisms is some kind of endomorphism generalization
   of similarity.

   We will be primarily interested in the case that
   $\mathcal C$ is a category whose morphisms are
   $\mathcal M(\mathcal S)$, the finite matrices over
   a semiring\begin{footnote}{A matrix ``over'' a set $\mathcal S$ is simply a
   matrix all of whose entries are in $\mathcal S$.
By a semiring\si{semiring} $\mathcal S$, we mean a set with addition
   and multiplication satisfying all the ring axioms except possibly
   existence of additive inverses.
   By definition, we require a semiring (in particular, a ring)
   to contain a multiplicative unit.}\end{footnote}
 $\mathcal S $,
   with composition of morphisms given by matrix multiplication,
   so $\mathcal E$ is the set of square matrices  over $\mathcal S$.
   The objects of $\mathcal C$ are the sets $\mathcal S^n$ ($\mathcal S^n$
   is the set of   $n$-tuples with entries in $\mathcal S$).
 With our default domain convention
 for morphisms, we take $\mathcal S^n$ to be a set of row vectors.
 If $\mathcal S$ is a ring $\mathcal R$, then
 by matrix multiplication
 an $m\times n$ matrix over $\mathcal R$ defines
    an $\mathcal R$-module homomorphism from
    $\mathcal S^m$ to $\mathcal S^n$,
 and  $\mathcal C$ can (of course) be identified with
    the category of $\mathcal R$-module homomorphisms of
    free, finitely generated left $\mathcal R$-modules.
    In this category, ESSE is not an equivalence relation.

Returning to a general category $\mathcal C$,
we define strong shift equivalence to be the equivalence
relation which is the transitive closure of ESSE.
That is, endomorphisms $A,B$ are {\it strong shift equivalent} if there
there are endomorphisms  $A_0, \dots , A_{\ell}$
and morphisms $U_i, V_i$,
such that
$A=A_0$, $B=A_{\ell}$,   and for $1 \leq  i  \leq \ell$
we have
$A_{i-1}=U_iV_i, A_{i}=
V_iU_i$. We refer to the
finite sequence $(U_i, V_i), 1\leq i \leq \ell,$ as  a
strong shift equivalence from $A$ to $B$.
The relation SSE, like ESSE,
is only defined for endomorphisms, and depends on
$\mathcal M$.
We use SSE\si{SSE}  as
a multipurpose abbreviation.
For a semiring $\mathcal S$, SSE-$\mathcal S$ denotes
SSE in
$\mathcal M = \mathcal M(\mathcal S)$, as described above.
SE and SSE arise in   various   settings
within
and outside
symbolic dynamics
\begin{footnote}
  {E.g., SE arises in ergodic theory
 \cite{LinRudolph}
    and Conley index theory
\cite{MischaikowWeibel2023}. 
    For SSE-$\Z_+$ in knot theory,
  see \cite{Gilmer1999,SilverWilliams1999}.
  For SSE related to the nonzero spectrum (and other invariants)
  of primitive real matrices, see
  \cite{BH91,BKR2013}.
For  a connection to category theory,
see \cite{jeandel}.
 For SE and SSE in other settings in symbolic dynamics, see
 Section \ref{subsec:otherrings}.}\end{footnote}.

The simplicity of the  SSE definition
is utterly deceptive.
For example,
SSE-$\mathcal \Z_+$
was introduced by Williams\ai{Williams, R.F.}
to classify shifts of finite type
up to topological conjugacy.
A half century later, we do not
know if the relation SSE-$\Z_+$
is even decidable.

   To study strong shift equivalence,
   one is naturally led (by Williams) to the  more tractable
   relation of shift equivalence.
   Given  $(U_i,V_i)$, $1 \leq i \leq \ell$, an SSE in
   a category $\mathcal C$ from
   $A=A_0$ to $B=A_{\ell}$, set  $U=U_1 \cdots U_{\ell }$
   and $V= V_{\ell }\cdots V_1 $. Then we have
     \begin{alignat*}{2}
       A^{\ell} &= UV\ ,  & \ \ \quad  B^{\ell} &= VU \\
       AU &=UB\ ,      \ \  \quad  &\quad    VA &= BV  \ \ .
     \end{alignat*}
     So,  endomorphisms $A,B$ are defined to be {\it shift equivalent}
     in $\mathcal C$
     if there are morphisms $U,V$ and a positive integer $\ell$
     satisfying the four displayed equations.
     One can check that SE is an equivalence relation on the endomorphisms
     in $\mathcal C$.
     Clearly SSE implies SE.
     The converse holds for $\mathcal M(\mathcal S)$ for many
     rings $\mathcal S$, but fails even for some integral domains,
     and fails for the semiring $\mathcal S=\Z_+$.

Let $\mathcal R$ be a ring.
Then,  SSE-$\mathcal R$ and SE-$\mathcal R$
can be described by other generating relations.
We say square matrices
$A,B$ in $\mathcal M(\mathcal R)$ are similar over $\mathcal R$
(SIM-$\mathcal R$)
if there exists a
matrix $U$ invertible over
$\mathcal \mathcal R$
such that
$B=U^{-1}AU$\begin{footnote}{A matrix is invertible over
    $\mathcal R$ if it has an inverse matrix all of whose entries
    are in $\mathcal R$. If a ring $\mathcal R$ has the invariant basis
    property (IBP), then an invertible matrix over $\mathcal R$
    must be square. If a $\mathcal R$ is a commutative ring, or
    an integral group ring, then it has the
    IBP \cite{WeibelBook}.
    For a ring without the IBP, ``similar'' square matrices can have
different sizes. }\end{footnote}.
Given a square matrix $A$, we define a
 nilpotent extension of $A$ to be
a square matrix with a block form
$\left( \begin{smallmatrix} A&X\\0 & N \end{smallmatrix} \right)$
or
$\left( \begin{smallmatrix} A&0\\X & N \end{smallmatrix} \right)$,
with $N$ a nilpotent matrix. A zero extension of $A$ is a nilpotent
extension in which $N$ is a zero matrix.
Then, the following hold  (by Theorem \ref{thm:endorelations}).

\begin{enumerate}
\item
  SSE-$\mathcal R$  is generated by similarity and
 zero extensions\begin{footnote}{I.e.,  SSE-$\mathcal R$ is the smallest
   equivalence relation on square matrices over $\mathcal R$
   containing SIM-$\mathcal R$
and closed under taking nilpotent extensions.}\end{footnote}.

\item
  SE-$\mathcal R$  is generated by similarity and
nilpotent extensions.
\end{enumerate}

          So, we study invariants
of similarity which persist under zero or nilpotent extensions.
We will see that this ``stable'' viewpoint
 can  be useful even for studying
 properties of real nonnegative matrices.

 Now, we turn to the organization of the sequel, and
supplement
 the detailed outline in
 the table of contents with some broader remarks.



In Section 1,
 we introduce shifts of finite type,
 the most fundamental   symbolic dynamical systems.
A square nonnilpotent matrix  $A$ over $\Z_+$
defines a shift of finite type (SFT), $\sigma_A$; we will see how
dynamical properties and relations correspond to
``stable algebra'' invariants of the matrix $A$.
By no means do we give a full introduction to SFTs.
We  say enough to explain
how the stable algebra invariants arise in symbolic dynamics --
to inform those interested in the dynamics, and provide
motivation and a concrete model system for those interested
in the algebra.

In Section 2, after  brief general remarks about
strong shift equivalence\si{strong shift equivalence},
we give an extensive discussion of
shift equivalence\si{shift equivalence}, including several example cases.
For a general ring $\mathcal R$, SE-$\mathcal R$ is characterized
by isomorphism of certain associated $\mathcal R[t]$-modules.

The theory of SE and SSE can be entirely recast in terms
of polynomial matrices; we do this in Section 3. This is
essential for later K-theory connections. One key is the presentation of
a directed graph by a polynomial matrix, a construction of interest for anyone
using directed graphs.


A classical problem of linear algebra, the
nonnegative inverse eigenvalue problem (NIEP),  asks which
multisets of complex numbers can be the spectrum of a
nonnegative matrix.  A stable version of the NIEP
asks which multisets of nonzero complex numbers can be the
nonzero\si{nonzero spectrum} part of the spectum of a nonnegative matrix.
In contrast to the notoriously difficult NIEP, the stable version has a
transparent solution (not a transparent proof).
We
review results and conjectures related to this
stable approach in Section 4, including new commentary on a
realization theorem of Tom Laffey.\ai{Laffey, Thomas J.}

The simple definition of SSE takes us
 to the higher mathematics of algebraic $K$-theory.
In section 5, we give
an introduction to enough algebraic $K$-theory to make the later statements
understandable, with a little context\begin{footnote}{A reader
  attracted by this glimpse  can look into
    various algebraic $K$-theory introductions, with their
    different scopes, prerequisites and publication dates
    (e.g., in order of publication,
    \cite{BassBook, MilnorBook, Silvester1981,RosenbergBook, WeibelBook}).
  The short exposition \cite{zakharevich2019} gives a lovely
  perspective on $K$-theory as a ``theory of assembly''.}\end{footnote}.

Let $\mathcal R$ be a ring. SE-$\mathcal R$ is a topic in
the theory of $\mathcal R$-modules;  the refinement of SE-$\mathcal R$
by SSE-$\mathcal R$ is understood (incompletely,
to date) with
algebraic $K$-theory.
In Section 6, we give the algebraic $K$-theoretic characterization
of the refinement of
SE-$\mathcal R$ by
SSE-$\mathcal R$. Here the class group of the category of nilpotent endomorphisms over $\calR$ plays a key role.

In Section 7, we give an overview of results on the automorphism
group of a shift of finite type: its actions and representations,
and related
  problems. We also (briefly)  introduce the mapping class group and
  the stabilized automorphism group of a shift of finite type.
We review two crucial representations (dimension and SGCC)
of the automorphism group
which
become key ingredients to the Kim-Roush\ai{Roush, F.W.}/Wagoner\ai{Wagoner, J.B.}\ai{Kim, K.H.} work giving
counterexamples to the Williams\ai{Williams, R.F.} Conjecture that
SE-$\Z_+$ implies SSE-$\Z_+$.

In Section 8, we give an
overview of the counterexample work showing
SE-$\Z_+$ does not imply SSE-$\Z_+$, even for primitive matrices.
This work takes place within
the machinery of Wagoner\ai{Wagoner, J.B.}'s CW complexes for SSE,
which give  a  different, homotopy based approach to SSE.
The counterexample
invariant of Kim\ai{Kim, K.H.} and Roush\ai{Roush, F.W.}
is  a relative sign-gyration number.
Wagoner\ai{Wagoner, J.B.} later formulated a different counterexample invariant,
with values in a certain group coming out of algebraic
$K$-theory.

Our exposition
is an elaboration of the course  we gave at the
2019 Yichang
G2D2 school,
accessible to graduate students,  but  of broader interest.
 The first four sections are by Boyle;
 the last four are by Schmieding.
 Each  section has the form of  a lecture, together  with an
 appendix including further details, proof and remarks.
 A numerical reference beginning
with Ap is a reference to such an appendix (e.g.,
\apr{sePrimitive})  is a reference to
an appendix item in Section \ref{sec:sesse}).

 The appendices are intended to enhance the value of this work
 both as exposition and reference.
 The format is intended to  keep the concise overview of the lectures, and
 (along with the detailed table of contents)
 facilitate  sampling by readers with disparate interests.



Fundamental problems remain open in the  theory of
strong shift equivalence\si{strong shift equivalence},
and the related theory of shifts of finite type and
their automorphism groups.
  We
hope our exposition may
encourage some contribution to their solution.

{\it Acknowledgements.} Each bibliography item includes the pages on which
it is cited; additional citations by name are in the index of authors.
For feedback and corrections on our manuscript, we thank Peter Cameron,
  Tullio Ceccherini-Silberstein, Sompong Chuysurichay,
  Ricardo Gomez, Emmanuel Jeandel, Charles Johnson,
  Johan Kopra, Wolfgang Krieger, Michael Maller,
 Alexander Mednykh,
Akihiro Munemasa, Michael Shub, Yaokun Wu and Yinfeng Zhu.
We thank especially Yaokun Wu, without whose vision and organization
this work would not exist.\ai{Cameron, Peter}\ai{Ceccherini-Silberstein, Tullio}\ai{Johnson, Charles}
\ai{Mednykh, Alexander}\ai{Munemasa, Akihiro}\ai{Wu, Yaokun}\ai{Zhu, Yinfeng}
\ai{Chuysurichay, Sompong}\ai{Gomez, Ricardo}\ai{Jeandel, Emmanuel}
\ai{Kopra, Johan}\ai{Krieger, Wolfgang}
\ai{Maller, Michael}\ai{Shub, Michael}


\newpage



%
%

\section{Basics}\label{sec:basics}

In this first section, we review some fundamentals of shifts of
finite type and the algebraic invariants of the matrices which
present them.

\subsection{Topological dynamics}
    By a topological dynamical system (or system), we will mean a
    homeomorphism of a  compact metric space to itself.  This is one setting for
    considering how points can/must/typically behave over time,
    i.e., how points move under iteration of the homeomorphism.
    A system/homeomorphism $S: X\to X$ is often written as a pair,
     $(X,S)$.
    Formally: we have a category, in which
\begin{itemize}
\item   an object is a topological dynamical system,
    \item
a morphism  $\phi: (X,S) \to (Y,T)$ is a continuous map
$\phi: X\to Y$ such that
$\phi T =  S\phi $, i.e. the
    following diagram commutes.
\[
\xymatrix{
X \ar[r]^{S} \ar[d]_{\phi } & X \ar[d]^{\phi} \\
Y \ar[r]^{T} & Y
}
\]
\end{itemize}
    The morphism is a topological conjugacy
    (an isomorphism in our category)
    if $\phi$ is a homeomorphism.

    Now, let $\phi: X\to Y$ be a topological conjugacy.
    We can think of  $\phi$ as follows:
    $\phi$  renames points without changing the mathematical
    structure of the system.
    \begin{itemize}
 \item        Because $\phi$ is a homeomorphism,  it gives new names to points
in essentially the same topological space.

\item Because $\phi T= S\phi $, the renamed points move as they did
with their original names. E.g., with $y=\phi x$,
\[
\xymatrix{
  \cdots \ \ar[r]^{S}  &  x \ar[r]^{S} \ar[d]_{\phi } &Sx \ar[r]^{S} \ar[d]_{\phi }
  & S^2x \ar[r]^{S}\ar[d]_{\phi }  & \ \cdots\\
  \cdots \ \ar[r]^{T}  &  y \ar[r]^{T}              &Ty \ar[r]^{T}
  & T^2y \ar[r]^{T}  & \ \cdots
}
\]
\end{itemize}
As $\phi$ respects the mathematical structure
     under consideration,  we see that the systems $(X,S)$ and $(Y,T)$
     are essentially  the same, just described in a different language.
         (Perhaps, points of $X$  are described in English, and
points of  $Y$ are described in Chinese, and  $\phi$ gives a
translation.)

We are  interested in ``dynamical'' properties/invariants  of
     a topological dynamical system --
     those which are  respected by topological conjugacy.
     For example, suppose
     $\phi :(X,S) \to (Y,T)$ is a topological conjugacy,
     and $x$ is a fixed point of $S$: i.e., $S(x)=x$.
     Then $\phi (x) $ is  a fixed point of $T$.
     The proof is trivial:
     $T(\phi (x)) =
     \phi (Sx)  = \phi (x) $. (We almost don't need a proof:
     however named, a fixed point is a fixed point.)
So, the cardinality of the fixed point set, $\text{card}(\text{Fix}(S))$,
     is a dynamical invariant.

\begin{nota}For $k\in \N$,  $S^k$ is $S$ iterated $k$ times.
     E.g., $S^2: x\mapsto S(S(x))$.
\end{nota}

If $\phi :(X,S) \to (Y,T)$ is a topological conjugacy, then
$\phi $ is also  is a topological conjugacy $(X,S^k) \to (Y,T^k)$, for all $k$
(because $\phi T= S\phi  \implies \phi T^k =  S^k\phi$).
So, the sequence
     $(|\text{card}(\text{Fix}(S^k)|)_{k\in \N}  $ is a dynamical invariant
     of the system $(X,S)$.

     \subsection{Symbolic dynamics}\label{subsec:symbolicdynamics}

Let $\mathcal A$ be a finite set.
Then $\mathbf{\mathcal A^{\mathbf{\Z}}}$ is the set of functions
     from $\Z$ to $\mathcal A$. We write an element $x$ of $\mathcal A^{\Z}$
     as  a doubly infinite sequence,  $x= \dots x_{-1}x_0x_1\dots$, with each
     $x_k$ an
 element of $\mathcal A$.
 (The bisequence defines the
 function $k\mapsto x_k$.) Often  $\mathcal A$ is called the alphabet, and its elements
 are called symbols.

     Let $\mathcal A$ have the discrete topology and let $\mathcal A^{\Z}$ have the product
     topology. Then $\mathcal A^{\Z}$ is a
{\it compact metrizable
space} \apr{zerodimetc}). For one metric compatible with
     the topology, given $x\neq y$ set
     $\dist (x,y) = 1/(M+1)$, where
     $M = \min \{|k|: x_k \neq y_k\}$. Points $x,y$ are
     close when they have the same central word,
     $x_{-M}\dots x_M = y_{-M}\dots y_M$, for large $M$.

The   {\it shift map} $\sigma: \mathcal A^{\Z}
   \to \mathcal A^{\Z}$ is defined by
     $(\sigma x)_n = x_{n+1}$.
     (This is the ``left shift'': visually, a symbol in box $n+1$ moves left into box $n$.)
     The  shift map $\sigma: \mathcal A^{\Z} \to \mathcal A^{\Z}$ is
     easily checked to be a homeomorphism.
     The system $(\mathcal A^{\Z}, \sigma )$ is called the
     full shift on $n$ symbols, if $n= |\mathcal A| $.
One notation: a dot over a symbol indicates it occurs in the zero coordinate.
Then
\begin{align*}
  \sigma: \ x &\mapsto\  \sigma(x) \\
\sigma :\ \dots x_{-2}x_{-1}\overset{\bullet}{x_0}x_1x_2\dots \ &\mapsto\
\dots x_{-2}x_{-1}x_0\overset{\bullet}{x_1}x_2\dots \ \ .
\end{align*}

     A {\it subshift} is a subsystem $(X,\sigma|_X)$ of some full shift
$(\mathcal A^{\Z}, \sigma)$
(i.e. $X$ is a closed subset of $\mathcal A^{\Z}$, and $\sigma (X)=X$).
For notational
     simplicity, we generally write $(X,\sigma|_X)$ as $(X,\sigma)$.
     Among the subshifts, the  subshifts of finite type
(also called {\it shifts of finite type},
     or {\it SFTs}) are a fundamental class, with varied
     applications \apr{introtosfts}).

\begin{de} \label{sftdefn} By definition, a subshift $(X, \sigma )$ is
  a subshift of finite type (SFT)
if if there is a finite set $\mathcal F$ of words on the alphabet $\mathcal A$
of $X$ such that $X$ is the subset of points $x$ in $\mathcal A^{\Z}$
such that no subword $x_m\cdots x_n$ is in $\mathcal F$.
\end{de}

     \subsection{Edge SFTs}\label{sec:edgesfts}

     We are interested in relating dynamical properties and relations
     of SFTs
     to matrix algebra. A key to this is given by a particular class
     of SFTs, the edge SFTs, which are presented by matrices.

     \begin{nota}  For us, always,
       ``{\it graph}'' means ``directed graph''.
Given an ordering of the vertices, $v_1, \dots v_n$, the
{\it adjacency matrix} $A$ of the graph is defined by setting
 $A(i,j)$ to be the number of
edges from vertex $v_i$ to vertex $v_j$. For simplicity we often
just refer to vertices $1, \dots , n$.
 \end{nota}

\begin{de} (Edge SFT)
          Given a square matrix $A$ over $\Z_+$,
          we let $ \Gamma_A$ denote a
     graph with adjacency matrix $A$.
     Let $\mathcal E$ be the set of edges of $ \Gamma_A$.
 $X_A$ is the set of
doubly infinite sequences $x = \dots x_{-2}x_{-1}x_{0}x_{1}x_{2} \dots$
such that each $x_n$ is in $\mathcal E$,
and for all $n$  the terminal vertex of $x_n$ equals the inital
vertex of $x_{n+1}$. (So, the points in $X_A$ correspond
  to doubly
infinite walks  through $ \Gamma_A$.)
The system $(X_A, \sigma)$,
is the edge shift, or edge SFT, defined by $A$
 \apr{edgesftnote}). We may also use the notation $\sigma_A$
 to denote the map $\sigma: X_A\to X_A$.
\end{de}
\begin{ex}
$(X_A,\sigma)$ is
the full shift on the two symbols $a,b$. The graph has a single
vertex,  denoted as 1.
     \[
A=\begin{pmatrix} 2  \end{pmatrix} \ , \qquad \quad \quad
\Gamma_A \quad =   \quad
\xymatrix{ *+[F-:<3pt>]{1} \ar@(lu,ld)_{a} \ar@(ru,rd)^{b}
  }
  \quad =   \quad
  \xymatrix{ \cdot \ar@(lu,ld)_{a}\ar@(ru,rd)^{b}
}
\]
\end{ex}
\begin{ex} The edge set $\mathcal E$ is $\{a,b,c,d\}$, and the
  vertex set is $\{1,2\}$:
     \[
A=\begin{pmatrix} 1& 2 \\ 1 & 0 \end{pmatrix} \ , \qquad \quad \quad
\Gamma_A \quad =   \quad
\xymatrix{  *+[F-:<3pt>]{1} \ar@(lu,ld)_{a}
   \ar@/^/[r]^{b}
  \ar@/^2pc/[r]^{c}
  & *+[F-:<3pt>]{2}     \ar@/^/[l]^{d} }
  \quad =
  \quad
\xymatrix{\cdot   \ar@(lu,ld)_{a}
   \ar@/^/[r]^{b}
  \ar@/^2pc/[r]^{c}
    &\cdot      \ar@/^/[l]^{d}
}
\]
Here $ ... aabdc ... $ can occur in a point of $X_A$,
but not $ ... bc ... $ .
\end{ex}
     \subsection{The continuous shift-commuting maps}\label{subsec:ctsshiftcommmaps}

\begin{nota}    For a subshift $(X,\sigma)$ and $n\in \N$,
   $\mathcal W_n(X)$ denotes the set of  $X$-words of length $n$:
\begin{align*}
\mathcal W_n(X) &=\{x_{0}\dots x_{n-1} : x\in X\}  \\
    &=
\{x_{i+n}\dots x_{i+n-1} : x\in X\},\quad \text{for every }i\in \Z \ .
\end{align*}
\end{nota}
\subsubsection{Block codes}

   Suppose $(X,\sigma)$
   and $(Y,\sigma)$ are subshifts.
   Suppose $\Phi : \mathcal W_N(X) \to \mathcal W_1(Y)$,
   and
   $j,k$ are integers, with $j+N-1= k$.
   Then for $x\in X$, we can define a bisequence $y= \phi (x)$
   by the rule $y_n = \Phi (x_{n +j} \dots x_{n+k})$, for all $n$.

   For example, with $N=4, j=-1$ and $k=2$:
   \begin{alignat*}{4}
  \cdots \ \ &x_{-1} x_0 x_1 x_2  &&\ \ \cdots\ \ && x_{n-1}x_nx_{n+1}x_{n+2}&&  \ \ \cdots \\
       &\quad   \ \   \downarrow &&  &&   \quad\quad \  \downarrow  \ &&\\
 \cdots \ \     & \quad\ \    y_0 &&\ \  \cdots\ \  && \quad \quad \  y_n && \ \ \cdots
   \end{alignat*}
   where  $y_0 = \Phi(x_{-1}x_0x_1x_2)$ and
   $y_n = \Phi(x_{n-1}x_nx_{n+1}x_{n+2})$. The point $y$ is defined by ``sliding'' the
   rule $\Phi$ along $x$. For some rules $\Phi$, the image of $\phi$ is contained
   in the subshift $(Y, \sigma)$.
\begin{de}
     The rule $\Phi$ above is called a {\it block code}.
     The map $\phi$ defined by $j$ and $\Phi$,
     is called a sliding block
     code (or a block code, or just a code).  The map $\phi$
     has range $n$ if $\Phi$ above can be chosen with
     $(j,k)= (-n,n)$ (i.e., $x_{-n}\cdots x_n$ determines
     $ (\phi x)_0$).
     \end{de}
The following result is
fundamental  for symbolic dynamics,
though it is  easy to prove \apr{chl}).

\begin{thm} (Curtis-Hedlund-Lyndon\si{Curtis-Hedlund-Lyndon Theorem})\label{thm:chlthm}\ai{Hedlund, G.A.}
\cite{Hedlund69}   Suppose $(X,\sigma)$ and $(Y,\sigma)$ are subshifts, and
$\phi : X \to Y$. The following are equivalent.
\begin{enumerate}
\item $\phi $ is continuous and $\sigma \phi = \phi \sigma$.

\item  $\phi$ is a  block code.
\end{enumerate}
\end{thm}
     The CHL Theorem tells us  the morphisms between subshifts are given by
     block codes.

We'll define some examples  by stating
the rule
     $(\phi x)_0 = \Phi (x_i \dots x_{i+N-1})$.
\begin{ex}
The shift map $\sigma$  and its powers are
sliding block codes. E.g.,
$(\sigma x)_0 = x_1$,
$(\sigma^2 x)_0 = x_2$  and $(\sigma^{-1} x)_0 = x_{-1}$ .
\end{ex}
\begin{ex}
Let $(X,\sigma)$ be the full shift on the two symbols $0,1$.
\\ Define  $\phi:(X,\sigma)\to (X,\sigma)$ by
$(\phi x)_0 = x_0 +x_1 $ (mod 2).
E.g. if $\ \  x = \dots 11\overset{\bullet}{0}\, 111\, 000\, 1 \dots $,
then $\phi (x) = \dots 01\overset{\bullet}{1}\, 001\, 001 \,
\dots $ .
\end{ex}
\subsubsection{Higher block  presentations}

Given a subshift $(X,\sigma)$ and $k\in \N$,
we define $X^{[k]}$ to be the image of $X$ under the block code
$\phi: x\mapsto y$, where for each $n$, the symbol $y_n $ is
$x_n\dots x_{n+k-1} $, the $X$-word (block) of length $k$
beginning at $x_n$.
We might put
parentheses around this word  for visual
clarity. E.g., with $k=2$,
\begin{align*}
x & = \dots x_{-1} \overd{x_0} x_1 x_2 x_3 x_4 \dots  \\
\phi (x) & = \dots (x_{-1}  x_0 )
\overd{(x_{0}  x_1 )}
(x_{1}  x_2 ) (x_{2}  x_3 )
(x_{3}  x_4)   \dots  \end{align*}
      For each $k$,  the map
      $\phi : X\to X^{[k]}$ is easily checked to be a topological conjugacy,
      $(X,\sigma) \to (X^{[k]}, \sigma)$. \\
\begin{de} The subshift      $(X^{[k]}, \sigma)$ is  the $k$-block presentation
      of $(X,\sigma)$.
\end{de}

\begin{prop}\label{conjugatetoedgesft} For a subshift $(X, \sigma)$, the following are equivalent.
  \begin{enumerate}
  \item
    $(X, \sigma)$ is a shift of finite type.
  \item
    $(X, \sigma)$  is topologically conjugate to an edge SFT.
    \end{enumerate}
\end{prop}

  We leave the (not difficult) proof of Proposition \ref{conjugatetoedgesft}
  to \apr{proofconjugatetoedgesft}).

By Proposition \ref{conjugatetoedgesft},
in order to relate the dynamical relations and properties of
arbitrary SFTs
to matrix algebra, it suffices to relate
the dynamical relations and properties of edge SFTs to their
defining matrices. So, we will be concerned from here almost
exclusively with SFTs which are edge SFTs.
For example, to classify SFTs up to topological conjugacy,
it suffices to determine when matrices $A,B$ over $\Z_+$
define edge SFTs $\sigma_A, \sigma_B $ which
are topologically conjugate\si{classification problem}.

\subsection{Powers of an edge SFT}

\begin{prop}                Let $n$ be a positive integer.
                Then the $n$th power
                system $(X_A, \sigma^n)$ is topologically conjugate to
the
     edge SFT
                $(X_{A^n}, \sigma)$
                 defined by $A^n$.
\end{prop}
The proposition  holds because in a graph with adjacency matrix $A$,
the number of paths of length $n$ from vertex
$i$ to vertex $j$ is $A^n(i,j)$ \apr{pathsandpowers}).

E.g. let $n=2$, and let $\mathcal V$ be the vertex set of
$\Gamma_A$. Let $\mathcal G$ be the graph with vertex set
$\mathcal V$ for which an edge from $i$ to $j$ is a two-edge path $(ab)$
from $i$ to $j$ in $\mathcal G$.
Since  $A^2$ is an adjacency matrix for $\mathcal G$,
 we may take for $(X_{A^2}, \sigma)$
the edge SFT on edge paths in $\mathcal G$.
We have \apr{notblock}) a topological   conjugacy
$\phi : (X_A, \sigma^2 ) \to (X_{A^2}, \sigma )$,
defined  by $(\phi x)_n = x_{2n}x_{2n+1}$ , for $n\in \Z$ :

\[
\xymatrix{
\dots x_{-2}x_{-1} \overd{x_0} x_1 x_2 x_3  \dots   \
\ar[r]^{\sigma^2 } \ar[d]_{\phi } & \
\dots x_{-2}x_{-1} x_0 x_1 \overd{x_2} x_3  \dots \ar[d]_{\phi }  \\
     \dots (x_{-2}x_{-1 }) \overd{(x_0 x_1)} (x_2 x_3)   \dots \
     \ar[r]^{\sigma } &
\ \dots (x_{-2}x_{-1 }) (x_0 x_1) \overd{(x_2 x_3)}   \dots
}
\]
The inverse system
$(X_A,\sigma^{-1})$ is conjugate to $(X_{A^T}, \sigma)$, the edge SFT defined by the
transpose of $A$ \apr{TransposeAndInverse}).

\subsection{Periodic points and nonzero\si{nonzero spectrum} spectrum}\label{perptssection}

  Given a subshift, let
  $\text{Fix}(\sigma^k)  = \{x\in X: \sigma^kx=x\}$.
  We can regard the sequence
  $(|\text{Fix} (\sigma^k)|)_{k \in \N}$
    as the {\it periodic data}\si{periodic data} of the system \apr{periodicdata}).
For an edge SFT $(X_A, \sigma_A)$, we will derive from $A$
 a complete  invariant for the periodic data\si{periodic data}.

\subsubsection{Periodic data\si{periodic data} $\ \leftrightarrow\ $ trace sequence\si{trace sequence} of A}

$x$ is a fixed point for $\sigma$
          iff          $x = ... aa\overd{a}aaa .... $
          for some edge $a$ with terminal vertex  =initial vertex.
The number of edges from vertex $i$ to vertex $i$ is $A(i,i)$.
So, in $X_A$,
\[
 |\text{Fix} (\sigma)| = \sum_i A(i,i) \ = \ \tr (A) \ .
\]
Likewise, a length $k$ path with initial vertex = terminal vertex
gives a fixed point of $\sigma^k$, and
\[
|\text{Fix}  (\sigma)^k| = \text{trace}(A^k)\  .
\]
Thus  $(|\text{Fix} (\sigma^k)|)_{k \in \N}\ =\
(\tr (A^k))_{k\in \N}$.
\subsubsection{Trace sequence\si{trace sequence} of A $\ \leftrightarrow\ $
  det(I-tA)}

There is a standard equation \apr{zetaeq})
\[
\frac 1{\det(I-tA)} = \exp \sum_{n=1}^{\infty} \frac 1n \tr (A^n) t^n \ .
\]
From this,  one sees the trace sequence\si{trace sequence} and $\det(I-tA)$ determine
each other \apr{proofwithzeta}). (This mutual determination holds for a matrix over
any torsion-free commutative ring \apr{NewtonIdentities})).

\subsubsection{det(I-tA)  $\ \leftrightarrow\ $ nonzero\si{nonzero spectrum} spectrum of A}

\begin{de} If a matrix $A$
has characteristic polynomial
$t^k \prod_{i=1}^m (t-\lambda_i)$, with the $\lambda_i$ nonzero\si{nonzero spectrum},
then the {\it nonzero\si{nonzero spectrum} spectrum} of $A$
is  $(\lambda_1, \dots , \lambda_m)$.
Here -- by abuse of notation \apr{abuse})) --
the $m$-tuple is used as notation for a multiset:
 the multiplicity     of entries of $(\lambda_1, \dots , \lambda_m)$
          matters, but not their  order.
For example,
          $(2,1,1)$ and $(1,2,1)$ denote the same nonzero\si{nonzero spectrum} spectrum,
          but $(2,1)$ is different.
\end{de}
          If $A$ has nonzero\si{nonzero spectrum} spectrum
$\Lambda = (\lambda_1, \dots , \lambda_m)$,
          then
          \[ \det (I-tA) = \prod_{i=1}^m (1-\lambda_i t) \ . \]
 For example,
\begin{alignat*}{2}
A &= \begin{pmatrix} 3&0&0&0 \\ 0&3&0&0 \\ 0&0&5&0\\ 0&0&0&0
\end{pmatrix}\ , \qquad &
I-t A&= \begin{pmatrix} 1-3t&0&0&0 \\ 0&1-3t&0&0 \\ 0&0&1-5t&0\\ 0&0&0&1
\end{pmatrix} \\
&&& \\
\Lambda &= (3,3,5) \ , &
\det (I-tA) &= (1-3t)^2(1-5t)\ .
\end{alignat*}
The nonzero\si{nonzero spectrum} spectrum and the polynomial $\det(I-tA)$ determine each other.
\subsection{Classification of SFTs}

\begin{prob}[Classification Problem\si{classification problem}]\label{prob:classificationprob} Given
square matrices $A,B$  over $\Z_+$,
determine whether they present SFTs which are topologically
          conjugate.
          \end{prob}
There are trivial ways to produce infinitely many
distinct matrices which define
the same SFT. E.g.,
\[
\begin{pmatrix} 2  \end{pmatrix}\ ,
\begin{pmatrix} 2&0\\0&0 \end{pmatrix}\ ,
\begin{pmatrix} 2&0\\1&0 \end{pmatrix}\ ,
\begin{pmatrix} 2&1\\0&0 \end{pmatrix}\ ,
\begin{pmatrix} 2&0&0\\1&0&0\\1&0&0 \end{pmatrix}\ ,
\begin{pmatrix} 2&1&1\\0&0&1\\0&0&0 \end{pmatrix}\ , \ \dots
\]
Every SFT $(X_A, \sigma)$ equals one which is defined by a matrix
which is {\it nondegenerate\si{nondegenerate}} (has no zero row and no zero column)
\apr{nondegenerate}).
We can avoid the trivial problem by considering only nondegenerate\si{nondegenerate} matrices.
Still,
in the nontrivial
case (the case that $X_A$ contains infinitely many points),
there are nondegenerate\si{nondegenerate} matrices of unbounded size
which define  SFTs topologically
conjugate to $(X_A, \sigma)$ \apr{higherblockedgesfts}).

\subsection{Strong shift equivalence\si{strong shift equivalence} of matrices, classification
of SFTs}\label{subsec:sseandclassification}

\begin{de} A {\it semiring}\si{semiring} is a set with  operations
        addition and multiplication
      satisfying all the ring axioms, except that an element is not required
      to have an additive inverse. In these lectures, the  semiring
is always assumed to       contain a multiplicative identity, 1.
\end{de}
      Below, $\mathcal S$ is a subset of a semiring \apr{boolean}) containing 0 and 1.
For $\mathcal S$ a subset of $\R$,
      $\mathcal S_+ $ denotes $\mathcal S \cap  \{x\in \R: x\geq 0\}$.
       We are especially interested in
 $\mathcal S =
     \Z, \Z_+ , \R , \R_+$.

     Let $A$ and $B$
     be square matrices over $\mathcal S$ (not necessarily
     of the same size).

\begin{de} $A$ and $B$ are {\it elementary strong shift equivalent}
     over $\mathcal S$ (ESSE-$\mathcal S$) if there exist matrices
     $R,S$ over $\mathcal S$ such that
     $A=RS$ and $B=SR$.
\end{de}
     Note, if a matrix $R$ is $m\times n$,
     and $S$ is a matrix such that $RS$ and $SR$ are well defined,
     then $S$ must be $n\times m$, and the matrices $RS$ and $SR$
     must be square.
 \begin{de}  $A$ and $B$ are  {\it strong shift equivalent}
     over $\mathcal S$ (SSE-$\mathcal S$) if there are matrices
     $A=A_0, A_1, \dots , A_{\ell} =B$
     over $\mathcal S$ such that $A_i$ and $A_{i+1}$ are ESSE-$\mathcal S$,
     $0\leq i < \ell $.

The number $\ell $ above is called the lag\si{lag} of the strong shift equivalence\si{strong shift equivalence}.
\end{de}
The relation ESSE-$\mathcal S$ is reflexive and symmetric. Easy
examples \apr{smallestlag}) show ESSE-$\mathcal S$ is
      not transitive.
     SSE-$\mathcal S$, the transitive closure of ESSE-$\mathcal S$,
is an     equivalence relation.
Williams\ai{Williams, R.F.} introduction of strong shift equivalence\si{strong shift equivalence} in
\cite{Williams73} -- the foundation for all later work on
the classification of shifts of finite type -- is explained
by the following theorem.

 \begin{thm}[Williams\ai{Williams, R.F.} 1973]\label{rfwtheorem} \apr{liberties}) Suppose
 $A$ and $B$ are square matrices over $\mathcal \Z_+$. The following
are equivalent.
\begin{enumerate}
\item  $A$ and $B$ are SSE-$\mathcal \Z_+$.

\item The SFTs defined by $A$ and $B$ are topologically conjugate.
\end{enumerate}
\end{thm}
\begin{proof}
The  difficult implication $(2) \implies (1)$ follows from
the Decomposition Theorem\si{Decomposition Theorem} \apr{decomp}).
We will prove the easy direction,
$(1) \implies (2)$.

It suffices to consider
 an ESSE over $\Z_+$,
   $A=RS, B=SR$. Define a square matrix $M$ with block form
   $\begin{pmatrix} 0&R \\ S&0\end{pmatrix}$,
     and  edge SFT $(X_M, \sigma)$.
Then $M^2 =
     \begin{pmatrix} RS&0 \\ 0&SR\end{pmatrix}
=    \begin{pmatrix} A&0 \\ 0&B\end{pmatrix}
  $.
  The system $(X_M, \sigma^2)$ is a disjoint union of two systems,
  $(X_1, \sigma^2|{X_1})$ and   $(X_2, \sigma^2|{X_2})$.
The shift map
  $\sigma: X_1 \to X_2$ gives a
  topological conjugacy between these subsystems.

For all $i,j$, we have $A(i,j)= \sum_k R(i,k)S(k,j)$.
Therefore, we may choose a bijection $\alpha: a\mapsto rs$ from the set of
$\Gamma_A$ edges to the set of $R,S $ paths in $\Gamma_M$
(an $R,S$ path is an $R$ edge followed
by an $S$ edge) which respects initial and terminal vertex.
Similarly we choose a bijection
$\beta: b\mapsto sr$
from $\Gamma_B$ edges to $S,R$ paths in $\Gamma_M$.
  We define a  conjugacy
$\phi_{\alpha}: (X_A, \sigma ) \to (X_1, \sigma^2|{X_1})$ ,
$\
\phi_{\alpha} :
  \dots x_{-1}x_0x_1 \dots  \mapsto
  \dots (r_{-1}s_{-1})(r_0s_0)(r_1s_1) \dots \ ,
 $
  by replacing
  each $x_n$ with $\alpha (x_n)$.
    We define a  conjugacy
  $\phi_{\beta}: (X_B, \sigma ) \to (X_2, \sigma^2|{X_2})$ in
  the same way.

We now have a conjugacy $c(R,S):X_A \to X_B$
as the composition, $c(R,S) = \phi_{\alpha} \sigma \phi_{\beta}^{-1}$,
\begin{align*}
c(R,S): \   \dots x_{-1}x_0x_1 \dots & \mapsto
  \dots (r_{-1}s_{-1})(r_0s_0)(r_1s_1) \dots \\
  & \mapsto
  \dots (s_{-1}r_0)(s_0r_1)(s_1r_2) \dots \mapsto
   \dots y_{-1}y_0y_1 \dots \ .
\end{align*}
\end{proof}

The technical statements of the next remark are not
needed at all before
Sections \ref{sectionAut} and \ref{sectionWagoner}\ai{Wagoner, J.B.}.

%
%

\begin{rem}\label{rem:crsconj}
Let $(R,S)$ be an ESSE-$\Z_+$, with $A=RS$ and $B=SR$.
Let $c(R,S)$  be a topological conjugacy from
$(X_A,\sigma)$ to $(X_B,\sigma)$ defined as in the proof above.
The conjugacy $c(R,S)$ is uniquely determined by $(R,S)$ when
all entries of $A$ and $B$ are in $\{0, 1\}$ (then, the
bijections $\alpha, \beta$ are unique).
But in general, the conjugacy depends on
the choice of those bijections.
With appropriate choice of
those bijections, we have the following:
\begin{enumerate}
\item
$c(R,S) c(S,R) = \sigma_A$, the shift map on $X_A$ .
\item
$(c(R,S))^{-1}=  c(S,R) \sigma_A^{-1}
  =   \sigma_B^{-1}c(S,R)  $.
\item
$c(I,A) = \text{Id}$,   and $c(A,I) = \sigma_A$.
\end{enumerate}
Also: with  $c(R,S)$, $\ x_0x_1$ determines $y_0$;
with $(c(R,S))^{-1}$, $\ y_{-1}y_0$ determines $x_0$.
\end{rem}

\subsection{Shift equivalence\si{shift equivalence}}

                Despite the seeming simplicity of its definition,
                SSE over $\Z_+$ is a very difficult relation to
                fully understand.
                Consequently, Williams\ai{Williams, R.F.} introduced shift equivalence\si{shift equivalence}.
\begin{de}
     Let $A,B$ be square matrices over a semiring $\mathcal S$.
     Then $A,B$ are shift equivalent over $\mathcal S$ (SE-$\mathcal S$)
     if there exist matrices $R,S$ over $\mathcal S$ and a positive
     integer $\ell$ such that the following hold:
     \[
     A^{\ell}=RS\ , \quad B^{\ell}= SR\ , \quad
     AR=RB\ ,\quad SA=BS\  .
     \]
                Here, $(R,S)$ is a shift equivalence\si{shift equivalence}
of lag\si{lag}  $\ell $
                from $A$ to $B$.
\end{de}
The next proposition is an easy exercise \apr{easyfacts}).
\begin{prop}  Let $\mathcal S$ be a  semiring.
\begin{enumerate}
\item SE over $\mathcal S$ is an equivalence relation.

\item  SSE over $\mathcal S$ implies SE over $\mathcal S$.
\end{enumerate}
\end{prop}
\subsection{Williams' Shift Equivalence\si{shift equivalence} Conjecture\si{Williams' Conjecture}}

\begin{conj} (Williams\ai{Williams, R.F.}, 1974)
  \cite{Williams73}\label{conj:williams}
Suppose $A,B$ are two square matrices
                which are SE-$\Z_+$.
                Then they are SSE-$\Z_+$.
\end{conj}

                Despite the seeming complexity of its definition,
         shift equivalence\si{shift equivalence} is  much easier to understand  than
                strong shift equivalence\si{strong shift equivalence}, as we'll see.  A positive
                solution to Williams\si{Williams' Conjecture}\ai{Williams, R.F.}' Conjecture would have been a very
                satisfactory solution to the classification problem\si{classification problem}
                for SFTs.
                                Alas ... there are  counterexamples to the
                                conjecture,
                due to Kim\ai{Kim, K.H.} and Roush\ai{Roush, F.W.} (building on work of Wagoner\ai{Wagoner, J.B.},
                and  Kim-Roush\ai{Roush, F.W.}-Wagoner\ai{Wagoner, J.B.}).
                The first Kim-Roush\ai{Roush, F.W.} counterexample was in 1992.
                We recall now a definition fundamental for the theory of nonnegative
matrices (as we will review in Lecture 4).

\begin{de}\label{primitiveDefinition}
A primitive\si{primitive matrix} matrix is a square matrix such that every entry
is a nonnegative real number and  for
some positive integer $k$, every entry of $A^k$ is positive.
\end{de}
                By far the most important
                case of Williams\si{Williams' Conjecture}\ai{Williams, R.F.}' Conjecture is the case that the
                matrices $A,B$ are primitive.
An edge  SFT defined from a nondegenerate\si{nondegenerate} matrix
                $A$ is {\it mixing}\begin{footnote}{See
    \cite{LindMarcus2021}\ai{Lind, Douglas}\ai{Marcus, Brian}
    for the definition of the
                dynamical property  ``mixing'', which we do not
                need.}\end{footnote}
 if and only if $A$ is primitive.
                The mixing SFTs play a role among SFTs very much analogous
                to the role played by primitive matrices in the
                theory of nonnegative matrices.

The Kim-Roush\ai{Roush, F.W.}\ai{Kim, K.H.} counterexample for primitive
                matrices came in 1999.

                Over twenty years later, we have no
                new theorem or counterexample for primitive matrices
                over $\Z_+$. The Kim-Roush\ai{Roush, F.W.}\ai{Kim, K.H.} counterexamples require quite special constructions
(reviewed in Section \ref{sectionWagoner}\ai{Wagoner, J.B.}).
                The proof method can  work only in special SE-$\Z$ classes,
and can never
show that there is
                an infinitely family of primitive  matrices which are
                SE-$\Z_+$ but are pairwise not SSE-$\Z_+$
                (see Sec. \ref{WagonerSubSecRemarks}).\ai{Wagoner, J.B.}

\subsubsection{The gap between SE-$\Z_+$ and
                     SSE-$\Z_+$?}

How big is the gap between
                   SE-$\Z_+$ and
                     SSE-$\Z_+$?                    We really don't know.

                   Suppose $A$ is ANY  square
 matrix over $\Z_+$ such that $A$  is primitive (for some $n$, every entry
 of $A^n$ is positive), and  $A\neq (1)$.
 (The case $A=(1)$ is trivial.)
 As we approach a half century following
 Williams\si{Williams' Conjecture}\ai{Williams, R.F.}' conjecture,
                 we cannot verify or rule out either
                 of the following statements.
\begin{enumerate}
\item There is an
                algorithm which takes as input any square matrix
                $B$ over $\Z_+$ and decides whether $A$ and $B$ are SSE-$\Z_+$.

\item There are infinitely many matrices which are
                SE-$\Z_+$ to $A$ and  which are pairwise
                not SSE-$\Z_+$.
\end{enumerate}

Regarding the first item above:
we
do not know upper bounds on the lag\si{lag} of a possible SSE or the sizes of
                the matrices in its chain of ESSEs.
(See
\apr{smallestlag}) - \apr{badArithmeticLag})
for more on  lag\si{lag} issues.)
Also, for example,  the ``$1 \times 1$ case'', in which $A$ in Conjecture \ref{conj:williams} is assumed
  to be $1\times 1$, is completely open.
This is called the ``Little Shift Equivalence Conjecture''
in \cite[Problem 3]{Bopen}.)
We will see \apr{1by1}) that a
square matrix over $\Z_+$ is SE-$\Z_+$ to $(k)$ $\iff$
its nonzero\si{nonzero spectrum} spectrum is $(k)$.
 But,
for every  positive integer $k>1$,  we do not know whether
a matrix SE over $\Z_+$ to $(k)$ must be
SSE over $\Z_+$ to $(k)$ \apr{ashley}).
Remarkably, even for two
$2\times 2$ matrices over $\Z_+$, we do not know whether
SE-$\Z_+$ implies SSE-$\Z_+$ (although, here there are significant
partial results, e.g. \cite{baker1983,baker1987,
CuntzKriegerDicyclic,Williams1992}).\ai{Williams, R.F.}

Nevertheless ... perhaps the situation is not hopeless.
\begin{enumerate}
\item If $A$ is a  matrix over $\R_+$ with $\det(I-tA) =1-\lambda t$,
then $A$ is SSE over $\R_+$ to $(\lambda )$.
(Over $\R_+$,
the ``$1 \times 1$ case'' is solved!)

Despite limited progress, I think the  proof framework for
this result of Kim\ai{Kim, K.H.} and Roush\ai{Roush, F.W.} is promising
for  proving SE-$\R_+$ implies
SSE-$\R_+$ for positive matrices \apr{RealSSEandBKR}).
\item In recent years we have (at last) gained a much better
(not complete) understanding of
strong shift equivalence\si{strong shift equivalence} over a ring, as discussed in Lecture 6.
This gives  more motivation for investigation, and
new ideas to explore.
\end{enumerate}

\subsection{Appendix 1} \label{a1}
This
subsection contains various remarks, proofs and comments referenced
in earlier parts of Section \ref{sec:basics}.

\begin{rem}\label{introtosfts} By way of Markov partitions,
  shifts of finite type are a fundamental tool in the theory of
  smooth dynamical systems. SFTs have application to
  coding theory, general topological dynamics, ergodic theory,
  $C^*$-algebras, cellular automata and   geometric group theory.
  Our focused introduction to the stable algebra related
  to SFTs will avoid all of this.
    For a comprehensive
    introduction to the theory of shifts of finite type and some
    related topics,
    including a supplement reviewing recent developments,
see
 the    2021 edition
\cite{LindMarcus2021}\ai{Lind, Douglas}\ai{Marcus, Brian} of the
 classic 1995 text   \cite{LindMarcus1995} of Lind and
Marcus\ai{Lind, Douglas}\ai{Marcus, Brian}.
This     crystal-clear  book is intended to be widely accessible;
a math graduate student can  easily  read it without guidance.
The 2021 edition includes a long supplement reviewing recent developments.
The lucid 1998 book
    \cite{Kitchens1998} of Kitchens\ai{Kitchens, Bruce P.} is also
    valuable, providing
    additional depth on various topics and
  developing the basic theory of
    countable state Markov shifts (a topic not covered by
    Lind and Marcus).
\end{rem}

\begin{rem} \label{zerodimetc}
Let the finite set  $\mathcal A$ have the discrete topology.
Then compactness of $\mathcal A^{\Z}$ follows from
a  diagonal  argument  from the chosen metric, or
from Tychonoff's Theorem (the product topology is the same
as the topology coming from the chosen metric).

Suppose $X$ is a closed nonempty subset of $\mathcal A^{\Z}$.
     A ``cylinder set'' is a set $C$ in $X$ of the following form:
there is a point      $x\in X$, and $i\leq j$ in $\Z$,  such that
     $C = \{y \in X: y_n = x_n \text{ if } i\leq n \leq j\}$ .
The cylinder sets form a basis for the topology on $X$.

The cylinder sets are closed open. A  subset of $X$ is closed open if and only if
     it is the union of finitely many cyinders.
By definition, a metric space is zero dimensional if there is a base for
the topology consisting of closed open sets.
Therefore $X$ is zero dimensional.
\end{rem}

\begin{rem} \label{edgesftnote}
To be careful, we'll be a little pedantic.

Two different but isomorphic graphs define different but isomorphic
SFTs. The topological conjugacy of SFTs in this case  is
rather trivial. If the graph isomorphism gives a map on edges
$e\mapsto \overline e$, then the topological conjugacy $\phi$
is defined by $(\phi x)_n = \overline{x_n}$, for all $n$.

In the other direction, given
just the matrix $A$, a graph $\mathcal G$ with adjacency matrix $A$ is only defined
up to graph isomorphism.  If $A$ is $n\times n$,
then there is an ordering of the  vertices, $\ \nu_1, \nu_2, \dots ,\nu_n$ ,
such that $A(i,j)$ is the number of edges from $\nu_i$ to $ \nu_j$.
For simplicity, we often just regard the vertex set as
$\{ 1,2, \dots , n \}$, with $\nu_i =i$.
\end{rem}

\begin{thm} [Curtis-Hedlund-Lyndon\si{Curtis-Hedlund-Lyndon Theorem}] \label{chl}
     Suppose $(X,\sigma)$ and $(Y,\sigma)$ are subshifts, and
$\phi : X \to Y$. TFAE.

     (1) $\phi $ is continuous and shift-commuting.

     (2) There are integers $j,k$ with $j\leq k$,
such that for $N=k-j+1$  there is
 a function
     $\Phi : \mathcal W_N(X) \to \mathcal W_1(Y)$,  such that
    for all $n$ in $\Z$ and $x$ in $X$,
     $(\phi x)_n = \Phi (x_{n +j} \dots x_{n+k})$ .
\end{thm}
\begin{proof}
(1) $\implies $ (2) Suppose $\phi$ is continuous, hence
uniformly continuous, on $X$. There is an $\epsilon >0$ such that
for $y, y'$ in $Y$, $y_0 \neq (y')_0 \implies \dist (y_0, (y')_0) > \epsilon $.
By the uniform continuity, there is $m\in \N $ such that for $x,w$ in $X$,
\[
x_{-m}\dots x_m = w_{-m}\dots w_m
   \implies
(\phi x)_0 = (\phi w)_0\ .
\]
 This gives a rule
$\Phi : \mathcal W_{2m+1} (X) \to \mathcal W_1 (Y) $
such that for all $x$ in  $X$,
$(\phi x)_0 = \Phi (x_{-m} \dots x_m)$. Because $\phi$ is shift commuting,
we then get for all $n$ that
\[
\Phi (x_{n-m} \dots x_{n+m})
= \Phi ((\sigma^n x)_{-m} \dots (\sigma^n x)_m)
= (\phi (\sigma^n x) )_0
 = ( \sigma^n(\phi x) )_0
= (\phi x)_n \ .
\]

We leave the proof of (2) $\implies $ (1) as an exercise.
\end{proof}
There are other, equivalent  ways to state the CHL Theorem.
(I didn't copy the original statement.)

\begin{prop}  \label{pathsandpowers}
Let a graph have adjacency matrix $A$.
Then the number of paths of length $n$ from vertex $i$ to
vertex $j$ is $A^n(i,j)$.
\end{prop}
\begin{proof}
A length 2 path from $i$ to $j$ is, for some vertex $k$,
 an edge from $i$ to $k$ followed by an edge from $k$ to $j$.
The number of such paths is $\sum_k A(i,k) A(k,j) = A^2 (i,j)$.
The claim for paths of length $n$ follows by induction,
considering paths of length $n-1$ followed by path of length 1.
\end{proof}

\begin{rem} \label{TransposeAndInverse}
Suppose $A$ is a square matrix over $\Z_+$, with transpose $A^T$.
From a graph $\mathcal G$ with adjacency matrix $A$,
let $\mathcal G^{\text{reversed}} $ be the graph with the same vertex set as $G$,
and edges with the same names but with reversed direction
(an edge $e$ from $i$ to $j$ in $\mathcal G$ becomes an edge
$e$ from $j$ to $i$ in  $\mathcal G^{\text{reversed}} $.
 Then
 $A^T$ is an adjacency matrix for $\mathcal G^{\text{reversed}} $.

Now, there is a topological conjugacy
 $\phi: (X_A, \sigma^{-1})\to (X_{A^T} , \sigma)$, defined by the
rule $(\phi x)_n = x_{-n}$, for $n\in \Z$.
\end{rem}

\begin{rem} \label{notblock}
The topological   conjugacy
$\phi : (X_A, \sigma^2 ) \to (X_{A^2}, \sigma )$
is not a block code.
 This does not contradict the
CHL Theorem, because $(X_A, \sigma^2)$ is not a subshift.
\end{rem}

\begin{rem}\label{periodicdata}
Formally, the ``periodic data\si{periodic data}'' for a system $(X,S)$ is
the isomorphism class of the system $(\text{Per}(S), S)$, with
the periodic points,  $\text{Per}(S)$, given the discrete topology
(i.e., ignore topology). (Here ``system'' relaxes our terminology
in these lectures that the domain must be compact.)

A complete invariant for the periodic data\si{periodic data} is one such that
two systems agree on the invariant if and only if they have
the same periodic data\si{periodic data}.

Now, one  complete invariant of the periodic data\si{periodic data} of a system
is simply the function which
assigns to $n$ the  cardinality of the set
of  points of least period $n$.  (A point has  least period $n$ if its
orbit is finite with cardinality $n$.)
For a subshift $(X,\sigma)$,
there is a finite number $q_n$ of points of least period $n$,
and the sequence $(q_n)$ is a complete invariant of the periodic data\si{periodic data}.
Let $\tau_n = |\text{Fix}(\sigma^n)|$.
The sequence $(q_n)$ determines the sequence
$(\tau_n)_{n=1}^{\infty}$. For our systems, each $\tau_n$ is a nonnegative
integer, and in this case the converse holds: the
sequence $(\tau_n)$ determines the sequence $(q_n)$. E.g.,
$q_1=\tau_1,\  q_2 = \tau_2 - \tau_1, \ \dots \ ,  q_6 = \tau_6
-\tau_3 -\tau_2 + \tau_1 , \ \dots $. (The formal device for producing
a systematic formula for this inclusion-exclusion pattern is
Mobius inversion.) So, ``we may regard'' $(\tau_n)$ as the periodic
data in the sense that it is a complete invariant for the periodic data\si{periodic data}.
\end{rem}

\begin{prop} \label{zetaeq}
Suppose $A$ is a matrix with entries in $\C$. Then
\begin{equation} \label{zetaseriesequation}
\frac 1{\det(I-tA)} \ = \ \exp \sum_{n=1}^{\infty} \frac 1n \tr (A^n) t^n \ .
\end{equation}
\end{prop}
\begin{proof}
Recall, $- \log (1-x) = x + \frac{x^2}2+ \frac{x^3}3 + \cdots $ .
Let   $(\lambda_1, \dots , \lambda_n)$
be the nonzero\si{nonzero spectrum} spectrum of $A$.
Then
\begin{align*}
& \ \exp \Big(\sum_{n=1}^{\infty} \frac 1n \tr (A^n) t^n\Big)
\ = \ \exp \Big(\sum_{n=1}^{\infty} \frac 1n \big(\sum_i \lambda_i^n\big) t^n \Big) \\
 =& \  \exp \Big(  \sum_i \big(\sum_{n=1}^{\infty} \frac 1n  (\lambda_it)^n\big) \Big)
\  = \ \prod_i \exp \Big(\sum_{n=1}^{\infty} \frac{(\lambda_i t)^n}n \Big) \\
 =& \ \prod_i \exp \big(- \log (1-\lambda_i t)\big)
\ = \ \prod_i \frac 1{(1- \lambda_i t)}
\ = \ \frac 1{\det (I-tA)} \ .
\end{align*}
\end{proof}
\noindent (The last proposition remains true as an equation in formal power series
if $\C$ is replaced by a torsion-free commutative
ring $\mathcal R$. In this case,
 $\N$ is a multiplicative subset of  $\mathcal R$ containing no zero divisor,
 and all the power series coefficients make sense in the
localization $\mathcal R[\N^{-1}]$.)

\begin{rem} (The zeta function\si{zeta function}) \label{zetafunction}
Suppose $(X,S)$ is a dynamical system such that for all $n$ in $\N$,
$|\text{Fix}(S^n)| < \infty $. Then the (Artin-Mazur) zeta function\si{zeta function} of
the system is defined to be
\[
\zeta (t) = \exp \Big(\sum_{n=1}^{\infty} \frac 1n |\text{Fix}(S^n)|
t^n \Big) \ .
\]
This is defined at least as a formal power series; it's defined as
an anaytic function inside the radius of convergence. The zeta
function\si{zeta function} (where it is defined) is the premier complete invariant of
the periodic data\si{periodic data}. For an edge SFT defined from a matrix $A$,
we see $\zeta (t) = 1/\det(I-tA)$.
\end{rem}

\begin{coro}\label{proofwithzeta}
Suppose $A$ is a square matrix over $\C$.
Then $\det (I-tA)$ and the sequence $(\tr (A^n))$ determine each other.
\end{coro}
\begin{proof}
The nontrivial implication, that the trace sequence\si{trace sequence} determines
 $\det (I-tA)$, follows from
 the  proposition.
The proposition also is easily used to
prove the reverse implication;  but we may also simply notice
that
 $\det(I-tA)$ determines the nonzero\si{nonzero spectrum} spectrum
$(\lambda_1, \dots , \lambda_n)$  of $A$, which determines
$\tr (A^k) = \sum_i (\lambda_i)^k$.
\end{proof}

\begin{rem} \label{NewtonIdentities}
For any square matrix $A$  over any commutative ring, the polynomial $\det(I-tA)$
determines the trace sequence\si{trace sequence} $(\tr (A^k))_{k=1}^{\infty}$; if the ring is
torsion-free, then conversely $(\tr (A^k))_{k=1}^{\infty}$ must determine
$\det(I-tA)$.
 To see this, let us write
$\det (I-tA)$ as $1-f(t) = 1 -f_1t -f_2t^2 \dots -f_Nt^N$,
and let $\tau_k$ denote $\tr (A^k)$.
Then the claimed determinations are easily proved by induction
from Newton's identities,\begin{footnote}{As $\det (I-tA)$ is the reversed
characteristic polynomial, Newton's identities
can alternately  be (and usually are)
stated in terms of coefficients of the characteristic polynomial.}
\end{footnote} valid over any commutative ring:
\begin{alignat*}{2}
\tau_k \ & = \ kf_k +\sum_{i=1}^{k-1}f_i\tau_{k-i}\ , \quad && \text{if } 1\leq k \leq N\ , \\
& =\  \sum_{i=1}^{N}f_i\tau_{k-i}\ , \quad && \text{if }  k > N\ .
\end{alignat*}
To see the torsion-free assumption is not extraneous,
let $\mathcal R$ be the ring $\Z_2 \times \Z_2$, and consider the matrices
\[
A = \begin{pmatrix} (0,1)
\end{pmatrix} \ , \qquad
B
= \begin{pmatrix} (0,0) & (1,1) \\ (1,0) & (0,1)
\end{pmatrix}
\]
Here,  $\det(I-tA) = 1 -t (0,1) \neq 1 -t (0,1) -t^2(1,0) = \det(I-tB) $,
but $\tr (A^n) = \tr (B^n) = (0,1)$ for every positive integer $n$.

One of the ways to prove Newton's identities is to take the derivative of the
log of both sides of \eqref{zetaseriesequation}, and equate coefficients in the resulting
equation of power series. This makes sense at the level of formal
power series when the ring is torsion free, in particular for a polynomial ring
$\Z [\{x_{ij}\}]$, where $\{ x_{ij} : 1\leq i,j \leq N\}$ is a set of $N^2$
commuting variables. Then, given $A$ over any commutative ring $\mathcal R$,
using the ring homomorphism $\Z [\{x_{ij}\}] \to \mathcal R$ induced by
$x_{ij} \mapsto A(i,j)$, from the Newton identities over
$\Z [\{x_{ij}\}]$ we obtain the Newton identities for $A$.
\end{rem}

\begin{rem} \label{nondegenerate}
A matrix is
{\it degenerate\si{degenerate}} if it has a zero row or a zero column.
 The {\it nondegenerate\si{nondegenerate} core} of a square matrix is the largest principal
submatrix $C$ which is nondegenerate.
If row $i$ or column $i$ of  $A$ is zero, then remove
row $i$ and column $i$. Continue until a nondegenerate\si{nondegenerate} matrix $C$ is reached.
This matrix is the nondegenerate\si{nondegenerate} core of $A$.

When the matrices have all entries in  $\Z_+$,
 $X_C=X_A$, because  if an edge occurs as $x_n$ for some  point of $X_A$,
then  the edge must
be followed and preceded by arbitrarily long paths
 in $\Gamma_A$.
 \end{rem}

\begin{rem} \label{proofconjugatetoedgesft}
   Proposition
   \ref{conjugatetoedgesft} states that every SFT is topologically
   conjugate to an edge SFT.
  Because a subshift is topologically conjugate to each of its
  higher block presentations, in order to prove
     Proposition
   \ref{conjugatetoedgesft}  it suffices to prove the next result.
We use $\mathcal W_k(X)$ to denote the set of $X$-words of length $k$.

  \begin{prop} \label{HigherBlockPresOfSFTisEdgeSFT}
    Suppose
 $(X, \sigma)$ is a subshift of finite type
  on alphabet $\mathcal A$. Let
  $\mathcal F$ be a finite set of words
  such that $X$ equals the set of points $x$ on alphabet $\mathcal A$
  such that no word
  $\mathcal F$ occurs in $x$. Suppose $N>1$ and
  $N \geq \max\{ \text{length of }W: W \in \mathcal F \}$.

  Then the $N$-block presentation of $(X,\sigma)$ is an edge SFT,
  with edge set $\mathcal W_N$.

  \begin{proof}
    We define a directed graph $G$. The vertex set is $\mathcal W_{N-1}(X)$.
    The edge set is $\mathcal W_{N}(X)$.
An edge  $W_1\dots W_N$ is an edge
  from vertex $W_1\dots W_{N-1}$ to vertex $W_2\dots W_{N}$.
 Clearly,
  the edge SFT $(X_A, \sigma)$
  defined from $G$ contains the $N$-block presentation
  $(X^{[N]}, \sigma )$.   Conversely, the condition on $\mathcal F$
  implies that every point of $(X_A, \sigma)$ is  a point
  in $X^{[N]}$.
  \end{proof}
\end{prop}

\end{rem}

\begin{rem} \label{higherblockedgesfts}
  For an edge SFT $X_A$, let $A^{[k]}$ be a transition
  matrix for its $k$-block presentation. Note, $A^{[k]}$ is not degenerate.
For example, let
$     A=A^{[1]}= (2)$ ,
with edge set $\mathcal E_1=\{a,b\}$.
The vertex sets $\mathcal W_2$ and $ \mathcal W_3$ for
2 and 3 block presentations are
$\{ a,b\}$ and
$\{aa,ab,ba,bb\}$. With the lexicographic orderings on these
sets (ordered as written),
we get the corresponding
adjacency matrices
\[
     A^{[2]} = \begin{pmatrix} 1&1\\1&1\end{pmatrix} \ , \qquad
 A^{[3]} = \begin{pmatrix} 1&1&0&0 \\0&0&1&1
\\ 1&1&0&0
         \\0&0&1&1
       \end{pmatrix}\  .
 \]
 In general, if  $X_A$ is infinite, then the size of $A^{[k]}$ must go
 to infinity with $k$.
\end{rem}

\begin{rem} \label{boolean}
The use of SSE over semirings goes beyond the study of
SSE over the positive set of an ordered ring.
 SSE over the Boolean semiring $\{ 0,1 \}$, in which $1+1=1$,
ends up being  quite relevant to some constructions over
$\R_+$ \cite{BKR2013}, and
 to relating topological conjugacy and flow equivalence\si{flow equivalence} of
SFTs \cite{B02posk}. The Boolean semiring cannot be embedded in a ring,
as  $1+1=1$ would then force $1=0$.
\end{rem}


\begin{rem} \label{liberties}
  For simplicity, I  take some liberties with the statement of
  the theorem. ``Edge SFTs'' don't appear in Williams' paper;
  he used a more abstract approach to
 associate SFTs to   matrices over $\Z_+$.

  Williams\ai{Williams, R.F.} 1973 paper \cite{Williams73}
   contained a ``proof'' (erroneous) of
  his conjecture. The 1974 Conjecture appeared in the erratum.
  One of the most important papers in symbolic dynamics also included
  perhaps its most famous mistake.
\end{rem}

\begin{rem} \label{decomp}
We say a little about the Decomposition Theorem\si{Decomposition Theorem}, even though we won't have
space to explain it well, because it is a very important feature of SSE.
Lind and Marcus give a nice presentation of the Decomposition Theorem\si{Decomposition Theorem}
\cite{LindMarcus2021}.\ai{Lind, Douglas}\ai{Marcus, Brian}

The Decomposition Theorem\si{Decomposition Theorem} tells us that when there is a conjugacy of
edge SFTs
 $\phi: (X_A, \sigma) \to (X_B, \sigma)$, there is another matrix $C$,
an SSE-$\Z_+$ from $C$ to $A$ given by a string of column amalgamations,
and an SSE -$\Z_+$ from $C$ to $B$
given by a string of row amalgamations, such that the associated  conjugacies
 $\alpha: (X_C, \sigma)\to (X_A, \sigma)$ and
$\beta: (X_C, \sigma)\to (X_A, \sigma)$
give $\phi =  \alpha^{-1}\beta$.

For $x$ in $X_C$: $(\alpha x)_0$ and $(\beta x)_0$ depend only on $x_0$.

A column amalgamation $C\to D$ is an ESSE
$C=RS $, $D=SR$,  such that   $S$ is a zero-one matrix
with each column containing exactly one nonzero entry.
For example,
\begin{align*}
C = \begin{pmatrix}1& 1&  5 \\ 2&2&3 \\1&1&2 \end{pmatrix} =
& \begin{pmatrix} 1&  5 \\ 2&3 \\1&2 \end{pmatrix}
\begin{pmatrix} 1& 1& 0 \\ 0&0& 1  \end{pmatrix} =RS \ , \\
D =\begin{pmatrix}      3&8 \\ 1&2 \end{pmatrix} =
& \begin{pmatrix} 1& 1& 0 \\ 0&0& 1  \end{pmatrix}
\begin{pmatrix} 1&  5 \\ 2&3 \\1&2 \end{pmatrix}=SR \ .
\end{align*}
Row amalgamations are correspondingly given by amalgamating rows
rather than columns.

The Decomposition Theorem\si{Decomposition Theorem}, or a relative,  is a tool for the
characterization of nonzero\si{nonzero spectrum} spectra of primitive real matrices
\cite{BH91}; for
Parry's\ai{Parry, William}
cohomological characterization of SSE-$\Z_+G$ \cite{BS05}; and for
studying SSE over dense subrings of $\R$ \cite{BKR2013}.
\end{rem}

\begin{prop} \label{easyfacts} Let $\mathcal S$ be a semiring.
\begin{enumerate}
\item
 SE over $\mathcal S$ is indeed an equivalence relation.
\item
 SSE over $\mathcal S$ implies SE over $\mathcal S$.
\end{enumerate}
\end{prop}
\begin{proof}
(1) If $(R_1,S_1)$ is a shift equivalence\si{shift equivalence}
of lag\si{lag}  $\ell_1 $
from $A$ to $B$, and
$(R_2,S_2)$ is a shift equivance
of lag\si{lag}  $\ell_2 $
from $B$ to $C$,
then
$(R_1R_2,S_2S_1)$ satisfies the equations to be
 a shift equivalence\si{shift equivalence}
of lag\si{lag}  $\ell_1+\ell_2 $
from $A$ to $C$.
(For example, $R_1R_2S_2S_1 = R_1B^{\ell_1}S_1 = R_1S_1A^{\ell_1} =
A^{\ell_2}A^{\ell_1} = A^{\ell_1+\ell_2}$.)

(2)
Suppose we are given a lag\si{lag} $\ell $ SSE from $A$ to $B$: \\
$A=A_0, A_1, \dots , A_{\ell}=B$;
$\quad \quad A_i=R_iS_i$ and $A_{i+1} =S_iR_i$, $\ \ $ for $0\leq i < \ell$ . \\
Set $R=R_1R_2\dots R_{\ell}$ , $\ \ S= S_{\ell}\dots S_2 S_1$. \\
Then $(R,S)$ is a shift equivalence\si{shift equivalence} of lag\si{lag} $\ell$ from $A$ to $B$.
$\qed$
\end{proof}

Next we state  one of the interesting partial results on Williams'
Conjecture\si{Williams' Conjecture}, which we will use later.

\begin{thm}\label{thm:baker}
\cite[K.Baker]{baker1983}
Suppose $A,B$ are positive $2\times 2$ integral matrices
with nonnegative determinant which are
similar over the integers. Then $A,B$ are strong shift equivalent
over $\Z_+$.
\end{thm}

\begin{rem}   \label{smallestlag}({\it Nilpotence and lag\si{lag}})
Let $A=RS,B=SR$ be an ESSE over a semiring $\mathcal S$.
Suppose $m\geq 2$ is the smallest positive integer such that
$A^m=0$. Then $B$ is also nilpotent (because $B^{m+1}=SA^mR=0$), but
$B^{m-2}\neq 0$ (because $B^{m-2}= 0$ would force $A^{m-1}=RB^{m-2}S=0$).
Thus if $\ell$ is the lag\si{lag} of an SSE-$\mathcal S$
from $A$ to a zero matrix,
then $\ell \geq m-1$.
For example, there is a lag\si{lag} 2 SSE-$\R$ from
$\left( \begin{smallmatrix} 0&1&0 \\ 0&0&1 \\ 0&0&0
\end{smallmatrix} \right)$ to $(0)$, but there is no
ESSE-$\R$ from
$\left( \begin{smallmatrix} 0&1&0 \\ 0&0&1 \\ 0&0&0
\end{smallmatrix} \right)$ to $(0)$.

  For an example involving primitive matrices,
  consider the matrix $A=(2)$ and its 3-block presentation
  matrix $B=A^{[3]}$ in
  Remark \ref{higherblockedgesfts}. There is a lag\si{lag} 2  SSE-$\Z_+$
  between $(2)$ and $B$. But there  cannot be an ESSE-$\R$
  of $B$ and $(2)$:
  if $RS= (2)$ and $SR=B$, then $R$ and $S$ have rank 1, so
  $SR$ has rank at most 1, contradicting $B$ having rank 2.
\end{rem}


The next example (extracted from Norbert Riedel's\ai{Riedel, Norbert}
paper \cite{riedel1983},
which has more)
shows that the lag\si{lag} of an SSE-$\Z_+$ is not just a matter of nilpotence.

\begin{ex}  \label{riedelexample} ({\it Bad lag\si{lag} at size 2 from geometry.})\ai{Riedel, Norbert}
For each positive integer $k$, set
$A_k= \left( \begin{smallmatrix} k&2\\1&k \end{smallmatrix} \right)$ and
$B_k= \left( \begin{smallmatrix} k-1&1\\1&k+1 \end{smallmatrix} \right)$.
For each $k$, the matrices $A_k$, $B_k$ are SSE over $\Z_+$.
However, the minimum lag\si{lag} of an SE-$\Z_+$ between
$A_k,B_k$ (and therefore the minimum lag\si{lag} of
an  SSE-$\Z_+$ between
$A_k,B_k$)  goes to infinity as
$k\to \infty$.
\end{ex}
\begin{proof}[Proof sketch]
First, $A_k$ and $B_k$ have the same nonzero\si{nonzero spectrum} spectrum
$(k+ \sqrt 2, k-\sqrt 2)$, and $\Z [k+\sqrt 2]=\Z [\sqrt 2]$,
and $\Z [k+\sqrt 2]=\Z [\sqrt 2]$.
$\Z [\sqrt 2]$ is the ring of algebraic integers in
$\Q [\sqrt 2]$, and this ring is well known to have
class number 1. By Theorem \ref{thm:taussky},
$A_k$ and $B_k$ are similar over $\Z$. Then, by Theorem \ref{thm:baker},
$A_k$ and $B_k$ are SSE over $\Z_+$. By induction one checks
that for each $n$, there are polynomials $P^{(n)}_1, P^{(n)}_2$
with positive integral coefficients
such that $\text{deg}( P^{(n)}_1) = \text{deg}( P^{(n)}_2) +1$
and for all $k,n$
\[
(A_k)^n =
\begin{pmatrix}
P^{(n)}_1(k) & 2P^{(n)}_2(k) \\
P^{(n)}_2(k) & P^{(n)}_1(k)
\end{pmatrix} \ .
\]

Now suppose $R,S$ are matrices over $\Z_+$ and $\ell \in \N$
such that $AR=RB, SA=BS, RS=A^{\ell}$.  The first two equations
force $R,S$ to have the forms
\begin{align*}
R &= \begin{pmatrix}
b-a & a+b \\ a & b
\end{pmatrix} ;
\ \ \ \quad
a,b,b-a \in \Z_+
 \\
S &=\begin{pmatrix}
b-a & 2a-b \\ a & b
\end{pmatrix}; \ \quad
a,b,b-a, 2a-b \in \Z_+
\end{align*}
and from this  one can check that $RS$ has the form
\[
RS = \begin{pmatrix}
a&2b \\ b & a
\end{pmatrix},
\ \ \ \quad
a,b,2b-a \in \Z_+ \ .
\]
For fixed $n$, $\lim_k 2P^{(n)}_2(k)/P^{(n)}_1(k)  = \infty $ .
Thus given $\ell_0 \in \N$, for all sufficiently large $k$
we have for $n\leq \ell_0$ that
$P^{(n)}_1(k) > 2P^{(n)}_2(k)$. Thus, for such $k$ the lag\si{lag} of
an SE-$\Z_+$ between $A_k$ and $B_k$ is greater than $\ell_0$.
\end{proof}

It is worth noting that Riedel's\ai{Riedel, Norbert}  argument showing the
smallest lag\si{lag} of an SE-$\Z_+$ goes to infinity with $k$ works
just as well with $\Q_+$ or $\R_+$ in place of $\Z_+$: bad lags\si{lag}
can happen for ``geometric'' reasons,
without nilpotence or arithmetic issues.
On the other hand, bad lags\si{lag} can happen for strictly
arithmetic reasons, as the next example shows.

\begin{ex}\label{badArithmeticLag}
({\it Bad  lag\si{lag} at size 2 from arithmetic.})
Given $\ell \in \N$, there are $2\times 2 $ positive
integral matrices $A,B$ such that
(i) $A,B$ are SE-$\Z_+$, with minimum lag\si{lag} at least $\ell$, and
(ii) $A,B$ are SE-$\Q_+$ with lag\si{lag} 2.\begin{footnote}
{In Example \ref{badArithmeticLag}, I don't know any obstruction
to existence of an example for which  condition (ii) is replaced
by \lq\lq $A,B$ are ESSE-$\Q_+$ and
SSE-$\Z_+$\rq\rq .}\end{footnote}
\end{ex}
\begin{proof}[Proof sketch]
We list  steps to check.
Given   a prime $q$, and positive integer $x$, set
$A_x=
\left( \begin{smallmatrix} q & x \\ 0 & 1
\end{smallmatrix}\right)  $.

{\it Step 1.}
Suppose $(R,S)$ gives an SE-$\Z$ from
$A_x$ to $A_y$: $A_xR=RA_y$, etc. Then (perhaps after replacing
$R,S$ with $-R,-S$)
$R$ has the form
$\left( \begin{smallmatrix}\pm q^k & z \\ 0 & 1
\end{smallmatrix}\right)  $,
where $k$ is a nonnegative  integer.
It follows that $\pm x \equiv q^ky  \mod (q-1)$.

{\it Step 2.}
Suppose there is a smallest positive integer
$k$ such that $q^kx \equiv \pm y \mod (q-1)$.
Then $A_x$, $A_y$ are SE-$\Z$, but any
such shift equivalence\si{shift equivalence} has lag\si{lag} at least $k$.

{\it Step 3.}
Choose $p$ prime such that $p-1> 2(2\ell +5)$.
Then choose $q$ prime such
that $p$ divides $q-1$ (this is possible
by Dirichlet's Theorem
 \cite{marcusDaniel}).
 Because $p-1 \geq 2\ell +5$,
by the Pigeonhole Principle
we may choose  $j$ a positive integer such that
$1\leq j \leq 2\ell +5$  and also for $1\leq k \leq \ell +2$
we have
$j\not\equiv \pm q^k \mod p$ .
 Define $x=(q-1)/p$ and
 $y=j(q-1)/p$. Then $A_x$ and $A_y$ are SE-$\Z$
 with minimum lag\si{lag} at least $\ell +2$. Also,
 $0< x< y < (1/2)q$ and $y< qx$.

{\it Step 4.}
For $z\in \{x,y \}$, define the positive integral matrix
\[
M_z = \begin{pmatrix} 1&0\\1&1 \end{pmatrix}
\begin{pmatrix} q&z\\0&1 \end{pmatrix}
\begin{pmatrix} 1&0\\-1&1 \end{pmatrix}
= \begin{pmatrix} q-z&z \\ \ q-z-1&1+z \end{pmatrix} \ ,
\]
This SIM-$\Z$ gives a lag\si{lag} 1 SE-$\Z$ between $A_z$ and $M_z$.
If follows that there can be no
SE-$\Z$ from $M_x$ to $M_y$ with lag\si{lag}
smaller than $\ell$.

{\it Step 5.}
It remains to produce the lag\si{lag} 2 SE-$\Q_+$ between $M_x$ and $M_y$.
For the eigenvalues $q$ and $1$,
$M_z$ has
 right eigenvectors
 $v=(1,1)^{\text tr}$ and $w_z= (-z,q-z-1)^{\text{tr}} $. Let $U$
be the $2\times 2$ matrix such that
$Uv=v$ and $Uw_x=w_y$. Then $R=M_yU, S=M_xU^{-1}$ gives a lag\si{lag} 2
SE-$\Q$ between $M_x$ and $M_y$. It remains to check $R,S$
are nonnegative. We have
\begin{align*}
R=M_yU &=
\begin{pmatrix} q-y&y \\ \ \ q-y-q & y+1
\end{pmatrix}
\frac 1{q-1}
\begin{pmatrix} q-x-1+y & x-y \\
-x+y & \  \ q+x-y-1
\end{pmatrix}
\\
&= \frac 1{q-1}
\begin{pmatrix}
q^2 - q(x+1) -xy
&
qx-y
\\
q^2 -q(x+2) -xy+1
&
\ \ q(x+1) - (y+1)
\end{pmatrix}
\end{align*}
From the last sentence of Step 3,
we see the entries of $M_yU$ are positive.
The matrix $S$ is obtained from $R$ by interchanging the
roles of $x$ and $y$, and $S$ is likewise positive.
\end{proof}

\begin{rem} \label{ashley}\ai{Ashley, J.J.}
  By the way, here is an example due to Jonathan Ashley (``Ashley's
eight by eight\si{Ashley's eight by eight}'',
from
\cite[Example 2.2.7]{Kitchens1998})\ai{Kitchens, Bruce P.}
 of
 a primitive matrix $A$   SE-$\Z$  to $(2)$, but
 not known to
be SSE-$\Z_+$  to $(2)$.
$A$ is the $8\times 8$ matrix which is the sum of the
permutation matrices for the permutations
(12345678) and (8)(1)(263754).
\end{rem}

  \begin{rem}  \label{RealSSEandBKR}
    For more on the problem of SSE over $\R$, focused on the
    case of positive matrices, see \cite{BKR2013}.
    (Kim\ai{Kim, K.H.} and Roush\ai{Roush, F.W.} proved that
    primitive matrices over $\R_+$
    are SSE-$\R_+$ to positive matrices. So, the case of SSE-$\R_+$ of
    positive  matrices  handles the primitive positive trace case.)
    The method here,
    due to Kim\ai{Kim, K.H.} and Roush\ai{Roush, F.W.},
        is to derive from a path of similar positive matrices an
        SSE-$\R_+$ between the endpoints.
        Kim\ai{Kim, K.H.} and Roush\ai{Roush, F.W.} were able to reduce to considering positive
        matrices of equal size, similar over $\R$;
        and in the ``$1\times 1$ case'',
        to produce such a path.

        However, even when both $A$ and $B$
        are $2\times 2$ positive real        matrices, the problem of
        when they are SSE-$\R_+$ is open. It is embarassing that we
        are not more clever.
  \end{rem}

\begin{rem} \label{ssez2ssezplus}
To understand  when
SE-$\Z_+$ matrices $A,B$ are
 SSE-$\Z_+$, it is best to focus on
 the fundamental case that
 $A$ and $B$ are primitive. (Then consider irreducible matrices, then
 general matrices, modulo a solution of the primitive case.)
 For primitive matrices over  $\Z_+$,
 SE-$\Z_+$ is equivalent to  SSE-$\Z$ (Proposition \ref{primiiveSE}).
For primitive matrices over $\Z_+$, it
has been important to study a reformulation of the problem:
  when does SSE-$\Z$ imply SSE-$\Z_+$?
 This  formulation  was essential for the Wagoner\ai{Wagoner, J.B.} complex
 setting for the Kim-Roush\ai{Roush, F.W.} counterexamples \cite{S13}\ai{Kim, K.H.}
 to Williams' Conjecture\si{Williams' Conjecture}, and for some arguments for
a general subring $R$ of  $\R$ (see \cite{BKR2013})\ai{Kim, K.H.}.
For  some subrings  $R$ of $\R$,
SE-$R_+$ does not even imply SSE-$\R$, as we will see.
 \end{rem}

\section{Shift equivalence and
    strong shift equivalence over a ring } \label{sec:sesse}

In this section, we present basic facts about shift
equivalence\si{shift equivalence}
and strong shift equivalence\si{strong shift equivalence} over rings, with
various example classes.



\subsection{SE-$\Z_+$: dynamical meaning and reduction to SE-$\Z$}

First we give the dynamical meaning of SE-$\Z_+$.

\begin{de}  Homeomorphisms $S$ and $T$ are eventually conjugate\si{eventually conjugate} if $S^n, T^n$
                are conjugate for all but finitely many positive integers $n$.
\end{de}

\begin{thm} Let $A,B$ be square matrices over $\Z_+$.
The following are equivalent  \apr{seeventual}).
\begin{enumerate}
   \item $A,B$ are shift equivalent over $\Z_+$.
\item
      The SFTs $(X_A,\sigma)$, $(X_B,\sigma)$
                are eventually conjugate\si{eventually conjugate}.
\end{enumerate}
\end{thm}

Next we consider how SE-$\Z$ and SE-$\Z_+$ are related.
Recall Definition \ref{primitiveDefinition}:
a primitive matrix is a square nonnegative real matrix such
that some power is positive.

\begin{ex}  The matrices $\begin{pmatrix} 1 \end{pmatrix} $ and
   $\begin{pmatrix} 1&1\\ 1&0 \end{pmatrix}$ are  primitive. \\
   The  matrices
$\begin{pmatrix} 1&1\\ 0&0 \end{pmatrix}$,
   $\begin{pmatrix} 1&1\\ 0&1 \end{pmatrix}$ and
   $\begin{pmatrix} 0&1\\ 1&0 \end{pmatrix}$ are not primitive.
\end{ex}
\begin{prop}\label{primiiveSE} \apr{sePrimitive})
 Suppose two primitive matrices over a subring $\mathcal R$ of the reals are
 SE over $\mathcal R$. Then they are SE over $\mathcal R_+$.
(Recall, $\mathcal R_+ = \mathcal R \cap \{x\in \R : x\geq 0\}$.)
\end{prop}
For primitive matrices, the  classification up to  SE-$\Z_+$ reduces
 to the tractable problem of classifying up to  SE-$\Z$.
 The Proposition becomes false if the hypothesis of primitivity is
 removed \apr{bkaplansky}).

 \subsection{Strong shift equivalence\si{strong shift equivalence}
   of matrices over a ring} \label{ssesubsec}

   Let $\mathcal R$ be a ring.
   Recall, $\GL (n, \mathcal R)$ is the group of $n\times n$ matrices invertible
   over $\mathcal R$;
   $U\in \GL (n,\mathcal R)$ if there is a matrix $V$ over $\mathcal R$
   with $UV=VU=I$.  This matrix $V$ is denoted $U^{-1}$.
      If $\mathcal R$ is commutative,  then $U\in \GL(n,\mathcal R)$
   iff $\det U$ is a unit in $\mathcal R$.

   Square matrices $A,B$ are similar over $\mathcal R$
   (SIM-$\mathcal R$) if there exists
   $U$ in $\GL(n,\mathcal R)$  such that $B =   U^{-1}AU$.

Our viewpoint:
SE and SSE of matrices over a ring $\mathcal R$
are  {\it stable versions of similarity}
of matrices over $\mathcal R$.

By  a  ``stable version of similarity''  we mean
 an equivalence relation on square matrices which
coarsens  the relation of  similiarity, and is
obtained by  allowing some kind of neglect of the nilpotent part of
the matrix multiplication\begin{footnote}{The term ``stable'' has
  had diverse use. We think of
  ``stable algebra of matrices''\si{stable algebra}
  as a large subject in which we
  consider one  meaningful topic.}\end{footnote}. (This will be less vague soon.)
\begin{prop}[Maller-Shub]\label{prop:MallerShub}\ai{Maller, Michael}\ai{Shub, Michael}
  SSE over
   a ring $\mathcal R$ is the equivalence
   relation on square matrices over $\mathcal R$ generated by
   the following relations on square matrices $A,B$ over $\mathcal R$.
 \begin{enumerate}
     \item
{\it (Similarity over  $\mathcal R$)}
For some $n$,    $A$ and $B$ are $n\times n$ and
there is a matrix $U$ in $\text{GL}(n,\mathcal R)$
   such that $A=U^{-1}BU$ .\\

\item {\it (Zero extension)\si{zero extension}} There exists
   a matrix $X$ over $\mathcal R$ such that in block form,
   $B= \begin{pmatrix} A&X\\ 0&0 \end{pmatrix} $ or
   $B= \begin{pmatrix} A&0\\ X&0 \end{pmatrix} $ .
\end{enumerate}
\end{prop}
\begin{proof}  A  similarity or a zero extension\si{zero extension}
   produces an ESSE:
   \\ \\
   If   $A=U^{-1}BU$, then
   $A= (U^{-1}B)\, U$ and $B = U\,(U^{-1}B)$.
   \\ \\
   If    $B= \begin{pmatrix} A&X\\ 0&0 \end{pmatrix} $,
   then $B=  \begin{pmatrix} I\\ 0 \end{pmatrix}
   \begin{pmatrix} A&X \end{pmatrix} $ and
   $A=
   \begin{pmatrix} A&X \end{pmatrix}
    \begin{pmatrix} I\\ 0 \end{pmatrix}$ .
   \\ \\
   If    $B= \begin{pmatrix} A&0\\ X&0 \end{pmatrix} $,
   then $B=  \begin{pmatrix} A\\ X \end{pmatrix}
   \begin{pmatrix} I&0 \end{pmatrix} $ and
   $A=
   \begin{pmatrix} I&0 \end{pmatrix}
    \begin{pmatrix} A\\ X \end{pmatrix}$ .

Conversely, given $A=RS$ and $B=SR$,
  Maller\ai{Maller, Michael}  and
  Shub\ai{Shub, Michael} constructed     in \cite{MallerShub1985}     a similarity
    of   zero  extensions\begin{footnote}{The paper \cite{MallerShub1985}
        did not consider general rings, or state Proposition
        \ref{prop:MallerShub} explicitly even for $\Z$. However,  Boyle heard
           Maller\ai{Maller, Michael}, in a
        talk in the 1980s, state and prove the content of
        Proposition \ref{prop:MallerShub} in full
        generality.}\end{footnote}:
\[
    \begin{pmatrix} I&0\\ S&I\end{pmatrix}
     \begin{pmatrix} A&R\\ 0&0 \end{pmatrix} =
     \begin{pmatrix} 0&R\\ 0&B \end{pmatrix}
     \begin{pmatrix} I&0\\ S&I \end{pmatrix} \ .
     \]
\end{proof}

    SSE-$\mathcal R$ coarsens         SIM-$\mathcal R$ by
    allowing ``zero extensions\si{zero extension}''.
     The analogue  of Proposition \ref{prop:MallerShub} for
     SE, Theorem \ref{thm:endorelations},
     will replace zero extensions with nilpotent extensions.

The relation SSE-$\mathcal R$ can be very subtle indeed, as we will
see.
Fortunately, if   $\mathcal R$ is $\Z$, or a field,
then      SSE-$\mathcal R$ = SE-$\mathcal R$.

%
%
%
%

\subsection{SE, SSE and det(I-tA)}\label{detI-tAsubsection}

Let $\mathcal R$ be a commutative ring, and $A$ a square
matrix over $\mathcal R$. As explained in Remark
\ref{NewtonIdentities}, the polynomial $\det(I-tA)$ determines
the trace sequence\si{trace sequence} $(\tr (A^n))_{n=1}^{\infty}$, and
that sequence determines $\det(I-tA)$ if $\mathcal R$
is torsion-free.

If $A$ and $B$ are SSE over $\mathcal R$, then one easily sees
$(\tr (A^n))_{n=1}^{\infty}=(\tr (B^n))_{n=1}^{\infty}$, simply because
$\tr (RS)= \tr (SR)$. To see that in addition
$\det(I-tA) = \det(I-tB)$, apply the
Maller-Shub\ai{Maller, Michael}\ai{Shub, Michael} characterization
Proposition \ref{prop:MallerShub}.

If there is a lag\si{lag} $\ell$ shift equivalence\si{shift equivalence} over $\mathcal R$  between
$A$ and $B$, then $(\tr (A^k))_{k=\ell}^{\infty} = (\tr (B^k))_{k=\ell}^{\infty}$.
We shall see below that if  $\mathcal R$ is an integral domain,
then $\det (I-tA)$ is also an invariant of SE-$\mathcal R$ (because
it is an invariant of shift equivalence\si{shift equivalence} over the field of fractions
of $\mathcal R$).

But in some cases,  the  trace of a matrix need not be an invariant of
SE-$\mathcal R$. Suppose $\mathcal R$ is a ring with a nilpotent element $a$
(i.e., $a\neq 0$ and $a^k=0$ for some positive integer $k$).
For example, let $\mathcal R = \Z[t]/(t^2)$ and $a=t$.
Consider the
$1\times 1$ matrices $A=(a)$ and $B=(0)$. Then $A$ and $B$ are SE-$\mathcal R$
but $\tr (A) \neq \tr (B)$.
This by the way gives an easy example of a ring $\mathcal R$ for which
SE-$\mathcal R$ and SSE-$\mathcal R$ are not the same
relation.\begin{footnote}{Somehow this easy example was missed for many years,
    perhaps because the rings arising in symbolic dynamics
    are generally without
nilpotents.}\end{footnote}

If a commutative ring $\mathcal R$ has no nilpotent element, then $\det(I-tA)  $
will be an invariant of the SE-$\mathcal R$ class of $A$
\apr{reducedrings}).

\subsection{Shift equivalence\si{shift equivalence} over a ring $\mathcal R$}

We consider shift equivalence\si{shift equivalence} over a ring $\mathcal R$, by cases.

\subsubsection{$\mathcal R$ is a field.}

     Suppose
     $A$ is a square matrix over $\mathcal R$.
     There is an invertible $U$ over $\mathcal R$ such that
     $U^{-1}AU$ has the form
     $\begin{pmatrix} A' & 0 \\ 0 &N \end{pmatrix}$,
     where $A'$ is invertible and $N$ is triangular with zero
     diagonal.
     (For $\mathcal R=\C$, use the Jordan form.)
\begin{ex}
        $\mathcal R=\R$,  $\ \ \ U^{-1}AU =
        \begin{pmatrix} \sqrt 2 &0&0&0 \\ 0&0&1&0 \\ 0&0&0&1 \\
          0&0&0&0 \end{pmatrix}$, $\ \ \ A'=(\sqrt 2)$.
       \end{ex}
      $A'$ (as a vector space endomorphism) is
     isomorphic to  the restriction of $A$ to the largest
     invariant subspace on
     which $A$ acts invertibly. Abusing notation, we
     call $A'$ {\bf the nonsingular
     part} of $A$ (keeping in mind that $A'$ is only well defined
     up to similarity over the field $\mathcal R$).
\begin{prop} A square matrix $A$ over a field $\mathcal R$
   is SSE over $\mathcal R$ to its
   nonsingular part, $A'$.
\end{prop}
\begin{proof}
$A'$ reaches $U^{-1}AU$ by a string of zero
     extensions.
\end{proof}
\begin{exer} Suppose $\det A =0$, and $U^{-1}AU
=\begin{pmatrix} A' & 0 \\ 0 &N \end{pmatrix}$, and
$\ell$ is the smallest positive integer such that $N^{\ell}=0$.
Then the smallest lag\si{lag} of an SSE over $\mathcal R$
from $A$ to $A'$ is $\ell $ \apr{smallestlag}).
\end{exer}

From the Proposition, square matrices  $A,B$ are
    SSE-$\mathcal R$
    if and only if their  nonsingular parts $A',B'$ are SSE-$\mathcal R$.
    Likewise for SE-$\mathcal R$.
\begin{prop}
Suppose $A,B$ are square nonsingular matrices over
        the field $\mathcal R$. The following are equivalent.
\begin{enumerate}
\item $A$ and $B$ have the same size and are similar over $\mathcal R$.

\item $A$ and $B$ are SE-$\mathcal R$.

\item $A$ and $B$ are SSE-$\mathcal R$.
\end{enumerate}
\end{prop}
\begin{proof}
        $(1) \implies (3) \implies (2)$. Clear.

          $(2) \implies (1)$.
  Let $(R,S)$ be  a lag\si{lag} $\ell $ SE over $\mathcal R$ from $A$ to $B$:
  \[
  A^{\ell} =RS\ ,\quad  B^{\ell} =SR\ ,\quad  AR=RB\ ,\quad  BS=SA\ .
  \]
  Suppose  $A$ is $m\times m$ and  $B$ is $n\times n$.
  Then $R$ is $m\times n$.
  Hence $m=n$, because
  \[
  m \ \ = \ \  \text{rank}(RS) \ \  \leq  \ \ \text{rank}(R)  \ \
  \leq  \ \ \min \{ m,n \}  \ \ \leq n
  \]
and likewise $n \leq m$.
Now $\det(A^{\ell}) =(\det R)( \det S)$, hence
$\det R \neq 0$.
Then $AR=RB$ gives $B=R^{-1}AR$ .
\end{proof}

\begin{coro}
Suppose matrices $A,B$ are SE over a field $\mathcal R$.
Then
 $\det(I-tA) = \det (I-tB)$ .
\end{coro}
\begin{proof}
By the proposition, the nonsingular parts
      $A',B'$ of $A,B$ are similar over $\mathcal R$.
      Therefore they have the same spectrum, which is the nonzero\si{nonzero spectrum}
      spectrum of $A$ and $B$. Therefore
            $\det (I-tA)= \det(I-tB)$.
\end{proof}

     When matrices
     $A,B$ have entries in a field $\mathcal R$  contained in $\C$,
     similarity over $\mathcal R$ is equivalent to
     similarity over $\C$, and the Jordan form of
the nonsingular part is a complete invariant for similarity of
     $A,B$ over $\mathcal R$.

\subsubsection{$\mathcal R$ is a Principal Ideal Domain}
The principal ideal domain of greatest interest to us is
$\mathcal R = \Z$, the integers.
 The PID  case  is like  the field case, but with
     more arithmetic structure. In place of the Jordan form, we use
     a  classical fact. Recall, an {\it upper triangular} matrix
     is a square matrix with only zero entries below the diagonal
     (i.e., $i > j \implies A(i,j) = 0$). A lower triangular
     matrix is a square matrix with only zero entries above the diagonal.
     A  matrix is triangular if it is upper or lower triangular.
Block triangular matrices are defined similarly,
for block structures on square matrices which use the same
index sets for rows and columns.

\begin{thm} [PID Block Triangular Form] \cite{Newman}\label{thm:PIDblocktri}
   Suppose $\mathcal R$ is a principal ideal domain (e.g.,
   $\mathcal R=\Z$ or a field).
   Suppose $A$ is a square matrix over $\mathcal R$ and
   $p_1, \dots , p_k$ are monic polynomials with coefficients in
   $\mathcal R$ such that the characteristic polynomial of $A$ is
   $\chi_A(t) = \prod_i p_i(t)$.

   Then $A$ is similar over $\mathcal R$ to a block triangular matrix,
   with diagonal blocks $A_i$, $1\leq i \leq k$ ,
   such that $p_i$ is the characteristic polynomial of $A_i$.
   \end{thm}
\begin{ex}
   Suppose $A$ is $6\times 6$
and $\chi_A(t)=(t-3)(t^2+5)(t+1)(t)(t)$.
   Then  there is some $U$ in $\text{GL}(6, \Z)$ such that
   $U^{-1}AU$ has the form
\[
\left( \begin{smallmatrix}
     3&*&*   &*&*&* \\
     0&a&b   &*&*&* \\
     0&c&d   &*&*&* \\
     0&0&0   &-1&*&* \\
     0&0&0   &0&0&* \\
     0&0&0   &0&0&0
   \end{smallmatrix}
   \right)
   \]
in which  $\begin{pmatrix} a&b\\c&d \end{pmatrix}$
   has characteristic polynomial $t^2 +5$ .
   \end{ex}
\begin{coro}[Corollary of PID Block Triangular Form]
 For $A$ square over the PID $\mathcal R$:
      $A$ is similar over $\mathcal R$ to a matrix
   with block form
$\begin{pmatrix} A' & X \\ 0 & N \end{pmatrix}$,
   where $\det (A') \neq 0$ and $N$ is upper triangular with
   zero diagonal.
\end{coro}
   As in the field case, we  call
   $A'$ {\it the nonsingular part} of $A$ ($A'$ is defined up to similarity
   over $\mathcal R$).
\begin{coro}[Nonsingularity]
For $\mathcal R$ a principal ideal domain,
   any nonnilpotent square matrix over $\mathcal R$ is
   SSE-$\mathcal R$  to its nonsingular part
   (hence, SE-$\mathcal R$ to its nonsingular part).
\end{coro}
\begin{proof}
 $A'$ reaches $U^{-1}AU$ by a string of zero
     extensions.
     \end{proof}

     The Corollary can easily fail even for a Dedekind domain,
     such as the algebraic integers in a number field \cite{BH93}.

\subsubsection{$\mathcal R$ is  $\Z$.}
\begin{exer}
\apr{1by1})
   Suppose $A$ is square over $\Z_+$
   and $\det (I-tA) = 1-nt$, where $n$ is a positive integer.
   Then $A$ is SE over $\Z_+$ to the $1\times 1$ matrix $(n)$.
   \end{exer}

               The classification of matrices over $\Z$ up to SE-$\Z$
   reduces to the classification of nonsingular matrices
   over $\Z$ up to SE-$\Z$.
   If $A,B$  are SE-$\Z$, then $A,B$  are SE-$\R$, so
   their nonsingular parts
are similar over  $\R$; in particular $A,B$ have the same
nonzero\si{nonzero spectrum} spectrum.

     It is NOT true that a square matrix over $\Z_+$ must be SE-$\Z_+$
     to a nonsingular matrix.

\begin{exer}  \apr{nonplussed})
The  primitive matrix
$A= \left(\begin{smallmatrix} 1&0&0&1 \\ 0&1&0&1 \\ 0&1&1&0 \\ 1&0&1&0
 \end{smallmatrix}\right) $  has
nonzero\si{nonzero spectrum} spectrum $(2,1)$. If $A$ is SE-$\Z_+$ to a nonsingular matrix $B$,
then $B$ must be a primitive $2\times 2$ with
nonzero\si{nonzero spectrum} spectrum $(2,1)$.
Prove that no such  $B$ exists.
\end{exer}

\begin{prop}Suppose $A$, $B$ are nonsingular matrices over
   $\Z$ with $|\det (A)|=1$. The following are equivalent.

   (1) $A,B$ are SE-$\Z$.

   (2) $A,B$ are SIM-$\Z$.
\end{prop}
\begin{proof} $(2) \implies (1)$ Clear.

 $(1) \implies (2)$
  An SE $(R,S)$ over $\Z$ from
   $A$ to $B$ is also an SE $(R,S)$ over the field $\Q$.
   Therefore    $A,B$ have the same size, $n\times n$.  Now
   $A^{\ell} = RS$ forces $\det A$ to divide $\det A^{\ell}$,
   so $|\det R| =1$. This implies $R\in \GL ( n, \Z)$.
   Then $AR=RB$ gives $B=R^{-1}AR$.
   \end{proof}
\begin{ex}[Nonsingular $A,B$ which are SIM-$\Q$, but not SE-$\Z$.]

   Let $A= \begin{pmatrix} 3&4\\1&1 \end{pmatrix}$
   and $B= \begin{pmatrix} 3&2\\2&1 \end{pmatrix}$. $A$ and $B$
   have the same characteristic polynomial, $p(t)= t^2 -4t-1$;
   as $p$ has
 no repeated root,
   $A$ and $B$  are similar over $\Q$.

   Now suppose $A,B$ are SE-$\Z$. Because
   $\det A = -1$, they are SIM-$\Z$: there is $R$ in
   $\text{GL}(2,\Z)$ such that $B=R^{-1}AR$. Therefore
   $(B-I)=R^{-1}(A-I)R$.
This is a contradiction, because
$B-I= \begin{pmatrix} 2&2\\2&0 \end{pmatrix}$ is zero mod 2, but
$A-I= \begin{pmatrix} 2&4\\1&0 \end{pmatrix}$ is not.
\end{ex}

What else?
Here is a quick overview.
\begin{thm}
Let $p$ be a monic polynomial in $\Z [t]$ with no zero root.
Let $\mathcal M(p)$ be the set of
matrices over $\Z$ with characteristic polynomial $p$.

If $p$ has no repeated root, then the following hold.
\begin{enumerate}
\item All matrices in $\mathcal M(p)$ are  SIM-$\Q$ (and therefore SE-$\Q$).

\item $\mathcal M(p) $ is the union of finitely many  SIM-$\Z$ classes
(hence finitely many SE-$\Z$ classes).

\item  It is can happen (depending on $p$) that in
 $\mathcal M(p)$, SIM-$\Z$  properly refines SE-$\Z$.
\end{enumerate}
If $p$ has a repeated root, then $\mathcal M(p) $
contains
infinitely many SE-$\Z$ classes, but only
finitely many SE-$\Q$ classes.
\end{thm}

\begin{ex} (Easily checked.) For $n\in \N$, the matrices
   $\begin{pmatrix} 1&n\\0&1 \end{pmatrix} $
   are similar over $\Q$, but pairwise not similar over $\Z$.
\end{ex}

Lastly, we will report on some decidablity issues for shift equivalence\si{shift equivalence}
over $\Z$.



\begin{thm}\ai{Grunewald, Fritz J.}
   Suppose $A,B$ are  square matrices over $\Z$.
   \begin{enumerate}
   \item (Grunewald \cite{grunewaldZDecide}; see also \cite{grunewaldSegalI}.)
     There is an algorithm to decide whether $A,B$ are SIM-$\Z$.
     The  general algorithm is not practical.
   \item (Kim\ai{Kim, K.H.} and Roush\ai{Roush, F.W.} \cite{S31})
     There is an algorithm to decide whether $A,B$ are SE-$\Z$.
     The general algorithm is not practical.
   \end{enumerate}
\end{thm}
\subsection{SIM-$\Z$ and SE-$\Z$:  some example classes}

The proof of the next result, from \cite{BH93},
 is an exercise.

\begin{thm} \label{thm:axb}
Suppose $a,b$ are integers and $a> |b| > 0$.
Let     $\mathcal M$ be the  set of $2\times 2$ matrices over $\Z$ with
     eigenvalues $a,b$. Then the following hold.
\begin{enumerate}
\item
   Every matrix in $\mathcal M$ is SIM-$\Z$ to a triangular matrix
   $M_x=\begin{pmatrix}a & x\\ 0 &b\end{pmatrix}$.
\item
  $M_x$ and $M_y$ are SIM-$\Z$ iff
   $x = \pm y$ mod ($a-b$).
\item     $M_x$ and $M_y$ are SE-$\Z$ iff $x\sim y$, where $\sim$
   is the equivalence relation generated by $x\sim y$ if
   $x = \pm qy$ mod ($a-b$) for a prime $q$ dividing $a$ or $b$.
\end{enumerate}
\end{thm}
\begin{ex}
   Suppose $a=6,b=1$.\\
   Then $\mathcal M$
is the union of three SIM-$\Z$ classes and two SE-$\Z$ classes.
\end{ex}
\begin{ex}
Suppose $a=6, b=2$. \\
  Then $\mathcal M$
is the union of two SIM-$\Z$ classes and one SE-$\Z$ class.
\end{ex}
\begin{exer} \label{transposeexercise}  \apr{transposeexerciseproof})
Use Theorem \ref{thm:axb} to prove the following: the matrix
$\begin{pmatrix} 256 & 7\\ 0 &1 \end{pmatrix}$ is not SE-$\Z$ to its
transpose. Then show $\begin{pmatrix} 256 & 7\\ 0 &1 \end{pmatrix}$
is SE-$\Z$ to a primitive matrix, which cannot be
SE-$\Z$ to its transpose.
\end{exer}

The next theorem states a result
relating a matrix similarity problem to algebraic number theory,
and the analagous result for shift equivalence\si{shift equivalence}.
The similarity result is a special case of a theorem of
Latimer and MacDuffee \cite{latimerMacDuffee};
Olga Taussky-Todd\ai{Taussky-Todd, Olga}
provided a simple
proof in this special case, which generalizes nicely to the SE-$\Z$
situation \apr{tausskytoddcase}).
In the next theorem, for $\mathcal R=\Z[\lambda]$ or
$\mathcal R=\Z[1/\lambda]$,
$\mathcal R$-ideals $I,I'$ are equivalent if they are equivalent as
$\mathcal R$-modules,
which in this case means there is a nonzero $c$ in $\Q [\lambda ]$
such that $cI=I'$. By an ideal class of $\mathcal R$
we mean an equivalence class
of nonzero $\mathcal R$-ideals.

\begin{thm}\label{thm:taussky} Suppose $p$ is monic irreducible in $\Z [t]$,
     and $p(\lambda)=0$, where
     $0\neq \lambda \in \C$.
   Let $\mathcal M$ be the set of matrices over $\Z$ with
   characteristic polynomial $p$.
   There are bijections:
\begin{enumerate}
\item    $\mathcal M/(\text{SIM}-\Z)$  $\leftrightarrow $
   Ideal classes  of $\Z[\lambda ]$ \
\cite{latimerMacDuffee,tausskyOnLatimerTheorem}\ai{Taussky-Todd, Olga}

\item
   $\mathcal M/(\text{SE}-\Z) \ \ \leftrightarrow $
    Ideal classes of $\Z[1/\lambda ]$ \  \cite{BMT1987}.
\end{enumerate}
\end{thm}

\begin{exer} \apr{finiteclass})
Let $\lambda$ be a nonzero algebraic integer, and
let $\mathcal O_{\lambda}$ be the ring of algebraic integers in the
number field $\Q [\lambda ]$.
It is a basic (and ``surprisingly easy to establish''
\cite[Ch.5]{marcusDaniel})
fact of algebraic number theory that
the class number of  $\mathcal O_{\lambda}$ is finite.
Use this fact to show that  $\Z[\lambda ]$ also has
finite class number.
\end{exer}

The number theory connection is useful. For example,
it follows from the exercise
that $\mathcal M$ in the theorem contains only finitely
many SIM-$Z$ classes (\cite[Theorem III.14]{Newman}).
In the case that  $\Z[\lambda ]$ is a full ring of
quadratic integers,  one can often simply look up the
class number of $\Z[\lambda ]$ in a table.

\subsection{SE-$\Z$ via direct limits}\label{sec:SEZdirectlimits}
   Let $A$ be an $n\times n$ matrix over
   $\Z$. We choose to let $A$ act on row vectors.
   From the action $A: \Z^n \to\Z^n $ one can form
   the direct limit group, on which there is a
   group automorphism $\hat A: G_A \to G_A$ induced by $A$.

   We will take a very concrete presentation, $\hat A: G_A \to G_A$,
   for the induced automorphism of the direct limit group
   \apr{directlimits}).
\subsubsection{The eventual image $V_A$}
      Define
   rational vector spaces
   $W_k = \{ vA^k: v \in \Q^n\}$ and
   $V_A = \cap_{k\in \N}  W_k$ .
   Then
\begin{align*}
& \Q^n \supset W_1 \supset W_{2}\supset W_3 \supset \cdots  \\
&\dim (W_{k+1})=\dim (W_k)\implies W_{k+1}=W_k \\
&   V_A = W_n \ .
\end{align*}
    $V_A$  is  the
   ``eventual image'' of $A$ as an endomorphism of the rational
   vector space $\Q^n$.
   $V_A$ is the largest invariant subspace of $\Q^n$ on which $A$ acts as
   a vector space isomorphism.

\subsubsection{The pair $(G_A, \hat A)$}

 $G_A$ is the subset of $V_A$ eventually mapped by $A$ into
      the integer lattice:
 $G_A := \{ v\in V_A: \exists k\in \N , vA^k \in \Z^n\}$ .
      The automorphism $\hat A$ of $G_A$ is defined by
restriction,
$\ \    \hat A : v\mapsto vA $ .

\begin{ex}  Suppose $|\det A| =1$. Then $V_A= \Q^n$, $G_A = \Z^n$.
\end{ex}
\begin{ex}
$A = (2)$.
Then
$V_A$ = $\Q$, and \\
$G_A$ is the group of dyadic rationals:
$\ G_A = \Z[1/2] = \{ n/2^{k}: n\in \Z, k \in \Z_+\}$.
\end{ex}
\begin{ex} Similarly, for a positive integer $k$, if $A=(k)$ then
$G_A = \Z [1/k]$.
For positive integers $k$ and $ m$, TFAE:

$\bullet \ $ $ \Z [1/k] = \Z [1/m]$ .

$\bullet \ $ $ k$ and $m$ are divisible by the same primes.

$\bullet \ $ $ \Z [1/k]$ and $ \Z [1/m]$ are isomorphic groups.
\end{ex}

\begin{ex}
   $A= \begin{pmatrix} 2 & 0\\ 0& 5 \end{pmatrix}$.
   $\ G_A = \{ (x,y): \ x\in \Z[1/2], \ y\in \Z[1/5] \} $.
   \end{ex}
\begin{ex}
   $B= \begin{pmatrix} 2 & 1\\ 0 &5 \end{pmatrix}$.  \\
The groups    $G_B$ and $G_A$ are not isomorphic.
($G_B$ is not the sum of a 2-divisible subgroup and a 5-divisible subgroup
 \apr{nonisomorphicpair}).) \end{ex}
\begin{de}  Two pairs $(G_A,\hat A), (G_B, \hat B)$ are isomorphic
if there is a group isomorphism $\phi : G_A \to G_B$ such that
$ \phi \hat B (x) =  \hat A \phi (x)$.
(In other words, $\hat A$ and $\hat B$ are isomorphic, in the
category of group automorphisms; or, equivalently, in the
category of group endomorphisms.)
\end{de}
\begin{prop}\label{prop:sezdirectlimit}  Let $A,B$ be square matrices over $\Z$.
   The following are equivalent \apr{seandpairiso}).
\begin{enumerate}
\item  $A$ and $B$ are SE-$\Z$.
\item  There is an isomorphism of direct limit pairs $(G_A, \hat A)$
   and $(G_B, \hat B)$.
   \end{enumerate}
\end{prop}

There is a natural way to make
$G_A$ above an ordered group
(in an important class of ordered groups,
the {\it dimension groups}\si{dimension group} \apr{dimensiongroupsremark})).
Then,
analogous to Proposotion \ref{prop:sezdirectlimit},
there an ordered group characterization
of SE-$\Z_+$ \apr{SEZ+andgroupsremark}).

\begin{ex}
      Let $A = (2)$ and $B=\begin{pmatrix} 1&1 \\ 1&1 \end{pmatrix}$. \\
These matrices are  SE-$\Z$ (and even ESSE-$\Z_+$).
\begin{itemize}
\item
$V_A$ = $\Q$ and
$G_A= \Z[1/2]$.
\item
$V_B = \{ (x,x): x\in \Q\}$, the eigenline for eigenvalue 2, and
$G_B  = \{ (x,x): x\in \Z[1/2] \}$.
\item
$\phi \colon x\mapsto (x,x)$ defines a group isomorphism $G_A \to G_B$
such that $\phi\hat B  (x) =  \hat A \phi(x)$ .
\end{itemize}
\end{ex}

\subsection{SE-$\Z$ via polynomials}

   It will be  important for us to put everything we've done
with shift equivalence   into a polynomial setting.
   (To our knowledge, the polynomial-shift equivalece
   connection was first explicitly pointed out
by Wagoner  \apr{WagonerSEPoly}).)
We use  $\Z[t]$, the ring of polynomials in one variable
   with integer coefficients.
   \\ \\
   Let $A$ be an $n\times n$ matrix over $\Z$. Recall
\begin{itemize}
\item    $V_A = \cap_{k\in \N}W_k = \cap_{k\in \N}\{ vA^k: v \in \Q^n\}$,
\item
   $G_A := \{ v\in V_A: \exists k\in \N , vA^k \in \Z^n\}$,
\item
   $\hat A$ is the  automorphism of $G_A$ given by
   $\hat A: x\mapsto xA$.
   \end{itemize}
We regard the direct limit group $G_A$  as a $\Z [t] $-module,
by letting $t$ act by $(\hat A)^{-1}$. (This
choice of $t$ action will match $\cok (I-tA)$ below.)
Isomorphism of the pairs $(G_A,\hat A), (G_B, \hat B)$ is
equivalent to isomorphism of
$G_A, G_B$ as $\Z [t]$-modules.
So, we sometimes simply refer to a
pair $(G_A, \hat A)$ as a $\Z[t]$-module.
To summarize, we have the following.
\begin{prop}
Suppose $A,B$ are square matrices over $\Z$. The
following are equivalent.
\begin{enumerate}
\item  $A$, $B$ are SE over $\Z$.
\item  $(G_A,\hat A)$ and $ (G_B,\hat B)$ are isomorphic $\Z [t]$-modules.
\end{enumerate}
\end{prop}
   Next we get another presentation of these $\Z[t]$-modules.

\subsection{Cokernel (I-tA), a $\Z [t]$-module}

   Given $A$ $n\times n$  over $\Z$,
   let $I$ be the $n\times n$ identity matrix.
   View $\Z [t]^n $ as a  $\Z [t]$-module: for $v$
   in $\Z [t]^n $ and $c\in \Z [t]$,
   the action  of $c$ is to send $v$ to $cv$, where
   $cv= c(v_1, \dots , v_n)=(cv_1, \dots , cv_n)$.
      The map
   $(I-tA): \mathcal R^n \to \mathcal R^n$, by
   $\ v\mapsto v(I-tA)\ $,
 is a $\Z [t]$-module homomorphism,
   as $(cv)(I-tA) = c(v(I-tA))$.

Now define
$   \text{cokernel}(I-tA) := \Z [t]^n / \text{Image} (I-tA) $,
where
$\text{Image} (I-tA) =
\{v(I-tA)\in \Z [t]^n:
   v\in \Z [t]^n\}  $ .
   An element of $\text{cokernel}(I-tA)$ is a coset,
   $v+ \text{Image} (I-tA)$, denoted $[v]$.
   $   \text{Cokernel}(I-tA)$ is a  $\Z [t]$-module, with
$   c: [v]  \mapsto [cv]$ .

   NOTE: we use row vectors to define the module.

\begin{prop}\label{cokanddirlimiso} Let $A$ be a square matrix over $\Z$.
      The $\Z [t]$-modules $\text{cok}(I-tA)$ and
$(G_A, \hat A)$ are isomorphic.
\end{prop}
\begin{proof}  Define
\begin{align*}
  \phi : G_A & \to \text{cokernel}(I-tA) \\
  x &\mapsto [x(tA)^k]
\end{align*}
where $k$ (dependng on $x$) is any nonnegative integer large enough that
$xA^k\in \Z^n$.  For a proof, check that this $\phi$ is a
well-defined isomorphism of $\Z[t]$-modules.
\end{proof}
Let us see how this works out in a concrete example.
\begin{ex}
   $A=(2)$. Here $G_A = \Z[1/2]
=   \cup_{k\geq 0}\, (\frac 12)^k\Z \ $  and
   $\ \Z[t] = \cup_{k\geq 0}\,  t^k \Z$ .
 The isomorphism
   $\ \phi :G_A \to \text{cokernel}(I-tA)\ $ is defined by
   \begin{align*}
  \phi : \Z[1/2] & \to \Z[t]/(1-2t)\Z[t] \   \\
 (1/2)^kn  & \mapsto [t^kn]\ , \qquad \qquad \ \
  \text{for }n \text{ in }\Z, \ k \in \Z_+ \ .
    \end{align*}
   The isomorphism $\phi$  takes
   $(1/2)^k \Z  $ to $[t^k\Z ]$.  The cokernel
   relation mimics the $G_A$
   relation $(1/2)^kn = (1/2)^{k+1}(2n)$.
   In more detail, to check
   that $\phi$  in this example is a $\Z[t]$-module isomorphism,
   check the following (some details are provided).
   \begin{itemize}
   \item
     $\phi$ is well defined. \\ Because
 $[x] \in \text{cokernel}(1-2t)$, we have $[x] = [2tx]$, so
\begin{align*}
 (1/2)^k n \ &\ \mapsto [t^k n]  \\
      (1/2)^{k+1}(2n)\  &\  \mapsto [t^{k+1} (2n)]
     = [(t^kn)(2t)] = [t^kn] \ .
\end{align*}
   \item $\phi$ is a group homomorphism.
   \item $\phi$ is a $\Z[t]$-module homomorphism :
\begin{align*}
t\phi ((1/2)^kn) \ &=\  t[t^kn] = [t^{k+1}n] \ ,  \\
     \phi (t((1/2)^kn))\ & =\ \phi (\hat A^{-1} ((1/2)^kn)
     = \phi ( (1/2)((1/2)^kn)) = [ t^{k+1}n] \  .
\end{align*}
   \item $\phi$ is surjective.
   \item $\phi$ is injective.
     \\
     Given $\phi ((1/2)^kn) = [t^kn]=[0]$, there exists $p$ in $\Z[t]$
     such that $t^kn = (1-2t)p$. This forces $n=0$. (Otherwise  $p\neq 0$,
and then     $(1-t)p=p-tp$
 with nonzero coefficients at different powers of $t$, contradicting
$t^kn = (1-2t)p$.)
   \end{itemize}
   \end{ex}

\begin{coro}  For square matrices $A,B$ over $\Z$, the following are equivalent.

      (1) The matrices $A,B$ are SE-$\Z$.

      (2)       $\text{cok}(I-tA), \text{cok}(I-tB)$
      are isomorphic $\Z [t]$ modules.
\end{coro}
\begin{rem} \label{LaurentRemark}   Consider now  $\Z[t,t^{-1}]$,
the ring of Laurent polynomials
in one variable.  Given the $\Z[t]$-module $G_A$, with $t$ acting by
$\hat A^{-1}$, there is a unique way to extend the $\Z[t]$-module action
on $G_A$ to a
$\Z[t,t^{-1}]$-module action ($t^{-1}$ must act by $\hat A$).
A map $\phi: G_A\to G_B$ is a $\Z[t]$-module isomorphism if and only if
it is a $\Z[t,t^{-1}]$-module isomorphism.

Consequently, SE-$\Z$ can be (and has been)  characterized using
$\Z[t,t^{-1}]$-modules above in place of the $\Z[t]$-modules.
\end{rem}

\subsection{Other rings for other systems}
\label{subsec:otherrings}

   We've looked at SFTs presented by matrices over $\Z_+$,
   and considered algebraic invariants in terms of these
   matrices (e.g.SE-$\Z$, SSE-$\Z$).
   There are cases  (Ap. \ref{othercases}, \ref{zgsse})
of SFTs with additional structure, or
   SFT-related systems, for which there is very much
   the same kind of theory, but with $\Z$ replaced by
   an integral group ring $\Z G$, and $\Z_+$
   replaced by $\Z_+G$. We will say  a
   little about one case, to indicate the pattern,
   and help motivate our interest in SSE over more
   general rings.

Let $G$ be a  finite group, and let $\Z_+G
 = \{ \sum_{g\in G} n_g g: n_g \in \Z_+ \}$,
 the
``positive'' semiring in $\Z G$.
By a  $G$-SFT we mean  an SFT together with a
free, continuous shift-commuting $G$-action.
A square matrix over $\Z_+G$ can be used to define an
SFT $T_A$ with such a  $G$-action.
Two $G$-SFTs  are isomorphic if there
is a topological conjugacy between them intertwining
the $G$-actions.
Every $G$-SFT is isomorphic to some $G$-SFT $T_A$.

\begin{rem} \label{correspremark}
We list below some correspondences
\apr{zgsse}).
\begin{enumerate}
\item \label{corresp1}SSE-$\Z_+G$ of matrices is equivalent
 to  conjugacy of their $G$-SFTs.

\item \label{corresp2}If $n$ is a positive integer, then
$(T_A)^n$ and $T_{A^n}$ are conjugate $G$-SFTs.

\item \label{corresp3} SE-$\Z_+G$ of matrices is equivalent to eventual conjugacy
of their $G$-SFTs.
\cite[Prop. B.11]{BoSc2}.

\item \label{corresp4} If $G$ is abelian,
    then   the polynomial $\det (I-tA)$ encodes the periodic data\si{periodic data}.

\item \label{corresp5} If $A$ is a square nondegenerate\si{nondegenerate} matrix over $\Z_+G$,
then the SFT $T_A$ is mixing if and only if $A$ is $G$-primitive
\cite[Prop. B.8]{BoSc2}.

\item \label{corresp6} $G$-primitive matrices are  SE-$\Z G$
if and only if they are  SE-$\Z_+ G$ \cite[Prop. B.12]{BoSc2}.
\end{enumerate}
\end{rem}

We add  comments
for some items in  Remark \ref{correspremark}.

\eqref{corresp2} The $G$-action for $(T_A)^n$  above is the
$G$-action given for $T_A$.

\eqref{corresp4}  The determinant is defined for commutative rings, and $\Z G$
is commutative iff the group $G$ is abelian.
Above, the polynomial $\det (I-tA)$ has coefficients in the ring
$\Z G$ .
For abelian $G$,
by definition  two $G$-SFTs have the same ``periodic data\si{periodic data}''
if there is a shift-commuting -- not necessarily continuous --
bijection between their periodic points which respects the $G$-action.


(\ref{corresp5}, \ref{corresp6})  By definition, a
 $G$-primitive matrix is a square matrix $A$
   over $\Z_+G$ such that for some positive integer $k$,
   every entry of $A^k$ has the form $\sum_{g\in G} n_g g$ with
   every $n_g$ a positive integer.

We note one feature of the $\Z$ situation which does NOT
translate to $\ZG$.
Recall, $\text{SE}-\Z \implies \text{SSE}-\Z$.
 In contrast, for many $G$,
the relationship of SE-$\Z G$  and
SSE-$\Z G$ is  highly nontrivial,
as we will see.

\subsection{The module-theoretic formulation of
  SE over a ring}\label{SEmoduletheory}

It is basically an observation that  arguments for SE-$\Z$
adapt to prove the  statements collected below in
Theorem \ref{RmodulestructureforSE}.
We will  outline how this goes.
We spell out  a few  details  in the appendix
 \apr{SEmoduletheorydetails})

Let $\mathcal R$ be a   (not necessarily commutative)
ring. By ``module'', we will mean left module.
For example,
$\mathcal R^n$ is an  $\mathcal R$-module,
with an element $r$ acting from the left by $v\mapsto rv$.

Let $A$ be an $n\times n$ matrix over $\mathcal R$.
The rule $v\mapsto vA$ 
 defines a map
$\mathcal R^n \to \mathcal R^n $. 
This is an $\mathcal R$-module endomorphism,
because $(rv)A = r(vA)$.
Considering $\mathcal R^n$
as an additive group,
let
$G$ be the direct limit group
defined as in
\apr{directlimits})
 by
 the action of $A$. We will
let
$\widehat A$ denote the group automorphism of $G$
defined, in the notation of \apr{directlimits}),
by $[(v,m)]\mapsto [(vA,m)]$.

Because $A$ is an $\mathcal R$-module endomorphism,
there is an induced $\mathcal R$-module structure on
the direct limit group
($r[v,m] = [(rv,m)]$), 
with respect to which
$\widehat A$ is an $\mathcal R$-module automorphism.
$G_A$ becomes an $\mathcal R[t,t^{-1}]$-module
  by having $t^{-1}$ act by $\widehat A$ and
 $t$ act by $\widehat A^{-1}$.
Call this the
direct limit  $\mathcal R[t,t^{-1}]$-module of $A$.
By restriction of action, it becomes the
direct limit  $\mathcal R[t]$-module of $A$.

The $n\times n$ matrix $I-tA$ acts on
the $\mathcal R[t]$-module $(\mathcal R[t])^n$ by
 $v\mapsto v(I-tA)$. The cokernel of this map,
$\cok (I-tA)$, is an $\mathcal R[t]$-module. The action of $t$
on $\cok (I-tA)$ has an inverse ($[v] \mapsto [vA]$), so
we may also consider $\cok (I-tA)$ as an $\mathcal R[t,t^{-1}]$-module.

\begin{thm} \label{RmodulestructureforSE}
  Let $\mathcal R$ be a ring,
  and $A$ a square matrix over $\mathcal R$. Then
  $\cok(I-tA)$ and the direct limit  module of $A$ are
  isomorphic, as $\mathcal R[t]$-modules and as
  $\mathcal R[t,t^{-1}]$-modules. For square matrices $A,B$
   over
  $\mathcal R$.
  The following are equivalent.
  \begin{enumerate}
  \item
    $A$ and $B$ are SE-$\mathcal R$.
  \item
    The direct limit  $\mathcal R[t]$-modules of $A$ and $B$
    are isomorphic.
  \item The $R[t]$-modules
    $\cok (I-tA)$ and $\cok(I-tB)$ are isomorphic.
  \item
    The direct limit  $\mathcal R[t,t^{-1}]$-modules of $A$ and $B$
    are isomorphic.
  \item
    The $R[t,t^{-1}]$-modules
    $\cok (I-tA)$ and $\cok(I-tB)$ are isomorphic.
      \end{enumerate}
  \end{thm}

\subsection{Appendix 2} \label{a2}

This subsection contains various remarks, proofs and comments referenced
in earlier parts of Section \ref{sec:sesse}.

\begin{prop}  \label{seeventual}
  For square matrices $A,B$ over $\Z_+$, The following are equivalent.
  \begin{enumerate}
\item $A$ and $B$ are SE-$\Z_+$
\item
  $A^k$ and $B^k$ are SSE-$\Z_+$, for all but finitely many $k$.
  \\
  (So, the SFTs defined by $A$ and $B$ are eventually conjugate\si{eventually conjugate}.)
\item   $A^k$ and $B^k$ are SE-$\Z_+$, for all but finitely many $k$.
\end{enumerate}
\end{prop}
\begin{proof}
\item (1) $\implies$ (2) Suppose matrices $R,S$ give a lag\si{lag} $\ell$
   SE-$\Z_+$ from $A$ to $B$.
   Because $AR=RB$ and $SA=BS$,  we  have for $k $ in $\Z_+$  that
   \begin{align*}
   (A^kR)(S)&=A^k(RS)=A^{k+\ell} \\
    (S)(A^kR) &=S(RB^k) = (SR)B^k=B^{k+\ell} \ .
\end{align*}
\item  $(2) \implies (3)$ This is trivial.
\item $(3) \implies (1)$
  This argument,
   due to Kim\ai{Kim, K.H.} and Roush\ai{Roush, F.W.},
   is not so trivial; see \cite{LindMarcus2021}.\ai{Lind, Douglas}\ai{Marcus, Brian}
   SE-$\Z$ of $A^k$ and $B^k$ does not
   always imply SE-$\Z$ of $A$ and $B$, because there
   are different choices of $k$th roots of eigenvalues. For
   example, consider $A=(3)$, $B= (-3)$ and $k=2$.
The very rough idea of the Kim-Roush\ai{Roush, F.W.}\ai{Kim, K.H.} argument is that when
   $k$ is  a prime very large (with respect to
   every number field generated by the eigenvalues), then the implication
   does reverse.
   \end{proof}


\begin{rem} \label{reducedrings}

If $A$ and $B$ are SE over a ring $\mathcal R$, then
by Theorem \ref{thm:ssesefibers} there is a nilpotent matrix $N$ over $\mathcal R$
such that $B$ is SSE over $\mathcal R$ to
$\left( \begin{smallmatrix} A & 0 \\ 0 & N \end{smallmatrix} \right)$.
For $\mathcal R$ commutative, it follows that
$\det(I-tA)$ fails to be
an invariant of SE-$\mathcal R$  if and only if there is a
nilpotent matrix $N$ over $\mathcal R$ such that $\det(I-tN) \neq 1$.
We check next that this requires $\mathcal R$ to contain a nilpotent
element.

\begin{prop} \label{nilpotentelementsprop}
Suppose $N$ is a nilpotent matrix over a commutative ring $\mathcal R$ and
$\det (I-tN) \neq 1$. Then $\mathcal R$ contains a nilpotent element.
\end{prop}
\begin{proof}
Let $\det(I-tN) = 1 + \sum_{i=1}^k c_it^i$, with $c_k \neq 0$.
Suppose $N$ is $n\times n$, and take $m$ in $\N$ such that $N^m = 0$. Then
the polynomial $\det((I-tN)^m)$ has degree at most $n(m-1)$.
For any $r$, $\det(I-tN)^m = (\det(I-tN))^m = (1 + c_1t+ \dots + c_kt^k)^r$.
This polynomial equals $(c_k)^rt^{kr}$ plus terms of lower degree.
So, for $r > n(m-1)$, we must have $(c_k)^r=0$.
\end{proof}
\end{rem}

\begin{rem}
By the way, it can happen that  matrices $A,B$
shift equivalent over a commutative ring $\mathcal R$ have $\tr (A^n)=\tr (B^n)$ for all $n$
while $\det (I-tA) \neq \det (I-tB)$. For example, let $\mathcal R$
be $\Z \cup \{a\}$, with $a^2=2a=0$. Then set
$A= \left( \begin{smallmatrix} 0 & 1 \\ a & 0 \end{smallmatrix} \right)$
and $B=(0)$.
\end{rem}

\begin{exer}  \label{nonplussed}
The  primitive matrix
$A= \left(\begin{smallmatrix} 1&0&0&1 \\ 0&1&0&1 \\ 0&1&1&0 \\ 1&0&1&0
 \end{smallmatrix}\right) $  has
nonzero\si{nonzero spectrum} spectrum $(2,1)$. Prove that $A$ is not SE-$\Z_+$ to a
nonsingular matrix.
\end{exer}
\begin{proof}
Such a matrix $B$ would be $2\times 2$ primitive with
diagonal entries  $1,2$ (a diagonal entry 3 would force $B$
to have spectral radius greater than 2).
 But, then $B$ has
spectral radius at least as large as the spectral radius of
$\left(\begin{smallmatrix} 1&1 \\ 1&2
 \end{smallmatrix} \right) $, which is greater than 2.
(A more informative obstruction, due to
Handelman,\ai{Handelman, David} shows that $A$ is not SE-$\Z_+$
to a matrix of size less than 4 \cite[Cor. 5.3]{BH93}.)
\end{proof}

In the proof above, we used the following  corollary of
Theorem \ref{PerronTheorem}:  for nonnegative square
matrices $C,B$, with $C\leq B$ and $C\neq B$ and $B$ primitive,
the spectral radius of $B$ is stricty greater than that of
$C$.

\begin{exer} \label{1by1}
   Suppose $A$ is square over $\Z$
   and $\det (I-tA) = 1-nt$, with $n$ a positive integer.
   Then $A$ is SE over $\Z$ to the $1\times 1$ matrix $(n)$.
\end{exer}
\begin{proof}
  (Here we use some basic theory of nonnegative matrices reviewed in Lecture IV.)
  There is a permutation matrix $P$ such that
  $P^{-1}AP$ is block triangular with each diagonal block either
  $(0)$ or an irreducible matrix. Because the nonzero\si{nonzero spectrum} spectrum is a singleton $(n)$,
  only one of these blocks is not zero, and this block $B$ must be primitive.
  There is an SSE-$\Z_+$ by zero extensions\si{zero extension} from $A$ to $B$. Now there is an SE-$\Z$
  from $B$ to $(n)$. Because  $B$ is primitive,
  this implies there is an SE-$\Z_+$ from $B$ to $(n)$.
  \end{proof}

\begin{prop} \label{sePrimitive}
 Suppose two primitive matrices over a subring $\mathcal R$ of the reals are
 SE over $\mathcal R$. Then they are SE over $\mathcal R_+$.
\end{prop}

\begin{proof}
  See  \cite{LindMarcus2021}\ai{Lind, Douglas}\ai{Marcus, Brian} for a proof.
  With  matrices $R,S$ giving a lag\si{lag} $\ell$ SE over $\mathcal R$
  from $A$ to $B$, the basic idea is to use linear
  algebra and the Perron Theorem (Lecture 4) to show
  (possibly after replacing $(R,S)$ with $(-R, -S)$)  that for large
  $n$, the matrices $A^nR$ and $SA^n$ will be positive. Then
  the pair $RA^n, A^nS$ implements an SE over $\mathcal R_+$
  with lag\si{lag} $\ell +2n$.
  \end{proof}

\begin{rem} \label{bkaplansky}
  An example of myself and Kaplansky,\ai{Kaplansky, Irving} recorded in
\cite{BoyleJordanForm1984},
  shows that two irreducible nonnegative matrices can be SE-$\Z$ but not
  SE-$\Z_+$. The example corrects \cite[Remark 4, Sec.5]{ParryWilliams} and
  shows that \cite[Lemma 4.1]{CuntzKriegerDicyclic} should be stated for
  primitive rather than irreducible matrices (the proof is fine for
  the primitive case).

\end{rem}

\begin{prop} \label{transposeexerciseproof}
The matrix
$\begin{pmatrix} 256 & 7\\ 0 &1 \end{pmatrix}$ is not SE-$\Z$ to its
transpose.
\end{prop}
\begin{proof}
First, suppose $a,b$ are integers such that $a> |b| >0$.  Let  $M_x$
denote the matrix $\left( \begin{smallmatrix} a & x \\ 0
    &b \end{smallmatrix} \right)$. Now suppose
$x,y$ are integers such that $xy = 1 \mod (a-b)$. Then
the matrices $(M_x)^{\text{tr}}$ and $M_y$ are SIM-$\Z$:
\[
\begin{pmatrix} a&0 \\ x&b \end{pmatrix}
\begin{pmatrix} a-b&y \\ x&(1-xy)/(b-a) \end{pmatrix}
=
\begin{pmatrix} a-b&y \\ x&(1-xy)/(b-a) \end{pmatrix}
\begin{pmatrix} a&y \\ 0&b \end{pmatrix} \ .
\]
Thus $M_x$ and $(M_x)^{\text{tr}}$ are SE-$\Z$ if and only if
$M_x$ and $M_y$ are SE-$\Z$.
Fix $a=256=2^8, b=1$.  Theorem \ref{thm:axb} implies that
$M_x$ and $M_y$ are SE-$\Z$ if and only if
there are integers $j,m$ such that
$2^mx =  \pm 2^jy \mod 255$. Because $2$ is a unit in $\Z/255\Z$,
and $xy = 1 \mod 255$,
this holds if and only if there is a nonnegative integer $n$ such that
$x^2= \pm 2^n \mod 255$. Because 2 and -2 are not squares mod 5, they
are not squares mod 255. Because $2^8=1 \mod 255$, the only squares
mod 255 in $\{ \pm 2^n: n\geq 0\}$ are 1,4,16 and 64.
The square 49 is not on this list. Therefore the matrix
$\left( \begin{smallmatrix} 256 & 7 \\ 0
    &1 \end{smallmatrix} \right)$ and its transpose are not SE-$\Z$.
\end{proof}

The following fact from  \cite{BH93}
facilitates constructions of primitive matrices realizing
the algebraic invariants above:
any $2\times 2$ matrix over $\Z$ with integer eigenvalues $a,b$ with
$a>|b|$ is SE-$\Z$ to a primitive matrix.
In our example,
\[
\begin{pmatrix} 1 & 0\\ 1 &1 \end{pmatrix}
\begin{pmatrix} 256 & 7\\ 0 &1 \end{pmatrix}
\begin{pmatrix} 1 & 0\\ -1 &1 \end{pmatrix}
\ =
\begin{pmatrix} 249 & 7\\ 248 & 8 \end{pmatrix} \ := B\ .
\]
It is an easy exercise to show that when matrices
$A,B$
are shift equivalent, if one of $A,B$ is shift equivalent to its
transpose then so is the other.
Consequently, the matrix $B$ displayed above cannot be SE-$\Z$ to its
transpose.

I haven't seen the method of Proposition \ref{transposeexerciseproof}
used  to distinguish the SE-$\Z$ classes of a primitive
matrix and its transpose,
but examples of such were produced long ago.
The matrix $A=\left( \begin{smallmatrix} 19&5\\4&1 \end{smallmatrix} \right)$
is an early example, due to K\"{o}llmer,\ai{K\"{o}llmer, Heinrich} of a primitive
matrix not SIM-$\Z$ (hence not SE-$\Z$, as $|\det (A)|=1$) to its transpose
(for an elementary proof, see
\cite[Ch.V, Sec.4]{ParryTuncel1982}).
The connection  of $\text{SL}(2,\Z)$ to continued fractions
leads to a computable characterization of SIM-$\Z$ for $2\times 2$
unimodular matrices, exploited by Cuntz\ai{Cuntz, Joachim}
  and Krieger\ai{Krieger, Wolfgang}
as another method to produce
$2\times 2$  primitive integer matrices not SE-$\Z$ to their transposes
(see \cite[Corollary 2.2]{CuntzKriegerDicyclic}).\begin{footnote}{For
      a dimension group viewpoint, read
\cite[Theorem 2.1]{CuntzKriegerDicyclic}) as:
  the SFTs defined by irreducible unimodular
  $2\times 2$ matrices over $\Z_+$
  are topologically conjugate if and only
  if they have isomorphic dimension groups and equal
  entropy.}\end{footnote}
Lind and Marcus\ai{Lind, Douglas}\ai{Marcus, Brian} use another connection to
$\Z[\lambda ]$ ideal classes to give an example of a primitive integral
matrix not SE-$\Z$ to its transpose,
\cite[Example 12.3.2]{LindMarcus2021}\ai{Lind, Douglas}\ai{Marcus, Brian}.
There are much earlier papers which give many cases
in which a square integer matrix  and its transpose must correspond to inverse
ideal classes of  an associated ring
(see \cite{tausskyInverseIdeal}  and its connections in the literature),
and these ideal classes may differ.
However, this still leaves the issue of realizing the algebraic
invariants in primitive matrices.

Next, we restate Theorem \ref{thm:taussky} and
sketch the proof coming out  of Taussky-Todd's work\ai{Taussky-Todd, Olga}
\cite{tausskyOnLatimerTheorem}.

\begin{thm}
\label{tausskytoddcase}
Suppose $p$ is monic irreducible in $\Z [t]$,
and $p(\lambda)=0$, with
$0\neq \lambda \in \C$.
Let $\mathcal M$ be the set of matrices over $\Z$ with
characteristic polynomial $p$.
Then there are bijections
\begin{align*}
\mathcal M/(\textnormal{SIM}-\Z)
& \to
\textnormal{ Ideal classes  of }\Z[\lambda ]\  ,
\quad \quad  \textnormal{ and } \\
\mathcal M/(\textnormal{SE}-\Z)
& \to
\textnormal{ Ideal classes  of }\Z[1/\lambda ]  \ .
\end{align*}
\end{thm}

\begin{proof}
If $A$ is in $\mathcal M$, then $A$ has a right eigenvector $r_A$
for $\lambda$. The eigenvector can be
chosen with entries in the field $\Q [ \lambda ]$
(solve $(\lambda I- A)r=0$   using
Gaussian elimination). Then, after multipying $r$ by a suitable element
of $\Z$ to clear denominators, we may assume the entries of $r_A$
are in $\Z[\lambda ]$. Let $I(r_A)$ be the ideal of the ring $\Z[\lambda ]$
generated by the entries of $r_A$. Let $\mathcal I (r_A)$ be the
ideal class of $\Z[\lambda ]$ which contains $ I (r_A)$.
Now it is routine to check that the map $A\mapsto \mathcal I (r_A)$
is well defined and   induces the first bijection.

For the second bijection, just repeat this Taussky-Todd argument,
with the ring $\Z[1/\lambda ]$ in place of $\Z [\lambda ]$,
and say $\mathcal J(r_A)$ denoting the $\Z[1/\lambda ]$ ideal
generated by the entries of $r_A$. The rule $\mathcal I(r_A)\mapsto
\mathcal J(r_A)$ induces a surjective map from the set of
ideal classes of $\Z[ \lambda ]$
to those of $\Z[ 1/\lambda ]$, which corresponds to the
lumping of SIM-$\Z$ classes to SE-$\Z$ classes.

There is more detail and comment on this in \cite{BMT1987}.
\end{proof}

It is important to note above that the ring $\Z [\lambda]$
is not in general equal to  $\mathcal O_{\lambda}$, the full ring of
algebraic integers in $\Q [\lambda ]$. When
$\Z [\lambda]$ is a proper subset of
$\mathcal O_{\lambda}$, its class number will strictly exceed that of
$\mathcal O_{\lambda}$
(in this case, a principal $\Z[\lambda ]$ ideal cannot be an
$\mathcal O_{\lambda}$  ideal).

\begin{prop} \label{finiteclass}
Suppose $\lambda$ is a nonzero algebraic integer.
Then the class number of $\Z[\lambda ]$ is finite.
\end{prop}
\begin{proof}
Let $n$ be the the dimension of
$\Q [\lambda ]$ as a rational vector space.
As  free abelian groups,
$\mathcal O_{\lambda}$ and
$\Z[\lambda ] $ (and all of their nonzero ideals) have
rank $n$. For $R$ equal to
$\Z[\lambda ]$
or $\mathcal O_{\lambda}$, the following are equivalent conditions on
$R$-ideals $I,I'$.
\begin{itemize}
\item
$I,I'$ are equivalent as $R$-ideals.
\item there is a nonzero $c\in \Q [\lambda] $ such that $cI=I'$.
\item $I,I'$ are isomorphic as $R$-modules.
\end{itemize}
Because the class number of $\mathcal O_{\lambda}$ is finite,
there is a finite set $\mathcal J$
of $\mathcal O_{\lambda}$ ideals such  that every nonzero
$\mathcal O_{\lambda}$ ideal is equivalent to an element of
$\mathcal J$.
Let $N$ be
a positive integer such that $N\mathcal O_{\lambda}
\subset \Z[\lambda ] $.

Now suppose $I$ is a $\Z [\lambda ]$ ideal, with
$\{\gamma_1, \dots , \gamma_n\}$ a $\Z$-basis of $I$.
Set
$J=\{ \sum_{i=1}^n r_i\gamma_i: r_i \in \mathcal O_{\lambda}, 1\leq i \leq n \}$.
There is a nonzero  $c$ in $\Q[\lambda ]$
such that $cJ \in \mathcal J$.
The $\Z[\lambda ]$ modules
$I, cI$ are
isomorphic.
We have $NJ \subset I \subset J$, and therefore $|J/I| \leq n^N$.
There are only finitely many abelian subgroups of $J$ with
index at most $n^N$ in $J$. It follows that there are only finitely
many possibilities for $cI$ as a $\Z[\lambda ]$ module,
and this finishes the proof.
\end{proof}

Proposition \ref{finiteclass} is a (very) special case of the
Jordan-Zassenhaus Theorem (see \cite{reinerMaximalOrders}).

\begin{rem} \label{directlimits}
We'll recall the general notion of direct limit of a group
endomorphism, and see in the $\Z$ case that our concrete
presentation really is isomorphic to  the general vesion.
The concrete version has its merits, but the general version
is essential.

For a group endomorphism $\phi : G\to G$, take the union of the
disjoint sets $(G,n)$, $n \in \Z_+$. Define an equivalence relation
on $\cup_{n\in \Z_+}(G,n)$: $(g,m) \sim (h,n)$ if there exist $j,k$ in
$\Z_+$ such that $(\phi^j(g),j+m) = \phi^k(h),n+k)$.
Define
$\varinjlim_{\phi} G$
to be the quotient set
$\big( \cup_{n\in \Z_+}(G,n) \big) /\sim $. The operation on
$\varinjlim_{\phi} G$
given
by $[(g,m)] + [(h,n)] = [(\phi^n(g) + \phi^m (h),m+n)]$ is
well defined and makes
$\varinjlim_{\phi} G$
a group. The endomorphism $\phi$
induces a group  automorphism $\widehat{\phi} $  given by
$\widehat{\phi}: [(g,n)]\mapsto [(\phi(g), n)]$. The inverse of
$\widehat{\phi}$ is defined by
$   [(g,n)]\mapsto [(g, n+1)]$.

In our case,  $A : \Z^n \mapsto \Z^n$ by $x\mapsto xA$,
we may define a map $\psi: G_A \to
\varinjlim_{\phi} G$
by $x\mapsto [(xA^m,m)]$ where $m=m(x)$ is sufficiently large that
$xA^m \in \Z^n$. One can check that
$\psi$ is a well defined group automorphism,
with $\psi \circ A = \widehat A \circ \psi$.
\end{rem}

\begin{rem}   \label{nonisomorphicpair}
$A= \left(\begin{smallmatrix} 2&0 \\ 0&5
\end{smallmatrix}\right)$ and
$B =  \left(\begin{smallmatrix} 2&1 \\ 0&5
\end{smallmatrix}\right)$, we will show that
$G_A$ is the sum of a 2-divisible group and a 5-divisible group,
but $G_B$ is not.

For $M=A \text{ or }B$, and $\lambda =2\text{ or } 5 $,
let $H_{M,\lambda } =
\{v \in G_M: \lambda^{-k}v\in G_M,\text{for all } k\in \N \}$.
An isomorphism $G_B\to G_A$ must send $H_{B,\lambda }$ to $H_{A,\lambda }$,
for $\lambda =2,5$.
For $M=A $ or $M=B$,
because the eigenvalues $2,5$  are relatively prime,
we can check
$H_{M,\lambda } = G_M \cap \{ v\in \Q^2: vM= \lambda v  \}  $.
Clearly $G_A = H_{A,2 }\oplus H_{A, 5 }$.
In contrast,
$G_B \neq H_{B,2 }\oplus H_{B, 5 }$. For example, $(1,0)\in G_B$, and
$(1,0)$ is uniquely a sum of vectors on the two eigenlines,
$(1,0) = (1/3) (3,-1)   +(1/3)(0,1)  $. But $(1/3)(0,1) \notin H_{B,5}$. $\qed$
\end{rem}

\begin{prop}  \label{seandpairiso}
Let $A,B$ be square  matrices over $\Z$.
Then the following are equivalent.
\begin{enumerate}
\item
$A$ and $B$ are SE-$\Z$.
\item
There is an isomorphism of direct limit pairs $(G_A, \hat A)$
and $(G_B, \hat B)$.
\end{enumerate}
\end{prop}
\begin{proof}
We will give a proof
with the general  direct limit definition
in Remark \ref{directlimits}, rather than using
the more concrete version of the group
involving eventual images. The general proof is easier.

(1) $\implies $ (2) Suppose $R,S$ gives the lag\si{lag} $\ell$ shift equivalence\si{shift equivalence}:
$A^{\ell}=RS$, etc. First note that the rule
$[(x,n)]\to [(xR,n)]$ gives a well defined map
$\phi: G_A \to G_B$, because
$ [(xAR,n+1)] =[ (xRB, n+1)] = [(xR,n)]$.
Check $\phi$ is a group homomorphism.
Similarly, define $\psi :G_B \to G_A$ by
$[(y, m)]\mapsto [(yS,m+\ell)]$.

Now, $\psi (\phi ([(x,n)])) = \psi ([(xR,n)])
=[(xRS,n+\ell)] = [(xA^{\ell}, n+\ell)] = [(x,n)]$.
Similarly,
$\phi (\psi ([(y,n)])) = [(y,n)]$.
Thus the homomorphism $\phi$ is an isomorphism,
with $\phi^{-1}=\psi$. Finally, $\widehat A \phi = \phi \widehat B$,
because
$\phi(\widehat A([x,n])) = \phi ([xA,n]) =
[(xAR,n)] = [(xRB,n)] = \widehat B([(xR,n)]) =
\widehat B(\phi ([(x,n)]))$.

(2) $\implies $ (1)
Suppose $\phi : G_A \to G_B$ gives the isomorphism of pairs.
Check that there must be $N >0$ and a matrix $R$ such that
$\phi: [(x,0)]\to [(xR, N)]$. After postcomposing with the
automorphism
$[(y,N)]\mapsto [(y,0)]$, we may suppose $N=0$.
There must be a matrix $S$ and $\ell >0$ such that the inverse map is
$[(y,0)] \mapsto [(yS, \ell )]$.

Now, $[(v,0)]=[(vA^{\ell},\ell)]=[(vRS,\ell)]$ for every $v$.
  So, if $v$ is a standard
basis vector, then for all large $k$,
$vA^{\ell +k}= vRSA^k$. Thus for all large $k$, $RSA^k = A^{k+\ell}$.
Similarly, for all large $k$ we have
$SRB^k= B^{k+\ell}$, $(AR)B^k = (RB)B^k$ and $(BS)A^k = (SA)A^k$.
Thus for all large $k$ we get a shift equivalence with lag $\ell+2k$,
\begin{alignat*}{2}
  B(SA^k) &= (SA^k)A \ ,\qquad &
  (RB^k)(SA^k) &= R(SA^k)A^k =A^{\ell +2k} \\
  A(RB^k)  &= (RB^k)B \ , \qquad  &
  (SA^k)(RB^k) &= S(A^kB)B^k = B^{\ell +2k} \ .
 \end{alignat*}
\end{proof}

\begin{rem} (Dimension groups) \label{dimensiongroupsremark}
The dimension groups\si{dimension group} are an important class of
ordered groups arising
from functional analysis \cite{Goodearl1986book},
with important  applications
in $C^*$-algebras \cite{Effros1981}
and topological dynamics
\cite{GPS95, GMPS10, DurandPerrin2022}.
We consider only countable groups.
As a group, a dimension group\si{dimension group}
is a direct limit  of the form
\[
\xymatrix{
\Z^{n_1}  \ \ar[r]^{A_1}  &
\Z^{n_2}  \ \ar[r]^{A_2}  &
\Z^{n_3}  \ \ar[r]^{A_3}   & \ \cdots
}
\]
for which nonnegative integral matrices $A_n$ defined the
bonding homomorphisms.
For $v$ in $\Z^{n_k}$ (we use row vectors),
the element $[(v,n_k)]$ of the group is
in the positive set if
$vA^{n_1}A^{n_2}\cdots A^{n_j} \in \Z_+$
for some (hence for every large)  nonnegative
integer $j$. Every torsion free countable abelian
group is isomorphic as an unordered group to a dimension group\si{dimension group}.
Effros,\ai{Effros, Edward G.}
Handelman\ai{Handelman, David} and Shen\ai{Shen, Chao Liang}
have given an elegant and important
abstract characterization of the ordered groups which
are isomorphic to dimension groups\si{dimension group} \cite{EHS1980}.

The dimension groups\si{dimension group} were introduced to the theory of SFTs
  (where they play a fundamental role) by Wolfgang Krieger\ai{Krieger, Wolfgang}
 \cite{KriegerDimMark1980}, in 1980.
\end{rem}

\begin{rem} \label{WagonerSEPoly}
I (and others) learned the  contents of Propositions 2.7.1,
2.8.1, Corollary 2.8.3 and Remark 2.8.4 from Wagoner in person, at
conferences or at MSRI, by some time in the 1980s or early 1990s.
Early references in print are perhaps somewhat scattered and implicit.
Wagoner's module viewpoint is evident (if not very quotable)
in a 1987 paper \cite[pp. 92,120]{MR908217}.
There is an  explicit statement of the correspondence of
$\mathbb Z[t,t^{-1}]$-module class and SE-$\mathbb Z$ class
in my 1993 review  \cite[Sec. 5.4]{Boyle91matrices}. 
The content of Proposition 2.7.1 is contained in
the standard  1995 Lind-Marcus text
\cite{LindMarcus1995}
(see Theorem 7.5.7, Exercise 7.5.7 and the credit to Wagoner
in the Ch. 7 notes);
 however,  there is no
expicit mention in \cite{LindMarcus1995}
of the $\Z[t]$ modules.
Ordered $\Z_+[t,t^{-1}]$-module versions of SE-$\Z_+$ are given
in  \cite[Lec. III; Secs. 2.2, 3.1, 3.2]{boyleTemuco2000}
and
\cite[Sec. 5]{BW04}.

Mischaikow and Weibel make good use of
the $\Z[t]$-module version of  shift equivalence
in their  study
\cite{MischaikowWeibel2023} of the (homological) Conley
index.
Interesting parts of this go beyond
the contents of our Section \ref{sec:sesse}.
\end{rem}

\begin{rem} \label{SEZ+andgroupsremark}
For a square matrix $A$ over $\Z_+$, the group $G_A$ above
becomes an ordered group,
$(G_A, G_A^+)$,
by defining the positive set
$G_A^+= \{ x\in G_A: \exists k\in \N, xA^k \geq 0 \}$.
The ordered group $(G_A, G_A^+)$ is a
dimension group\si{dimension group} (set every bonding map $A_n$ equal to $A$).
Now  $(G_A, G_A^+,\hat A)$
is an ordered $\Z[t]$ module (the action of $t$ takes
$G_A^+$ to $G_A^+$), and
is sometimes called a dimension module.  (Sometimes the
unordered group $G_A$ is referred to as a dimension group\si{dimension group}.
We have tried to avoid this.)

For $A,B$ over $\Z_+$, SE-$\Z_+$ of $A,B$ is equivalent to
existence of an isomorphism $G_A\to G_B$ which intertwines
$\hat A$ and $\hat B$ and sends $G_A^+$ onto $G_B^+$.
For more on this, see Lind and Marcus\ai{Lind, Douglas}\ai{Marcus, Brian}
\cite{LindMarcus2021}.
\end{rem}

\begin{rem} \label{othercases}
Parry\ai{Parry, William} and Tuncel\ai{Tuncel, Selim} made the first
beyond-$\Z$ connection of this sort in
\cite{ParryTuncel1982stoch}, as they studied
conjugacies of SFTs taking one Markov measure
to another. The matrices they considered are
not taken explicitly from a group ring,
but the connection to an integral group
ring of a finitely generated free abelian
group emerges in  \cite{MarcusTuncelwps}.\ai{Marcus, Brian}
\end{rem}

\begin{rem}
\label{zgsse} It was Bill Parry\ai{Parry, William} who introduced the presentation of
$G$-SFTs by matrices over $\Z_+G$, and the conjugacy/SSE-$\Z_+G$  correspondence.
Parry never published a
proof (although one can see the ideas emerging
from the earlier paper with Tuncel,\ai{Tuncel, Selim}
\cite{ParryTuncel1982stoch}).
For an  exposition with proofs, see
\cite{BS05} and
\cite[Appendices A,B]{BoSc2}.
The items (\ref{corresp2}, \ref{corresp4})
the list in Remark \ref{correspremark} are not proved
explicitly in \cite{BS05},
but they should not be difficult to verify following
the exposition of \cite{BS05}.
For further development of relations between the $\Z_+G$ matrices and
their $G$-SFTs, see
\cite[Appendix]{bce:gfe}.
The exposition in \cite{BS05}
includes Parry's\ai{Parry, William}  connection between SSE-$\Z_+G$  and
cohomology of functions
\cite[Theorem  2.7.1]{BS05},
which is the heart of the matter.
When $G$ is not abelian, one needs to be careful about
left vs. right actions;
\cite[Appendix A]{BoSc2}
explains
this,  and corrects a left/right error in the presentation
in \cite{BS05}.
\end{rem}

\begin{rem}\label{SEmoduletheorydetails}
We spell out some details of the outline given in
Section  \ref{SEmoduletheory}.

By the action $v\mapsto rv$ of $\mathcal R$ on $\mathcal R^n$,
we mean $(v_1, \dots , v_n)\mapsto (rv_1, \dots , rv_n)$,
with $rv_i$ given by multiplication in $\mathcal R$.

With notation as in Remark \ref{directlimits}, an element of
the direct limit group $G=G(A)$ has the form $[(v,m)]$,
with $v\in \mathcal R^n$ and $m$ an integer.
In $G$, $[(v,m)]=[(vA,m+1)]$. Given $r\in \mathcal R$,
for the rule $r[(v,m)]= [(rv,m)]$ to give a well defined
map on $G$ we need   $[(rv,m)]=[(rvA,m+1)]$.
This holds because $[(rv,m)]=[(rv)A,m+1)]=[r(vA),m+1)]$.
Then  $\widehat A$ is an $\mathcal R$-module homomorphism,
because
$r\widehat A: [(v,m)]\mapsto [(rvA,m)] =  \widehat Ar$.
Similarly,
$r\widehat A^{-1} : [(v,m)]\mapsto [(rv,m+1)] =\widehat A^{-1}r$.

By definition,
$\cok(I-tA)= (\mathcal R[t])^n/\{ v(I-tA): v\in (\mathcal R[t])^n\}$.
This cokernel  is an $\mathcal R[t]$-module:
an element  $p$ in $\mathcal R[t]$ acts on $\cok(I-tA)$ by the
  rule $[v] \mapsto [pv]$. The rule
is well defined on  $\cok(I-tA)$ because
 $\{ pv(I-tA): v\in \mathcal R^n\}
  \subset \{ v(I-tA): v\in \mathcal R^n\}$.
     Then  $[v]\mapsto [vA]$  defines
    an $\mathcal R[t]$-module endomorphism ($f$, say) of
    $\cok(I-tA)$ which is an inverse to the action of $t$,
    because $f(t[v])= f([tv])= [tvA]
    =[tvA + v(I-tA)] = [v]$.
\end{rem}

The claimed equivalences of the items in
Theorem \ref{RmodulestructureforSE}
are proved just as for $\mathcal R = \Z$,
with the observations that the  isomorphisms constructed
(e.g., for Prop. \ref{cokanddirlimiso})
give module isomorphisms as required.

\section{Polynomial matrices}\label{sec:polynomial}

We will define SFTs, and the algebraic and classification structures
around them, using polynomial matrices. This is essential for
the K-theory connections to come.
\subsection{Background}\label{subsec:backgroundstuff}
Before we move on to the polynomial matrices,
we review background on  flow equivalence and vertex SFTs. Later,
this  will be  context  for
the polynomial approach.
\subsubsection{Flow equivalence\si{flow equivalence} of SFTs}

Two homeomorphisms are  {\it flow equivalent\si{flow equivalence}} if
there is a
homeomorphism between their mapping tori which takes orbits
onto orbits preserving the direction
of the suspension flow
\apr{floweqbackground}).
Roughly speaking: two homeomorphisms are flow equivalent\si{flow equivalence} if their
suspension flows
move in the same way, but at different speeds.
If SFTs  are topologically
conjugate, then they are flow equivalent\si{flow equivalence}, but the converse is
not true.

An $n\times n$ matrix
$C$ over $\Z$ defines a map $\Z^n \to \Z^n$, $v\mapsto vC$,
with $\text{Image} (C) = \{vC: v\in \Z^n\}$, and
cokernel group
$\cok_{\Z}(C)  =
\Z^n / \text{Image} (C)$.

\begin{thm} If SFTs  defined by $\Z_+$ matrices $A,B$ are
flow equivalent\si{flow equivalence}, then
\begin{enumerate}
\item
$\det (I-A) = \det (I-B)$ , and
\item
$\cok_{\Z}(I-A)$ and $\cok_{\Z}(I-B)$
are isomorphic abelian groups.
\end{enumerate}
\end{thm}
Above, (1) is due to Bill Parry\ai{Parry, William} and Dennis
Sullivan \cite{parrysullivan};\ai{Sullivan, Dennis}
(2)
is due to Rufus Bowen\ai{Bowen, Rufus} and John Franks\ai{Franks, John}
\cite{BowenFranks1977}. The group
$\cok_{\Z}(I-A)$ is called the Bowen-Franks\si{Bowen-Franks group} group of the SFT
defined by  $\Z_+$-matrix $A$.  The group
$\cok_{\Z}(I-A)$  determines $|\det (I-A)|$, except for
the sign of
$\det (I-A)$ in the case $\det (I-A)\neq 0$
\apr{bfgroup}).

When $A$ is irreducible  and
$A$ is not a permutation matrix, the converse
of the theorem holds (John Franks, \cite{franksFlowEq1983}).\ai{Franks, John}
So, in this case the Bowen-Franks\ai{Bowen, Rufus}
group\si{Bowen-Franks group} determines the flow equivalence\si{flow equivalence} class, up to
knowing the sign of $\det (I-A)$.

\subsubsection{Vertex SFTs}

Once upon a time, before edge SFTs, SFTs were presented only by
matrices with entries in $\{0,1\}$.
Such a matrix can be viewed as the adjacency matrix of a graph
without parallel edges (i.e., for each vertex pair $(i,j)$,
there is  at most one edge from $i$ to $j$).
We can then define a ``vertex SFT'' as we defined edge SFT,
but using bisequences of vertices rather than bisequences
of edges to describe infinite walks through the graph.

A vertex SFT is quite natural,
especially if one starts from subshifts.
A subshift $(X, \sigma)$ is a ``topological Markov shift''
if whenever points $x,y$ satisfy $x_0=y_0$,  the bisequence
$z= \dots x_{-3}x_{-2}x_{-1}x_0y_1y_2y_3 \dots $
is also a point in
$X$. (That is, the past of $x$ and the future of $y$ can be glued
together at their common present
to form a point. This is a topological analogue of the
independence property of a Markov measure.)
One can check that   a topological Markov shift
is the same object as  a vertex SFT, with the
alphabet of the subshift being the vertex set
\apr{topmarkovshift}).

\begin{rem}
Defining  SFTs (as edge SFTs) with matrices over $\Z_+$
has some significant  advantages over defining SFTs (as vertex SFTs) with
matrices over $\{0,1\}$, as follows.
\begin{itemize}
\item
{\it Functoriality.}
Recall,
$(X_A, (\sigma_A)^n)$ is conjugate to the edge SFT defined by
$A^n$,  whereas $A^n$ cannot define a vertex
SFT if $A^n$ has an entry greater than 1.

\item
{\it Conciseness. }
E.g., an edge SFT defined by the perfectly transparent
$2\times 2 $ matrix
$ A= \begin{pmatrix} 1 & 4 \\ 4& 15 \end{pmatrix}$ has
a (rather large)  alphabet of 24 symbols;
as a vertex SFT, it would be defined by a $24 \times 24$
zero-one matrix. And while $A^n$ is $2\times 2$ for all $n$,
the size of the matrix presenting
the vertex SFT $(X_A, (\sigma_A)^n)$ goes to infinity as $n\to \infty$.

\item
{\it  Proof techniques.}  Defining the SFTs directly with matrices
over $\Z_+$ allows other proof techniques \apr{prooftechniquesforedgesft}).
\end{itemize}
\end{rem}
We'll see that some  advantages of
defining
SFTs with $\Z_+$ rather than $\{0,1\}$ matrices are repeated,
as we compare defining SFTs with polynomial rather than $\Z_+$
matrices.

\subsection{Presenting SFTs with polynomial matrices}

The {\it length} of a path $e_1 \dots e_n$ of
$n$ edges in a graph is $n$.  (We also think of $n$ as the time taken at unit
speed to traverse the path.)
An $n\times n$ matrix $A$ with polynomial entries in $t\Z_+[t]$
presents a graph
$\Gamma_A$ as follows.
\begin{itemize}
\item $\{ 1, \dots , n\}$ is a subset of the vertex set of $\Gamma_A$.
\item For each monomial entry $t^k$ of $A(i,j)$, there is a distinct path of
$k$ edges from
vertex $i$ to vertex $j$.
We call such a path an {\it elementary path} in
$\Gamma_A$.
(E.g. if $A(i,j) = 2t^3$, then from $i$ to $j$ there are
two elementary paths of length 3.)
\item There are no other edges, and distinct elementary
paths do not intersect at intermediate vertices.
\end{itemize}
Above,    the vertex set  $\{ 1, \dots , n\}$  is a  {\it rome}
\apr{rome})
for the graph $\Gamma_A$:
every sufficiently long path hits the rome.
(\lq\lq All roads lead to Rome \dots\rq\rq)

\begin{ex} \label{Asharpgraphexample}
Below, the rome
vertex set is $\{1,2\}$;  the additional vertices are unnamed black
dots; and there are five elementary paths in $\Gamma_A$.
\[
A=   \begin{pmatrix}  2t & t^2 + t^3 \\  t^2&0  \end{pmatrix}\ ,
\quad \quad
\Gamma_A \ = \quad   \xymatrix{
& \bullet     \ar@/^1pc/[rr]
& \bullet \ar@/^/[rrd]
& \bullet        \ar@/^/[rd]
& \\
*+[F-:<2pt>]{1} \ar@(ul,u) \ar@(ul,l) \ar@/^/[ru]\ar@/^/[rru]
& & \bullet \ar[ll]
& & *+[F-:<2pt>]{2} \ar[ll]
}
\]
\end{ex}
%

Given $A$ over $t\Z_+[t]$, let
$A^{\sharp}$ be the  adjacency
matrix  for the graph $\Gamma_A$.
In Example \ref{Asharpgraphexample},
$A^{\sharp}$ would be $6\times 6$.
(The vertex set of the graph is the rome,
together with $k-1$ additional vertices for each monomial $t^k$.)
We can
think of $A$ as being a way to present the edge SFT defined
by the matrix         $A^{\sharp}$.

{\it  Conciseness.}
Obviously, we can present many SFTs (and, various interesting
families of SFTs) much more concisely with
polynomial matrices than with  matrices over $\Z_+$.
For example, a theorem of D. Perrin shows that
any number which can be the entropy of
an SFT is the entropy of an SFT defined by a $2\times 2$
matrix over $t\Z_+[t]$. \apr{perrin})

\begin{de} An elementary matrix is a square matrix equal to the identity
except in at most a single offdiagonal entry.
\end{de}

The polynomial presentation offers more than
conciseness. To see this, we need a little preparation.
$I_k$ denotes the $k\times k$ identity matrix.
\begin{de}  Suppose $\mathcal R$ is a ring.
{\it Stabilized elementary equivalence\si{stabilized elementary equivalence}} is the equivalence relation
$\sim$ on
square matrices $C$ over $\mathcal R$ generated by the following
two relations.
\begin{enumerate}
\item
$C \sim C\oplus I_k$ , for $k\in \N.\ \ $
(E.g., $(2)\sim \begin{pmatrix} 2&0\\0&1
\end{pmatrix}$ . )

\item
$C \sim D$ if there is an elementary matrix $E$ such that
$ D= CE$ or  $D=EC$.
\end{enumerate}
\end{de}
Above, condition (1) is the ``stabilized'' part.
A stabilized elementary equivalence\si{stabilized elementary equivalence} from $C$ to $D$
is a finite sequence of
the elementary matrix moves, taking $C$ to $D$.

Given $C\sim D$,  for either type of relation,  we have
\begin{enumerate}
\item     $\det C = \det D$, if $\mathcal R$ is commutative, and
\item  the  $\mathcal R$-modules
$\cok (C) $,  $\cok (D)$ are isomorphic
\apr{fromequivtocok}).
\end{enumerate}
When  working in a stable setting, we often say just ``elementary
equivalence'' instead of ``stabilized elementary equivalence\si{stabilized elementary equivalence}''.

\subsection{Algebraic invariants in the polynomial setting}
\begin{ex}If all nonzero entries of $A$ have degree one,
then the relation of $A$ and $A^{\sharp} $ is obvious: for
example,
\[
A=\begin{pmatrix} t& 2t \\ t & 0 \end{pmatrix} \
= \ t         \begin{pmatrix} 1& 2 \\ 1 & 0 \end{pmatrix} \
= \ t A^{\sharp}
, \qquad \quad \quad
\Gamma_A \quad =
\quad
\quad
\xymatrix{\cdot   \ar@(lu,ld)
\ar@/^/[r]
\ar@/^2pc/[r]
&\cdot      \ar@/^/[l]
}
\]
Here, $I-A$ equals $I-tA^{\sharp}$. It follows, of course, that
\begin{align*}
\det (I-A) &= \det (I-tA^{\sharp}) \ , \quad \text{and} \\
\cok_{\Z[t]} (I-A) &\cong \cok_{\Z[t]}  (I-tA^{\sharp})  \ .
\end{align*}
\end{ex}
These two statements hold for general $A$ over $t\Z_+[t]$,
for the following reason.
\begin{prop}  There is a stabilized elementary equivalence\si{stabilized elementary equivalence} over
the ring $\Z[t]$
from $I-A$ to $I-tA^{\sharp}$.
\end{prop}
Next we'll see the essential ideas of the proof of the proposition.
Given $n\times n$ $A$ over $t\Z_+[t]$, let $\mathcal H_A$ be the
$n\times n$  labeled  graph in which a monomial $t^k$ of $A(i,j)$
gives rise to an edge from $i$ to $k$ labeled $t^k$.
\begin{ex}
\[
A\ =\    \begin{pmatrix}  2t & t + t^4 \\ t^2&0  \end{pmatrix}\ ,
\qquad \qquad
\mathcal H_A\ =\
\xymatrix{
*+[F-:<2pt>]{1}
\ar@(u,ul)_t \ar@(d,dl)^t
\ar@/^/[rr]^{t}
\ar@/^2pc/[rr]^{t^4} & &
*+[F-:<2pt>]{2}   \ar@/^/[ll]^{t^2 }
}
\]
Note:  the graph $\Gamma_A$
with adjacency matrix $A^{\sharp}$ is  obtained from $\mathcal H_A$
by replacing each path labeled $t^k$ with a path of length $k$.
The graph $\mathcal H_{tA^{\sharp}}$ is
the graph $\Gamma_A$ with each edge
labeled by $t$.

We  can decompose
the graph move $\mathcal H_A \to \mathcal H_{tA^{\sharp}}$ into  steps,
$\mathcal H_0 \to \mathcal H_1 \to \cdots \to \mathcal H_4$,
with one vertex  added at each step.
The labeled graph
$\mathcal  H_{i+1}$
is obtained from $\mathcal H_{i}$ by replacing
some  edge labeled $t^k$ with a path of two edges:
an edge labeled $t$ followed by an edge labeled $t^{k-1}$.
There will be  matrices
$A_i$ over $t\Z_+[t]$ such that
$\mathcal H_i = \mathcal H_{A_i}$, with $A_0=A$ and $A_4=tA^{\sharp}$.
Here is the data for the  step $\mathcal H_0 \to\mathcal H_1$:
\[
A=A_0\ =\    \begin{pmatrix}  2t & t + t^4 \\ t^2&0  \end{pmatrix}\ ,
\qquad \qquad
\mathcal H_A\ =\ \mathcal H_0 \ = \
\xymatrix{
*+[F-:<2pt>]{1}
\ar@(u,ul)_t \ar@(d,dl)^t
\ar@/^/[rr]^{t}
\ar@/^2pc/[rr]^{t^4} & &
*+[F-:<2pt>]{2}   \ar@/^/[ll]^{t^2 }
}
\]
\[
B=   A_1\ =\   \begin{pmatrix}  2t & t & t \\ t^2&0&0 \\ 0 &t^3 &0
\end{pmatrix}\ ,
\qquad \qquad
\mathcal H_{A_1}\ =\
\mathcal H_{1}\ =\
\xymatrix{
&   *+[F-:<2pt>]{3} \ar@/^/[dr]^{t^3} & \\
*+[F-:<2pt>]{1}
\ar@(u,ul)_t \ar@(d,dl)^t
\ar@/^/[rr]^{t}
\ar@/^/[ru]^{t} & &
*+[F-:<2pt>]{2}   \ar@/^/[ll]^{t^2 }
}
\]

Let us see how the move  $A \to A_1$ in the example above
is accomplished at the matrix level, by a stabilized elementary
equivalence\si{stabilized elementary equivalence} over the ring $\Z[t]$.

First, define the matrix
$A\oplus 0  = \begin{pmatrix}  2t & t + t^4 & 0\\
t^2&0& 0\\ 0&0&0  \end{pmatrix}$.
The move $A \to A\oplus 0$
is the same as the elementary stabilization move
$(I-A) \to (I-A)\oplus 1$.
Then  multiply $(I-A)\oplus 1$ by elementary matrices
to get $(I-A_1)$. This is a small computation:
\begin{align*}
(I-B)&=
\begin{pmatrix}  1 & 0 & 0 \\ 0&1&0 \\ 0 & 0 & 1
\end{pmatrix}
-    \begin{pmatrix}  2t & t & t \\ t^2&0&0 \\ 0 &t^3 & 0
\end{pmatrix} \\
(I-B)E_1 &=
\begin{pmatrix}  1-2t & -t & -t \\ -t^2&1&0 \\ 0 &-t^3 & 1
\end{pmatrix}
\begin{pmatrix} 1&0&0 \\ 0&1&0 \\ 0& t^3 & 1 \end{pmatrix} \\
& =  \begin{pmatrix}  1-2t & -t -t^4 & -t \\ -t^2&1&0 \\ 0 &0 & 1
\end{pmatrix}
= I -  \begin{pmatrix}  2t & t +t^4 & t \\ t^2&0&0 \\ 0 &0 & 0
\end{pmatrix}
:= I-C
\\
E_2 (I-C) &=
\begin{pmatrix} 1&0&t \\ 0&1&0 \\ 0& 0 & 1 \end{pmatrix}
\begin{pmatrix}  1-2t & -t -t^4 & -t \\ -t^2&1&0 \\ 0 &0 & 1
\end{pmatrix} \\
& =
\begin{pmatrix}  1-2t & -t -t^4 & 0 \\ -t^2&1&0 \\ 0 &0 & 1
\end{pmatrix}
\\
&=
\begin{pmatrix}  1 & 0 & 0 \\ 0&1&0 \\ 0 & 0 & 1  \end{pmatrix}
-
\begin{pmatrix}  2t & t +t^4 & 0 \\ t^2&0&0 \\ 0 &0 & 0
\end{pmatrix}
= (I-A) \oplus 1 \ .
\end{align*}
\end{ex}
The example computation above contains the ideas of the
general proof that  there is a stabilized elementary equivalence\si{stabilized elementary equivalence}
from $(I-A)$ to $(I-tA^{\sharp})$.

\begin{coro}
Let $A$ be a square matrix over $t\Z_+[t]$, with
$A^{\sharp}$ the adjacency matrix of $\Gamma_A$. Then
\begin{enumerate}
\item
$\det (I-A) = \det (I-tA^{\sharp})$.
\item  The  $\Z[t]$-modules
$\cok (I-A) \ ,\  \cok (I-tA^{\sharp} )$ are isomorphic.
\end{enumerate}
\end{coro}
\begin{proof}
The claim follows because the matrices $(I-A)$, $(I-tA^{\sharp})$ are
related by a string of the two relations $\sim$ generating
stabilized elementary equivalence\si{stabilized elementary equivalence}.
\end{proof}

Thus algebraic data of the polynomial matrix $(I-A)$
captures
\begin{enumerate}
\item the nonzero\si{nonzero spectrum} spectrum
(by $\det (I-A)$),  and
\item  the SE-$\Z$ class of $A^{\sharp}$
(by the isomorphism class of the $\Z[t]$-module
$\cok (I-A)$).
\end{enumerate}

\subsection{Polynomial matrices: from  elementary equivalence
to conjugate SFTs}
For square matrices $A,B$ over $\Z_+$,
the SFTs $(X_A, \sigma), (X_B, \sigma)$ are
topologically conjugate  if and only if
$A,B$ are SSE-$\Z_+$.
We will find a relation on polynomial matrices corresponding
to topological conjugacy of the SFTs they define.
\begin{nota} With $i\neq j$,  let $E_{ij}(x)$ denote
the  elementary  matrix
with $(i,j)$ entry defined to
be $x$, and other entries matching the identity.
The size of the square  matrix $E_{ij}(x)$ is suppressed from
the notation (but evident in context).
E.g.,
$E_{12}(t^2)$ could denote $ \begin{pmatrix} 1&t^2\\0&1 \end{pmatrix}$ or
$ \begin{pmatrix} 1&t^2&0\\0&1&0 \\ 0&0&1  \end{pmatrix}$.
\end{nota}
There is now a very pleasant surprise.
\begin{thm} \label{thm:ElEqtoConj}
Suppose  $A,B$ are square matrices over $t\Z_+[t]$,
with $E(I-A) = (I-B)$ or $(I-A)E=(I-B)$, where
$E=E_{ij}(t^k)$.

Then $A,B$
define topologically
conjugate SFTs (i.e., $B^{\sharp}$ and $A^{\sharp}$ define
topologically  conjugate edge SFTs).  \end{thm}
\begin{rem}  Suppose
$E=E_{ij}(t^k)$, $A$ is square with entries in $t\Z_+[t]$, and
$(I-B)=E(I-A)$ or $(I-B)=(I-A)E$.
Then one easily checks (it will be obvious from the next example)
that the following are equivalent:
\begin{enumerate}
\item
The entries of $B$ are in $t\Z_+[t]$.
\item
$A(i,j) -t^k \in t\Z_+[t]$.
\end{enumerate}
\end{rem}
\begin{proof}[Proof ideas for Theorem \ref{thm:ElEqtoConj}]
The ideas of the proof of Theorem \ref{thm:ElEqtoConj}
should be clear from  the next example.
\begin{ex} \label{proofideasexample} Suppose $A$ is  matrix over
$t\Z_+[t]$,
$A=\begin{pmatrix} a&b+t^3 &c \\ d&e&f \\ g&h&i  \end{pmatrix} $,
with $b\in t\Z_+[t]$ (i.e., not only $A(1,2)$, but
also $A(1,2) -t^3$, is in  $t\Z_+[t]$).
Now multiply $I-A$ from the left by the elementary matrix
$E=E_{12}(t^3)$,
\begin{align*}
E(I-A) & =
\begin{pmatrix} 1&t^3&0 \\ 0&1&0\\ 0&0&1 \end{pmatrix}
\begin{pmatrix} 1-a&-b-t^3 &-c \\ -d&1-e&-f \\
-g&-h&1-i  \end{pmatrix} \\
&=       \begin{pmatrix} 1-a \mathbf{+ t^3(-d)}
&-b -t^3 \mathbf{+t^3(1-e)}
&-c \mathbf{+t^3(-f)} \\ -d&1-e&-f \\
-g&-h&1-i  \end{pmatrix}
\\
&=
\begin{pmatrix} 1&0&0 \\ 0&1&0\\ 0&0&1 \end{pmatrix} -
\begin{pmatrix} a +t^3d & b +t^3e &c +t^3f \\ d&e&f \\
g&h&i  \end{pmatrix}         \ .
\end{align*}

We then define a matrix $B$ over $t\Z_+[t]$
by setting $I-B=E(I-A)$, so,
\[
A=\begin{pmatrix} a&b+t^3 &c \\ d&e&f \\ g&h&i  \end{pmatrix} , \quad
B=\begin{pmatrix} a +t^3d &b+t^3e &c +t^3f\\ d&e&f \\ g&h&i  \end{pmatrix}\ .
\]
\end{ex}
{\it Producing $\Gamma_A$ from $\Gamma_B$.}
Suppose
$\tau=\tau_1\tau_2\tau_3$ is the elementary
path in $\Gamma_A$ from vertex 1 to
vertex 2
corresponding to the term $t^3$ above. Let $|p|$ denote the
length (number of edges) in a graph path $p$.
We obtain  $\Gamma_B$ from $\Gamma_A$ as follows.
\begin{enumerate}
\item
Remove the elementary path $\tau$ from $\Gamma_A$;
\item  For each elementary path $\nu$ of $\Gamma_A$ beginning
at vertex 2, put in an elementary path
$\widetilde{\nu}$ beginning at vertex 1, such that
\begin{enumerate}
\item
$|\widetilde{\nu}| = |\tau| + |\nu| = 3+ |\nu| $,
\item the terminal vertices of $\nu$ and $\widetilde{\nu}$ agree.
\end{enumerate}
\end{enumerate}
For example,
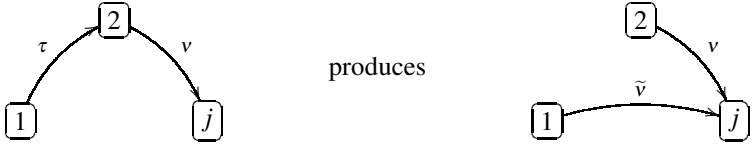
\[
\xymatrix{
&   *+[F-:<2pt>]{2} \ar@/^/[dr]^{\nu} & \\
*+[F-:<2pt>]{1}
\ar@/^/[ru]^{\tau}
& &
*+[F-:<2pt>]{j}
}
\qquad \qquad
\begin{matrix} \\ \\ \\ \text{produces} \end{matrix}
\qquad \qquad
\xymatrix{
&   *+[F-:<2pt>]{2} \ar@/^/[dr]^{\nu} & \\
*+[F-:<2pt>]{1}
\ar@/^/[rr]^{\widetilde{\nu}}
& &
*+[F-:<2pt>]{j}
}
\]
with  $|\widetilde{\nu}| = |\tau| + |\nu|$ .

{\it Defining the conjugacy  $\phi: X_{A^{\sharp}} \to X_{B^{\sharp}}$.}
Wherever the elementary path  $\tau$
occurs in a point $x$ of $X_A$,
it must be followed
by an elementary path $\nu$.
Now define  $ \phi (x)$ be replacing each
path $\tau \nu$
with the elementary path $\widetilde{\nu}$:
\begin{itemize}
\item
If $x_{k+1}\dots x_{k+|\tau\nu|} = \tau \nu$ ,
$\ \ \ \text{with } \nu$ an elementary path in $\Gamma_A$, \\
then $(\phi x)_{k+1}\dots (\phi x)_{k+|\tau\nu|} = \widetilde{\nu}$.
\item Otherwise, $(\phi x)_n=x_n$ .
\end{itemize}
If we look at a succession of elementary paths $\tau$ and
$\nu_i$, the code looks like:
\[
\xymatrix{
\dots \ \nu_{-1} \ \
\tau \nu_1 \ \nu_2 \ \tau \nu_3 \ \nu_4\ \tau \nu_5 \
\ \nu_6 \ \nu_7   \ \dots
\ar[d]_{\phi } \\
\dots \ \nu_{-1} \ \ \
\widetilde{\nu_1}  \ \ \nu_2
\ \ \widetilde{\nu_3} \ \ \nu_4\ \
\widetilde{\nu_5}\ \ \ \nu_6 \ \nu_7  \ \dots
}
\]
This map $\phi$ is well defined because
an elementary path $\nu$ following $\tau$
has no edge in common  with $\tau$ (because
the initial and terminal vertices of $\tau$ are different)
\apr{tautau}).
Given that $\phi$ is well defined, it is straightforward
to check that $\phi$ defines a topological conjugacy
$(X_{A^{\sharp}}, \sigma)\to  (X_{B^{\sharp}}, \sigma) $.

Above, we considered $E(I-A) = (I-B)$.
Suppose instead we define a matrix $C$  by $(I-A)E = (I-C)$.
No surprise:   the matrix $C$ also defines an SFT
conjugate to that defined by $A$. In this case, instead of a
conjugacy $X_A\to X_B$ based on $\tau \nu\mapsto \widetilde{\nu}$ as above,
we have a conjugacy $X_A\to X_C$ based on
$\nu \tau \mapsto \widetilde{\nu}$, where $\nu$
is an elementary path in $\Gamma_A$ with {\it terminal} vertex 1.
\end{proof}

If $C$ is a square matrix, then $C\oplus 1$ is the square matrix with
block form $\begin{pmatrix} C&0 \\ 0&1 \end{pmatrix}$.
\begin{de} (Positive equivalence\si{positive equivalence}) \apr{poseqterminology})
Suppose $\mathcal P$ is a subset of a ring $\mathcal R$.
Let  $\mathcal S$ be a set of square matrices over a ring
$\mathcal R$ which is ``1-stabilized'' :
\[
C\in \mathcal S \implies (C \oplus 1) \in \mathcal S\ .
\]
Positive equivalence\si{positive equivalence} of matrices in $\mathcal S$
(with respect to $\mathcal P$)
is the equivalence relation on $\mathcal S$
generated by the following
relations
(where $C,D$ must {\it both} be in $\mathcal S$):
\begin{enumerate}
\item $C \sim C\oplus 1$ .
\item $EC=D$ or $ CE=D$ , where
$E=E_{ij}(r)$, with $i\neq j$ and $r\in \mathcal P$.
\end{enumerate}
If $\mathcal P$ is not specified, then by default we assume
$\mathcal P= \mathcal R$.
\end{de}
For $I-A$ in $\mathcal S$,
the requirement that $\mathcal S$ is closed under the move
$(I-A) \to (I-A)\oplus 1 $
is equivalent to the requirement that
the set $\{A \colon I-A \in \mathcal{S}\}$
is closed under   the
move $A\to A\oplus 0$.

Now suppose
$\mathcal R = \Z [t]$. We let
$\mathcal M(X)$ denote the set of matrices with entries in
a set $X$, and set $I- \mathcal M(X)= \{ I-A: A \in \mathcal M(X)\}$.
Suppose
$\{ A,B \} \subset \mathcal M(\Z_+[t])$;
$E=E_{ij}(f) $, with $f\in \Z [t]$; and
$(I-B)=E(I-A)$. Writing $f=c-d$ with $\{c,d\} \subset \Z_+[t]$,
  we see $A_{ij}-c \in \Z_+[t]$, and then $E_{ij}(c)(I-A)=(I-C)$ with
  $C\in \mathcal M(\Z_+[t])$. This equivalence is a composition of
  equivalences of the form $E_{ij}(t^k)(I-A_r)=(I-A_{r+1})$,  with
  $\{ A_r , A_{r+1}\}\subset \mathcal M(\Z_+[t])$ and $k\geq 0$.
  Considering likewise  $(I-C) = E_{ij}(d)$, we see that the
  positive equivalence\si{positive equivalence} of
  $I-A$, $I-B$ with respect to $\mathcal P = \mathcal R$
  gives rise to a
    positive equivalence\si{positive equivalence}
    with respect to $\mathcal P = \{ t^k: k\geq 0\}$. The same holds if
    $(I-A)E=(I-B)$. To summarize, matrices in
    $I-\mathcal M(\Z_+[t])$ are positive equivalent with respect
    to $\mathcal P= \mathcal R$ if and only if they are
    positive equivalent with respect to $\mathcal P=
    \{ t^k: k \geq 0 \} $. Given $\{A,B\} \subset
    \mathcal M(t\Z_+[t])$, the possibility $k=0$ can be excluded,
    and the positive equivalence\si{positive equivalence} for $\mathcal P=\mathcal R$
    gives rise to a positive equivalence\si{positive equivalence} over
    $\mathcal P = \{ t^k: k > 0 \} $.

    It then follows from Theorem \ref{thm:ElEqtoConj}
    that positive equivalent matrices in
    $I-\mathcal M(t\Z_+[t])$  define topologically
conjugate SFTs.
Also, if $I-A \in I-\mathcal (t\Z_+[t])$,
then so is $(I-A)\oplus 1$, and
$A$ and $A \oplus 1$ define conjugate SFTs.
Consequently we have the following.

\begin{thm} \label{tMatricesAndTopConjug}
Suppose matrices $(I-A)$ and $(I-B)$ are positive equivalent in
$I -\mathcal M(t\Z_+[t])$.
Then $A,B$ define topologically conjugate SFTs.
\end{thm}

The converse of Theorem \ref{tMatricesAndTopConjug} is  ``true
up to a technicality'' \apr{exampleforneednzc}).
For a true converse, we
expand the collection of matrices allowed to present SFTs,
from $\mathcal M(t\Z_+[t])$ to a slightly larger class,
NZC\si{NZC condition}.
(On first exposure, it is fine to pretend
$\text{NZC}=\mathcal M(t\Z_+[t])$.
But we'll give statements for NZC\si{NZC condition}, just to tell
the truth.)

\subsection{Classification of SFTs by positive equivalence\si{positive equivalence} in I-NZC}

For a matrix $M$ over $\Z[t]$,
let $M_0$ be $M$ evaluated at $t=0$.

\begin{de}  Let NZC\si{NZC condition} be the set of square matrices $A$ over
$\Z_+[t]$ such that $A_0$ is nilpotent.
\end{de}

\begin{ex}   $A$ and $B$ are in NZC\si{NZC condition}; $C$ and $D$ are not:
\begin{alignat*}{4}
A& =  \begin{pmatrix} t^3 + t & 3t^5 \\ t & 3t^5 \end{pmatrix}\ , \quad
& B& = \begin{pmatrix} t^3  & 1 \\ t & 3t^5 \end{pmatrix}\ , \quad
& C& = \begin{pmatrix} 1  \end{pmatrix} \ , \quad
&  D & = \begin{pmatrix} t^3  & 5t^2 +2 \\ 1+t^7 & 3t^5 \end{pmatrix} \ , \\
A_0  & =  \begin{pmatrix} 0 & 0 \\ 0 & 0 \end{pmatrix}\ , \quad
&   B_0 & = \begin{pmatrix} 0  & 1 \\  0 & 0 \end{pmatrix}\ , \quad
&   C_0 & = \begin{pmatrix} 1  \end{pmatrix} \ , \quad
&   D_0 & = \begin{pmatrix} 0  & 2 \\ 1 & 0 \end{pmatrix} \ .
\end{alignat*}
\end{ex}
If $A$ is in $\mathcal M(t\Z_+[t])$, then $A_0=0$. A matrix in NZC\si{NZC condition}
can have some entries with nonzero constant term, but not too many.
\\ \\
Why the term NZC?\si{NZC condition}  Here is the heuristic.
\\ \\
If e.g.  $A(i,j) =t^4$, then
in the graph with adjacency matrix $A^{\sharp}$,
there is an elementary path, from $i$ to $j$,  of 4
edges.  We consider this a path taking 4 units of time to
traverse. The time to traverse concatenations of elementary
paths is  the sum of the times for its elementary paths.
A nonzero term $1$ in $A_0$ is considered as $1=t^0$, giving
a path taking zero
time to traverse. ``NZC\si{NZC condition}'' then refers to ``No Zero Cycles'',
where a zero cycle is a  cycle
taking zero time to traverse.
\\ \\
In the case NZC\si{NZC condition}, one can make good sense of this heuristic,
and everything works \apr{evenmore}).
But for a matrix $A$ over $\Z_+[t]$
with zero cycles,
we can't make sense of how $A$ defines an SFT (let alone
how multiplication by elementary matrices might induce topological
conjugacies).

We do get a classification statement parallelling the
SSE-$\Z_+$ setup of Williams\ai{Williams, R.F.}.

\begin{thm}
For matrices  $A,B$
in $ \text{NZC}$\si{NZC condition}, The following are equivalent.
\begin{enumerate}
\item  $(I-A)$ and $(I-B)$ are positive equivalent in $I-\text{NZC}$\si{NZC condition}.
\item  $A,B$ define topologically conjugate SFTs. \\
\end{enumerate}
\end{thm}

\begin{proof} $(1)\implies (2)$ We have seen this for positive equivalence\si{positive equivalence}
in $I-\mathcal M(t\Z_+[t])$. This works similarly  for matrices
in $I-\text{NZC}$\si{NZC condition} \cite {BW04,B02posk}.
Note, for matrices $A,B$ from $\mathcal M(\Z_+[t])$, and
  $E$ a basic elementry matrix with $(I-B)=E(I-A)$ or
  $(I-B)=(I-A)E$,
  we have
  $A \in NZC \iff B  \in NZC$\si{NZC condition}. One can check this by considering a
  correspondence of cycle paths, similar to the correspondence of paths
in  Example \eqref{proofideasexample}. Alternately, one can use
  that for $A\in \mathcal M(\Z_+[t])$, we have
  $A \in NZC \iff (\det(I-A))|_{t=0} =1$.\si{NZC condition}

$(2)\implies (1)$ E.g., $(I-A)$ is positive equivalent in I-NZC\si{NZC condition}
to the matrix $I-tA^{\sharp}$, likewise $(I-B)$. So it suffices
to get the positive equivalence\si{positive equivalence} for matrices $(I-tA^{\sharp})$,
$(I-tB^{\sharp})$, assuming the edge SFTS for $A^{\sharp}, B^{\sharp}$
are conjugate, i.e. the $\Z_+$ matrices $A^{\sharp}, B^{\sharp}$
are SSE over $\Z_+$. It suffices to show the positive equivalence\si{positive equivalence}
given an elementary SSE, $A^{\sharp}=RS,  B^{\sharp}=SR$.
For this, define matrices $A_0, A_1, \dots ,A_4$ in NZC\si{NZC condition} :
\[
\begin{pmatrix} tRS & 0 \\ 0 & 0 \end{pmatrix}
\ , \
\begin{pmatrix} tRS & 0 \\ tS & 0 \end{pmatrix}
\ , \
\begin{pmatrix} 0 & R \\ tS & 0 \end{pmatrix}
\ , \
\begin{pmatrix} 0 & 0 \\ tS & tSR \end{pmatrix}
\ , \
\begin{pmatrix} 0 & 0 \\ 0 & tSR \end{pmatrix} \ .
\]
(Notice, $A_2$ is in NZC\si{NZC condition}, but
is not in $\mathcal M(t\Z_+[t])$.)
The following Polynomial Strong
Shift Equivalence Equations (PSSE Equations), taken from
\cite{BW04},
give a positive equivalence\si{positive equivalence} in $I-NZC$\si{NZC condition} between
$(I-A_i)$ and $(I-A_{i+1})$, for $0\leq i < 4$.
\begin{alignat*}{2}
\begin{pmatrix} I-tRS & 0 \\ -tS & I \end{pmatrix}
\begin{pmatrix} I & 0 \\ tS & I \end{pmatrix}
& =
\begin{pmatrix} I-tRS & 0 \\ 0 & I \end{pmatrix}
\quad &\text{ is }\quad
(I-A_1)E_1&=(I-A_0) \ ,
\\
\begin{pmatrix} I & R \\ 0 & I \end{pmatrix}
\begin{pmatrix} I & -R \\ -tS & I \end{pmatrix}
&=
\begin{pmatrix} I-tRS & 0 \\ -tS & I \end{pmatrix}
\quad &\text{ is }\quad
E_2(I-A_2)&=(I-A_1)  \ ,
\\
\begin{pmatrix} I & -R \\ -tS & I \end{pmatrix}
\begin{pmatrix} I & R \\ 0 & I \end{pmatrix}
&=
\begin{pmatrix} I & 0 \\ -tS & I -tSR \end{pmatrix}
\quad &\text{ is }\quad
(I-A_2)E_3 &=(I-A_3)  \ ,
\\
\begin{pmatrix} I & 0 \\ tS & I \end{pmatrix}
\begin{pmatrix} I & 0 \\ -tS & I -tSR \end{pmatrix}
&=
\begin{pmatrix} I & 0 \\ 0 & I -tSR \end{pmatrix}
\quad &\text{ is }\quad
E_4(I-A_3) &=(I-A_4)  \ .
\end{alignat*}
One can  check
that each of the four equivalences given by
the PSSE equations is a composition of basic positive equivalences\si{positive equivalence}
in NZC\si{NZC condition}.
That finishes the proof.
\end{proof}

{\it   Tools for construction.}
One way to construct a conjugacy between SFTs defined
by  matrices $A,B$ over $\Z_+$ is to find an SSE over $\Z_+$
from $A$ to $B$.
The polynomial matrix setting gives another way:
find a chain of elementary positive equivalences\si{positive equivalence} from
$I-tA$ to $I-tB$. This is not  a strict advantage;
it's an alternative tool. There are results for which
the only known proof uses this tool \apr{polytool}).

\subsection{Functoriality: flow equivalence\si{flow equivalence} in the polynomial setting}

We will consider one satisfying feature of presenting SFTs
by matrices in NZC\si{NZC condition} (or, just in $\mathcal M(t\Z_+[t])$)
\apr{otherpolyfeatures}).
With the  $\Z_+$ matrix presentation,  the algebraic invariant for
conjugacy, SE-$\Z$, does not  have an obvious natural
relationship to algebraic invariants for flow equivalence\si{flow equivalence}
(e.g. Bowen-Franks\si{Bowen-Franks group} group,\ai{Franks, John}\ai{Bowen, Rufus}
$\det(I-A)$).
In the polynomial setting, we do see that natural relationship.

Let $\mathcal M$ be the set of matrices $I-A$ with $A$ in NZC\si{NZC condition}.
Say matrices $(I-A), (I-B)$ are related by changing
positive powers, $(I-A) \sim_+ (I-B)$, if they  become equal after
changing positive powers of $t$ to other positive powers.
For example,
\[
\begin{pmatrix}  1-t^2-t^5 & -t - t^3 \\ -t^2&1  \end{pmatrix}
\sim_+
\begin{pmatrix}  1-t^2-t^3 & -t^4 -t^5  \\ -t^7&1  \end{pmatrix}
\sim_+
\begin{pmatrix}  1-2t & -2t  \\ -t&1  \end{pmatrix} \ .
\]
The next result is one version for SFTs of the
Parry-Sullivan\ai{Parry, William}\ai{Sullivan, Dennis}
characterization of flow equivalence\si{flow equivalence} of subshifts.

\begin{thm}
Suppose $A,B$ are matrices in NZC\si{NZC condition}.
The following are equivalent.
\begin{enumerate}
\item $A,B$ define flow equivalent\si{flow equivalence}  SFTs.
\item $(I-A),(I-B)$ are equivalent, under the equivalence
relation generated by
(i) positive equivalence\si{positive equivalence} in $I$-NZC\si{NZC condition} and (ii) $\sim_+$ .
\end{enumerate}
\end{thm}
We won't give a proof for this theorem \apr{parrysullianetc}).
But, it is intuitive:  flow equivalence\si{flow equivalence}
arises from conjugacy and time changes, and the time changes  are
addressed by the  $\sim_+$  relation.

Given a matrix $A=A(t)$ in $\mathcal M(t\Z_+[t])$, or in NZC\si{NZC condition}, let $A(1)$
be the matrix defined entrywise by the (augmentation) homomorphism
$ \Z[t] \to \Z$ which sends $t$ to 1. For example,
\[
A=  A(t)=\begin{pmatrix} 3t \end{pmatrix} \ , \quad
B=  B(t)=\begin{pmatrix} t^2+2t^3 \end{pmatrix} \ , \quad
A(1) = \begin{pmatrix} 3 \end{pmatrix} = B(1) \ .
\]
From the Theorem, one can check for SFTs defined by $A,B$
from $\text{NZC}$\si{NZC condition}:
\begin{enumerate}
\item
Flow equivalent\si{flow equivalence} SFTs
defined by $A(t),B(t)$ from NZC\si{NZC condition}    produce isomorphic groups
$\cok_{\Z} (I-A(1)), \cok_{\Z} (I-B(1))$.
\item
  $\cok_{\Z} (I-A(1))$ is the {\it Bowen-Franks group\si{Bowen-Franks group}}\ai{Franks, John}\ai{Bowen, Rufus}
  of the SFT
defined by $A$. \apr{bfgroup})
\end{enumerate}
We sometimes use notation $\cok_{\mathcal R}$ to emphasize that a cokernel is an
$\mathcal R$-module. (A $\Z$-module is just an abelian group.)

Recall, for $A=A(t)$ in NZC\si{NZC condition},
the isomorphism class of $\cok_{\Z[t]} (I-A(t))$
is  the SE-$\Z$ class of the SFT.
There is a  functor, induced by $t\mapsto 1$:
\begin{align*}
\Z[t]\text{-modules } &\to \ \Z[1]\text{-modules} =\Z\text{-modules} =
\text{abelian groups} \\
\cok_{\Z[t]} (I-A(t)) &\mapsto \cok_{\Z} (I-A(1))\ .
\end{align*}
So, this functor gives a presentation of
\begin{align*}
  \text{SE}-\Z \text{ class } & \to \text{ Bowen-Franks group}\si{Bowen-Franks group}\ .\ai{Franks, John}
  \ai{Bowen, Rufus}
\qquad   \qquad   \qquad   \qquad \
\end{align*}
This shows us how  algebraic invariants of flow equivalence\si{flow equivalence} and topological
conjugacy are naturally related in the polynomial setting.

\begin{ex}
Let  $  A=\begin{pmatrix} 3t \end{pmatrix}$ and
$  B=\begin{pmatrix} t^2+2t^3 \end{pmatrix} $.
The $\Z[t]$-modules $\cok (I-A)$ and $ \cok(I-B)$ are not isomorphic.
(For example, $\det(I-A) \neq \det(I-B)$.)
However they do define SFTs
which are flow equivalent\si{flow equivalence},  with Bowen-Franks\si{Bowen-Franks group}\ai{Franks, John}
  \ai{Bowen, Rufus}  group
\[
\cok (I-A(1)) = \cok (I-B(1)) =
\cok \begin{pmatrix}  -2 \end{pmatrix} =
\Z /(-2)\Z = \Z /2\Z \ . \\
\]
\end{ex}
There is
a useful  analog  of positive equivalence\si{positive equivalence}
for constructing maps which give a flow equivalence\si{flow equivalence},
using multiplications by elementary matrices over $\Z$
rather than $\Z[t]$
\apr{feposeq}).
Also, the passage from  SSE-$\Z_+$  of matrices $A,B$
to positive equivalence\si{positive equivalence} of matrices $(I-tA), (I-tB)$
works with an integral group ring $\Z_+ G$
in place of $\Z_+$,
as noted in
\cite{BW04,B02posk}.

\begin{rem}[Category theory.]\label{categorytheoryremark} For an
approach to the classification of  SFTs
(and flow equivalence)\si{flow equivalence}
through category theory\si{category theory},
see the substantive recent paper \cite{jeandel}
of Jeandel. In Jeandel's work, again matrices over $\Z_+$ and $\Z_+[t]$
play roles related to strong shift equivalence and flow equivalence.
 For an earlier approach to flow equivalence
through category theory, see the paper \cite{CostaSteinberg} of
Costa and Steinberg.
\end{rem}

\subsection{Appendix 3} \label{a3}

This subsection contains various remarks, proofs and comments referenced
in earlier parts of Section \ref{sec:polynomial}.

\begin{rem}[Flow equivalence\si{flow equivalence} background] \label{floweqbackground}
It takes more space than we will spend to give a reasonably
understandable introduction to flow equivalence\si{flow equivalence};
see e.g. \cite{bce:fei,bce:fei:corr}
for background and definitions for flow equivalence\si{flow equivalence} of
subshifts.  However, the description to come
of  the Parry-Sullivan Theorem\ai{Parry, William}\ai{Sullivan, Dennis}
\cite{parrysullivan} for SFTs
will be quite adequate for our purposes, as a description of
what flow equivalence\si{flow equivalence} is equivalent to.

Unexpectedly, tools developed for flow equivalence\si{flow equivalence} of SFTs
turned out to be quite useful for certain classification
problems in $C^*$-algebras (see e.g.
\cite{errs:completeDuke2021, Restorff2006, Rordam1995} and their references).
\end{rem}

\begin{rem}  \label{bfgroup}
For a square matrix $C$ over $\Z$,
one can check that the group $\cok_{\Z}(C)$ is infinite when $\det C=0$,
and $|\cok_{\Z}(C)|= |\det C|$ when $\det C\neq 0$. The groups arising
as $\cok_{\Z}(C)$ are the finitely generated abelian groups.
The group
$\cok_{\Z}(C)$ may be determined algorithmically
by computing the Smith normal
form of $C$.

We refer to ``the'' Bowen-Franks\ai{Bowen, Rufus}\ai{Franks, John}
group\si{Bowen-Franks group}
class associated
to an SFT. Formally, the group depends on the presentation; really,
we are talking about ``the'' group up to isomorphism.
Americans of a certain age may remember Bill Clinton being mocked
for a reply, \lq\lq It depends on what you mean by the word
\lq is\rq.\rq\rq
In math, we really do need to keep track.

\end{rem}

\begin{rem} \label{topmarkovshift}
Topological Markov shifts were defined as ``intrinsic Markov chains''
by Bill Parry\ai{Parry, William} in the 1964 paper \cite{Parry1964}.
Parry's paper has the independence of past and future
conditioned on the present, and this being presented by
a zero-one transition matrix, essentially as a vertex
shift.

But before Parry, there was Claude Shannon's\ai{Shannon, Claude} astonishing
monograph
\cite{ShannonWeaver1949} in the 1940s, which launched
information theory. Shannon already was looking at
something we could understand as a Markov shift, with
half of the   variational principle proved in Parry's\ai{Parry, William} paper. Shannon
even used polynomials to present those Markov shifts,
just as we describe.

A zero-one  matrix can be used to define an edge SFT
or a vertex SFT. Yes, they are topologically  conjugate SFTs.
(The two block presentation of the vertex SFT is the edge SFT.)
\end{rem}

\begin{rem}
\label{prooftechniquesforedgesft}
As a postdoc, I heard a talk of John Franks\ai{Franks, John}
  \ai{Bowen, Rufus}  on his classification
of irreducible SFTs up to flow equivalence\si{flow equivalence}. Edge SFTs were a bit
new; he announced for the suspicious that for his proofs, zero-one
matrices just weren't enough.
\end{rem}

\begin{rem} \label{rome}
The  ``rome'' term was introduced in the paper \cite{bgmy},
which also gave a proof that $\det(I-A)= \det(I-tA^{\sharp})$.
\end{rem}

\begin{rem} \label{perrin}
The entropy of an SFT defined by a matrix $B$ over $\Z_+$ is
the log of the spectral radius $\lambda$ of $B$.
Given $\lambda >1 $ the spectral radius of a primitive matrix over
$\Z$, Perrin
constructs a $2\times 2$ $A$ over $t\Z_+[t]$ such
that $A^{\sharp}$ is primitive
with spectral radius $\lambda$ \cite{Perrin1992}.
(The condition that $A^{\sharp}$ is primitive is a significant part of the result.)
\end{rem}

\begin{prop}        \label{fromequivtocok}
Suppose $U,C,V$ are matrices over a ring $\mathcal R$; $U$ and $V$
are invertible over $\mathcal R$; and $D=UCV$.
Then $\cok_{\mathcal R} C$ and $\cok_{\mathcal R} D$ are isomorphic as
$\mathcal R$-modules.
\end{prop}
\begin{proof}
We consider $C,D$ acting by matrix multiplication
on row vectors; of course, the same fact holds for
the action on column vectors. Corresponding to the action
being on row vectors, we are considering
left $\mathcal R$-modules
($c$ in $\mathcal R$ sends $v$ to $cv$), so that matrix multiplication
gives an $\mathcal R$-module homomorphism
(e.g. $(cv)D=c(vD)$).

Let $C$ be $j\times k$, and let $D$ be $m\times n$.
Then
\begin{align*}
\cok C &=\mathcal R^k /    \text{image}(C )=
\mathcal R^k / \{ vC: v\in \mathcal R^j \} \ , \\
\cok D &=\mathcal R^n /    \text{image}(D )=
\mathcal R^n / \{ vD: v\in \mathcal R^m \} \ .
\end{align*}
Define an $\mathcal R$-module isomorphism
$\phi : \mathcal R^k \to \mathcal R^n$ by
$\phi : w \mapsto wV$. (For most rings of interest,
necessarily $j=m$ and $k=n$.)
To show $\phi$ induces the isomorphism
$\cok C \to \cok D$, it suffices to
show  $\phi: \text{image}(C) \to \text{image}(D)$  and
$\phi^{-1}: \text{image}(D) \to \text{image}(C)$ .
For  $xC \in \text{image}(C)$,
\[
\phi (xC) = xCV= (xU^{-1})(UCV) = (xU^{-1}) D \in
\text{image}(D) \ .
\]
For $yD \in \text{image}(D)$,
\[
\phi^{-1} (yD) =  yDV^{-1} = y(UCV)V^{-1} = yUC \in \text{image}(C) \ .
\]
\end{proof}

Recall, for $A=A(t)$ in NZC\si{NZC condition},
the isomorphism class of $\cok_{\Z[t]} (I-A(t))$
determines  the SE-$\Z$ class of the SFT, and conversely.

\begin{rem}  \label{tautau}
If the initial and terminal vertices of $\tau$ were the same,
then we could apply the $\phi$ ``rule''
to a point $x= \dots \tau\tau\overd{\tau}\tau\tau \dots$ (with $\tau$ beginning
at $x_0$)
in contradictory
ways, according to the two groupings
\begin{align*}
\dots\   (\tau \tau) (\tau \tau)(& \overd{\tau} \tau) (\tau \tau) (\tau \tau)
(\tau \tau) (\tau \tau) \dots \\
\dots\  (\tau \tau) (\tau \tau)    (\tau
&  \overd{\tau}) (\tau \tau) (\tau \tau)
(\tau \tau) (\tau \tau) \dots \ \ .
\end{align*}
\end{rem}

\begin{rem} \label{poseqterminology}
The move to polynomial algebraic invariants was pushed by Wagoner\ai{Wagoner, J.B.},
who wanted to exploit analogies between SFT invariants and algebraic
K-theory.
Positive equivalence\si{positive equivalence} was born in the
Kim-Roush\ai{Roush, F.W.}-Wagoner\ai{Kim, K.H.}\ai{Wagoner, J.B.}
papers \cite{KRWForumI, KRWForumII}
 as a tool for constructions, and taken further in
\cite{BW04} (see also \cite{B02posk}).
The framework developed from  considering conjugacy of SFTs
via positive equivalence\si{positive equivalence}
is called ``Positive K-theory'' (or, Nonnegative K-theory).
This reflects the  heuristic connection
to algebraic K-theory.
We will see that the  connection
is more than heuristic.

The term ``positive equivalence\si{positive equivalence}'' arises from its
genesis in  our application.
We defined positive equivalence\si{positive equivalence} rather
generally; there is nothing a priori about $\mathcal M$
which must involve positivity. Also, if
$E_{ij}(-t^k)(I-A)=(I-B)$, then
$E_{ij}(t^k)(I-B)=(I-A)$ -- so, multiplications by
elementary matrices $E_{ij}(-t^k)(I-A)$ are allowed.
If $U$ is a product of elementary matrices over
$\mathcal R [t]$,
such that  $U(I-A)=(I-B)$,
with $A,B$ in $I-\text{NZC}$,
it need not be the case $I-A$ and $I-B$ are positive
equivalent. Each elementary step must be from a matrix
in $\mathcal M$ to a matrix in $\mathcal M$.
\end{rem}

\begin{rem}
\label{exampleforneednzc}
If  $A,B$ in $\mathcal M (t\Z_+[t])$ have  polynomial entries with all
coefficients in $\{0,1\}$,  and define topologically conjugate SFTs,
then one can show that $I-A$ and $I-B$ are positive equivalent
in $I- \mathcal M(t\Z_+[t])$.
But in general, the
converse of the theorem is not true; for example,
the matrices $\begin{pmatrix} 1-2t \end{pmatrix}$ and
$\begin{pmatrix} 1-t& -t \\ -t &1-t \end{pmatrix}$ define
SFTs which are conjugate;  but, there is not a string of
elementary positive equivalences\si{positive equivalence} of square matrices over
$t\Z_+[t]$, from one to the other.
To see this, check the following claim:
if $E=E_{ij}(t^k)$ with $k\geq 0$, and $A,B$ are positive
equivalent in
$\mathcal M(t\Z_+[t])$, and $A(i,i) = 2t + \sum_{k\geq 2} a_kt^k$,
then $B(i,i) = 2t + \sum_{k\geq 2} b_kt^k$.
\end{rem}

\begin{rem} \label{evenmore} We can expand NZC\si{NZC condition} further, and
consider matrices $A$ over $\Z [t,t^{-1}]$ with no cycles
taking zero time or negative time, and  make
good sense of their presenting SFTs, and positive equivalence\si{positive equivalence}
of these matrices $I-A$ as classifying  SFTs.
This isn't necessary for classification of SFTs,
but might be convenient for some construction.
\end{rem}

\begin{rem} \label{polytool}
For example,
constructions of SFTs and topological conjugacies between them,
using polynomial matrices and basic positive equivalences\si{positive equivalence},
were the proof method for the  result in
\cite{KRWForumI, KRWForumII}
of Kim, Roush\ai{Roush, F.W.} and Wagoner\ai{Wagoner, J.B.}\ai{Kim, K.H.} (a result quite important for SFTs).
The hardest step  was a construction of brutal complication.
But without their proof, we would have no proof at all.
\end{rem}

\begin{rem}
\label{otherpolyfeatures}
Edge SFTs are related
in a simple and transparent way to their defining matrices over $\Z_+$.
When using a matrix $A$ in NZC\si{NZC condition}, or even just in $\mathcal M(t\Z_+[t])$,
to define an SFT--we did it by way of the edge SFT defined from
$A^{\sharp}$. The relationship between $A$ and $A^{\sharp}$ is not
very tight -- there is some freedom about what matrix $A^{\sharp}$
is produced. That can be eliminated by precise choices, but these
in generality become complicated and rather artificial.

So, for $A$ in NZC,\si{NZC condition}
one would like to have a presentation of an SFT more simply
and transparently related to $A$, and with
an elementary positive equivalence\si{positive equivalence}
presented transparently.
There is such a presentation -- the ``path SFT''
presented by $A$ (see \cite{B02posk}).

For a square matrix $A$  over $\Z_+$, and a positive integer
$n$, the systems $(X_A, \sigma^n)$ and $(X_{A^n}, \sigma)$ are
topologically conjugate. For a polynomial matrix $A$,
$A^n$ generally does not define an SFT conjugate to the
$n$th power system of the SFT defined by $A$. But,
in the path SFT presentation, we recover a natural way to
pass to powers of the SFT, which
works equally well for negative powers (no passage to transposes needed).

{\it Caveat.}
We considered three matrix presentations of SFTs: by matrices over
$\{0,1\}$, $\Z_+$ and $t\Z_+[t]$. The polynomial presentations
have the greatest scope. But we certainly
still need edge SFTs -- usually the most convenient
choice, sometimes
the only choice, as for Wagoner's\ai{Wagoner, J.B.} SSE-$\Z_+$ complex.

We also  need vertex SFTs. Every topological Markov shift
in the sense of Parry\ai{Parry, William}
(also known as a 1-step shift of finite type)
is a vertex SFT, up to naming of symbols.
But, not every topological Markov shift is
equal to an edge SFT up to naming of symbols.
For an example, consider the vertex SFT with adjacency matrix
$\left(\begin{smallmatrix} 1&1\\ 1&0 \end{smallmatrix}\right) $. This vertex SFT
cannot be an edge SFT after renaming symbols as edges in
some directed graph, because a nondegenerate\si{nondegenerate} adjacency matrix  for a
graph with exactly two edges is either $(2)$,
$\left( \begin{smallmatrix} 1&0 \\ 0&1 \end{smallmatrix} \right)$
or $\left( \begin{smallmatrix} 0&1 \\ 1&0 \end{smallmatrix} \right)$.
\end{rem}

\begin{rem} \label{parrysullianetc}
See \cite{B02posk} for a proof of this
version of the Parry-Sullivan result \cite{parrysullivan}.\ai{Sullivan, Dennis}
For a careful discussion of flow equivalence\si{flow equivalence} for
subshifts, and related issues, see \cite{bce:fei,bce:fei:corr}, which includes
references and a detailed proof of the Parry-Sullivan
result.\ai{Parry, William}
\end{rem}

\begin{rem} \label{feposeq} For this version of positive equivalence\si{positive equivalence},
see the paper \cite{B02posk} and papers citing it.
\end{rem}


\section{Inverse problems for nonnegative matrices}
\label{sec:inverseprobs}

In this section, we study certain inverse spectral problems, and related
problems, for nonnegative matrices. We are especially interested in
inverse problems
which involve the realization of
\lq\lq stable algebra\si{stable algebra}\rq\rq\ invariants, such as the nonzero\si{nonzero spectrum} spectrum.

\subsection{The NIEP}

\begin{de} A matrix is nonnegative if every entry is in $\R_+$.
A matrix is positive if every entry is positive.
\end{de}

We recall some definitions.
If $A$ has characteristic polynomial
$\chi_A(t)=\prod_{i=1}^n(t-\lambda_i)$, then the {\it spectrum} of $A$ is
$(\lambda_1, \dots , \lambda_n)$. We refer to the spectrum as an
$n$-tuple by
abuse of notation \apr{abuse}): the ordering of the $\lambda_i$
does not matter but the multiplicity does matter. The $\lambda_i$ are in $\C$.
Similarly, if $\chi_A(t)= t^j\prod_{i=1}^k(t-\lambda_i)$, with the
$\lambda_i$ nonzero, then the {\it nonzero\si{nonzero spectrum} spectrum} of $A$ is
$(\lambda_1, \dots , \lambda_k)$.

\begin{prob}
{\it The NIEP} (nonnegative inverse eigenvalue problem):
What can be the spectrum of an $n\times n$
nonnegative matrix $A$ over $\R$?
\end{prob}

Work on the NIEP goes back to (at least)
the following result.
\begin{thm}[Suleimanova 1949]\cite{suleimanova}
Suppose $\Lambda = (\lambda_1, \dots , \lambda_n)$ is a list of
real numbers; $\sum_i \lambda_i >0$; and $i>1 \implies \lambda_i <0$.
Then $\Lambda$ is the spectrum of a nonnegative matrix.
\end{thm}
(In fact, under the assumptions of Suleimanova's Theorem,
the companion matrix of the polynomial $\prod_i(t-\lambda_i)$
is nonnegative  \apr{friedland}).)

There is a huge and active
literature  on the NIEP; see the  survey
\cite{NIEP2018} for an overview and extensive bibliography.
Despite a rich variety of interesting results,
a complete solution is not known at size $n$ if $n > 4$.

\begin{thm} [Johnson-Loewy-London Inequalities] \apr{jll})
Suppose $A$ is an $n\times n$  nonnegative matrix.
Then for all $k,m$ in $\N$,
\[
\tr (A^{mk})\geq \frac{\big(\tr (A^m)\big)^k}{n^{k-1}} \ .
\]
\end{thm}

The JLL inequalities, proved independently by Johnson\ai{Johnson, Charles} and by Loewy and
London,
give a
quantitative version of an easy compactness result:
for $n\times n$
nonnegative matrices $A$ with $\text{trace}(A)\geq \tau >0$,
there is a positive lower bound to
$\text{trace}(A^k)$ which depends only on $\tau , n, k$.
We will use the JLL inequalities later.

\subsection{Stable variants of the NIEP}
Throughout this lecture, $\mathcal R$ denotes a subring of $\R$.
\begin{prob}[Inverse problem for nonzero\si{nonzero spectrum} spectrum]\label{inversespecprob}
What can be the nonzero\si{nonzero spectrum} spectrum
of a  nonnegative matrix $A$ over $\mathcal R$?
What can be the nonzero\si{nonzero spectrum} spectrum
of an irreducible or primitive matrix over $\mathcal R$?
\end{prob}
The case  $\mathcal R = \Z$ asks,
what are the possible periodic data\si{periodic data} for shifts of finite type?
This is the connection to  ``stable algebra\si{stable algebra}'' for symbolic dynamics
(and the original impetus for the paper \cite{BH91}).
Later, we will   also consider
the realization in nonnegative matrices of more refined stable
algebra structure.

To begin we review relevant parts of the
Perron-Frobenius theory of nonnegative matrices
\apr{pfnotes}). This will let us reduce the different flavors of
Problem \ref{inversespecprob} to the primitive case.

\subsection{Primitive matrices}

Recall Definition \ref{primitiveDefinition}:
a  primitive matrix\si{primitive matrix} is a square nonnegative matrix $A$
such that for some positive integer $k$, $A^k$ is positive. (Then,
$A^n$ is positive    for all $n\geq k$.)
The next theorem is the heart of the theory of nonnegative
matrices \apr{whyperron}). Recall, the {\it spectral radius} of a
square matrix with real (or complex) entries is the maximum of
the moduli of the eigenvalues (i.e., the radius of the smallest
circle in $\C$ with center 0 which contains the spectrum).

\begin{thm}[Perron] \label{PerronTheorem}
Suppose $A$ is primitive, with spectral radius $\lambda$.
Then the following hold.
\begin{enumerate}
\item $\lambda$ is a simple root of the characteristic
polynomial $\chi_A$.
\item If $\nu$ is another root of     $\chi_A$,
then $|\nu|< \lambda$.
\item There are left and right eigenvectors $\ell , r$
of $A$ for $\lambda$ which have all entries positive.
\item The only nonnegative eigenvectors of $A$ are the eigenvectors for the
spectral radius. \\
\end{enumerate}
\end{thm}
\begin{ex}
We list three nonprimitive nonnegative matrices for which a conclusion
of the Perron Theorem fails.
\[
A=\begin{pmatrix} 0&-1 \\ 1& 0 \end{pmatrix} \ , \qquad
B=\begin{pmatrix} 1&0 \\ 0& 1 \end{pmatrix}\ , \qquad
C = \begin{pmatrix} 0&1 \\ 1& 0 \end{pmatrix} \ .
\]
$A$  has spectrum $(i, -i)$;
the spectral radius of $A$ is 1, but  1 is not an eigenvalue of $A$.
$B$  has spectrum $(1,1)$;
the spectral radius of $B$ is 1, and 1
is a repeated root of $\chi_B$.
$C$  has spectrum $(1, -1)$;
the spectral radius  1 is an eigenvalue, but $1 = |-1|$.
\end{ex}
\begin{ex}
The matrix
$A=\left(\begin{smallmatrix} 0&3 \\ 4& 1 \end{smallmatrix}\right)$ is primitive
with spectrum  $(4 ,-3)$.
There is a positive left eigenvector for eigenvalue $4$, but not for
$3$:
\[
(1,1)\begin{pmatrix} 0&3 \\ 4& 1 \end{pmatrix} =
4(1,1) \quad \quad
\text{and}
\quad \quad
(-4,3)\begin{pmatrix} 0&3 \\ 4& 1 \end{pmatrix} = -3(-4,3) \ .
\]
\end{ex}

\subsection{Irreducible matrices}

\begin{de} An irreducible matrix is an $n\times n$
nonnegative matrix $A$ such that
\[
\{ i,j\} \subset \{1, \dots,   n\}
\ \ \implies \  \
\exists k>0 \text{ such that }
A^k(i,j) >0 \  .
\]
\end{de}
Every primitive matrix is irreducible.
\begin{ex}
\[
A = \begin{pmatrix} 1&1&1 \\ 1& 1& 1 \\ 0&0&0\end{pmatrix}  \quad \quad
B= \begin{pmatrix} 1&1 \\ 0& 1 \end{pmatrix}  \quad \quad
C=\begin{pmatrix} 0&1&0 \\ 0& 0& 1 \\ 1&0&0\end{pmatrix}  \quad \quad
D=\begin{pmatrix} 0&0&1 \\ 0& 0& 1 \\ 1&1&0\end{pmatrix} \ .
\]
For all $n\in \N$, we see sign patterns:

\[
A^n = \begin{pmatrix} +&+&+ \\ +& +& + \\ 0&0&0\end{pmatrix}  \ \
B^n= \begin{pmatrix} +&+ \\ 0& + \end{pmatrix}  \ \
C^{3n}=\begin{pmatrix} +&0&0 \\ 0& +& 0 \\ 0&0&+\end{pmatrix}  \ \
D^{2n}=\begin{pmatrix} +&+&0 \\ +& +& 0 \\ 0&0&+\end{pmatrix} \ .
\]
$A$ and $B$ are not irreducible.
$C$ and $D$ are irreducible, but not primitive.
\end{ex}


\subsubsection{Block permutation structure}
If $n>1$, then an $n\times n$ cyclic-permutation matrix
is irreducible but  not primitive.
This is  representative of the
general irreducible case.

\begin{thm}  For a square nonnegative matrix $A$, the following are equivalent.
\begin{enumerate}
\item $A$ is irreducible.
\item There is a permutation matrix $Q$
and a positive integer $p$ such that $Q^{-1}AQ$
has the block structure of a cyclic permutation,
\[
Q^{-1} A Q =
\begin{pmatrix}
0 & A_1 & 0 & 0 &\dots & 0 \\
0 & 0 & A_2 & 0 &\dots & 0 \\
&   &\dots   &&& \\
0 & 0 & 0 & 0 &\dots & A_{p-1}  \\
A_{p} & 0 & 0 & 0 &\dots & 0
\end{pmatrix}
\]
such that each of the cyclic products  $\ D_1=A_1A_2\cdots A_p$,
$\ D_2= A_2A_3\cdots A_1$, $\ \ \dots \ \ $,
$\ D_p= A_pA_1\cdots A_{p-1} \ \ $
is a  primitive matrix.
\end{enumerate}
\end{thm}
The integer $p$ above is called the period of
the irreducible matrix $A$. (If $p=1$,
then $A$ is primitive.) For $A$ above,  $A^p$ is block
diagonal, with diagonal blocks $D_1, \dots , D_p$.

From the block permutation structure, one can show
the following (in which  $D$ could be any of the
matrices $D_i$ above).

\begin{thm}[Irreducible to primitive reduction]
Suppose $A$ is an irreducible matrix with period
$p$. Then there is a primitive matrix $D$ such that
$\det (I-tA) =\det (I-t^pD)$.
\end{thm}
It is not hard to check that the  converse of this theorem is also true
\apr{AforDproof}).

\subsubsection{Reduction in terms of nonzero\si{nonzero spectrum} spectrum}

Recall,
$A$ has nonzero\si{nonzero spectrum} spectrum $ (\lambda_1, \dots , \lambda_k)$
if and only if  $\det (I-tA) = \prod_{i=1}^k (1-\lambda_i t)$.
The statement
$\det (I-tA) =\det (I-t^pD)$ has an equivalent description
\apr{nzspecequiv}): \\
if $\Lambda$ is the nonzero\si{nonzero spectrum} spectrum of $D$, then
$\Lambda^{1/p}$ is the nonzero\si{nonzero spectrum} spectrum of $A$.
Here,
$ \Lambda^{1/p}$ is defined
by replacing
each entry of $\Lambda$ with the list of its $p$th roots in $\C$.
If $\Lambda$ is $k$ entries, then
$\Lambda^{1/p}$ has $pk$ entries.

\begin{ex}
Suppose
$  \det(I-tD) = (1-8t)(1-7t)^2$ and
$\det(I-tA) = \det (I-t^3D)$.
Let $\xi = e^{2\pi i/3}$.       The
nonzero\si{nonzero spectrum} spectrum of $D$
is
$\Lambda = (8,7,7)$.  The nonzero\si{nonzero spectrum} spectrum of $A$ is
\[
\Lambda^{1/3}  =  \big(\ \, 2,\, \xi 2, \, \xi^2 2,\, \ \ \
7^{1/3} ,\    \xi 7^{1/3} ,\xi^2 7^{1/3} , \, \ \ \
7^{1/3} ,\    \xi 7^{1/3} ,\xi^2 7^{1/3} \, \ \ \big)
\ .
\]
\end{ex}
\subsubsection{Multiplicity of zero in the spectrum}
Apart from one exception: if a nonzero\si{nonzero spectrum} spectrum
is realized by an irreducible matrix over $\mathcal R$ of size $n\times n$,
then it can also be realized at any larger size, by an
irreducible matrix  over $\mathcal R$ of the same period.

The one exception: if $\mathcal R = \Z$, then
an irreducible matrix with spectral radius 1 can only
be a cyclic permutation matrix.

Also: if $n\times n$ is the smallest size primitive matrix
realizing a nonzero\si{nonzero spectrum} spectrum $\Lambda$, then
$pn\times pn$ is the smallest size irreducible matrix
realizing  $\Lambda^{1/p}$.

{\it Conclusion.}
Knowing the possible spectra of irreducible matrices
over a subring $\mathcal R$ of $\R$ reduces to
knowing the possible nonzero\si{nonzero spectrum} spectra of primitive matrices over
$\mathcal R$, and
the smallest dimension in which they can be realized.
\subsection{Nonnegative matrices}

\begin{exer} \apr{pfnotes}) Suppose $A$ is a square nonnegative matrix.
Then there is a permutation matrix $P$ such
$P^{-1}AP$ is block triangular, such that each diagonal
block is either irreducible or $(0)$.

For $A$ nonnegative as above,
let $A_i$ be the $i$th diagonal block, with
characteristic polynomial $p_i$.
Then  the characteristic polynomial of $A$ is
$\chi_A(t)=\prod_i p_i(t)$, and the nonzero\si{nonzero spectrum}
spectrum is given by
$\det (I-tA) = \prod_i \det (I-tA_i)$.
\end{exer}

So, the  spectrum of a
nonnegative matrix is an arbitrary disjoint union of
spectra of irreducible matrices, together with an arbitrary repetition of 0.

There are constructions and constraints which work best at the level
of nonnegative matrices (e.g., JLL).
Still, one approach  to the NIEP
is to focus on the primitive
case (which gives the irreducible case, and then the general case).
Obstructions might be more simply formulated in this case.
Moreover,  in applications a nonnegative matrix must often
be irreducible or primitive. (For symbolic dynamics: definitely.) A
realization statement for
nonnegative matrices does not  give a realization statement for
irreducible or
primitive matrices.
So, we focus on  primitive matrices. But even
in this restricted case,
no satisfactory general
characterization is known or conjectured.

{\it Conclusion.}
We will focus on
the nonzero\si{nonzero spectrum} spectrum of primitive
matrices. And here, at last, we find simplicity.

\subsection{The Spectral Conjecture}

Let $\Lambda= (\lambda_1, \dots , \lambda_k)$ be a $k$-tuple of
nonzero complex numbers.
We will give three simple
conditons $\Lambda$ must satisfy to be the nonzero\si{nonzero spectrum} spectrum
of a primitive matrix over  $\mathcal R$.
\\ \\
\begin{de}
For a  tuple $\Lambda= (\lambda_1, \dots , \lambda_k) $
of complex numbers:
\begin{itemize}
\item
$\lambda_i$ is a   {\it Perron value}  for $\Lambda$
if $\lambda_i$ is a positive real number and  $i\neq j \implies
\lambda_i > |\lambda_j|$ .
\item
$\tr (\Lambda) = \sum_{i=1}^k \lambda_i $ .
\item $\Lambda^n = \big( (\lambda_1)^n, \dots , (\lambda_k)^n )\ ,\quad$
if $n\in \N$.
\end{itemize}
\end{de}

\begin{prop}[Necessary conditions]
Suppose $\Lambda $ is the nonzero\si{nonzero spectrum} spectrum of a primitive matrix over a
subring $\mathcal R$ of $\R$. Then the following hold.
\begin{enumerate}
\item \textnormal{Perron Condition}:
\\
$\Lambda$ has a Perron value.

\item \textnormal{Coefficients Condition \apr{coeff})}:
\\
The  polynomial
$p(t)= \prod_{i=1}^k (t-\lambda_i)$
has all its coefficients in $\mathcal R$.

\item \textnormal{Trace Condition}: \\
If  $\mathcal R \neq \Z$, then for all positive integers $n,k$:
\begin{enumerate}
\item   (i)   $\tr (\Lambda^n) \geq 0$, and
\item   (ii)  $\tr (\Lambda^n) > 0 \implies \tr ( \Lambda^{nk}) > 0 $ .
\end{enumerate}
If $\mathcal R=\Z$, then for all positive integers $n$,
$\tr_n(\Lambda) \geq 0\ .  $
\end{enumerate}
\end{prop}
(We define $\tr_n(\Lambda)$, the $n$th net trace of $\Lambda$, below.)
\begin{proof}
(1) By the Perron Theorem, $\Lambda$ has a Perron value.

(2) The characteristic polynomial of a matrix over a ring has
coefficients in the ring. For some $k\geq 0$,
the characteristic polynomial of $A$ is
$t^kp(t)$. So, $p$ has coefficients in the ring.

(3) (i)  $\tr (\Lambda)= \tr (A)$, and
$\tr (\Lambda^n)=\tr(A^n)$. The trace of a
nonnegative matrix is nonnegative.
(ii) Suppose $\tr (\Lambda^n)>0$. Then
$\tr(A^n)>0 $ and $A^n \geq 0$.  Therefore $\tr{(A^n)^k}>0$ .
But,
$ \tr{(A^n)^k} = \tr (A^{nk}) =\tr (\Lambda^{nk}) $ .

(3) Suppose $\mathcal R = \Z$. Conditions (i) and (ii) hold,
but a stronger condition holds.

Consider $A$ as the adjacency matrix of a graph.
A loop is a path with the same minimal and terminal vertex.
The number of loops of length $n$ is trace($A^n$).

A loop is minimal if it is not a concatenation of copies
of a shorter loop.
So, for example,
\begin{align*}
\text{number of minimal loops of length }1  & =
\text{trace}  (A) \\
\text{number of minimal loops of length }2  &
= \text{trace}(A^2) - \text{trace}  (A) \ .
\end{align*}
For example, let $\Lambda = (2,i,-i,i,-i,1)$. Then
$\tr ( \Lambda^2 )- \tr  (\Lambda) = 1 -3 = -2 <0$.
This $\Lambda$ cannot be the nonzero\si{nonzero spectrum} spectrum of a
matrix over $\Z_+$, even though $\Lambda$ satisfies
conditions 1,2,3(i) and 3(ii).

The number of minimal loops of
length $n$, $\tr_n(\Lambda)$,  can be expressed as a function
of the traces of powers of $\Lambda$
using Mobius inversion:
\[
\tr_n(\Lambda):= \sum_{d|n}
\mu (n/d)\,  \tr ( \Lambda^d) \ ,
\]
where $\mu $ is the Mobius function,
\begin{align*}
\mu : \N &\to \{-1,0,1\}  \\
:    n & \mapsto 0 \quad \text{if }n\text{ is not squarefree} \\
: n  & \mapsto (-1)^e \quad \text{if }n\text{ is the product of }e
\text{ distinct primes. }
\end{align*}
\end{proof}
\begin{conj}[Spectral Conjecture, Boyle-Handelman  1991 \cite{BH91}] \label{spectralconjecture}\ai{Handelman, David}
Let $\mathcal R$ be a  subring of $\R$. Suppose $\Lambda
= (\lambda_1, \dots , \lambda_k)$ is an $k$-tuple of complex numbers.
Then $\Lambda$ is the nonzero\si{nonzero spectrum} spectrum of some primitive matrix over
$\mathcal R$ if and only the above conditions (1), (2), (3) hold.
\end{conj}

\begin{ex} (unbounded realization size) Suppose $\mathcal R=\R$. Given
$0< \epsilon < (1/2)$, set
\[ \Lambda_{\epsilon}  \ =\
\Big(\ 1\ ,\ i\sqrt{(1-\epsilon )/2}\ ,\ -i\sqrt{(1-\epsilon )/2} \
\Big) \ .
\]
This  $\Lambda_{\epsilon} $
satisfies the conditions of the Spectral Conjecture.

But,
if a nonnegative $n\times n$ matrix $A$ has nonzero\si{nonzero spectrum} spectrum
$\Lambda_{\epsilon}$, then
\begin{align*}
\tr (\Lambda_{\epsilon}^2) &\geq \frac{(\tr \Lambda_{\epsilon})^2}n
\ , \quad
\text{by the JLL inequality, and therefore }
\\
\epsilon &\geq \frac{1^2}n  = 1/n \ .
\end{align*}
So, as $\epsilon$ goes to zero, the size of $A$ must
go to infinity.
\end{ex}

\begin{de}
A matrix $A$ is {\bf eventually positive (EP)}  if
for all large $k>0$, $A^k$ is positive.
\end{de}

\begin{thm}[Handelman] \apr{HIEP})\ai{Handelman, David}
Suppose $A$ is a square matrix over $\mathcal R$ whose spectrum has a
Perron value.
\begin{enumerate}
\item
If $\mathcal R \neq \Z$, then $A$ is similar over $\mathcal R$
to an EP matrix \cite{HandelmanJOpTh1981}.
\item
If $\mathcal R = \Z$, then $A$ is SSE over $\mathcal R$ to an EP matrix
\cite{HandelmanEvPosRational1987}.
\end{enumerate}
\end{thm}

In particular, the Spectral Conjecture would  be true if we were allowed
to replace $\Lambda$ with $\Lambda^k$, $k$ large. With $\mathcal R\neq \Z$,
and $\Lambda$ an $n$-tuple, we could even realize
$\Lambda^k$ with a positive matrix which is $n\times n$.

Let's consider existing results on the Spectral Conjecture.

\subsection{Boyle-Handelman Theorem}\ai{Handelman, David}
\begin{thm}[Boyle-Handelman \cite{BH91}] The Spectral
Conjecture is true if $\mathcal R=\R$.
\end{thm}
\begin{rem}

The focus on  nonzero\si{nonzero spectrum} spectra in \cite{BH91}
grew out of symbolic dynamics, as indicated by
Section \ref{perptssection}.
  However,
in his
  1981 paper \cite{Johnson1981},
  Charles Johnson\ai{Johnson, Charles}  had already called attention to
  the potential impact of adding zeros to a
  candidate spectrum of a nonnegative matrix \apr{johnsonzeros}).
\end{rem}

\begin{rem} The problem of determining the possible nonzero\si{nonzero spectrum} spectra
of primitive symmetric matrices is quite
different. If an $n$-tuple is the nonzero\si{nonzero spectrum} spectrum of a
{\bf nonnegative symmetric matrix}, then it is achieved
by a matrix whose size
is bounded above by a function of $n$
\cite{JohnsonLaffeyLoewySymmetric1996}\ai{Laffey, Thomas J.}. Adding more zeros
to the spectrum doesn't help.
\end{rem}
The Boyle-Handelman Theorem is a  corollary of a stronger result.

\begin{thm}[Subtuple Theorem \cite{BH91}]\ai{Handelman, David}
Suppose $\Lambda  $ satisfies the conditions of the Spectral Conjecture,
and a subtuple of $\Lambda $ containing the Perron value of $\Lambda$
is
the nonzero\si{nonzero spectrum} spectrum of a primitive matrix over $\mathcal R$.
(For example, this holds if the Perron value is in $\mathcal R$.)
Then $\Lambda $  is
the nonzero\si{nonzero spectrum} spectrum of a primitive matrix over $\mathcal R$.
\end{thm}
The proof of the Suptuple Theorem
uses ideas from symbolic dynamics.
The proof     is constructive,
in the sense that one could make it a
formal algorithm. But the construction is very complicated,
and uses   matrices of enormous size.
It has no practical value as a general algorithm.

\begin{thm} [Boyle-Handelman-Kim-Roush\ai{Roush, F.W.} \cite{BH91}]\ai{Kim, K.H.}
  \ai{Handelman, David}
Suppose $\Lambda  $ satisfies the conditions of the Spectral
Conjecture, $\tr(\Lambda )>0$ and   $\mathcal R\neq  \Z$.
Then  $\Lambda $ is the nonzero\si{nonzero spectrum} spectrum of a primitive matrix
over
$\mathcal R$.
\end{thm}
\begin{proof} We will outline  the proof.
\begin{enumerate}
\item
By the B-H Theorem, there is a
primitive matrix $A$ over $\R$
with nonzero\si{nonzero spectrum} spectrum $\Lambda $.
\item
Given $A$ primitive with positive trace,  a theorem of
Kim\ai{Kim, K.H.} and Roush\ai{Roush, F.W.} produces a positive  matrix $B$
which is SSE-$\R_+$ to $A$ (hence, has the same nonzero\si{nonzero spectrum}
spectrum as $A$).
\item
There are matrices $U,C$ over $\mathcal R$
such that $U^{-1}CU = A$  and $\det U =1$.
\item $U$ is a product of  elementary matrices over $\R$, equal to $I$
except in a single off diagonal entry. By
density of $\mathcal R$ in $\R$, these can be perturbed to elementary
matrices over $\mathcal R$. Thus $U$ can be perturbed to a
matrix $V$ over $\mathcal R$ with determinant 1.
\item Because $U^{-1}CU >0$, if $V$ is close enough to $U$ then
$V^{-1}CV>0$.
\end{enumerate}
\end{proof}

\begin{rem}
Suppose $\mathcal R \neq \Z$.
It would be very satisfying to see the Spectral Conjecture
proved
in the remaining case, $\tr (\Lambda) =0$,
by some analogous
perturbation argument. I have no idea how to do this, or if
it can be done.
\end{rem}
\subsection{The Kim-Ormes-Roush Theorem}

\begin{thm} [Kim-Ormes-Roush\ai{Roush, F.W.}] \cite{S8}\ai{Kim, K.H.}
For $\mathcal R = \Z$, the Spectral Conjecture is true.
\end{thm}
Let us note an immediate corollary.
\begin{coro}
For $\mathcal R = \Q$, the Spectral Conjecture is true.
\end{coro}
\begin{rem} Polynomial
matrices and formal power series
play a fundamental role in the KOR proof.
The KOR Theorem gives us a complete understanding of the
possible periodic data\si{periodic data} for SFTs. The proof, though quite
complicated, is much more tractable than the proof of the B-H
Theorem. The use of power series leads to an interesting
analytical approach to the NIEP
\cite{LaffeyLoewySmigocPowerSeries2016}.\ai{Laffey, Thomas J.}
\end{rem}

\subsection{Status of the Spectral Conjecture}
The conjecture is true for $\R$, $\Q$ and $\Z$;  in the positive
trace case; under the Subtuple Theorem assumption; and in
other special cases. It is very hard to doubt the
conjecture.

One expects  the case $\mathcal R=\Z$ to be the hardest case.
Perhaps it is feasible to prove the Spectral
Conjecture by adapting the Kim-Ormes-Roush\ai{Roush, F.W.} proof.\ai{Kim, K.H.}

\subsection{Laffey's Theorem}\ai{Laffey, Thomas J.}

Laffey\ai{Laffey, Thomas J.} \cite{Laffey2012}
gave a constructive version of the Boyle-Handelman Theorem in
the case that $\mathcal R= \R$ and
the candidate spectrum $\Lambda$,  satisfying the
necessary conditions
of the Spectral Conjecture for $\R$, also satisfies
\[
\tr (\Lambda^k)>0\ , \quad k\geq 2 \ .
\]
The
primitive  matrix which Laffey\ai{Laffey, Thomas J.} constructs
to realize $\lambda$ has
a rather classical  form, and there is a comprehensible formula
giving an upper bound on the size of the smallest $N$ given
by the construction.
From here, we give some remarks on Laffey's\ai{Laffey, Thomas J.} theorem.

{\it The Coefficients Condition.}
In the case of $\R$,
the Coefficients Condition  of the Spectral Conjecture
follows automatically from the  Trace Conditions
\apr{coeff}),

{\it Laffey's\ai{Laffey, Thomas J.}
  upper bound on the smallest size $N$  of a primitive matrix $A$
realizing a given $\Lambda = ( \lambda_1, \dots, \lambda_n)$.}
If $A$ is primitive with nonzero\si{nonzero spectrum} spectrum $\Lambda$ and $c>0$,
then $cA$ is primitive with nonzero\si{nonzero spectrum} spectrum $c\Lambda
= (c\lambda_1, \dots, c\lambda_n)$. So, to consider an
upper bound $N$,
for simplicity we  consider just the
special case that the Perron value of $\Lambda$
is $\lambda_1 =1$.

Laffey's\ai{Laffey, Thomas J.} explicit, computable formula giving an
upper bound for $N$ is rather complicated.
But, using the Perron value $\lambda_1=1$,
and considering only the nontrivial case $n\geq 2$, it can be shown
that Laffey's\ai{Laffey, Thomas J.} bound implies
\begin{equation} \label{LaffeysN}
N \leq \kappa_n \Bigg(\frac{1}{MG}\Bigg)^n \
\end{equation}
where
$\ \kappa_n $ depends only on $n$,
\begin{align*}
G \ &=\ 1 - \max \{|\lambda_i|: 2\leq i \leq n\}\ ,  \quad
\\
M\ &=\ \min \{ \tr (\Lambda )^n : n \geq 2 \ \} \ .
\end{align*}
The numbers $\kappa_n$ obtained from the estimate grow very
rapidly; e.g.  $\kappa_n \geq n^n$ .

This bound is certainly nonoptimal!   For example,
suppose $0< \epsilon  <1$.
The nonzero\si{nonzero spectrum} spectrum $(1,\,  -1+\epsilon )$ is realized
by the $2\times 2$ primitive matrix
$\begin{pmatrix} 0&1 \\ 1-\epsilon & \epsilon
\end{pmatrix} $. But here, $\epsilon =G$, and as
$\epsilon$ goes to zero the upper bound in \apr{LaffeysN}) goes
to infinity.

Nevertheless: this is a transparent and
meaningful bound. The bound involves only $n$, $M$ and $G$.
The spectral gap $G$ appears repeatly in the use of primitive
matrices (and more generally), e.g. for convergence rates.
Also, neither of the terms $1/M$ and $1/G$ can simply be deleted, as we
note next.

{\it The tracial floor term $1/M$.}
If Laffey's\ai{Laffey, Thomas J.} formula for
an upper bound on $N$ were
replaced by a formula of the form
$N \leq  f(G,n)$, then even at $n=3$
the formula could not give a correct bound,
on acount of the JLL Inequalities  \apr{reflectjll}).

{\it The spectral gap term $(1/G)$.}
If Laffey's\ai{Laffey, Thomas J.} formula for
an upper bound on $N$ were
replaced by a formula of the form
$N \leq  f(M,n)$, then even at $n=4$,
the formula could not give a correct bound
\apr{gaprefl}).

\begin{ex}
Let $\Lambda = (\, 1.1,\,  \xi ,\,  \overline{\xi}\, )$,
where $\xi =  \exp{(\pi i/10)}$ and
$\overline{\xi}$ is its complex conjugate.
Laffey\ai{Laffey, Thomas J.} stated there is a $128\times 128$ primitive matrix
realizing this nonzero\si{nonzero spectrum} spectrum.
\end{ex}
{\it The matrix form.}
The primitive matrix with nonzero\si{nonzero spectrum} spectrum
$\Lambda$ has (for sufficiently large $k$)  the banded form
\[
\begin{pmatrix}
x_1    & 1      & 0       & 0   &\cdots      & 0      &   0   &  0  \\
x_2    & x_1    & 2       & 0   &\cdots      & 0      &   0   &  0  \\
x_3    & x_2    & x_1     & 3    &\cdots      & 0      &   0   &  0  \\
x_4    & x_3    & x_2     & x_1   &\cdots      & 0      &   0   &  0  \\
\cdots &\cdots  &\cdots  &\cdots  &\cdots     &\cdots  &\cdots  &\cdots  \\
x_{k-2} & x_{k-3} &x_{k-4}  &\cdots   &\cdots        &    x_1 &  k-2  &  0    \\
x_{k-1} & x_{k-2} &x_{k-3}  & \cdots  &\cdots        &    x_2 &  x_1  &  k-1   \\
x_k    & x_{k-1} & x_{k-2}  &\cdots   &\cdots        &   x_3  &  x_2  &  x_1
\end{pmatrix} \ .
\]
For the relation of the matrix entries  to $\Lambda$, see Laffey's\ai{Laffey, Thomas J.}
paper
\cite{Laffey2012}.\ai{Laffey, Thomas J.}

{\it Limits of the argument.}
A lot of the  complication of the B-H proof involves
complications of
$\tr (\Lambda^n) = 0 $ for a variety of sets
of $n$.  These general difficulties aren't addressed in Laffey's\ai{Laffey, Thomas J.} result.  Laffey's\ai{Laffey, Thomas J.} argument also proves the Spectral Conjecture over
any  subfield $\mathcal R$ of $\R$, under the restriction $\tr (\Lambda^k) >0$ for $k>1$.
But it does not work for all $\mathcal R$.
The argument uses division by integers in $\mathcal R$.

\subsection{The Generalized Spectral Conjectures}

The NIEP refines to an even harder question:
what can be the Jordan form of a square nonnegative
matrix over $\R$?
We refer to \cite[Sec.9]{NIEP2018} for a discussion.
A rather sobering example of Laffey\ai{Laffey, Thomas J.} and Meehan \cite{laffeyMeehan1999}
shows that $(3+t, 3-t,-2,-2,-2)$ is the spectrum of a $5\times 5$
nonnegative matrix if $t> (16\sqrt{6})^{1/2} -39 \approx 0.437\dots$,
but it is the spectrum of  a diagonalizable nonnegative matrix
if and only if $t\geq 1$.

Suppose $A$ is a nonnilpotent square matrix over $\R$.
The {\it nonsingular part} of $A$ is a nonsingular matrix $A'$ over $\R$
such that $A$ is similar to the direct sum of $A'$ and
a nilpotent matrix.  ($A'$ is only defined up to
similarity over $\R$.)
Analagous to the
Spectral Conjecture
\eqref{spectralconjecture}, we have the following.

\begin{conj}[Boyle-Handelman] \label{realgsc}\ai{Handelman, David}
If $B$ is a square real matrix satisfying the necessary
conditions of the Spectral Conjecture,
then $B$ is the nonsingular part of some primitive
matrix over $\R$.
\end{conj}

Let $A,B$ be square matrices over $\R$, with
nonsingular parts $A',B'$.
Recall, the following are equivalent:
\begin{enumerate}
\item
$A'$ and $B'$ are SIM-$\R$ (similar over $\R$).

\item $A$ and $B$ are SE-$\R$ (shift equivalent over $\R$).

\item $A$ and $B$ are SSE-$\R$ (strong shift equivalent over $\R$).
\end{enumerate}
So,  the conjecture above is a special case of either of
the following conjectures \apr{gscrefs}).

\begin{conj}[ (Weak) Generalized Spectral Conjecture, Boyle-Handelman 1991]
Suppose $A$ is a square matrix over a subring $\mathcal R$ of $\R$,
and the nonzero\si{nonzero spectrum} spectrum of $A$ satisfies the necessary
conditions of the Spectral Conjecture.

Then $A$ is SE-$\mathcal R$  to
a primitive matrix.
\end{conj}
\begin{conj} [(Strong) Generalized Spectral Conjecture,
Boyle-Handelman 1993]\ai{Handelman, David}
Suppose $A$ is a square matrix over a subring $\mathcal R$ of $\R$,
and the nonzero\si{nonzero spectrum} spectrum of $A$ satisfies the necessary
conditions of the Spectral Conjecture.

Then $A$ is SSE-$\mathcal R$  to
a primitive matrix.
\end{conj}
The Strong GSC is the strongest viable conjecture we know which reflects
the idea
that the only obstruction to expressing stable algebra\si{stable algebra} in
a primitive matrix is the nonzero\si{nonzero spectrum} spectrum obstruction.

In the next result, a \lq\lq nontrivial unit\rq\rq\ is a
unit in the ring not equal to $\pm 1$. (The assumption
of a nontrivial unit is
probably an artifact of the proof.)

\begin{thm} \cite[Theorem 3.3]{BH93}
Let $\mathcal R$ be a unital subring of $\R$. Suppose that either
$\mathcal R= \Z$ or $\mathcal R$ is a Dedekind domain with a nontrivial
unit. Let $A$ be a square matrix with entries from $\mathcal R$ whose
nonzero\si{nonzero spectrum} spectrum $\Lambda$ satisfies the necessary conditions
of the Spectral Conjecture and consists of elements of $\mathcal R$.
Then $A$ is algebraically shift
equivalent\begin{footnote}{\lq\lq Algebraically shift
equivalent
over $\mathcal R $\rq\rq\ was the notation in \cite{BH93} for what we
are calling SE-$\mathcal R$, shift equivalence\si{shift equivalence} over the ring $\mathcal R$.
Also, \cite[Prop.2.4]{BH93} established that SE and SSE are equivalent
over a Dedekind domain, so the conclusion could  have been
stated for strong shift equivalence\si{strong shift equivalence}.}\end{footnote} over
$\mathcal R $
to a primitive matrix.
\end{thm}

The following corollary is immediate.
\begin{coro}
Conjecture \ref{realgsc} is true under the additional assumption
that the spectrum of $B$ is real.
\end{coro}

For example, the corollary covers the case that
$B$ in Conjecture \ref{realgsc} is a diagonal matrix.
(For example, if $B$ is diagonal with a Laffey-Meehan\ai{Laffey, Thomas J.} spectrum
$(3+t,3-t,-2,-2,-2)$, for any $t>0$).
On the other hand, the following (embarassing) open problem indicates
how little we know.

\begin{prob} Suppose $A$ is a $2\times 2$ matrix over $\Z$
with irrational eigenvalues satisfying the conditions of the
Spectral Conjecture. Prove that $A$ is SE-$\Z$ to a primitive
matrix.
\end{prob}

When the Generalized Spectral Conjectures were made, it was not known
whether SE-$\mathcal R$ implied SSE-$\mathcal R$ for every ring
$\mathcal R$. We now know that there are many rings over which
SSE properly refines SE \cite{BoSc1}, including some subrings of $\R$.
So, the weak and strong conjectures are not
{\it a priori}
equivalent.
Nevertheless, it can be proved, for every subring $\mathcal R$
of $\R$, that
if any matrix in a given SE-$\mathcal R$ class is
primitive, then every matrix in that SE-$\mathcal R$ class is
SSE-$\mathcal R$ to a primitive matrix
\cite{BoSc3}. So, we now know the weak and strong conjectures
are equivalent.

\subsection{Appendix 4} \label{a4}

This subsection contains various remarks, proofs and comments referenced
in earlier parts of Section \ref{sec:inverseprobs}.

\begin{rem} \label{abuse}
``Abuse of notation'' is a use of notation to mean something
it does not literally represent, for simplicity.
For example,  describing
the spectrum correctly as a multiset (set with multiplicities)
seems to divert more mental energy than one uses to be aware that
an $n$-tuple is not literally a multiset.
\end{rem}

\begin{rem} \label{friedland}  The companion matrix characterization
for Suleimanova's Theorem is
attributed in \cite{Laffey2012}\ai{Laffey, Thomas J.}
to Shmuel Friedland\ai{Friedland, Shmuel}.
\end{rem}

\begin{thm}\label{jll} (\textnormal{JLL Inequalities})
Let $A$ be an $n\times n$  nonnegative matrix.
Then for all $k,m$ in $\N$ :
\[
\textnormal{trace} (A^{mk}) \geq \frac{\big(
\textnormal{trace} (A^m)\big)^k}{n^{k-1}} \ .
\]
\end{thm}
Expressed at $m=1$ in
  terms of the spectrum $(\lambda_1, \dots , \lambda_n\}$,
  the inequality becomes
  \begin{equation}\label{jllformnk}
    \sum_{i=1}^n (\lambda_i)^{k} \geq
    \frac{\big(\sum_{i=1}^t \lambda_i\big)^k }{n^{k-1}} \ .
    \end{equation}

This result was proved independently
by
Loewy and London \cite{LoewyLondon}, and
by  Johnson\ai{Johnson, Charles}
\cite{Johnson1981}. (Johnson's\ai{Johnson, Charles}  explicit statement was only
for $m=1$, the essential case.)
The proof of this insightful
result is not difficult.
\begin{enumerate}
\item If $B$ is an $n\times n$  nonnegative matrix, and $k\in \N$,
then $\tr (B^k) \geq \sum_{i=1}^n \big(B(i,i)\big)^k$. \\
(Because: the  $B(i,i)^k $ are some of the terms contributing to
$\tr (B^k)$, and the other terms are nonnegative.)
\item Now  suppose $\tau = \tr (B) >0$, and
solve the  problem: if $x_1, \dots, x_n$
are nonnegative numbers with positive sum $\tau$, what is the minimum possible
for the sum $s_k= \sum_{i=1}^n (x_i)^k$ ?
\end{enumerate}
You can check (with Lagrange multipliers, say, or
H\"{o}lder's inequality) that the minimum is achieved
at $(x_1, \dots , x_n)=(\tau /n, \tau /n, \dots , \tau /n)$.
(Intuitively, there is no other candidate,
because there is a minimum and the minimum is not achieved at $(x_1, \dots , x_n)
=(\tau_1, 0, \dots , 0)$.)
Then, for $B$ and for $B=A^m$,
\begin{align*}
\tr (B^k) \geq s_k & \geq \sum_{i=1}^n (\tau/n)^k
= n (\tau/n)^k  = \tau^k / n^{k-1}  \\
\tr (B^k)  & \geq  \tau^k / n^{k-1} \\
\tr (A^{mk}) =\tr ((A^m)^k)    & \geq  (\tr (A^m))^k/n^{k-1} \ .  \qed
\end{align*}

\begin{rem}\label{pfnotes}
There are a number of excellent works on
the Perron-Frobenius theory of nonnegative matrices;
Seneta's classic book
\cite{SenetaBook2006}
one introduction.
My short exposition
\cite{boylepfnotes},
appealing
  to an argument of Michael Brin,\ai{Brin, Michael} covers the
  heart of the theory (statements
  in this chapter), but not all parts of it.
\end{rem}

\begin{rem}                  \label{whyperron}
Briefly: why is the Perron theorem so important?

Suppose $A$ is primitive with
spectral radius $\lambda$. Let $  \ell, r$ be
be positive  left, right
eigenvectors for $\lambda$, such that
$\ell r = (1)$.
The Perron Theorem implies that for
many purposes, for large $n$,  $A^n$ is very
well approximated by the rank one positive matrix
$\lambda^n r\ell$.

What is ``very well approximated''?
Let $\mu $ be the second highest
eigenvalue modulus.
There is a matrix $B$ with spectral radius $\mu$
such that $A = (\lambda r\ell) + B$,
with $(\lambda r\ell) B =0 = B(\lambda r\ell)$, so
$A^n = (\lambda^n r\ell) + B^n$.
Entries of  $B^n$ cannot grow at an exponential rate greater than
$\mu^n$; but every entry of $A^n $ grows
at the exponentially greater rate  $\lambda^n$.
\end{rem}

\begin{prop} \label{AforDproof}
Suppose $D$ is a primitive matrix over a subring $\mathcal R$
of $\R$, and $p$ is a positive integer. Then there is
an irreducible matrix $A$ over $\mathcal R$ with period
$p$ such that
$\det (I-tA) =\det (I-t^pD)$.
\end{prop}
\begin{proof}
We give a proof for $p=4$ (which should make
the general case obvious).
Define
$A = \left(\begin{smallmatrix} 0 & D & 0 & 0 \\ 0  & 0 & I& 0 \\
0&0 & 0 &  I  \\  I & 0 & 0 & 0
\end{smallmatrix} \right)$.
We compute a product
\[
(I-tA)U =
\left(\begin{smallmatrix} I & -tD & 0 & 0 \\ 0  & I & -tI& 0 \\
0&0 & I &  -tI  \\  -tI & 0 & 0 & I
\end{smallmatrix} \right)
\left(\begin{smallmatrix} I & 0 & 0 & 0 \\
t^3I  & I & 0 &  0 \\
t^2I&0 & I &  0  \\
tI &0 & 0 &  I
\end{smallmatrix}\right)
=
\left(\begin{smallmatrix} I -t^4D & -tD & 0 & 0 \\ 0  & I & -tI& 0 \\
0&0 & I &  -tI  \\  0 & 0 & 0 & I
\end{smallmatrix} \right)
\]
Noting $\det U = 1$, we see
$\det(I-tA) = \det\big((I-tA)U\big) =\det(I-t^4D)$\ .
\end{proof}

\begin{prop} \label{nzspecequiv}
Suppose $A,D$ are square real matrices,
$\det(I-tA) = \det (I-t^pD)$, and
the nonzero\si{nonzero spectrum} spectrum of $D$ is
$\Lambda=(\lambda_1, \dots , \lambda_k)$.

Then the nonzero\si{nonzero spectrum} spectrum of $A$ is $\Lambda^{1/p}$.
\end{prop}
\begin{proof}
Because the nonzero\si{nonzero spectrum} spectrum of $D$ is
$\Lambda=(\lambda_1, \dots , \lambda_k)$, we can write
$\det (I-tD) = \prod_{i=1}^k(1-\lambda_it)$.
Therefore,
$\det (I-t^pD) = \prod_{i=1}^k(1-\lambda_it^p)$.
Given $\lambda_i$, let $\mu_{i1}, \dots , \mu_{ip}$ be a list
of its $p$th roots in $\C$. Then,
\[
(1-\lambda_it^p) = \prod_{j=1}^p(1-\mu_{ij} t) \ .
\]
Thus, $\det (I-tA) =\prod_{i=1}^k \prod_{i=1}^p(1-\mu_{ij} t) $,
and it follows that
the nonzero\si{nonzero spectrum} spectrum of $A$ is $\Lambda^{1/p}$.
\end{proof}

\begin{rem} \label{coeff}
The Coefficients Condition of the Spectral Conjecture
holds
if the ring $\mathcal R$ contains $\Q$
and if $\tr (\Lambda^n) \in \mathcal R$ for all positive
integers $n$.
For a self contained proof of this,
consider the
companion matrix $C$ to  the polynomial
\[
p(x)
= \prod_{i=1}^k(t-\lambda_i) = t^k -c_1t^{k-1} -c_2t^{k-2}-   \dots  \ \ .
\]
Clearly $c_1\in \mathcal R$ iff $\tr (\Lambda) \in \mathcal R$.
Now suppose $c_1, \dots , c_{j-1}$ are in $\mathcal R$, and $j\leq k$.
From this assumption and the form
of $C$,  we have that $jc_j$ equals an element of $\mathcal R$ plus
$\tr (\Lambda^j)$.

In particular, the Coefficients Condition
is redundant if $\mathcal R=\R$, because
$\tr (\Lambda^n) \geq 0$ implies
$\tr (\Lambda^n) \in \R$.
\end{rem}

  \begin{rem} \label{johnsonzeros}
   In \cite[Section 4]{Johnson1981},
      Johnson\ai{Johnson, Charles}
    wrote the
    following.
``Suppose the set of numbers $\{\lambda_1, \dots , \lambda_m\}$
is not the spectrum of an $m\times m$ nonnegative matrix. Is it
possible to ``save'' this set by appending $n-m>0$ zeros, that is,
might $\{ \lambda_1, \dots , \lambda_m , 0, \dots , 0\}$ be
the spectrum of a nonnegative matrix?''
Johnson\ai{Johnson, Charles} also (among the many results in \cite{Johnson1981})
gave an example of such a ``save'' with $m=4$ (minimum possible);
established the JLL inequality \eqref{jllformnk};
and noted that it
  ``is more likely to be
satisfied as zeroes are added to the proposed spectrum.''
\end{rem}

\begin{rem} \label{HIEP}
Given $A$ over $\mathcal R\neq \Z$ with a Perron value $\lambda$,
Handelman finds $U$ invertible over $\mathcal R$ such that $U^{-1}AU$
has positive left and right eigenvectors for $\lambda$. This matrix
$U^{-1}AU$ must be eventually positive. He also exhibits an obstruction
to this in the case $\mathcal R=\Z$: if $\ell, r$ are left, right
integral eigenvectors for $\lambda$, then the minimum inner product
$\ell \cdot r$ does not improve with similarity, and if it is smaller
than the size of $A$ then it is impossible to find $U$ invertible over
$\Z$ such that $U^{-1}AU$ has the positive left, right eigenvectors.
But, if needed, Handelman\ai{Handelman, David} produces an SSE-$\Z$ to a larger matrix
(of smallest size possible) for which
he produces the desired
$U$.
\end{rem}

\begin{prop} \label{reflectjll}
If Laffey's\ai{Laffey, Thomas J.} formula for
an upper bound on $N$ were
replaced by a formula of the form
$N \leq  f(G,n)$, then even at $n=3$
the formula could not give a correct bound.
\end{prop}
\begin{proof}
To show this, it suffices to exhibit a family $\{ \Lambda_{\epsilon}: 0 <
\epsilon < 1/2 \} $ of
3-tuple nonzero\si{nonzero spectrum} spectra of primitive matrices, with spectral gaps
bounded away from zero,  which cannot be realized
by matrices of bounded size.

Set  $ \Lambda_{\epsilon}
= (1, i\sqrt{(1-\epsilon)/2}, -i\sqrt{(1-\epsilon)/2})$.
Each $ \Lambda_{\epsilon} $  satisfies the conditions of
the Spectral Conjecture for $\R$, with  spectral gap
greater than 1/2.
But, if
$\Lambda_{\epsilon} $ is the nonzero\si{nonzero spectrum} spectrum of
an $N\times N$ matrix,
we have already seen
from
the JLL inequalities that $N \geq 1/\epsilon$.
\end{proof}

\begin{prop} \label{gaprefl}
If Laffey's\ai{Laffey, Thomas J.} formula for
an upper bound on $N$ were
replaced by a formula of the form
$N \leq  f(M,n)$, then even at $n=4$
the formula could not give a correct bound.
\end{prop}

\begin{proof}
It suffices to find a family
$\{ \Lambda_{\epsilon}: 0 < \epsilon <\epsilon_0\}  $ of
4-tuple nonzero\si{nonzero spectrum} spectra of primitive matrices, with
\[
\inf_{\epsilon} \inf
\{ \tr ( (\Lambda_{\epsilon})^k) : k\in \N  \} \ > \ 0 \  ,
\]
such that the $\Lambda_{\epsilon}$ cannot be
the  nonzero\si{nonzero spectrum} spectra of matrices of bounded size.

Let $\Lambda_{\epsilon} = (1, 1-\epsilon , .9i, -.9i)$,
with $0< \epsilon < \epsilon_0 = . 0001$ (to avoid computation).
Each $\Lambda_{\epsilon}$ is the nonzero\si{nonzero spectrum} spectrum of a primitive
matrix over $\R$.  For
$\Lambda= (1,1, .9i , - .9i)$, for $n\in \N$,
$\tr (\Lambda^{2n}) = 2$ and
$\tr (\Lambda^{2n+1}) = 2 -2(.9)^n$ and therefore
$\tr (\Lambda^{n}) \geq  2 -2(.9) = .2 $. With $\epsilon_0$ small
enough, likewise
$ \inf_{\epsilon} \inf
\{ \tr ( (\Lambda_{\epsilon})^k) : k\in \N  \}  >  0 $ .

Suppose for some positive integer $K$, for each $\Lambda_{\epsilon}$
there is a nonnegative matrix $A_{\epsilon}$ of size $K\times K$
with nonzero\si{nonzero spectrum} spectrum
$\Lambda_{\epsilon}$. Then by compactness,
there is  a subsequence of the sequence $(A_{1/n})$
which converges to a nonnegative matrix $A$. The spectrum is a
continuous function of the matrix entries, so $A$ has nonzero\si{nonzero spectrum} spectrum
$\Lambda= (1,1, .9i , - .9i)$. By the Perron-Frobenius
spectral constraints,  $A$ cannot be irreducible, and
$\Lambda$  is the union of nonzero\si{nonzero spectrum} spectra
of irreducible matrices, $(1)$ and $(1, .9i, -.9i)$.
But  $1^2 + (.9i)^2 + ( -.9i)^2= - 1.8 < 0$, a contradiction.
\end{proof}
\begin{rem} \label{gscrefs}
The Weak Generalized Spectral Conjecture was stated in the
1991 publication
\cite{BH91}.
The Strong Generalized Spectral Conjecture was stated in
the 1993 publication
\cite{Boyle91matrices}. Although Handelman\ai{Handelman, David} was not a coauthor
of the latter paper,
the Strong conjecture was a conjecture by both of us.

\end{rem}

\section{A brief introduction to algebraic K-theory}
Shift equivalence\si{shift equivalence} and strong shift equivalence\si{strong shift equivalence} are relations on sets of matrices over a semiring. Algebraic K-theory offers many tools for such a setting\footnote{At the beginning of the book {\it Algebraic K-theory and Its Applications} \cite{RosenbergBook}, the author Jonathan Rosenberg writes ``Algebraic K-theory is the branch of algebra dealing with linear algebra over a ring''.}, so it is natural to suspect algebraic K-theory might be useful for studying the relations of shift and strong shift equivalence\si{strong shift equivalence}. This suspicion is correct, and we will present two cases where this happens:

\begin{enumerate}
\item
For a general ring $\calR$, the refinement of SE-$\calR$ by SSE-$\calR$.
\item
Wagoner's\ai{Wagoner, J.B.} obstruction map detecting a difference between SE-$\mathbb{Z}_{+}$ and SSE-$\mathbb{Z}_{+}$
\end{enumerate}
The first is a purely algebraic problem, motivated by applications to symbolic dynamics, and to topics in algebra. The second, Wagoner's\ai{Wagoner, J.B.} obstruction map, is concerned with an ``order'' problem, and is one of two known methods to produce counterexamples to Williams' Conjecture\si{Williams' Conjecture} (discussed in Lecture 1).\\
\indent Lectures 5 and 6 will focus on addressing the first item above. Lecture 7 will discuss automorphisms of shifts of finite type, an important topic in its own right. Lecture 7 is also used partly to prepare for Lecture 8, which addresses the second item above.\\

To begin, we introduce some necessary background from algebraic K-theory, relevant for Lecture 6.

\subsection{$K_{1}$ of a ring\si{$K_{1}$ of a ring} $\calR$}\label{k1ofring}
Given a ring $\calR$, consider the group $GL_{n}(\calR)$ of invertible $n \times n$ matrices over $\calR$. If one wishes to understand the structure of this group, a natural question one may ask is: what is the abelianization of $GL_{n}(\calR)$? While the answer may be fairly complicated depending on $n$ and $\calR$, Whitehead, in 1950 in \cite{Whitehead1950}, made a beautiful observation: by stabilizing, the commutator subgroup becomes more accessible.\\

To describe Whitehead's result, first let us say that by stabilizing, we mean the following.
\begin{de}\label{def:stabilization}
For any $n$, there is a group homomorphism
\begin{gather*}
GL_{n}(\calR) \hookrightarrow GL_{n+1}(\calR)\\
A \mapsto \begin{pmatrix} A & 0 \\ 0 & 1 \end{pmatrix}
\end{gather*}
and we define
$$GL(\calR) = \varinjlim GL_{n}(\calR).$$
\end{de}

The group $GL(\calR)$ is often called the {\it stabilized general linear group} (over the ring $\calR$).\\

An important collection of invertible matrices are the {\it elementary matrices}. A matrix $E \in GL_{n}(\calR)$ is an elementary matrix if $E$ agrees with the identity except in at most one off-diagonal entry. The following observation may be familiar from linear algebra: if $E$ is an $n \times n$ elementary matrix and $B$ is any $n \times n$ matrix then
\begin{enumerate}
\item
$EB$ is obtained from $B$ by an elementary row operation (adding a multiple of one row of $B$ to another row of $B$).
\item
$BE$ is obtained from $B$ by an elementary column operation (adding a multiple of one column of $B$ to another column of $B$).
\end{enumerate}

We define $El_{n}(\calR)$ to be the subgroup of $GL_{n}(\calR)$ generated by $n \times n$ elementary matrices.\\

Like $GL(\calR)$, we can also stabilize the elementary subgroups. The homomorphisms in Definition \ref{def:stabilization} map $El_{n}(\calR)$ to $El_{n+1}(\calR)$, and we define
$$El(\calR) = \varinjlim El_{n}(\calR).$$

If $X \in El(\calR)$, then $X$ can be written as a product of elementary matrices
$$X = \prod_{i=1}^{k}E_{i}.$$
It follows that, for any matrix $A \in GL(\calR)$, $XA$ is obtained from $A$ by performing a sequence of row operations, and $AX$ is obtained from $A$ by performing a sequence of column operations.\\

Note that when we write $AX$ and $XA$, $A$ and $X$ may be of different sizes. However, the process of stabilization allows us replace $A$ with $A \oplus I$ or $X$ with $X \oplus I$ as necessary to carry out the multiplication.\\

The group $El(\calR)$ turns out to be the key to analyzing the abelianization of $GL(\calR)$.

\begin{thm}[Whitehead]\label{thm:whitehead}
For any ring $\calR$, $[GL(\calR),GL(\calR)] = El(\calR)$.
\end{thm}
A proof of this can be found in a number of places; for example, see \cite[Chapter III]{WeibelBook}. To see why the commutator $[GL(\calR),GL(\calR)]$ is contained in $El(\calR)$, one can check that if $A \in GL_{n}(\calR)$, then
$$\begin{pmatrix} A & 0 \\ 0 & A^{-1} \end{pmatrix} = \begin{pmatrix} 1 & A \\ 0 & 1 \end{pmatrix} \begin{pmatrix} 1 & 0 \\ -A^{-1} & 1 \end{pmatrix} \begin{pmatrix} 1 & A \\ 0 & 1 \end{pmatrix} \begin{pmatrix} 0 & -1 \\ 1 & 0 \end{pmatrix}$$
and that the last matrix in the above lies in $El(\calR)$, so that $\begin{pmatrix} A & 0 \\ 0 & A^{-1} \end{pmatrix}$ is always in $El(\calR)$. Now observe that we have
$$\begin{pmatrix} ABA^{-1}B^{-1} & 0 \\ 0 & I \end{pmatrix} = \begin{pmatrix} A & 0 \\ 0 & A^{-1} \end{pmatrix} \begin{pmatrix} B & 0 \\ 0 & B^{-1} \end{pmatrix} \begin{pmatrix} (BA)^{-1} & 0 \\ 0 & BA \end{pmatrix}$$
so any commutator lies in $El(\calR)$.

\begin{de}
For a ring $\calR$, the first algebraic K-group\si{$K_{1}$ of a ring} (of $\calR$) is defined by
$$K_{1}(\calR) = GL(\calR)_{ab} = GL(\calR) / El(\calR).$$
\end{de}

We use $[A]$ to refer to the class of a matrix $A$ in $K_{1}(\calR)$.\\

The second equality in the above definition is precisely Whitehead's Theorem. We note a few things regarding $K_{1}$:

\begin{enumerate}
\item
$K_{1}(\calR)$ is always an abelian group.
\item
As noted before, multiplying a matrix $A$ by an elementary matrix from the left (resp. right) corresponds to performing an elementary row (resp. column) operation on $A$. Thus the group $K_{1}(\calR)$ coincides with equivalence classes of (stabilized) invertible matrices over $\calR$, where two matrices are equivalent if one can be obtained from the other by a sequence of elementary row and column operations.
\item
The group operation in $K_{1}(\calR)$ is, by definition,
$$[A][B] = [AB]$$
where again the product $AB$ is defined because we have stabilized. However, the group operation is equivalently defined by
$$[A] + [B] = \left[\begin{pmatrix} A & 0 \\ 0 & B \end{pmatrix}\right].$$
To see this, as we noted before, for any $A \in GL(\calR)$, the matrix $\begin{pmatrix} A & 0 \\ 0 & A^{-1} \end{pmatrix}$ is in $El(\calR)$. Since we have stabilized, we may assume that $A$ and $B$ are the same size, and
\begin{align*}
  [AB] &= \left[\begin{pmatrix} A & 0 \\ 0 & I\end{pmatrix} \begin{pmatrix} B & 0 \\ 0 & I \end{pmatrix} \right]\left[\begin{pmatrix} B^{-1} & 0 \\ 0 & B \end{pmatrix}\right]\\
  &= \left[\begin{pmatrix} A & 0 \\ 0 & I\end{pmatrix} \begin{pmatrix} B & 0 \\ 0 & I \end{pmatrix} \begin{pmatrix} B^{-1} & 0 \\ 0 & B \end{pmatrix}\right] = \left[\begin{pmatrix} A & 0 \\ 0 & B \end{pmatrix}\right]\ .
  \end{align*}
\end{enumerate}

Historically, one of Whitehead's main motivations was to define what is now called {\it Whitehead torsion}. If $f \colon X \to Y$ is a homotopy equivalence between two finite CW complexes, Whitehead showed how to define a certain torsion class $\tau(f)$ in $K_{1}(\mathbb{Z}\pi_{1}(X))$. He showed that $f$ is a simple homotopy equivalence (one obtained through some finite sequence of elementary moves) if and only if $\tau(f) = 0$. For more on this, see \cite[Section 2.4]{RosenbergBook}.\\

What about computing $K_{1}(\calR)$? In general this is a difficult problem, but there are many cases where the answer is accessible, and we'll give some examples shortly. \\

Before discussing these examples, suppose now that $\calR$ is commutative. Then there is a determinant homomorphism
\begin{gather*}
\deter \colon K_{1}(\calR) \to \calR^{\times}\\
\deter([A]) = \deter(A).
\end{gather*}
The kernel of the determinant map is denoted by
$$SK_{1}(\calR) = \ker \deter.$$
Since the determinant map is surjective and right split (by identifying $\calR^{\times}$ with $GL_{1}(\calR)$), we get an exact sequence of abelian groups
$$0 \to SK_{1}(\calR) \longrightarrow K_{1}(\calR) \stackrel{\deter}\longrightarrow \calR^{\times} \to 0$$
and
$$K_{1}(\calR) \cong SK_{1}(\calR) \oplus \calR^{\times}.$$

The determinant map turns out to be very useful in actually computing $K_{1}(\calR)$; often, it is actually an isomorphism.\\

Here are a few examples of $K_{1}$ for some rings.
\begin{enumerate}
\item
When $\calR$ is a field, or even a Euclidean domain, the group $SK_{1}(\calR)$ is trivial, and $K_{1}(\calR) \cong \calR^{\times}$. When $\calR$ is a field, this is just the classical fact that, over a field, any invertible matrix $A$ can be row and column reduced to the matrix $\det A \oplus 1$. When $\calR$ is a Euclidean domain, $SK_{1}(\calR) = 0$ as well (see \cite[Ex. 1.3.5]{WeibelBook}). Thus for example
$$K_{1}(\mathbb{Z}) \cong \mathbb{Z}/2\mathbb{Z} = \{1,-1\}$$
where we've identified $\{1,-1\}$ with the group of units in $\mathbb{Z}$.
\item
If $\calR$ is an integrally closed subring of a finite field extension $E$ of $\mathbb{Q}$, then $SK_{1}(\calR) = 0$ (this is a deep theorem of Bass, Milnor, and Serre; see \cite[4.3]{BMS67}).
\item
When $G$ is an abelian group, the integral group ring $\mathbb{Z}G$ is commutative, so $SK_{1}(\mathbb{Z}G)$ is defined. There are finite abelian groups $G$ for which $SK_{1}(\mathbb{Z}G) \ne 0$; for example, if $H = \mathbb{Z}/4\mathbb{Z} \times \mathbb{Z}/2\mathbb{Z} \times \mathbb{Z}/2\mathbb{Z}$, then $SK_{1}(\mathbb{Z}H) \cong \mathbb{Z}/2\mathbb{Z}$ \cite[Example 5.1]{OliverBook}). In general, the calculation of $SK_{1}(\mathbb{Z}G)$ is very nontrivial (see \cite{OliverBook}).
\end{enumerate}

This last example is especially important in topology (see \cite[Section 4]{RosenbergBook} for a brief discussion of this), and in addition, has applications to symbolic dynamics; see \cite{BoSc2}.\\

\subsection{$NK_{1}(\calR)$\si{$NK_{1}$ of a ring}}
We introduce now a certain algebraic $K$-group called $NK_{1}(\calR)$\si{$NK_{1}$ of a ring}. This group will play a key role for us later, when we discuss strong shift equivalence\si{strong shift equivalence} and shift equivalence\si{shift equivalence} over a ring $\calR$.\\

Any homomorphism of rings $f \colon \calR \to \calS$ induces, for each $n$, a homomorphism of groups $GL_{n}(\calR) \to GL_{n}(\calS)$ and hence a group homomorphism $GL(\calR) \to GL(\calS)$. The homomorphism $f$ then induces a group homomorphism on $K_{1}$
$$f_{*} \colon K_{1}(\calR) \to K_{1}(\calS).$$
In fact, the assignment $\calR \to K_{1}(\calR)$ defines a functor from the category of rings to the category of abelian groups. For any ring $\calR$, we may consider the ring of polynomials $\calR[t]$ over $\calR$, and there is a ring homomorphism
\begin{gather*}
ev_{0} \colon \calR[t] \to \calR\\
p(t) \mapsto p(0).
\end{gather*}
This induces a homomorphism on $K_{1}$
$$(ev_{0})_{*} \colon K_{1}(\calR[t]) \to K_{1}(\calR)$$
and the kernel of this map is denoted by\si{$NK_{1}$ of a ring}
$$NK_{1}(\calR) = \ker \left(K_{1}(\calR[t]) \stackrel{(ev_{0})_{*}}\longrightarrow K_{1}(\calR)\right).$$

Thus by definition, $NK_{1}(\calR)$ is a subgroup of $K_{1}(\calR[t])$. In particular, it is always an abelian group.\\

The group $NK_{1}(\calR)$ is important in algebraic K-theory. It appears (among other places) in the Fundamental Theorem of Algebraic K-theory, relating the K-groups of $\calR[t]$ and $\calR[t,t^{-1}]$ to the K-groups of $\calR$
(see \cite[Chapter III]{WeibelBook}).\\



Here are a few facts about $NK_{1}(\calR)$:
\begin{enumerate}
\item
  If $\calR$ is a Noetherian regular ring
  (see \cite[Chapter 3]{RosenbergBook}), then $NK_{1}(\calR) = 0$. In particular, if $\calR$ is a field, a PID, or a Dedekind domain, then $NK_{1}(\calR) = 0$ (see \cite[III.3.8]{WeibelBook}).
\item
A theorem of Farrell \cite{Farrell} shows that if $NK_{1}(\calR) \ne 0$, then it is not finitely generated as an abelian group.
\end{enumerate}

Thus, to summarize the above two items: $NK_{1}(\calR)$ very often vanishes, but when it doesn't vanish, it's large (as an abelian group).\\

There are rings $\calR$ for which $NK_{1}(\calR) \ne 0$. For an easy example, take any commutative ring $\calR$, and let $\calS = \calR[s] / (s^{2})$. Then $NK_{1}(\calS) \ne 0$. Indeed, over the ring $\calS[t]$, the matrix $(1+st)$ is invertible, and hence we can consider its class $[(1+st)] \in K_{1}(\calS[t])$. Clearly $[(1+st)]$ lies in $NK_{1}(\calS)$, and the class $[1+st]$ is nontrivial in $K_{1}(\calS[t])$ since $\det(1+st) \ne 1$.\\

Here are some more interesting examples:

\begin{enumerate}\label{exa:nk1example}
\item
$NK_{1}(\mathbb{Q}[t^{2},t^{3},z,z^{-1}]) \ne 0$ (see \cite{SchmiedingExamples} for details on this calculation). This is a nontrivial fact: since the ring $\mathbb{Q}[t^{2},t^{3},z,z^{-1}]$ is reduced (has no nontrivial nilpotent elements), we have $NK_{1}(\calR) \subset SK_{1}(\calR[t])$ (see Exercise \ref{exer:reducednk1sk1} below), and often it is not easy to determine whether $SK_{1}$ vanishes\footnote{To be convinced of the difficulties in determining whether $SK_{1}$ vanishes, see the introduction of Oliver's very thorough book \cite{OliverBook}.}.
\item
There are finite groups $G$ for which $NK_{1}(\mathbb{Z}G) \ne 0$; for example, for $G = \mathbb{Z}/4\mathbb{Z}$, $NK_{1}(\mathbb{Z}[\mathbb{Z}/4\mathbb{Z}]) \ne 0$ (details for this particular $G$ can be found in \cite{Weibelnk1calc}).
\end{enumerate}

See \cite{BoSc3} for an application of the example $(1)$ above. The example $(2)$ above of integral group rings of finite groups is relevant for applications to symbolic dynamics (see \cite{BoSc2}). In general, the calculation of $NK_{1}(\mathbb{Z}G)$ for $G$ a finite group is complicated, and not fully known (see e.g. \cite{Harmon1987}, \cite{Weibelnk1calc}).\\

The following is a very useful tool for studying $NK_{1}(\calR)$. The result is often referred to as Higman's Trick.
\begin{thm}[Higman]\label{lemma:higmantrick}
Let $\calR$ be a ring and let $A$ be a matrix in $GL(\calR[t])$ such that $[A] \in NK_{1}(\calR)$. Then there exists a nilpotent matrix $N$ over $\calR$ such that $[A] = [I-tN]$ in $NK_{1}(\calR)$.
\end{thm}
\begin{proof}[Sketch of proof]
Use the fact that we are in the stabilized setting to kill off powers of $t$ from $A$ using elementary operations, arriving at a matrix of the form $A_{0}+A_{1}t$. Since $[A] \in NK_{1}(\calR)$, $[A_{0}] = 0 \in K_{1}(\calR)$, so $[A] = [I+B_{1}t]$ for some $B_{1}$ over $\calR$. Since the matrix $I+B_{1}t$ is invertible over $\calR[t]$, $B_{1}$ must be nilpotent.
\end{proof}

A more detailed proof of Theorem \ref{lemma:higmantrick} may be found in \cite[III.3.5.1]{WeibelBook}.

\begin{exer}\label{exer:reducednk1sk1}
\apr{appexer:reducednk1sk1})
Suppose $\calR$ is a commutative ring which is reduced, i.e. $\calR$ has no nontrivial nilpotent elements. Then $NK_{1}(\calR) \subset SK_{1}(\calR[t])$.
\end{exer}

\begin{exer}\label{exer:nk1vanishpid}
\apr{appexer:nk1vanishpid})
If $\calR$ is a principal ideal domain, then $NK_{1}(\calR) = 0$.
\end{exer}


\subsection{$Nil_{0}(\calR)$\si{$Nil_{0}$ of a ring}}\label{sec:subsecnil0R}
Higman's Trick suggests there is a connection between the group $NK_{1}(\calR)$ and the structure of nilpotent matrices over the ring $\calR$. This is indeed the case, and we'll describe this relationship quite explicitly in this subsection \apr{rem:nilviewlocalization}). To begin, we first define another group coming from algebraic K-theory, the {\it class group of the category of nilpotent endomorphisms over $\calR$}. That's quite a long name, and we usually just call it ``nil zero (of $\calR$)'', since it's denoted by $Nil_{0}(\calR)$\si{$Nil_{0}$ of a ring}.\\

\begin{de}\label{def:nil0def1}
Let $\calR$ be a ring. Define $Nil_{0}(\calR)$\si{$Nil_{0}$ of a ring} to be the free abelian group on the set of generators
$$\{[N] \mid N {\it  \textnormal{ is a nilpotent matrix over} } \calR\}$$
together with the following relations:
\begin{enumerate}
\item
$[N_{1}] = [N_{2}]$ if $N_{1} = P^{-1}N_{2}P$ for some $P \in GL(\calR)$.
\vspace{.07in}
\item
$[N_{1}] + [N_{2}] = \left[ \begin{pmatrix} N_{1} & B \\ 0 & N_{2} \end{pmatrix} \right]$ for any matrix $B$ over $\calR$.
\vspace{.07in}
\item
$[0] = 0$.
\end{enumerate}
\end{de}
Where does the group $Nil_{0}(\calR)$ come from? First let us recall some definitions.
Consider the category $\textbf{Nil}\calR$ whose objects are pairs $(P,f)$ where $P$ is a finitely generated projective $\calR$-module and $f$ is a nilpotent endomorphism of $P$, and where a morphism from $(P,f)$ to $(Q,g)$ is given by an $\calR$-module homomorphism $\alpha \colon P \to Q$ for which the square
\[
\xymatrix{
P \ar[r]^{f} \ar[d]_{\alpha} & P \ar[d]^{\alpha} \\
Q \ar[r]^{g} & Q \\
}
\]
commutes. The category $\textbf{Nil}\calR$ has a notion of exact sequence by defining
$$(P_{1},f_{1}) \to (P_{2},f_{2}) \to (P_{3},f_{3})$$
to be exact if the corresponding sequence of $\calR$-modules
$$P_{1} \to P_{2} \to P_{3}$$
is exact, i.e. $\image(P_{1} \to P_{2}) = \ker(P_{2} \to P_{3})$ (see \apr{rem:exactcategory}) regarding how $\textbf{Nil}\calR$ with this notion of exact sequence fits into a more general setting). Given this, define $K_{0}(\textbf{Nil}\calR)$ to be the free abelian group on isomorphism classes of objects $(P,f)$ in $\textbf{Nil}\calR$, together with the relation:
\begin{gather*}
[(P_{1},f_{1})] + [(P_{3},f_{3})] = [(P_{2},f_{2})] \\
\textnormal{ whenever } \\
0 \to (P_{1},f_{1}) \to (P_{2},f_{2}) \to (P_{3},f_{3}) \to 0\\
\textnormal{ is exact. }
\end{gather*}
Let $\textbf{Proj}\calR$ denote the category of finitely generated projective $\calR$-modules and consider the standard notion of an exact sequence in $\textbf{Proj}\calR$. We can likewise define the group $K_{0}(\textbf{Proj}\calR)$ to be the free abelian group on isomorphism classes of objects in $\textbf{Proj}\calR$ with the similar relations:
\begin{gather*}
[P_{1}] + [P_{3}] = [P_{2}] \\
\textnormal{ whenever } \\
0 \to P_{1} \to P_{2} \to P_{3} \to 0 \\
\textnormal{ is exact in } \textbf{Proj}\calR.
\end{gather*}
These relations are equivalent to the set of relations
$$[P_{1}] + [P_{2}] = [P_{1} \oplus P_{2}], \qquad P_{1},P_{2} \textnormal{ in } \textbf{Proj}\calR$$
since any exact sequence of projective $\calR$-modules splits. Thus $K_{0}(\textbf{Proj}\calR)$ is isomorphic to the group completion of the abelian monoid of isomorphism classes of finitely generated projective $\calR$-modules under direct sum, which is often given as the definition of the group $K_{0}(\calR)$.\\

There is a functor $\textbf{Nil}\calR \to \textbf{Proj}\calR$ given by $(P,f) \mapsto P$, and this functor respects exact sequences, so there is an induced map on the level of the $K_{0}$ groups defined above
$$K_{0}(\textbf{Nil}\calR) \to K_{0}(\textbf{Proj}\calR).$$
The kernel of this map is isomorphic to $Nil_{0}(\calR)$ (details of this isomorphism can be found in \cite[Chapter II]{WeibelBook}).\\

The following formalizes the connection between $NK_{1}(\calR)$ and nilpotent matrices over $\calR$.
\begin{thm}\label{thm:nil0nk1iso}
The map
\begin{equation}\label{eqn:nil0nk1iso}
\begin{gathered}
\Psi \colon Nil_{0}(\calR) \to NK_{1}(\calR)\\
\Psi \colon [N] \mapsto [I-tN]
\end{gathered}
\end{equation}
is an isomorphism of abelian groups.
\end{thm}

\begin{exer}\apr{appexer:psiwelldefined})
Show the map $\Psi$ defined in \eqref{eqn:nil0nk1iso} is a well-defined group homomorphism.
\end{exer}

Towards showing $\Psi$ is an isomorphism, given Higman's Theorem \ref{lemma:higmantrick} above, one obvious thing to try is to define an inverse map
\begin{equation}\label{eqn:nk1nil0}
\begin{gathered}
NK_{1}(\calR) \to Nil_{0}(\calR)\\
[I-tN] \mapsto [N].
\end{gathered}
\end{equation}

This in fact works: this map turns out to be well-defined, and is an inverse to the map $\Psi$. This is classically done, in algebraic K-theory, using a fair amount of machinery and long exact sequences coming from localization results (e.g. \cite[III.3.5.3]{WeibelBook}). Later we will see there is an alternative, more elementary, proof using strong shift equivalence\si{strong shift equivalence} theory.\\

We will make frequent use of the isomorphism \eqref{eqn:nil0nk1iso} above in later lectures.

\begin{exer}\apr{appexer:uppertrivanish})
Consider an upper triangular matrix $N$ over $\calR$ with zero diagonal. Then $I-tN$ lies in $El(\calR[t])$, and hence $[I-tN] = 0$ in $NK_{1}(\calR)$. Using the relations defining $Nil_{0}(\calR)$, show the class of such an $N$ must be zero in $Nil_{0}(\calR)$.
\end{exer}

The isomorphism $NK_{1}(\calR) \cong Nil_{0}(\calR)$ is only one instance of a larger phenomenon, which, loosely speaking, relates the K-theory of polynomial rings $\calR[t]$ (in fact, certain localizations of them) to the K-theory of endomorphisms over the ring $\calR$ \apr{rem:ktheoryendomorphisms}). The strong shift equivalence\si{strong shift equivalence} theory also fits nicely into this framework, and we'll describe this in a little more detail later.

\subsection{$K_{2}$ of a ring\si{$K_{2}$ of a ring} $\calR$}\label{subsec:k2ofaring}
This short subsection gives a definition and a few very basic properties of the group $K_{2}$ of a ring, motivated by its appearance later in Lecture 8. For a more thorough introduction to $K_{2}$, see either \cite{MilnorBook} or \cite[III. Sec. 5]{WeibelBook}.\\

Roughly speaking, $K_{2}(\calR)$ measures the existence of ``extra relations'' among elementary matrices over $\calR$. We'll make this more formal below, but the idea is that elementary matrices always satisfy a certain collection of relations which do not depend on the ring. The group $K_{2}(\calR)$ is a way to detect additional relations coming from the ring.\\

Let $\calR$ be a ring. Given $n \ge 1$ and $1 \le i \ne j \le n$, let $e_{i,j}(r)$ denote the matrix which has $r$ in the $i,j$ entry, and agrees with the identity matrix everywhere else. Recall the group $El_{n}(\calR)$ of $n \times n$ elementary matrices over $\calR$ is generated by matrices $e_{i,j}(r)$, $i \ne j$. It is straightforward to check that $El_{n}(\calR)$ always satisfies certain relations: for any $r, s \in \calR$, we have
\begin{enumerate}
\item
$e_{i,j}(r)e_{i,j}(s) = e_{i,j}(r+s)$.
\item
$[e_{i,j}(r),e_{k,l}(s)] =
\begin{cases}
1 & \mbox{if } i \ne l \mbox{ and } j \ne k\\
e_{i,l}(rs) & \mbox{if } i \ne l \mbox{ and } j = k\\
e_{k,j}(-sr) & \mbox{if } j \ne k \mbox{ and } i = l.
\end{cases}$
\end{enumerate}

The key here is that these relations are satisfied by $El_{n}(\calR)$ for {\it every} ring. This perhaps motivates defining the following group:

\begin{de}
Let $\calR$ be a ring and $n \ge 3$. The $n$th Steinberg group $St_{n}(\calR)$ has generators $x_{i,j}(r)$, where $1 \le i \ne j \le n$ and $r \in \calR$, and relations:
\begin{enumerate}
\item
$x_{i,j}(r)x_{i,j}(s) = x_{i,j}(r+s)$.
\item
$[x_{i,j}(r),x_{k,l}(s)] =
\begin{cases}
1 & \mbox{if } i \ne l \mbox{ and } j \ne k\\
x_{i,l}(rs) & \mbox{if } i \ne l \mbox{ and } j = k\\
x_{k,j}(-sr) & \mbox{if } j \ne k \mbox{ and } i = l.
\end{cases}$
\end{enumerate}
\end{de}

The map
$$x_{i,j}(r) \mapsto e_{i,j}(r)$$
defines a surjective group homomorphism
$$\theta_{n} \colon St_{n}(\calR) \to El_{n}(\calR).$$

The relations for $St_{n}(\calR)$ and $St_{n+1}(\calR)$ imply there is a well-defined group homomorphism
\begin{equation*}
\begin{gathered}
St_{n}(\calR) \to St_{n+1}(\calR)\\
x_{ij}(r) \mapsto x_{ij}(r)\\
\end{gathered}
\end{equation*}
and we define
$$St(\calR) = \varinjlim St_{n}(\calR)$$
and assemble the $\theta_{n}$'s to get a group homomorphism
$$\theta \colon St(\calR) \to El(\calR).$$

Finally, we define\si{$K_{2}$ of a ring}
$$K_{2}(\calR) = \ker \theta.$$

It turns out the sequence
$$K_{2}(\calR) \to St(\calR) \to El(\calR)$$
is the universal central extension of the group $El(\calR)$. The group $K_{2}(\calR)$ is precisely the center of $St(\calR)$, and so is always abelian. Furthermore, the assignment $\calR \to K_{2}(\calR)$ is functorial; see \cite[III, Sec.5]{WeibelBook} for more details on this.\\

An observation we'll make use of later is the following. An expression of the form
$$\prod_{i=1}^{k}E_{i} = 1$$
where $E_{i}$ are elementary matrices can be used to produce an element of $K_{2}(\calR)$: lift each $E_{i}$ to some $x_{i}$ in $St(\calR)$ and consider
$$x = \prod_{i}^{k}x_{i} \in St(\calR).$$
Then $x \in K_{2}(\calR)$, although in general this element may depend on the choice of lifts.\\

\begin{ex}
Let $\calR = \mathbb{Z}$, and consider
$$E = e_{1,2}(1)e_{2,1}(-1)e_{1,2}(1) = \begin{pmatrix} 0 & 1 \\ -1 & 0 \end{pmatrix}.$$
One can check directly that
$$E^{4} = I$$
so we can consider the element of $K_{2}(\mathbb{Z})$
$$x = \big(x_{1,2}(1)x_{2,1}(-1)x_{1,2}(1)\big)^{4}.$$
Milnor in \cite[Sec. 10]{MilnorBook} proves that $x$ is nontrivial in $K_{2}(\mathbb{Z})$, $x^{2} = 1$, and $x$ is actually the only nontrivial element of $K_{2}(\mathbb{Z})$. Thus we have
$$K_{2}(\mathbb{Z}) \cong \mathbb{Z}/2\mathbb{Z}.$$
\end{ex}

It turns out (see \cite[Chapter V]{WeibelBook}) that $K_{2}(\mathbb{Z}[t]) \cong K_{2}(\mathbb{Z})$. Given $m \ge 1$, there is a split surjection $K_{2}(\mathbb{Z}[t]/(t^{m})) \to K_{2}(\mathbb{Z})$, and we can define the group $K_{2}(\mathbb{Z}[t]/(t^{m}),(t))$ to be the kernel of this split surjection. In \cite{vdK}, van der Kallen proved that $K_{2}(\mathbb{Z}[t]/(t^{2}),(t)) \cong \mathbb{Z}/2\mathbb{Z}$, a fact which will prove to be useful later in Lecture 8. More generally, the following was proved by Geller and Roberts.
\begin{thm}[{\cite[Section 7]{Roberts1979}}]\label{thm:k2calctruncatedpoly}
For any $m \ge 2$, the group $K_{2}(\mathbb{Z}[t]/(t^{m}),(t))$ is isomorphic to $\bigoplus_{k=2}^{m}\mathbb{Z}/k\mathbb{Z}$.
\end{thm}

\subsection{Appendix 5}
This appendix contains some remarks, proofs, and solutions of exercises for Lecture 5.

\begin{exer}\label{appexer:reducednk1sk1}
Suppose $\calR$ is a commutative ring which is reduced, i.e. $\calR$ has no nontrivial nilpotent elements. Then $NK_{1}(\calR) \subset SK_{1}(\calR[t])$.
\end{exer}
\begin{proof}
  If $\calR$ is commutative and reduced, the only units in $\calR[t]$ are degree zero. Thus for a nilpotent matrix $N$ over $\calR$, since $I-tN$ is invertible,
  we have $\det(I-tN)= 1$.  So,  together with Higman's Trick (Theorem \ref{lemma:higmantrick}), we have $NK_{1}(\calR) \subset SK_{1}(\calR[t])$.
\end{proof}

\begin{exer}\label{appexer:nk1vanishpid}
If $\calR$ is a principal ideal domain, then $NK_{1}(\calR) = 0$.
\end{exer}
\begin{proof}
By Higman's Trick (Theorem \ref{lemma:higmantrick}), it suffices to show that if $N$ is a nilpotent matrix over $\calR$ then $[I-tN] = 0$ in $K_{1}(\calR[t])$. Given $N$ nilpotent, by Theorem \ref{thm:PIDblocktri} from Lecture 2, there exists some $P \in GL(\calR)$ such that $P^{-1}NP$ is upper triangular with zero diagonal. Then
$$[I-tN] = [P^{-1}(I-tN)P] = [I-t(P^{-1}NP)].$$
in $K_{1}(\calR[t])$. Since $P^{-1}NP$ is upper triangular with zero diagonal, $I-t(P^{-1}NP)$ lies in $El(\calR[t])$, and hence $[I-t(P^{-1}NP)] = 0$ in $K_{1}(\calR[t])$.
\end{proof}

\begin{exer}\label{appexer:psiwelldefined}
The map
\begin{equation}
\begin{gathered}
\Psi \colon Nil_{0}(\calR) \to NK_{1}(\calR[t])\\
\Psi \colon [N] \mapsto [I-tN]
\end{gathered}
\end{equation}
is a well-defined group homomorphism.
\end{exer}
\begin{proof}
Since $[I-tN_{1}] + [I-tN_{2}] = [I-t(N_{1} \oplus N_{2})] = [(I-tN_{1}) \oplus (I-tN_{2})]$ in $K_{1}(\calR[t])$, $\Psi$ respects the group operations. To see it is well-defined it suffices to check $\Psi$ on the relations for $Nil_{0}(\calR)$. For the first relation of $Nil_{0}(\calR)$, if $N$ is a nilpotent matrix over $\calR$ and $P \in GL(\calR)$ then
$$[I-tN] = [P^{-1}(I-tN)P] = [I-t(P^{-1}NP)]$$
in $NK_{1}(\calR)$. For the second relation, suppose $N_{1},N_{2}$ are nilpotent matrices and $B$ is some matrix over $\calR$ and consider
$$\begin{pmatrix} I-tN_{1} & -tB \\ 0 & I - tN_{2} \end{pmatrix}.$$
Since $I-tN_{2}$ is invertible over $\calR[t]$, we can consider the block matrix in $El(\calR[t])$ given by
$$E = \begin{pmatrix} I & tB(I-tN_{2})^{-1} \\ 0 & I \end{pmatrix}.$$
Then
$$E \begin{pmatrix} I-tN_{1} & -tB \\ 0 & I - tN_{2} \end{pmatrix} = \begin{pmatrix} I-tN_{1} & 0 \\ 0 & I-tN_{2} \end{pmatrix}$$
so the second relation is preserved by $\Psi$. The third relation is obvious.
\end{proof}

\begin{exer}\label{appexer:uppertrivanish}
Consider an upper triangular matrix $N$ over $\calR$ with zero diagonal. Then $I-tN$ lies in $El(\calR[t])$, and hence $[I-tN] = 0$ in $NK_{1}(\calR)$. Using the relations defining $Nil_{0}(\calR)$, show the class of such an $N$ must be zero in $Nil_{0}(\calR)$.
\end{exer}
\begin{proof}
If $N$ is size one or two then this is immediate from relation (2) in the definition of $Nil_{0}(\calR)$. Now if $N$ is upper triangular of size $n \ge 2$ with zero diagonal, then there is some matrix $B$ such that
$$N = \begin{pmatrix} N_{1} & B \\ 0 & 0 \end{pmatrix}$$
where $N_{1}$ is upper triangular of size $n-1$ with zero diagonal. Now use relation $(2)$ of $Nil_{0}(\calR)$ and induction.
\end{proof}

\begin{rem}\label{rem:nilviewlocalization}
For a more abstract viewpoint, the connection between $NK_{1}$ and the class group $Nil_{0}(\calR)$ of nilpotent endomorphisms over $\calR$ essentially comes from the localization sequence in algebraic K-theory, together with identifying the category of $\calR[t]$-modules of projective dimension less than or equal to 1 which are $t$-torsion (i.e. are annihilated by $t^{k}$ for some $k$) with the category of pairs $(P,f)$ where $P$ is a finitely generated projective $\calR$-module and $f$ is a nilpotent endomorphism of $P$; see \cite[Chapter III]{WeibelBook} for more on this viewpoint.
\end{rem}

\begin{rem}\label{rem:exactcategory}
The category $\textbf{Nil}\calR$ equipped with the notion of exact sequence as defined here is a particular case of the more general concept, introduced by Quillen, of an {\it exact category}, a category equipped with some notion of exact sequences which satisfy some conditions. Such a category has enough structure to define $K$-groups of the category; our definition of $K_{0}(\textbf{Nil}\calR)$ coincides with $K_{0}$ of the exact category $\textbf{Nil}\calR$. See \cite[II Sec. 7]{WeibelBook} for details regarding this viewpoint.
\end{rem}

\begin{rem}\label{rem:ktheoryendomorphisms}
One may also define a class group for endomorphisms over a ring $\calR$. Define $\textbf{End}\calR$ to be the category whose objects are pairs $(P,f)$ where $P$ is a finitely generated projective $\calR$-module and $f \colon P \to P$ is an endomorphism, and a morphism $(P,f) \to (Q,g)$ is given by an $\calR$-module homomorphism $h \colon P \to Q$ such that $hf = gh$. Analogous to $\textbf{Nil}\calR$, we call a sequence
$$(P_{1},f_{1}) \to (P_{2},f_{2}) \to (P_{3},f_{3})$$
in $\textbf{End}\calR$ exact if the associated sequence of $\calR$-modules
$$P_{1} \to P_{2} \to P_{3}$$
is exact. Then $K_{0}(\textbf{End}\calR)$ is defined to be the free abelian group on isomorphism classes of objects $(P,f)$ in $\textbf{End}\calR$ together with the relations
\begin{equation}
\begin{gathered}
{[(P_{1},f_{1})]} + [(P_{3},f_{3})] = [(P_{2},f_{2})]\\
\textnormal{whenever}\\
0 \to (P_{1},f_{1}) \to (P_{2},f_{2}) \to (P_{3},f_{3}) \to 0\\
\textnormal{is exact in } \textbf{End}\calR.
\end{gathered}
\end{equation}
There is a forgetful functor $\textbf{End}\calR \to \textbf{Proj}\calR$ given by $(P,f) \mapsto P$ and an induced group homomorphism on the level of $K_{0}$
\begin{equation}
\begin{gathered}
K_{0}(\textbf{End}\calR) \to K_{0}(\textbf{Proj}\calR)\\
[(P,f)] \mapsto [P].
\end{gathered}
\end{equation}
Now define $End_{0}(\calR)$ to be the kernel of this homomorphism. The group $End_{0}(\calR)$ has a presentation analogous to the one given in Definition \ref{def:nil0def1}: $End_{0}(\calR)$ is the free abelian group on the set of generators
$$\{[A] \mid A \textnormal{ is a square matrix over } \calR\}$$
together with the relations
\begin{enumerate}
\item
$[A_{1}] = [A_{2}]$ if $A_{1} = P^{-1}A_{2}P$ for some $P \in GL(\calR)$.
\vspace{.07in}
\item
$[A_{1}] + [A_{2}] = \left[ \begin{pmatrix} A_{1} & B \\ 0 & A_{2} \end{pmatrix} \right]$ for any matrix $B$ over $\calR$.
\vspace{.07in}
\item
$[0] = 0$.
\end{enumerate}
There is an equivalence relation on square matrices over $\calR$ defined by $A \sim_{end} B$ if $[A]=[B]$ in $End_{0}(\calR)$. A natural question is how this relation compares to the relations of strong shift equivalence\si{strong shift equivalence} and shift equivalence\si{shift equivalence} over $\calR$. In fact, this is settled in the commutative case by the following theorem of Almkvist (which was also proved, and greatly generalized, by Grayson in \cite{Grayson1977}). In the theorem, for $\calR$ commutative we let $\tilde{\calR}$ denote the multiplicative subgroup of $1+t\calR[[t]]$ given by
$$\tilde{\calR} = \left\{\frac{p(t)}{q(t)} \mid p(t), q(t) \in \calR[t] \textnormal{ and } p(0)=q(0)=1 \right\}.$$
\begin{thm}[\cite{Almkvist78,AlmkvistErratum}]\label{thm:almkvistend0}
Let $\calR$ be a commutative ring. The map
\begin{equation}\label{eqn:end0iso}
\begin{gathered}
End_{0}(\calR) \to \tilde{\calR}\\
[A] \mapsto \det(I-tA)
\end{gathered}
\end{equation}
is an isomorphism.
\end{thm}
There is an extension of Theorem \ref{thm:almkvistend0} to general (i.e. not necessarily commutative) rings due to Sheiham \cite{SheihamWhitehead}.\\

As a consequence of the theorem, if $\mathcal R$ is an integral domain then the relation $\sim_{end}$ is coarser than shift equivalence\si{shift equivalence} over $\calR$. For example, when $\calR = \mathbb{Z}$ and $A$ over $\mathbb{Z}_{+}$ presents a shift of finite type $(X_{A},\sigma_{A})$, knowing the class $[A]$ in $End_{0}(\mathbb{Z})$ is the same as knowing the zeta function\si{zeta function} $\zeta_{\sigma_{A}}(t)$. Also (see Section \ref{detI-tAsubsection}),
$\det(I-tA)$ is an invariant of SSE-$\mathcal R$ for any commutative ring $\mathcal R$,
but there are commutative rings for which
the trace is not an invariant of
shift equivalence\si{shift equivalence}, and for such a ring $\mathcal R$,
SE-$\mathcal R$ does not refine $\sim_{end}$.
\\

\indent For a symbolic system presented by a matrix $A$ over a noncommutative ring (for example, the integral group ring $\mathbb{Z}G$ where $G$ is nonabelian), Theorem \ref{thm:almkvistend0} suggests the class $[A]$ in $End_{0}(\calR)$ can serve as an analogue of the zeta function\si{zeta function} of the symbolic system presented by $A$.

\end{rem}
\section[Strong shift equivalence\si{strong shift equivalence} vs. shift equivalence\si{shift equivalence} via algebraic K-theory]{The algebraic K-theoretic characterization of
the refinement of strong shift equivalence\si{strong shift equivalence} over a
ring by shift equivalence\si{shift equivalence}}
Let $\calR$ be a semiring. Recall that square matrices $A,B$ are {\it elementary strong shift equivalent over $\calR$} (ESSE-$\calR$ for short, denoted $\esse{A}{B}$) if there exists matrices $R,S$ over $\calR$ such that
$$A=RS, \hspace{.1in} B=SR.$$

Recall also from Lecture 2 the following two equivalence relations defined on the collection of square matrices over $\calR$:
\begin{enumerate}
\item
Square matrices $A$ and $B$ are {\it strong shift equivalent over $\calR$} (SSE-$\calR$ for short, denoted $\sse{A}{B}$) if there exists a chain of elementary strong shift equivalence\si{strong shift equivalence}s over $\calR$ from $A$ to $B$:
$$A = \esse{A_{0}}{A_{1}} \begin{tiny} \reallywidesim{\hspace{.03in} esse-$\mathcal{R}$ } \end{tiny} \cdots \begin{tiny} \reallywidesim{\hspace{.03in} esse-$\mathcal{R}$ } \end{tiny} \esse{A_{n-1}}{A_{n}} = B.$$
\item
Square matrices $A$ and $B$ are {\it shift equivalent over $\calR$} (SE-$\calR$ for short, denoted $\se{A}{B}$) if there exists matrices $R,S$ over $\calR$ and a number $l \in \mathbb{N}$ such that
\begin{gather*}
A^{l} = RS, \hspace{.1in} B^{l} = SR\\
AR = RB, \hspace{.1in} BS = SA.
\end{gather*}
\end{enumerate}

The group $\mathcal R^n$ is a (left) $\mathcal R$-module,
  by the obvious definition $r:x\mapsto rx$.
For an $n \times n$ square matrix $A$ over $\calR$ there is an $\calR$-module endomorphism $\calR^{n} \to \calR^{n}$ given by $x \mapsto xA$. We can form the direct limit $\calR$-module
$$G_{A} = \varinjlim\{\calR^{n},x \mapsto xA\}.$$
This was introduced in the case $\calR = \mathbb{Z}$ in Section \ref{sec:SEZdirectlimits} of Lecture 2. The $\calR$-module $G_{A}$ becomes an $\calR[t,t^{-1}]$-module by defining $x \cdot t^{-1} = xA$.
The following result
was a part of Theorem \ref{RmodulestructureforSE}.

\begin{prop}
For square matrices $A,B$ over $\calR$, we have
\begin{gather*}
\se{A}{B}\\
\textnormal{ if and only if }\\
G_{A} \textnormal{ and } G_{B} \textnormal{ are isomorphic as } \calR[t,t^{-1}]-\textnormal{modules}.
\end{gather*}
\end{prop}

This proposition shows that shift equivalence\si{shift equivalence} over a ring $\calR$ has a nice classical algebraic interpretation.

\subsection{Comparing shift equivalence\si{shift equivalence} and strong shift equivalence\si{strong shift equivalence} over a ring}\label{subsec:comparingseandsse}

Recall that for any semiring $\calR$ and square matrices $A,B$ over $\calR$,
$$\sse{A}{B} \Longrightarrow \se{A}{B}.$$

Lectures 1 and 2 discussed various aspects of both shift equivalence\si{shift equivalence} and strong shift equivalence\si{strong shift equivalence}, especially in the central case of $\calR = \mathbb{Z}_{+}$ and $\mathbb{Z}$. Recall Conjecture \ref{conj:williams} from Lecture 1:

\begin{conj}[Williams' Shift Equivalence\si{shift equivalence} Conjecture, 1974]
If $A$ and $B$ are square matrices over $\mathbb{Z}_{+}$ which are shift equivalent over $\mathbb{Z}_{+}$, then $A$ and $B$ are strong shift equivalent over $\mathbb{Z}_{+}$.
\end{conj}

There are counterexamples to Williams' Conjecture\si{Williams' Conjecture}; we'll discuss some of this in Lecture 8. We can generalize the conjecture in the obvious way to arbitrary semirings, and rephrase as a more general problem:

\begin{prob}[General Williams\ai{Williams, R.F.} Problem]
Suppose $\calR$ is a semiring, and $A,B$ are square matrices over $\calR$. If $A$ and $B$ are shift equivalent over $\calR_{+}$, must $A$ and $B$ be strong shift equivalent over $\calR_{+}$?
\end{prob}

Williams' original Shift Equivalence\si{shift equivalence} Conjecture concerns the case $\calR = \mathbb{Z}_{+}$, and is most immediately linked to shifts of finite type, through its relation to topological conjugacy (as discussed in Lecture 1). It turns out, even in the case $\calR = \mathbb{Z}_{+}$ the answer to Williams\ai{Williams, R.F.} Problem is `not always'. We will talk more about this in Lecture 4, but for now let us consider the following picture, which outlines how the General Williams' Problem can be approached:

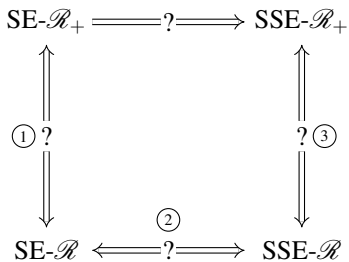
\begin{equation*}
\begin{tikzpicture}
\path
(-1.6,1) node (a) {SE-$\mathcal{R}_{+}$}
(0,1) node (b) {?}
(1.75,1) node (c) {SSE-$\mathcal{R}_{+}$}
(-1.6,-0.5) node (d) {?}
(1.75,-0.5) node (e) {?}
(-1.6,-2) node (f) {SE-$\mathcal{R}$}
(1.75,-2) node (g) {SSE-$\mathcal{R}$}
(0,-2) node (h) {?};
\node[shape=circle,draw,inner sep=2pt,scale=0.6] at (-1.9,-0.5) (char){1};
\node[shape=circle,draw,inner sep=2pt,scale=0.6] at (2.05,-0.5) (char){3};
\node[shape=circle,draw,inner sep=2pt,scale=0.6] at (0,-1.6) (char){2};
\draw[double,double equal sign distance] (-1,1) -- (-0.1,1);
\draw[double,double equal sign distance, -implies] (0.1,1) -- (1,1);

\draw[double,double equal sign distance,implies-] (-1.6,0.7) -- (-1.6,-0.3);
\draw[double,double equal sign distance,-implies] (-1.6,-0.7) -- (-1.6,-1.7);
\draw[double,double equal sign distance,implies-] (1.75,0.7) -- (1.75,-0.3);
\draw[double,double equal sign distance,-implies] (1.75,-0.7) -- (1.75,-1.7);
\draw[double,double equal sign distance, implies-] (-1,-2) -- (-0.1,-2);
\draw[double,double equal sign distance, -implies] (0.1,-2) -- (1,-2);
\end{tikzpicture}
\end{equation*}


Looking at the picture above, Williams' Problem concerns the top arrow. The picture describes how the problem can be broken down into a few parts: an `algebra' part ($\circled{2}$ in the picture), and two `order' parts ($\circled{1}$
and $\circled{3}$ in the picture). In key cases, the answer to $\circled{1}$ is yes for a fundamental subclass of matrices over $\calR_{+}$. Recall, for $\calR \subset \mathbb{R}$, a matrix $A$ is primitive if there exists $k$ such that $A^{k}$ has all positive entries. Then as shown in Proposition \ref{sePrimitive} in Lecture 2, we have:
\begin{thm}
Suppose $\calR \subset \mathbb{R}$ and $\calR_{+} = \calR \cap \mathbb{R}_{+}$. If $A$ and $B$ are primitive matrices over $\calR$, then $A$ and $B$ are SE-$\calR$ if and only if they are shift equivalent over $\calR_{+}$.
\end{thm}

This result says that, when the ring is a subring of $\mathbb{R}$, we can reduce the question of SE-$\calR_{+}$ of primitive matrices to the purely algebraic question of SE-$\calR$.\\

Part $\circled{2}$ is the main topic of this and the next lecture. Part $\circled{3}$, in the case of $\calR = \mathbb{Z}_{+}$, we will discuss in Lecture 4, and contains the remaining core of Williams' Problem.

\subsection{The algebraic shift equivalence\si{shift equivalence} problem\si{Algebraic shift equivalence problem}}

We consider now $\circled{2}$, which we can restate as:

\begin{prob}[Algebraic Shift Equivalence\si{shift equivalence} Problem\si{Algebraic shift equivalence problem},
    \cite{Wagoner99}]\label{algseproblem}\ai{Wagoner, J.B.}
Let $\calR$ be a ring and $A,B$ be square matrices over $\calR$. If $A$ and $B$ are shift equivalent over $\calR$, must $A$ and $B$ be strong shift equivalent over $\calR$?
\end{prob}

Williams\ai{Williams, R.F.} gave an argument in \cite[Lemma 4.6]{Williams1970} (which needed an additional step, later given in \cite{Williams1992}) showing that, when $\calR = \mathbb{Z}$, the answer to Problem \ref{algseproblem} is yes. Effros also gave a similar argument, in an unpublished work, in the case $\calR = \mathbb{Z}$ , and it was observed in \cite{BH93} that both arguments work in the case $\calR$ is a principal ideal domain. It was then shown by Boyle and Handelman\ai{Handelman, David} \cite{BH93} that the answer to Problem \ref{algseproblem} is also yes when $\calR$ is a Dedekind domain. The Boyle-Handelman paper \cite{BH93} was published in 1993, and after that point no further progress was made; in fact, it was still not known whether the answer to Problem \ref{algseproblem} might be yes for {\it every} ring. Now, from recent work \cite{BoSc1}, we know the answer to Problem \ref{algseproblem} is not always yes, and we have a pretty satisfactory characterization (Corollary \ref{cor:nk1vanishandsse}) of the rings $\calR$ for which the relations SE-$\calR$ and SSE-$\calR$ are the same. It turns out to depend on some K-theoretic properties of the ring $\calR$ in question, and we'll spend the remainder of the lecture discussing how this works.\\

In short, the answer to Problem \ref{algseproblem} turns out to depend on the group $NK_{1}(\calR)$. Before getting into the precise statements, recall from Proposition \ref{prop:MallerShub} in Lecture 2 that SSE-$\calR$ is the relation generated by similarity and extensions by zero. Since the direct limit module associated to a nilpotent matrix is clearly trivial, it is reasonable to suspect that determining the strong shift equivalence\si{strong shift equivalence} classes of nilpotent matrices is connected to determining which nilpotent matrices over the ring can be obtained from the zero matrix (up to similarity) by extensions by zero. In fact this is the case, and the question of which nilpotent matrices over the ring can be obtained from the zero matrix (up to similarity) by extensions by zero turns out to be governed by $Nil_{0}(\calR)$.\\

Fix now a ring $\calR$. For a matrix $A$ over $\calR$, we let $[A]_{sse}$, $[A]_{se}$ denote the strong shift equivalence\si{strong shift equivalence} (respectively shift equivalence\si{shift equivalence}) class of $A$ over $\calR$ (we suppress the $\calR$ in the notation, as it is cumbersome). We define the following sets
\begin{gather*}
SSE(\calR) = \{[A]_{sse} \mid A \textnormal{ is a square matrix over } \calR\}\\
SE(\calR) = \{[A]_{se} \mid A \textnormal{ is a square matrix over } \calR\}.
\end{gather*}

Since matrices which are strong shift equivalent over $\calR$ must be shift equivalent over $\calR$, there is a well-defined map of sets
\begin{equation}\label{eqn:ssetosemap}
\begin{gathered}
\pi \colon SSE(\calR) \to SE(\calR)\\
\pi \colon [A]_{sse} \mapsto [A]_{se}.
\end{gathered}
\end{equation}

Problem \ref{algseproblem} is equivalent to determining whether $\pi$ is injective. We'll discuss when this happens, and in fact, we will do much more: we will describe the fiber over a class $[A]_{se}$ in terms of some K-theoretic data involving $NK_{1}(\calR)$.\\

\subsection{Strong shift equivalence\si{strong shift equivalence} and elementary equivalence\si{elementary equivalence}}

From here on, we identify a square matrix $M$ over $\calR[t]$ with its class in the stabilization of matrices given by
\begin{gather*}
M_{n}(\calR[t]) \hookrightarrow M_{n+1}(\calR[t])\\
M \mapsto \begin{pmatrix} M & 0 \\ 0 & 1 \end{pmatrix}.
\end{gather*}

\begin{de}
Let $\calR$ be a ring. We say square matrices $M, N$ over $\calR[t]$ are {\it elementary equivalent over $\calR[t]$}\si{elementary equivalence}, denoted $\eleq{M}{N}$, if there exist $E,F \in El(\calR[t])$ such that
$$EMF = N.$$
\end{de}

Note that, as when we first met the definition of $K_{1}$ of a ring, matrices $M,N$ over $\calR[t]$ are elementary equivalent\si{elementary equivalence} over $\calR[t]$ if and only if they (after stabilizing!) can be transformed into each other through a sequence of elementary row and column operations.\\

\begin{rem}
Given square matrices $M,N$ over $\calR[t]$, it may be tempting to ask why we don't just define $\eleq{M}{N}$ if and only if $[M] = [N]$ in $K_{1}(\calR[t])$, but this doesn't make sense: since $M,N$ may not be invertible over $\calR[t]$, we can't consider their class in $K_{1}(\calR[t])$.
\end{rem}

The following is one of the key results for studying strong shift equivalence\si{strong shift equivalence} over a ring $\calR$.

\begin{thm}[{\cite[Theorem 7.2]{BoSc1}}]\label{thm:sseeleeq}
Let $\calR$ be a ring. For any square matrices $A,B$ over $\calR$, we have
$$\sse{A}{B} \hspace{.1in} \textnormal{ if and only if } \hspace{.1in} \eleq{I-tA}{I-tB}.$$
\end{thm}

To see how this fits into the endomorphism $\leftrightarrow$ polynomial philosophy, consider a finitely generated free $\calR$-module $P$ and an endomorphism $f \colon P \to P$ (one may allow more generally $P$ to be finitely generated projective; see \apr{apprem:projectivetofree})). The endomorphism $f$ gives $P$ the structure of an $\calR[t]$-module with $t$ acting by $f$, and the similarity class of $f$ (over $\calR$) corresponds to the isomorphism class of the $\calR[t]$-module. The direct limit $\calR[t,t^{-1}]$-module $\mathcal{M}_{f} = \varinjlim \{P,v \mapsto f(v)\}$ is isomorphic as an $\calR[t,t^{-1}]$-module to $P \otimes_{\calR[t]} \calR[t,t^{-1}]$, and it follows that passing from the similarity class of $f$ to the shift equivalence\si{shift equivalence} class of $f$ is, in the polynomial world, the same as `localizing at $t$' (note that, besides here, our convention for the direct limit modules is that $t^{-1}$ acts by $f$). On the endomorphism side, the strong shift equivalence\si{strong shift equivalence} relation lies between the similarity relation and the shift equivalence\si{shift equivalence} relation, and Theorem \ref{thm:sseeleeq} tells us the meaning of the strong shift equivalence\si{strong shift equivalence} relation in the polynomial world.\\
\indent We can summarize the above in the following chart \apr{apprem:endtopolyeq}), where $A_{f}$ denotes a matrix over $\calR$ representing $f \colon P \to P$ in a chosen basis for $P$ as a free $\calR$-module:\\


\begin{center}
    \begin{tabular}{ c | c | c }
    \hline
     & & \\
 Endomorphisms
     over $\calR$ \hspace{.05in} & \hspace{.05in} $\calR[t]$-endomorphism relation \hspace{.05in} & \hspace{.05in} $\calR[t]$-module
       relation \hspace{.05in}\\
       & & \\ \hline
            & & \\
    Similarity class of $A_{f}$ & $Gl$-$\calR[t]$-conjugacy class & Isomorphism class of \\
         & of $t-A_{f}$ & the $\calR[t]$-module $P$ \\
              & & \\
    SSE-$\calR$ class of $A_{f}$ & $El$-$\calR[t]$-equivalence class & ?? \\
         &  of $1-tA_{f}$  & \\
              & & \\
    SE-$\calR$ class of $A_{f}$ & $Gl$-$\calR[t]$-equivalence class  & Isomorphism class of\\
         & of $1-tA_{f}$ & the $\calR[t,t^{-1}]$-module\\
         & &  $P \otimes \calR[t,t^{-1}]$\\ \hline
    \end{tabular}
\end{center}



In the chart, $GL-\calR[t]$-equivalence of (stabilized) matrices $C$ and $D$ means there exists $U,V \in GL(\calR[t])$ such that $UCV = D$. The ?? entry indicates that we do not have a good intrinsic interpretation of SSE-$\calR$ at the $\calR[t]$-module level.\\

Theorem \ref{thm:sseeleeq} determines the algebraic relation in the polynomial world corresponding to SSE-$\calR$. It also gives another idea of how strong shift equivalence\si{strong shift equivalence} arises algebraically in a natural way. Recall for $A$ and $B$ invertible over $\calR[t]$ we have
$$[A] = [B] \textnormal{ in } K_{1}(\calR[t]) \textnormal{ if and only if }
A \begin{tiny} \reallywidesim{\hspace{.03in}  El-$\mathcal{R}[t]$ } \end{tiny} B.$$
In light of Theorem \ref{thm:sseeleeq}, if one tries to naively extend the K-theory of $\calR[t]$ to not necessarily invertible matrices over $\calR[t]$, then strong shift equivalence\si{strong shift equivalence} naturally appears.\\

See \apr{apprem:endo0intable}) for a discussion of how $\det(1-tf)$ fits into the above table in the case $\calR$ is commutative.\\

As a nice corollary of Theorems \ref{thm:sseeleeq} and \ref{thm:nil0nk1iso}, we have the following:

\begin{coro}\label{cor:nilsse}
Let $\calR$ be a ring, and let $N$ be a nilpotent matrix over $\calR$. Then
$$0 = [N] \textnormal{ in } Nil_{0}(\calR) \hspace{.11in} \textnormal{ if and only if } \hspace{.11in} [N]_{sse} = [0]_{sse}.$$
\end{coro}

In other words, a nilpotent matrix is strong shift equivalent over $\calR$ to the zero matrix if and only if its class in $Nil_{0}(\calR)$ is trivial. One can use this together with Theorem \ref{thm:sseeleeq} to show that the map defined in \eqref{eqn:nk1nil0} is injective.

\subsection{The refinement of shift equivalence\si{shift equivalence} over a ring by strong shift equivalence\si{strong shift equivalence}}

Theorem \ref{thm:sseeleeq} gives us a key tool to understand the refinement of shift equivalence\si{shift equivalence} by strong shift equivalence\si{strong shift equivalence} over $\calR$, obtaining a description of the fibers of the map $\pi$ above. We do this as follows.\\

In light of Corollary \ref{cor:nilsse} above, there is a well-defined action $\mathfrak{N}$ of the group $Nil_{0}(\calR)$ on the set $SSE(\calR)$ by
$$\mathfrak{N}([N]) \colon [A]_{sse} \mapsto \left[ \begin{pmatrix} A & 0 \\ 0 & N \end{pmatrix} \right]_{sse}, \hspace{.11in} [N] \in Nil_{0}(\calR), \hspace{.11in} [A] \in SSE(\calR).$$

The following gives a description of the fibers of the map
$$\pi \colon SSE(\calR) \to SE(\calR)$$
defined in \eqref{eqn:ssetosemap}.

\begin{thm}[{\cite[Theorem 6.6]{BoSc1}}]\label{thm:ssesefibers}
Let $\calR$ be a ring, and let $A$ be a square matrix over $\calR$. There is a bijection
\begin{gather*}
\pi^{-1}([A]_{se}) \stackrel{\cong}\longrightarrow \mathfrak{N}\textnormal{-orbit of } [A]_{sse}.
\end{gather*}
In other words, there is a bijection between the set of strong shift equivalence\si{shift equivalence}\si{strong shift equivalence} classes of matrices which are shift equivalent to $A$, and the orbit of $[A]_{sse}$ under the action of $Nil_{0}(\calR)$.
\end{thm}

As a corollary, we get the following.

\begin{coro}\label{cor:nk1vanishandsse}
Let $\calR$ be a ring. Then $NK_{1}(\calR) = 0$ if and only if, for all square matrices $A,B$ over $\calR$,
$$\se{A}{B} \textnormal{ if and only if } \sse{A}{B}.$$
\end{coro}
\begin{proof}
First suppose $NK_{1}(\calR)=0$. By Theorem \ref{thm:nil0nk1iso}, this implies $Nil_{0}(\calR) = 0$, so the action $\mathfrak{N}$ is trivial. Thus if $A$ is any square matrix over $\calR$, by Theorem \ref{thm:ssesefibers}, the fiber $\pi^{-1}([A]_{se})$ is also trivial. It follows that if $B$ is any square matrix over $\calR$, then $\se{A}{B} \Leftrightarrow \sse{A}{B}$ as desired.\\
\indent Now suppose matrices which are SE-$\calR$ must be SSE-$\calR$. Given $N$ nilpotent over $\calR$, $N$ is clearly shift equivalent over $\calR$ to the zero matrix, and hence by assumption, strong shift equivalent over $\calR$ to the zero matrix. By Corollary \ref{cor:nilsse}, this implies $[N]=0$ in $Nil_{0}(\calR)$. Since $N$ was a general nilpotent matrix, it follows that $Nil_{0}(\calR)=0$.
\end{proof}

So what is the behavior of the action $\mathfrak{N}$? In general, its orbit structure is far from trivial. Given $A$ over $\calR$, define the $\mathfrak{N}$-stabilizer of $[A]_{sse}$ to be
$$\sta_{\mathfrak{N}}(A) = \{[N] \in Nil_{0}(\calR) \mid [A \oplus N]_{sse} = [A]_{sse}\}.$$
Note the $\mathfrak{N}$-stabilizer depends only on the SSE-$\calR$ class of a matrix $A$, but to avoid cumbersome notation, we write simply $\sta_{\mathfrak{N}}(A)$ instead of $\sta_{\mathfrak{N}}([A]_{sse})$.\\

The notation used here differs from what is used in \cite{BoSc1}; see \apr{apprem:elstnotation}).\\

There is a bijection between the $\mathfrak{N}$-orbit of $[A]_{sse}$ and
$Nil_{0}(\calR) / \sta_{\mathfrak{N}}(A)$, the quotient  given by mapping a coset of $[N]$ in $Nil_{0}(\calR) / \sta_{\mathfrak{N}}(A)$ to $[A \oplus N]_{sse}$.\\

For a commutative ring $\calR$, we define
$$SNil_{0}(\calR) = \{[N] \in Nil_{0}(\calR) \mid \deter(I-tN) = 1\}.$$
It is straightforward to check that $SNil_{0}(\calR)$ is a subgroup of $Nil_{0}(\calR)$, and that $SNil_{0}(\calR)$ is precisely the pullback via the isomorphism \eqref{eqn:nk1nil0} of the subgroup $NK_{1}(\calR) \cap SK_{1}(\calR[t])$ in $NK_{1}(\calR)$.

\begin{thm}[{\cite[Theorems 4.7, 5.1]{BoSc1}}]\label{thm:stabilizerthm1}
For any ring $\calR$, both of the following hold:
\begin{enumerate}
\item
If $A$ is nilpotent or invertible over $\calR$, then $\sta_{\mathfrak{N}}(A)$ is trivial.
\item
If $\calR$ is commutative, then
$$\bigcup_{[A]_{sse} \in SSE(\calR)} \sta_{\mathfrak{N}}(A) = SNil_{0}(\calR).$$
\end{enumerate}
\end{thm}

\begin{exer}\apr{appex:redsnil0})
If $\calR$ is commutative and reduced (has no nontrivial nilpotent elements), then the groups $SNil_{0}(\calR)$ and $Nil_{0}(\calR)$ coincide.
\end{exer}

The nilpotent case of $(1)$ is straightforward, and is Exercise \ref{exer:nilpotentstavanishexer} below. Both $(2)$ and the invertible case of $(1)$ are nontrivial to prove. Part $(1)$ uses localization and K-theoretic techniques for localization; in the non-commutative case, this requires some deep K-theoretic results of Neeman and Ranicki about non-commutative localization of rings. Part $(2)$ uses work of Nenashev on presentations for $K_{1}$ of exact categories. We will not go into more detail about the structure of these proofs, but instead refer the reader to \cite{BoSc1}.

\begin{exer}\apr{appex:nilstvanish})\label{exer:nilpotentstavanishexer}
If $A$ is a nilpotent matrix over $\calR$ then $\sta_{\mathfrak{N}}(A)$ vanishes.
\end{exer}

There are rings for which $SNil_{0}(\calR)$ does not vanish. For example, the ring $\mathbb{Q}[t^{2},t^{3},z,z^{-1}]$ is commutative and reduced, and has nontrivial $NK_{1}$ (see Example (1) in \ref{exa:nk1example}).\\

By part (2) of the above theorem, it follows that the $\mathfrak{N}$-stabilizers can be nontrivial, and can change depending on the matrix. There is a conjectured analogous version of part $(2)$ in the non-commutative case, which is more technical to state (see \cite[Conjecture 5.20]{BoSc1}). In general, we do not have a complete understanding of the groups $\sta_{\mathfrak{N}}(A)$, and the following problem was posed in \cite[Problem 5.21]{BoSc1}:
\begin{prob}\label{prob:elementarystabl}
Given a square matrix $A$ over $\calR$, give a satisfactory description of the elementary stabilizer $\sta_{\mathfrak{N}}(A)$. In particular, determine when $\sta_{\mathfrak{N}}(A)$ is trivial.
\end{prob}

\subsection{The SE and SSE relations in the context of endomorphisms}
Using the results above, we can now give another view on what the relations SE-$\calR$ and SSE-$\calR$ mean in the context of endomorphisms, and how they fit in with the similarity relation over a ring. Given square matrices $A,B$ over $\calR$, we say
\begin{enumerate}
\item
  $B$ is a zero extension\si{zero extension} of $A$ if there exists some matrix $C$ over $\calR$ such that $B = \left(\begin{smallmatrix} A & C \\ 0 & 0 \end{smallmatrix}\right)$ or $B = \left(\begin{smallmatrix} A & 0 \\ C & 0 \end{smallmatrix}\right)$.
\item
$B$ is a nilpotent extension\si{nilpotent extension} of $A$ if there exists some matrix $C$ over $\calR$ and some nilpotent matrix $N$ over $\calR$ such that $B = \left(\begin{smallmatrix} A & C \\ 0 & N \end{smallmatrix}\right)$ or $B = \left(\begin{smallmatrix} A & 0 \\ C & N \end{smallmatrix}\right)$.
\end{enumerate}

Zero and nilpotent extensions\si{nilpotent extension} fit nicely into the context of the category $\textbf{Nil}\calR$; see \apr{apprem:extsequences}).

\begin{thm}\label{thm:endorelations}
Let $\calR$ be a ring.
\begin{enumerate}
\item
SSE-$\calR$ is the equivalence relation on square matrices over $\calR$ generated by:
\begin{enumerate}
\item
Similarity
\item
Zero extensions\si{zero extension}
\end{enumerate}
\item
SE-$\calR$ is the equivalence relation on square matrices over $\calR$ generated by:
\begin{enumerate}
\item
Similarity
\item
Nilpotent extensions\si{nilpotent extension}
\end{enumerate}
\end{enumerate}
\end{thm}
\begin{proof}
Part $(1)$ is Proposition \ref{prop:MallerShub} from Lecture 2.\\
\indent For part $(2)$, one direction is easy. If $A$ and $B$ are similar, they are certainly SE-$\calR$. To see why $A$ is shift equivalent to $\left(\begin{smallmatrix} A & B \\ 0 & N \end{smallmatrix}\right)$ for any $B$ and $N$ nilpotent, note there exists $l$ such that
$$\begin{pmatrix} A & B \\ 0 & N \end{pmatrix}^{l} = \begin{pmatrix} A^{l} & C \\ 0 & 0 \end{pmatrix}$$
for some $C$. Then $\left(\begin{smallmatrix} A & B \\ 0 & N \end{smallmatrix}\right)$ and $A$ are shift equivalent with lag\si{lag} $l$ using
$$R = \begin{pmatrix} I \\ 0 \end{pmatrix}, \hspace{.1in} S = \begin{pmatrix} A^{l} & C \end{pmatrix}.$$
For the other direction, suppose $A$ and $B$ are SE-$\calR$. Then the classes $[A]_{sse}, [B]_{sse}$ lie in the same fiber of the map $\pi$, so by Theorem \ref{thm:ssesefibers} above, there exists a nilpotent matrix $N$ over $\calR$ such that $A \oplus N$ is SSE-$\calR$ to $B$. Then $A \oplus N$ and $B$ are connected by a chain of similarities and extensions by zero. Since $A$ and $A \oplus N$ are related by an extension by a nilpotent, the result follows.\\
\indent The proof that $A$ is shift equivalent over $\calR$ to $\left(\begin{smallmatrix} A & 0 \\ B & N \end{smallmatrix}\right)$ for any $B$ and nilpotent $N$ is analogous.
\end{proof}

\subsection{Appendix 6}
This appendix contains some remarks and solutions for exercises for Lecture 6.

\begin{rem}\label{apprem:projectivetofree}
For a ring $\calR$, shift equivalence\si{shift equivalence} and strong shift equivalence\si{strong shift equivalence} may be defined in the context of finitely generated projective $\calR$-modules as follows. If $f \colon P \to P, g \colon Q \to Q$ are endomorphisms of finitely generated projective $\calR$-modules, then:
\begin{enumerate}
\item
$f$ and $g$ are strong shift equivalent (over $\textbf{Proj}\calR$) if there exists module homomorphisms $r \colon P \to Q, s \colon Q \to P$ such that $f=sr, g = rs$.
\item
$f$ and $g$ are shift equivalent (over $\textbf{Proj}\calR$) if there exists module homomorphisms $r \colon P \to Q, s \colon Q \to P$ and $l \ge 1$ such that $f^{l}=sr, g^{l} = rs$.
\end{enumerate}

Suppose now $f \colon P \to P$ is an endomorphism of a finitely generated projective $\calR$-module. There exists a finitely generated projective $\calR$-module $Q$ such that $P \oplus Q$ is free, and $f$ is strong shift equivalent over $\textbf{Proj}\calR$ to $f \oplus 0 \colon P \oplus Q \to P \oplus Q$ using $r \colon P \to P \oplus Q$ given by $r(x) = (x,0)$ and $s \colon P \oplus Q \to P$ given by $s(x,y) = f(x)$. It follows that, when considering strong shift equivalence\si{strong shift equivalence} and shift equivalence\si{shift equivalence} over $\textbf{Proj}\calR$, we may without loss of generality work with free modules.
\end{rem}

\begin{rem}\label{apprem:endtopolyeq}
For an example of the relationship between endomorphisms of $\calR$-modules and certain classes of modules over the polynomial ring $\calR[t]$ worked out more formally, see Theorem 2 in \cite{Grayson1977} and the discussion on page 441 there.
\end{rem}
\begin{rem}\label{apprem:endo0intable}
Given a commutative ring $\calR$ and an endomorphism $f \colon P \to $ of a finitely generated projective $\calR$-module where $\calR$ is commutative, one may add the polynomial $\det(I-tf)$ as an additional entry to the chart. The data $\det(I-tf)$ corresponds to, on the endomorphism side, the class of $[f]$ in the endomorphism class group $End_{0}(\calR)$ (see \apr{rem:ktheoryendomorphisms})).
\end{rem}

\begin{rem}\label{apprem:elstnotation}
The presentation here of the elementary stabilizers differs from the one given in \cite{BoSc1}. Roughly speaking, here we use $Nil_{0}(\calR)$ and the endomorphism side, whereas in \cite{BoSc1} the notation and definitions are in terms of $NK_{1}(\calR)$ and the polynomial matrix side. More precisely, in \cite{BoSc1} the elementary stabilizer of a polynomial matrix $I-tA$ is defined to be
$$E(A,\calR) =  \{U \in GL(\calR[t]) \mid U\textnormal{Orb}_{El(\calR[t])}(I-tA) \subset \textnormal{Orb}_{El(\calR[t])}(I-tA)\}$$
where $\textnormal{Orb}_{El(\calR[t])}(I-tA)$ denotes the set of matrices over $\calR[t]$ which are elementary equivalent\si{elementary equivalence} over $\calR[t]$ to $I-tA$. There it is observed that $E(A,\calR)$ is a subgroup of $NK_{1}(\calR)$. Given $A$ over $\calR$, the map
\begin{equation}
\begin{gathered}
E(A,\calR) \to \sta_{\mathfrak{N}}(A)\\
[I-tN] \mapsto [N]
\end{gathered}
\end{equation}
defines a group isomorphism between $E(A,\calR)$ and $\sta_{\mathfrak{N}}(A)$.
\end{rem}

\begin{exer}\label{appex:redsnil0}
If $\calR$ is commutative and reduced (has no nontrivial nilpotent elements), then the groups $SNil_{0}(\calR)$ and $Nil_{0}(\calR)$ coincide.
\end{exer}
\begin{proof}
This is essentially Exercise \ref{appexer:reducednk1sk1}, just in the nilpotent endomorphism setting: if $\calR$ is commutative and reduced, then $\det(I-tN)=1$ for any nilpotent matrix $N$ over $\calR$.
\end{proof}

\begin{exer}\label{appex:nilstvanish}
If $A$ is a nilpotent matrix over $\calR$ then $\sta_{\mathfrak{N}}(A)$ vanishes.
\end{exer}
\begin{proof}
If $A$ is a nilpotent matrix over $\calR$ and $[N] \in \sta_{\mathfrak{N}}(A)$, then $[A \oplus N]_{sse} = [A]_{sse}$. Since both $A$ and $N$ are nilpotent, $A \oplus N$ is nilpotent, and by Theorem \ref{thm:sseeleeq} this implies $[A \oplus N] = [A]$ in the group $Nil_{0}(\calR)$. Thus $[N]=0$ in $Nil_{0}(\calR)$.
\end{proof}

\begin{rem}\label{apprem:extsequences}
The notion of zero and nilpotent extension\si{nilpotent extension} can also be defined in terms of endomorphisms. Recall from \apr{rem:ktheoryendomorphisms}) the category $\textbf{End}\calR$ whose objects are pairs $(P,f)$ where $f \colon P \to P$ is an endomorphism of a finitely generated projective $\calR$-module and a morphism from $(P,f)$ to $(Q,g)$ is an $\calR$-module endomorphism $h \colon P \to Q$ such that $hf = gh$. Given $(P,f),(Q,g)$ in $\textbf{End}\calR$, we say
\begin{enumerate}
\item
$(Q,g)$ is a zero extension\si{zero extension} of $(P,f)$ if there exists some $\calR$-module $P_{1}$ such that $Q = P \oplus P_{1}$ and either of the following happen:
\begin{enumerate}
\item
  There exists an $\calR$-module homomorphism $h \colon P_{1} \to P$ such that $g = \left(\begin{smallmatrix} f & h \\ 0 & 0 \end{smallmatrix}\right)$.

\item
There exists an $\calR$-module homomorphism $h \colon P \to P_{1}$ such that $g = \left(\begin{smallmatrix} f & 0\\ h & 0 \end{smallmatrix}\right)$.
\end{enumerate}
\item
$(Q,g)$ is a nilpotent extension\si{nilpotent extension} of $(P,f)$ if there exists some $\calR$-module $P_{1}$ such that $Q = P \oplus P_{1}$ and either of the following happen:
\begin{enumerate}
\item
There exists an $\calR$-module homomorphism $h \colon P_{1} \to P$ and a nilpotent endomorphism $j \colon P_{1} \to P_{1}$ such that $g = \left(\begin{smallmatrix} f & h \\ 0 & j \end{smallmatrix}\right)$
\item
There exists an $\calR$-module homomorphism $h \colon P \to P_{1}$ and a nilpotent endomorphism $j \colon P_{1} \to P_{1}$ such that $g = \left(\begin{smallmatrix} f & 0 \\ h & j \end{smallmatrix}\right)$.
\end{enumerate}
\end{enumerate}
Recall in $\textbf{End}\calR$ we say a sequence $(P_{1},f_{1}) \to (P_{2},f_{2}) \to (P_{3},f_{3})$ is exact if the corresponding sequence of $\calR$-modules $P_{1} \to P_{2} \to P_{3}$ is exact. Zero extensions\si{zero extension} and nilpotent extensions\si{nilpotent extension} have a nice interpretation in terms of certain exact sequences in the endomorphism category. Note that in $\textbf{End}\calR$ the pair $(P,0)$ means the zero endomorphism of the $\calR$-module $P$, and $(0,0)$ means the zero endomorphism of the zero $\calR$-module. Since $(0,0)$ serves as a zero object in $\textbf{End}\calR$, we can consider short exact sequences in $\textbf{End}\calR$, by which we mean an exact sequence of the form
$$(0,0) \to (P_{1},f_{1}) \to (P_{2},f_{2}) \to (P_{3},f_{3}) \to (0,0).$$
Given this, the following shows that zero extensions\si{zero extension} and nilpotent extensions\si{nilpotent extension} are given (up to isomorphism) by certain short exact sequences in $\textbf{End}\calR$.
\begin{prop}
Let $\calR$ be a ring, and suppose
$$(0,0) \to (P_{1},f_{1}) \stackrel{\alpha_{1}}\longrightarrow (P_{2},f_{2}) \stackrel{\alpha_{2}}\longrightarrow (P_{3},f_{3}) \to (0,0)$$
is a short exact sequence in $\textbf{End}\calR$.
\begin{enumerate}
\item
If $f_{1}=0$ then $(P_{2},f_{2})$ is isomorphic to a zero extension\si{zero extension} of $(P_{3},f_{3})$.
\item
If $f_{3}=0$, then $(P_{2},f_{2})$ is isomorphic to a zero extension\si{zero extension} of $(P_{1},f_{1})$.
\item
If $f_{1}$ is nilpotent, then $(P_{2},f_{2})$ is isomorphic to a nilpotent extension\si{nilpotent extension} of $(P_{3},f_{3})$.
\item
If $f_{3}$ is nilpotent, then $(P_{2},f_{2})$ is isomorphic to a nilpotent extension\si{nilpotent extension} of $(P_{1},f_{1})$.
\end{enumerate}
\end{prop}
\begin{proof}
  We will prove 4; the other are analogous. Since the sequence is exact there is a splitting map $\alpha_{1}^{\prime} \colon P_{2} \to P_{1}$ such that
  $\alpha_{1}\alpha_{1}^{\prime}  = id$ on $P_{1}$, and an $\calR$-module isomorphism $\beta \colon P_{2} \to P_{1} \oplus P_{3}$ given by $\beta(x) = (\alpha_{1}^{\prime}(x),\alpha_{2}(x))$ so that the following diagram commutes
\begin{equation*}
\xymatrix{
0 \ar[r] & P_{1} \ar[r]^{\alpha_{1}} \ar[d]_{id} & P_{2} \ar[d]_{\beta} \ar[r]^{\alpha_{2}} & P_{3} \ar[r] \ar[d]^{id} & 0\\
0 \ar[r] & P_{1} \ar[r]^{i} & P_{1} \oplus P_{3} \ar[r]^{q} & P_{3} \ar[r] & 0\\
}
\end{equation*}
where $i \colon P_{1} \to P_{1} \oplus P_{3}$ by $i(x) = (x,0)$ and $q \colon P_{1} \oplus P_{3} \to P_{3}$ by $q(x,y) = y$. Define $g = \beta f_{2} \beta^{-1}$, so $g \colon P_{1} \oplus P_{3} \to P_{1} \oplus P_{3}$. We may write $g = \left(\begin{smallmatrix} g_{1} & g_{2} \\ g_{2}^{\prime} & g_{3} \end{smallmatrix}\right)$ where $(P_{1},g_{1}),(P_{3},g_{3})$ are in $\textbf{End}\calR$, and $g_{2} \colon P_{3} \to P_{1}, g_{2}^{\prime} \colon P_{1} \to P_{3}$. For any $x \in P_{1}$ we have
$$g(x,0) = \beta f_{2} \beta^{-1}(x,0) = \beta f_{2} \alpha_{1}(x) = \beta \alpha_{1} f_{1}(x) = (f_{1}(x),0)$$
so $g_{2}^{\prime}(x) = 0$ and $g_{1}(x) = f_{1}(x)$. Since $x$ was arbitrary, it follows that $g_{2}^{\prime}=0$ and $g_{1} = f_{1}$. Likewise, one can check that $g_{3} = f_{3}$. Altogether $f_{2}$ is isomorphic to $g = \left(\begin{smallmatrix} f_{1} & g_{2} \\ 0 & f_{3} \end{smallmatrix}\right)$, and since $f_{3}$ is nilpotent, this is a nilpotent extension\si{nilpotent extension} of $(P_{1},f_{1})$.
\end{proof}
\end{rem}


\section{Automorphisms of SFTs}\label{sectionAut}

We turn now to discussing automorphisms of shifts of finite type. In general, an automorphism of a dynamical system is simply a self-conjugacy of the given system. The collection of all automorphisms of a given system forms a group, the size of which can vary greatly depending on the system in question. It turns out that a nontrivial mixing shift of finite type possesses a very rich group of automorphisms.\\
\indent It's maybe unsurprising that, even in the context of the classification problem\si{classification problem} for shifts of finite type (Problem \ref{prob:classificationprob} in Lecture 1), the study of automorphisms plays an important role. Partly, this role is indirect: various tools and ideas which were originally introduced to study automorphism groups\si{automorphism group} of shifts of finite type (e.g. sign-gyration, introduced later in this lecture) in fact turned out to be important tools for the conjugacy problem. For example, the dimension representation\si{dimension representation} plays a role in constructing counterexamples to Williams' conjecture\si{Williams' Conjecture} in the reducible case (see \cite{S21}). Some of this we will discuss in Lecture 8.\\

The goal of this lecture is only to give a brief tour through some of main ideas in the study of automorphism groups\si{automorphism group} for shifts of finite type. At the end of the lecture we mention some newer developments, as well as a small collection of problems and conjectures that have guided some of the direction for studying the automorphism groups\si{automorphism group}.\\


We will continue to use the following notation. For a matrix $A$ over $\mathbb{Z}_{+}$, we let $(X_{A},\sigma_{A})$ denote the edge shift of finite type (as defined in Section \ref{sec:edgesfts} of Lecture 1) corresponding to the graph associated to $A$ (i.e. the graph $\Gamma_{A}$ as defined in Lecture 1). Since any shift of finite type is topologically conjugate to an edge shift $(X_{A},\sigma_{A})$ for some $\mathbb{Z}_{+}$-matrix $A$ (see e.g. \cite[Theorem 2.3.2]{LindMarcus2021})\ai{Lind, Douglas}\ai{Marcus, Brian}, and automorphism groups\si{automorphism group} of topologically conjugate systems are isomorphic, we will only consider edge shifts $(X_{A},\sigma_{A})$. The fundamental case is when $A$ is primitive with the topological entropy of the shift satisfying $h_{top}(\sigma_{A}) > 0$; with this in mind we make the following standing assumption.\\

\textbf{Standing Assumption: } Throughout this lecture, unless otherwise noted, when considering an SFT $(X_{A},\sigma_{A})$ we assume $A$ is primitive with $\lambda_{A} > 1$, where $\lambda_{A}$ denotes the Perron-Frobenius eigenvalue of $A$.\\

Since $h_{top}(\sigma_{A}) = \log \lambda_{A}$, where $h_{top}(\sigma_{A})$ is the topological entropy of the shift $\sigma_{A}$, the assumption on $\lambda_{A}$ is equivalent to the system $(X_{A},\sigma_{A})$ having positive entropy.\\

Now let us say more precisely what we mean by an automorphism. We begin with a general definition, and specialize to shifts of finite type later. Recall by a topological dynamical system $(X,f)$ we mean a self-homeomorphism $f$ of a compact metric space $X$.

\begin{de}
  Let $(X,f)$ be a topological dynamical system. An automorphism of $(X,f)$ is a homeomorphism $\alpha \colon X \to X$ such that $\alpha f = f \alpha$. The collection of automorphisms of $(X,f)$ forms a group under composition, which we call the group of automorphisms of $(X,f)$, and we denote this group by $\aut(f)$.
We define composition in $\aut(f)$ left to right: given $f,g$ in
    $\aut(f)$ and an input $x$, the output  $(fg)(x) $ is
    $g(f(x))$.\begin{footnote}{The choice of left-to right
          composition will imply that the dimension
  representation, defined later in this section, is a group
  homomorphism.}\end{footnote} 
\end{de}

In other words, an automorphism of $(X,f)$ is simply a self-conjugacy of the system $(X,f)$, and the automorphism group\si{automorphism group} is the group of all self-conjugacies of $(X,f)$.\\

It is straightforward to check that if two systems $(X,f)$ and $(Y,g)$ are topologically conjugate then their automorphism groups\si{automorphism group} $\aut(f)$ and $\aut(g)$ are isomorphic.

\begin{ex}
Let $X$ be a Cantor set and $f \colon X \to X$ be the identity map, i.e. $f(x) = x$ for all $x \in X$. Then $\aut(f) = \textnormal{Homeo}(X)$ is the group of all homeomorphisms of the Cantor set.
\end{ex}

Recall from Section \ref{subsec:symbolicdynamics} a subshift is a system $(X,\sigma)$ which is a subsystem of some full shift $(\mathcal{A}^{\mathbb{Z}},\sigma)$.

\begin{ex}
For a subshift $(X,\sigma)$, the shift $\sigma$ is itself is always an automorphism of $(X,\sigma)$, i.e. $\sigma \in \aut(\sigma)$. Whenever $(X,\sigma)$ has an aperiodic point, $\sigma$ is clearly infinite order in the group $\aut(\sigma)$.
\end{ex}

\begin{ex}\label{exa:0blockcodeauto}
Let $(X_{3},\sigma_{3})$ denote the full shift on the symbol set $\{0,1,2\}$ and define an automorphism $\alpha \in \aut(\sigma_{3})$ using the block code
$$\alpha_{0} \colon x \mapsto x + 1 \textnormal{ mod } 3, \hspace{.03in} x \in \{0,1,2\}.$$
Thus for example, $\alpha$ acts like the following:
\begin{gather*}
\ldots 01020102011\overd{0}202220102110 \ldots \\
\downarrow \hspace{.03in} \alpha \\
\ldots 12101210122\overd{1}010001210221 \ldots
\end{gather*}
This automorphism is order $3$, i.e. $\alpha^{3}=\textnormal{id}$.
\end{ex}

As we'll see later, automorphism groups\si{automorphism group} of shifts of finite type contain a large supply of nontrivial automorphisms. Here is an interesting example of a subshift whose only automorphisms are powers of the shift.

\begin{ex}\label{ex:sturmiansubshift}
Let $\alpha$ be an irrational, and consider the rotation map $R_{\alpha} \colon [0,1) \to [0,1)$ given by $R_{\alpha}(x) = x + \alpha \textnormal{ mod } 1$. Consider the indicator map $I_{\alpha} \colon [0,1) \to \{0,1\}$ given by $I_{\alpha}(z)=0$ if $z \in [0,1-\alpha)$ and $I_{\alpha}(z)=1$ if $z \in [1-\alpha,1)$. Now we can define a subshift $(X_{\alpha},\sigma_{X_{\alpha}})$ of the full shift on two symbols $(\{0,1\}^{\mathbb{Z}},\sigma)$ to be the orbit closure of locations of orbits of points under the map $R_{\alpha}$, i.e. we let
$$X_{\alpha} = \overline{\{I_{\alpha}(R_{\alpha}^{k}(z)) \mid k \in \mathbb{Z}, z \in [0,1)\}}.$$
The subshift $(X_{\alpha},\sigma_{X_{\alpha}})$ is known as a Sturmian subshift, and it is a folklore result (see \cite{Olli2013} or \cite{DDMP2016} for a proof) that $\aut(\sigma_{X_{\alpha}}) = \langle \sigma_{X_{\alpha}} \rangle$, so as a group $\aut(\sigma_{X_{\alpha}})$ is isomorphic to $\mathbb{Z}$.
\end{ex}

The subshift in the last example has zero topological entropy, and the structure of its automorphism group\si{automorphism group} is very easy to understand (as a group it's just $\mathbb{Z})$. In many cases, the automorphism groups\si{automorphism group} of subshifts with such ``low-complexity'' dynamics (of which Example \ref{ex:sturmiansubshift} is an example) have more constrained automorphism groups\si{automorphism group}, in contrast to the automorphism groups\si{automorphism group} of shifts of finite type (see \apr{apprem:lowcomplexity}) for a brief discussion of this, and for what we mean here by low-complexity).\\

By the Curtis-Hedlund-Lyndon Theorem\si{Curtis-Hedlund-Lyndon Theorem} (Theorem \ref{thm:chlthm}), any automorphism of a subshift $(X,\sigma)$ is induced by a block code. This leads immediately to the following observation:
\begin{prop}
If $(X,\sigma)$ is a subshift, then $\aut(\sigma)$ is a countable group.
\end{prop}

Thus for a shift of finite type $(X_{A},\sigma_{A})$, $\aut(\sigma_{A})$ is always a countable group. Under our assumptions that $(X_{A},\sigma_{A})$ is mixing with positive entropy, $\aut(\sigma_{A})$ is also always infinite.

It turns out that $\aut(\sigma_{A})$ possesses a rich algebraic structure. Example \ref{exa:0blockcodeauto} above was induced by a block code of range $0$, but for arbitrarily large $R \in \mathbb{N}$ there are automorphisms which can only be induced by block codes of range $R$ or greater (indeed, given an SFT $(X_{A},\sigma_{A})$ and a non-negative number $R$, there are only finitely many automorphisms in $\aut(\sigma_{A})$ having range $\le R$). To give an indication that $\aut(\sigma_{A})$ is quite large, consider the following results regarding different types of subgroups that can arise in $\aut(\sigma_{A})$.

\begin{thm}\label{thm:largesubgroups}
Let $(X_{A},\sigma_{A})$ be a shift of finite type where $A$ is a primitive matrix with $\lambda_{A} > 1$.
\begin{enumerate}
\item
  (Boyle-Lind-Rudolph\ai{Lind, Douglas}
  in \cite{BLR88}) The group $\aut(\sigma_{A})$ contains isomorphic copies of each of the following groups:
\begin{enumerate}
\item
Any finite group.
\item
$\bigoplus\limits_{i=1}^{\infty}\mathbb{Z}$.
\item
The free group on two generators $\mathbb{F}_{2}$.
\end{enumerate}
\item
(Kim-Roush\ai{Roush, F.W.}\ai{Kim, K.H.} in \cite{S27}) For any $n \ge 2$, let $(X_{n},\sigma_{n})$ denote the full shift on $n$ symbols. Then $\aut(\sigma_{n})$ is isomorphic to a subgroup of $\aut(\sigma_{A})$.
\item
  (Kim-Roush\ai{Roush, F.W.} in \cite{S27})\ai{Kim, K.H.}
  Any countable, locally finite, residually finite group embeds into $\aut(\sigma_{A})$.
\end{enumerate}
\end{thm}

In particular, by part $(1)$, $\aut(\sigma_{A})$ is never amenable. By part $(2)$, for full shifts, the isomorphism types of groups that can appear as subgroups of $\aut(\sigma_{n})$ is independent of $n$.\\

Recall a group $G$ is residually finite if the intersection of all its subgroups of finite index is trivial. A finitely presented group $G$ is said to have solvable word problem if there is an algorithm to determine whether a word made from generators is the identity in the group.

\begin{exer}\apr{appex:denppresfin})
  \label{exer:denseppresfinite}
If $(X,\sigma)$ is a subshift whose periodic points are dense in $X$, then $\aut(\sigma)$ is residually finite.
\end{exer}

\begin{prop}\label{prop:resfinsolwp}
Let $(X_{A},\sigma_{A})$ be a shift of finite type where $A$ is a primitive matrix with $\lambda_{A} > 1$. Then both of the following hold:
\begin{enumerate}
\item
The group $\aut(\sigma_{A})$ is residually finite.
\item
The group $\aut(\sigma_{A})$ contains no finitely generated group with unsolvable word problem.
\end{enumerate}
\end{prop}
\begin{proof}
  Such an SFT has a dense set of periodic points (see
  \cite[Sec. 6.1]{LindMarcus2021}),\ai{Lind, Douglas}\ai{Marcus, Brian} so (1) follows from Exercise \ref{exer:denseppresfinite}. For (2), see \cite[Prop. 2.8]{BLR88}.
\end{proof}

Since a subgroup of a residually finite group must be residually finite, both parts of the previous proposition give some necessary conditions for a group to embed as a subgroup of $\aut(\sigma_{A})$. For example, it follows that the additive group of rationals $\mathbb{Q}$ cannot embed into $\aut(\sigma_{A})$, since $\mathbb{Q}$ under addition is not residually finite (however, the additive group $\mathbb{Q}$ can embed into the automorphism group\si{automorphism group} of a certain minimal subshift - see \cite[Example 3.9]{BLR88}). Still, we do not have a good understanding of what types of countable groups can be isomorphic to a subgroup of $\aut(\sigma_{A})$.\\

An important tool for constructing automorphisms in $\aut(\sigma_{A})$ is the use of ``markers''. We'll forego describing marker methods here, instead referring the reader to \cite[Sec. 2]{BLR88}; but we note that, for example, all three parts of Theorem \ref{thm:largesubgroups} make use of markers. We'll see another perspective on marker automorphisms when discussing simple automorphisms\si{simple automorphism} below.

\subsection{Simple Automorphisms\si{simple automorphism}}
In \cite{Nasu88}, Nasu\ai{Nasu, Masakazu} introduced a class of automorphisms known as simple automorphisms\si{simple automorphism}, which we'll define
shortly.
The set of automorphisms built from compositions of these simple automorphisms\si{simple automorphism} encompasses the collection of automorphisms defined using marker methods (see \cite{BoyleNasu's} for a presentation of this), and give rise to an important subgroup of $\aut(\sigma_{A})$ \apr{appnasu}).\\

Let $A$ be a square matrix over $\mathbb{Z}_{+}$, and let $\Gamma_{A}$ be its associated directed graph. A {\it simple graph symmetry}\footnote{We use the term graph symmetry instead of graph automorphism to avoid confusion between automorphisms of graphs and automorphisms of subshifts.} of $\Gamma_{A}$ is a graph automorphism of $\Gamma_{A}$ which fixes all vertices. A simple graph symmetry of $\Gamma_{A}$ gives a 0-block code and hence a corresponding automorphism in $\aut(\sigma_{A})$. Given $\alpha \in \aut(\sigma_{A})$, we call $\alpha$ a {\it simple graph automorphism} if it is induced by a simple graph symmetry of $\Gamma_{A}$, and we call $\alpha \in \aut(\sigma_{A})$ a {\it simple automorphism}\si{simple automorphism} if it is of the form
$$\alpha = \Psi \gamma \Psi^{-1}$$
where $\Psi \colon (X_{A},\sigma_{A}) \to (X_{B},\sigma_{B})$ is a conjugacy to some shift of finite type $(X_{B},\sigma_{B})$ and $\gamma \in \aut(\sigma_{B})$ is a simple graph automorphism in $\aut(\sigma_{B})$.\\

\begin{ex}
Let $A = \begin{pmatrix} 2 & 2 \\ 1 & 1 \end{pmatrix}$ and label the edges of $\Gamma_{A}$ by $a,\cdots,f$. The graph automorphism of $\Gamma_{A}$ drawn below defined by permuting the edges $c$ and $d$ is a simple graph symmetry of $\Gamma_{A}$, and the corresponding simple graph automorphism in $\aut(\sigma_{A})$ is given by the block code of range 0 which swaps the letters $c$ and $d$ and leaves all other letters fixed.
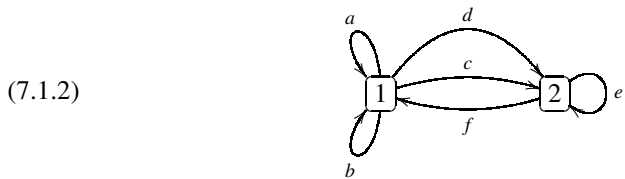
\begin{equation}\label{fig:simplegraphsymmetry1}
\begin{gathered}
\xymatrix{
*+[F-:<2pt>]{1}
\ar@(u,ul)_a \ar@(d,dl)^b
\ar@/^/[rr]^{c}
\ar@/^2pc/[rr]^{d} & &
*+[F-:<2pt>]{2}
\ar@/^/[ll]^{f}
\ar@(ru,rd)^e
}
\end{gathered}
\end{equation}
\end{ex}

We define $\simp{A}$ to be the subgroup of $\aut(\sigma_{A})$ generated by simple automorphisms\si{simple automorphism}. It is immediate to check that $\simp{A}$ is a normal subgroup of $\aut(\sigma_{A})$.\\

\begin{ex}
There is a conjugacy from the full 3-shift $(X_{3},\sigma_{3})$ on symbols $\{0,1,2\}$ to the edge shift of finite type $(X_{A},\sigma_{A})$ presented by the graph given in Figure \ref{fig:simplegraphsymmetry1} on symbol set $\{a,b,c,d,e,f\}$. Here the matrix $A$ is given by $A = \begin{pmatrix} 2 & 2 \\ 1 & 1 \end{pmatrix}$, and a conjugacy
$$\Psi \colon (X_{3},\sigma_{3}) \to (X_{A},\sigma_{A})$$
is given by the block code:
\begin{equation*}
\begin{array}{ccccc}
00 \mapsto a & \qquad & 10 \mapsto b & \qquad & 20 \mapsto f\\
01 \mapsto a & \qquad & 11 \mapsto b & \qquad & 21 \mapsto f\\
02 \mapsto d & \qquad & 12 \mapsto c & \qquad & 22 \mapsto e\\
\end{array}
\end{equation*}
with inverse given by
\begin{equation*}
\begin{array}{ccc}
a \mapsto 0 & \qquad & d \mapsto 0\\
b \mapsto 1 & \qquad & c \mapsto 1\\
e \mapsto 2 & \qquad & f \mapsto 2\\
\end{array}
\end{equation*}
Let $\gamma$ denote the simple automorphism\si{simple automorphism} in $\aut(\sigma_{A})$ induced by the simple graph symmetry of $\Gamma_{A}$ shown in Figure \ref{fig:simplegraphsymmetry1}, which permutes the edges $c$ and $d$, and let $\beta = \Psi\gamma\Psi^{-1}$. Then $\beta \in \simp{3}$, and acts for example like
\begin{gather*}
\ldots 112002\overd{0}2120011 \ldots\\
\Psi \Big\downarrow\\
\ldots bcfadf\overd{d}fcfaab \ldots\\
\gamma \Big\downarrow\\
\ldots bdfacf\overd{c}fdfaab \ldots\\
\Psi^{-1} \Big\downarrow\\
\ldots 102012\overd{1}202001 \ldots\\
\end{gather*}

Notice that $\beta$ essentially scans a string of $0,1,2$'s, and swaps $12$ with $02$.
\end{ex}

$\simp{A}$ is an important subgroup of $\aut(\sigma_{A})$, and we'll come back to it later.

\subsection{The center of $\aut(\sigma_{A})$}



Understanding the structure of $\aut(\sigma_{A})$ as a group is not easy. One useful result is the following, proved by Ryan in '72/'74.
\begin{thm}[\cite{Ryan1,Ryan2}]\label{thm:ryanstheorem}
If $A$ is irreducible (in particular, if $A$ is primitive) then the center of $\aut(\sigma_{A})$ is generated by $\sigma_{A}$.
\end{thm}

Ryan's Theorem\si{Ryan's Theorem} essentially says the center of $\aut(\sigma_{A})$ is as small as it could possibly be. In fact, for $A$ irreducible, every normal amenable subgroup of $\aut(\sigma_{A})$ is contained in the subgroup generated by $\sigma_{A}$; see \apr{apprem:amenradical}).\\

In
\cite{Kopra2020SFTgliders},
Kopra proved a finitary version of Ryan's Theorem\si{Ryan's Theorem}: namely, for any nontrivial irreducible shift of finite type, there exists a subgroup generated by two elements whose centralizer is generated by the shift map.
In
\cite{KopraGlidersSofic2022},
  Kopra extended this result to nontrivial transitive sofic shifts,
  and showed that it fails to hold for nonsofic S-gap shifts.
Prior to Kopra's work, Salo in \cite{Salo2019} had proved there is a finitely generated subgroup (needing more than two generators) of the automorphism group\si{automorphism group} of the full shift on four symbols whose centralizer is generated by the shift map.\\

Ryan's Theorem\si{Ryan's Theorem} can be used to distinguish, up to isomorphism, automorphism groups\si{automorphism group} of certain subshifts of finite type. The idea is to use Ryan's Theorem\si{Ryan's Theorem} in conjunction with the set of possible roots of the shift. For a subshift $(X,\sigma)$, define the root set of $\sigma$ to be $\textnormal{root}(\sigma) = \{k \in \mathbb{N} \mid \textnormal{ there exists } \alpha \in \aut(\sigma) \textnormal{ such that } \alpha^{k}=\sigma \}$. The following exercise demonstrates this technique.

\begin{exer}\apr{appexer:ryanthmdist})\label{exer:ryanthmdist}
\begin{enumerate}
\item
Show that if $(X_{A},\sigma_{A})$ and $(X_{B},\sigma_{B})$ are irreducible shifts of finite type such that $\aut(\sigma_{A})$ and $\aut(\sigma_{B})$ are isomorphic, then $\textnormal{root}(\sigma_{A}) = \textnormal{root}(\sigma_{B})$.
\item
Let $(X_{2},\sigma_{2}), (X_{4},\sigma_{4})$ denote the full shift on 2 symbols and on 4 symbols, respectively. Show that $\textnormal{root}(\sigma_{2}) \ne \textnormal{root}(\sigma_{4})$. Use part (1) to conclude that $\aut(\sigma_{2})$ and $\aut(\sigma_{4})$ are not isomorphic as groups.
\end{enumerate}
\end{exer}

The exercise above can be generalized to some other values of $m$ and $n$; one can find this written down in \cite{HKS2022} (also see \cite[Ex. 4.2]{BLR88} for an example where the method is used to distinguish automorphism groups\si{automorphism group} in the non-full shift case). For a full shift $(X_{n},\sigma_{n})$, it turns out that $k \in \textnormal{root}(\sigma_{n})$ if and only if $n$ has a $k$th root in $\mathbb{N}$ (see \cite[Theorem 8]{Lind84}).\ai{Lind, Douglas} \\

Currently, the technique of using Ryan's Theorem\si{Ryan's Theorem} in conjunction with $\textnormal{root}(\sigma_{A})$ is the only method known to us which can show two explicit nontrivial mixing shifts of finite type have non-isomorphic automorphism groups\si{automorphism group}. We do not at the moment know how to distinguish automorphism groups\si{automorphism group} with identical root sets; in particular, despite being introduced by Hedlund in the 60's, we still do not know whether $\aut(\sigma_{2})$ and $\aut(\sigma_{3})$ are isomorphic (see Problem \ref{prob:fullshiftautisoproblem} in Section \ref{subsec:openproblemsaut}).

\subsection{Representations of $\aut(\sigma_{A})$}
So how can we study $\aut(\sigma_{A})$? One way is to try to find good representations of it. There are two main classes of representations that we know of:
\begin{enumerate}
\item
Periodic point representations, and representations derived from these.
\item
The dimension representation\si{dimension representation}.
\end{enumerate}

The first, the periodic point representations (and ones derived from them), are quite natural to consider. They also lead to the sign and gyration maps, which are also quite natural (once defined). The second, the dimension representation\si{dimension representation}, is essentially a linear representation, and is based on the dimension group\si{dimension group} associated to the shift of finite type in question.

We start with the second one, the dimension representation\si{dimension representation}.
\subsection{Dimension Representation\si{dimension representation}}\label{subsec:dimrep}
We briefly recall the definition, introduced in Section \ref{sec:SEZdirectlimits} in Lecture 2, of the dimension group\si{dimension group} associated to a $\mathbb{Z}_{+}$-matrix. Given an $r \times r$ matrix $A$ over $\mathbb{Z}_{+}$ the eventual range subspace of $A$ is $ER(A) = \mathbb{Q}^{r}A^{r}$ (we will have matrices act on row vectors throughout), and the dimension group\si{dimension group} associated to $A$ is
$$G_{A} = \{ x \in ER(A) \mid x A^{k} \in \mathbb{Z}^{r} \cap ER(A) \textnormal{ for some }k \ge 0\}.$$
Recall also the group $G_{A}$ comes equipped with an automorphism (of abelian groups) $\delta_{A} \colon G_{A} \to G_{A}$ (the automorphism $\delta_{A}$ was denoted by $\hat{A}$ in Lecture 2, but we'll use the notation $\delta_{A}$). The automorphism $\delta_{A}$ of $G_{A}$ makes $G_{A}$ into a $\mathbb{Z}[t,t^{-1}]$-module by having $t$ act by $\delta_{A}^{-1}$, but we will usually just refer to the pair $(G_{A},\delta_{A})$ to indicate we are considering both $G_{A}$ and $\delta_{A}$ together. Then by an automorphism of $(G_{A},\delta_{A})$ we mean a group automorphism $\Psi \colon G_{A} \to G_{A}$ which satisfies $\Psi \delta_{A} = \delta_{A} \Psi$; in other words, an automorphism of the pair is equivalent to an automorphism of $G_{A}$ as a $\mathbb{Z}[t,t^{-1}]$-module. Let $\aut(G_{A})$ denote the group of automorphisms of the pair $(G_{A},\delta_{A})$.
\\

The group $G_{A}$ is isomorphic, as an abelian group, to the direct limit
 $\varinjlim \{\mathbb{Z}^{r}, x \mapsto xA\}$.\\


When $A$ is over $\mathbb{Z}_{+}$ (which is the case for a matrix presenting an edge shift of finite type), $G_{A}$ has a positive cone $G_{A}^{+} = \{v \in G_{A} \mid vA^{k} \in \mathbb{Z}^{r}_{+} \textnormal{ for some } k\}$ making $G_{A}$ into an ordered abelian group. The automorphism $\delta_{A}$ maps $G_{A}^{+}$ into $G_{A}^{+}$, and when we want to keep track of the order structure we refer to the triple $(G_{A},G_{A}^{+},\delta_{A})$. An automorphism of the triple $(G_{A},G_{A}^{+},\delta_{A})$ then means an automorphism of $(G_{A},\delta_{A})$ which preserves $G_{A}^{+}$.\\

\begin{exer}\apr{appexer:dimgroupfullshiftcomp})\label{exer:dimgroupfullshiftcomp}
When $A = (n)$ (the case of the full-shift on $n$ symbols), the triple $(G_{n},G_{n}^{+},\delta_{n})$ is isomorphic to the triple $(\mathbb{Z}[\frac{1}{n}],\mathbb{Z}_{+}[\frac{1}{n}],m_{n})$, where $m_{n}$ is the automorphism of $\mathbb{Z}[\frac{1}{n}]$ defined by $m_{n}(x) = x \cdot n$.
\end{exer}

The following exercise shows that for a mixing shift of finite type $(X_{A},\sigma_{A})$, the group of automorphisms of $(G_{A},G_{A}^{+},\delta_{A})$ is index two in $\aut(G_{A},\delta_{A})$.

\begin{exer}\apr{appexer:indextwoorderautos})\label{exer:indextwoorderautos}
Let $A$ be a primitive matrix and suppose $\Psi$ is an automorphism of $(G_{A},\delta_{A})$. By considering $G_{A}$ as a subgroup of $ER(A)$, show that $\Psi$ extends to a linear automorphism $\tilde{\Psi} \colon ER(A) \to ER(A)$ which multiplies the Perron eigenvector of $A$ by some quantity $\lambda_{\Psi}$. Show that $\Psi$ is also an automorphism of the ordered abelian group $(G_{A},G_{A}^{+},\delta_{A})$ if and only if $\lambda_{\Psi}$ is positive.
\end{exer}

Krieger\ai{Krieger, Wolfgang}
gave a definition of a triple $(D_{A},D_{A}^{+},d_{A})$ which is isomorphic to the triple $(G_{A},G_{A}^{+},\delta_{A})$ using only topological/dynamical data intrinsic to the system $(X_{A},\sigma_{A})$ \apr{rem:kriegerconstruction}).\\

A topological conjugacy between shifts of finite type $\Psi \colon (X_{A},\sigma_{A}) \to (X_{B},\sigma_{B})$ induces an isomorphism $\Psi_{*} \colon (G_{A},G_{A}^{+},\delta_{A}) \stackrel{\cong}\longrightarrow (G_{B},G_{B}^{+},\delta_{B})$. This is easiest to see
using Krieger's\ai{Krieger, Wolfgang}
intrinsic definition of $(G_{A},G_{A}^{+},\delta_{A})$ (see \apr{rem:kriegerconstruction})). One can also see this in terms of the conjugacy/strong shift equivalence\si{strong shift equivalence} framework developed in Lecture 2, as follows. Given a conjugacy $\alpha \colon (X_{A},\sigma_{A}) \to (X_{B},\sigma_{B})$, from Lecture 2 we know that corresponding to $\alpha$ is some strong shift equivalence\si{strong shift equivalence} from $A$ to $B$
$$A = R_{1}S_{1}, A_{2} = S_{1}R_{1}, \ldots, A_{n} = R_{n}S_{n}, B = S_{n}R_{n}.$$
Then we  define an isomorphism
$\pi (\alpha)$ from $(G_{A},G_{A}^{+},\delta_{A})$ to $(G_{B},G_{B}^{+},\delta_{B})$ by
$$\pi(\alpha) : \, v \mapsto vR_1\cdots R_n\ .
$$
A priori, it is not clear that $\pi(\alpha)$ is actually well-defined, since the strong shift equivalence\si{strong shift equivalence} we choose to associate to $\alpha$ may not be unique. However, it turns out that $\pi(\alpha)$ is indeed well-defined; this will be a consequence of material in Lecture 8.\

Since an automorphism of $(X_{A},\sigma_{A})$ is just a self-conjugacy of $(X_{A},\sigma_{A})$, it follows that any $\alpha \in \aut(\sigma_{A})$ induces an isomorphism
$\alpha_{*} \colon (G_{A},G_{A}^{+},\delta_{A})
\stackrel{\cong}\longrightarrow (G_{A},G_{A}^{+},\delta_{A})$.
Morever, if for $i=1,2$ we have $\alpha_i: v\mapsto vR_i$,
  then $\alpha_1\alpha_2: v\mapsto vR_1R_2$ (because composition in
  $\aut(\sigma_{A})$ is defined left to
  right), hence $(\alpha_1)_{*}(\alpha_2)_{*} =
  (\alpha_1\alpha_2)_{*}$.
Thus the rule $\alpha \mapsto \alpha_{*}$ defines a group homomorphism 
\begin{equation}
\pi_{A} \colon \aut(\sigma_{A}) \to \aut(G_{A},G_{A}^{+},\delta_{A})
\end{equation}
which
is known 
as the {\it dimension representation}\si{dimension representation}
of $\aut(\sigma_{A})$\begin{footnote}{It can
      happen that $(\alpha_1)_{*}$ and
      $(\alpha_2)_{*}$ do not commute. In this case,
      the map $\pi_A$ would be well
      defined, but would not be a group
      homomorphism.}\end{footnote}.

and there is a well-defined homomorphism
\begin{equation}
\pi_{A} \colon \aut(\sigma_{A}) \to \aut(G_{A},G_{A}^{+},\delta_{A}).
\end{equation}
The homomorphism $\pi_{A}$ is known as the {\it dimension representation}\si{dimension representation} of $\aut(\sigma_{A})$.

\begin{ex}
The automorphism $\sigma_{A} \in \aut(\sigma_{A})$ corresponds to the strong shift equivalence\si{strong shift equivalence}
$$A = (A)(I), \hspace{.03in} A = (I)(A).$$
In particular, we have for any shift of finite type $(X_{A},\sigma_{A})$
$$\pi_{A}(\sigma_{A}) = \delta_{A} \in \aut(G_{A},G_{A}^{+},\delta_{A}).$$
\end{ex}

\begin{ex}
When $A = (3)$, $ER(A) = \mathbb{Q}$, and as mentioned above, the dimension triple is isomorphic to $(\mathbb{Z}[\frac{1}{3}],\mathbb{Z}_{+}[\frac{1}{3}],m_{3})$ where $m_{3}(x) = 3x$. Thus $\aut(G_{3},G_{3}^{+},\delta_{3}) \cong \mathbb{Z}$, where $\mathbb{Z}$ is generated by $\delta_{3}$. The dimension representation then looks like
$$
\pi_{3} \colon \aut(\sigma_{3}) \to \aut(\mathbb{Z}[\frac{1}{3}],\mathbb{Z}_{+}[\frac{1}{3}],\delta_{3}) \cong \mathbb{Z} = \langle \delta_{3} \rangle $$
$$
\pi_{3} \colon \sigma_{3} \mapsto \delta_{3}.
$$
\end{ex}
More generally, the following proposition describes how the dimension representation behaves for full shifts. Given $n \in \mathbb{N}$, let $\omega(n)$ denote the number of distinct prime divisors of $n$.

\begin{prop}\label{prop:fullshiftdimautos}
Given $n \ge 2$, there is an isomorphism $\aut(G_{n},G^{+}_{n},\delta_{n}) \cong \mathbb{Z}^{\omega(n)}$ and the map $\pi_{n} \colon \aut(\sigma_{n}) \to \aut(G_{n},G^{+}_{n},\delta_{n})$ is surjective.
\end{prop}
\begin{proof}
From Exercise \ref{exer:dimgroupfullshiftcomp} we know $(G_{n},G_{n}^{+},\delta_{n}) \cong (\mathbb{Z}[\frac{1}{n}],\mathbb{Z}_{+}[\frac{1}{n}],\delta_{n})$. The result follows since the group $\aut(\mathbb{Z}[\frac{1}{n}],\mathbb{Z}_{+}[\frac{1}{n}],\delta_{n})$ is free abelian with basis given by the maps $\delta_{p_{i}} \colon x \mapsto x\cdot p_{i}$ where $p_{i}$ is a prime dividing $n$. For the surjectivity part of $\pi_{n}$, see \cite{BLR88}.
\end{proof}

In general, the dimension representation may not be surjective (see \cite{S19}), and the following question is still open:\\

\begin{prob}\label{prob:dimreprange}
Given a mixing shift of finite type $(X_{A},\sigma_{A})$, what is the image of the dimension representation $\pi_{A} \colon \aut(\sigma_{A}) \to \aut(G_{A},G_{A}^{+},\delta_{A})$?
\end{prob}

Problem \ref{prob:dimreprange} is of relevance for the classification problem\si{classification problem} (see \apr{apprem:dimreprangeclassification})).\\

In \cite[Theorem 6.8]{BLR88} it is shown that if the nonzero eigenvalues of $A$ are simple, and no ratio of distinct eigenvalues is a root of unity, then for all sufficiently large $m$ the dimension representation $\pi_{A}^{(m)} \colon \aut(\sigma_{A}^{m}) \to \aut(G_{A^{m}},G_{A^{m}}^{+},\delta_{A^{m}})$ is onto. Long \cite{Long2013} showed the ``elementary'' construction method of \cite[Theorem 6.8]{BLR88} is not in general sufficient to reveal the full image of the dimension representation.\\

An automorphism $\alpha \in \aut(\sigma_{A})$ is called {\it inert}\si{inert automorphism} if $\alpha$ lies in the kernel of $\pi_{A}$, and we denote the subgroup of inerts by
$$\inert{A} = \ker \pi_{A}.$$

The subgroup $\inert{A}$\si{inert automorphism} is, roughly speaking, the heart of $\aut(\sigma_{A})$, and in general, we do not know how to distinguish the subgroup of inert automorphisms\si{inert automorphism} among different shifts of finite type. The following exercise shows that constructions using marker methods or simple automorphisms\si{simple automorphism} always lie in $\inert{A}$\si{inert automorphism}.

\begin{exer}\apr{appexer:simpininert})\label{exer:simpininerts}
For any shift of finite type $(X_{A},\sigma_{A})$, we have $\simp{A} \subset \inert{A}$\si{inert automorphism}. (Hint: Use \apr{rem:kriegerconstruction}))
\end{exer}

\begin{rem}
As evidence that $\inert{A}$\si{inert automorphism} contains much of the complicated algebraic structure of $\aut(\sigma_{A})$, consider the case of a full shift over a prime number of symbols, i.e. $A = (p)$ for some prime $p$. In this case, $\aut(G_{p},G_{p}^{+},\delta_{p}) \cong \mathbb{Z}$ is generated by $\delta_{p}$, and the map
$$\pi_{p} \colon \aut(\sigma_{p}) \to \aut(G_{p},G_{p}^{+},\delta_{p})$$
is a split surjection, with a splitting map being given by $\delta_{p} \mapsto \sigma_{p}$. This shows $\aut(\sigma_{p})$ is isomorphic to a semi-direct product of $\inert{p}$\si{inert automorphism} and $\mathbb{Z}$. Since $\sigma_{p}$ lies in the center of $\aut(\sigma_{p})$, in fact this semi-direct product is isomorphic to a direct product, and we have
$$\aut(\sigma_{p}) \cong \inert{p} \times \mathbb{Z}.$$

\end{rem}

\subsection{Periodic point representation\si{periodic point representation}}
For an SFT $(X_{A},\sigma_{A})$ and $k \in \mathbb{N}$ we let $P_{k}$ denote the $\sigma_{A}$-periodic points of least period $k$, and $Q_{k}$ the set of $\sigma_{A}$-orbits of length $k$ (both $P_{k}$ and $Q_{k}$ depend on $\sigma_{A}$ of course - we suppress this in the notation since it's usually clear from context). For a shift of finite type, the set $P_{k}$ is always finite, and we have
$$|P_{k}| = k|Q_{k}|.$$

Let $\alpha \in \aut(\sigma_{A})$ and let $k \in \mathbb{N}$. Since $\alpha$ is a bijection which commutes with $\sigma_{A}$, $\alpha$ maps $P_{k}$ to itself and thus induces a permutation of $P_{k}$ which we'll denote by
$$\rho_{k}(\alpha) \in \sym(P_{k})$$
where $\sym(P)$ of a set $P$ denotes the group of permutations of $P$ (we use the convention that if $P = \emptyset$ then $\sym(P)$ is the group containing only one element).

It is straightforward to check that this assignment $\alpha \mapsto \rho_{k}(\alpha)$ defines a homomorphism
$$\rho_{k} \colon \aut(\sigma_{A}) \to \sym(P_{k}).$$

The automorphism $\alpha$ must also respect $\sigma_{A}$-orbits, and it follows that $\alpha$ induces a permutation of the set $Q_{k}$ which we denote
$$\xi_{k}(\alpha) \in \sym(Q_{k}).$$
Thus, we also get a homomorphism
$$\xi \colon \aut(\sigma_{A}) \to \sym(Q_{k}).$$

These homomorphisms assemble into homomorphisms
\begin{equation}
\begin{gathered}
\rho \colon \aut(\sigma_{A}) \to \prod_{k=1}^{\infty}\sym(P_{k})\\
\rho(\alpha) = (\rho_{1}(\alpha),\rho_{2}(\alpha),\ldots).
\end{gathered}
\end{equation}
and
\begin{equation}
\begin{gathered}
\xi \colon \aut(\sigma_{A}) \to \prod_{k=1}^{\infty}\sym(Q_{k})\\
\xi(\alpha) = (\xi_{1}(\alpha),\xi_{2}(\alpha),\ldots).
\end{gathered}
\end{equation}
The map $\rho$ is called the {\it periodic point representation}\si{periodic point representation} of $\aut(\sigma_{A})$, and $\xi$ is called the {\it periodic orbit representation}.\\

When $A$ is irreducible, the map $\rho$ is injective (this follows from the fact that for irreducible $A$, periodic points are dense in $(X_{A},\sigma_{A})$ - see \cite[Sec. 6.1]{LindMarcus2021}).\ai{Lind, Douglas}\ai{Marcus, Brian} Clearly $\xi$ can not be injective since $\sigma_{A} \in \xi$. However, it turns out $\sigma_{A}$ generates the whole kernel of $\xi$, from a theorem of Boyle-Krieger.\ai{Krieger, Wolfgang}

\begin{thm}\label{thm:injximap}
If $(X_{A},\sigma_{A})$ is an irreducible shift of finite type, then $\ker \xi = \langle \sigma_{A} \rangle$.
\end{thm}


Fix $k \in \mathbb{N}$ and $\alpha \in \aut(\sigma_{A})$. The periodic point representation\si{periodic point representation} $\rho_{k}(\alpha)$ is obtained by restricting $\alpha$ to the finite subsystem $P_{k}$ of $(X_{A},\sigma_{A})$, and $\rho_{k}(\alpha)$ lies in the automorphism group\si{automorphism group} $\aut(\sigma_{A}|_{P_{k}})$ of this finite system. It was observed in \cite{MR887501} that the automorphism group\si{automorphism group} $\aut(\sigma_{A}|_{P_{k}})$ is isomorphic to the semidirect product $\left(\mathbb{Z}/k\mathbb{Z}\right)^{Q_{k}} \rtimes \sym(Q_{k})$ \apr{apprem:semidirectaut}), and this leads to considering possible abelian factors of these automorphism groups\si{automorphism group} $\aut(\sigma_{A}|_{P_{k}})$. This motivates the following gyration maps, which were introduced by Boyle and
Krieger in \cite{MR887501}.\ai{Krieger, Wolfgang}

\begin{de}\label{def:sg}
Fix $k \in \mathbb{N}$. We define the $k$th gyration map $g_{k} \colon \aut(\sigma_{A}) \to \mathbb{Z}/k\mathbb{Z}$ as follows. Let $\alpha \in \aut(\sigma_{A})$, let $Q_{k} = \{O_{1},\ldots,O_{I(k)}\}$ denote the set of orbits in $Q_{k}$, and choose, for each $1 \le i \le k$, some representative point $x_{i} \in O_{i}$. Then $\alpha(x_{i}) \in O_{\xi_{k}(\alpha)(i)}$, so there exists some $r(\alpha,i) \in \mathbb{Z} / k\mathbb{Z}$ such that $\alpha(x_{i}) = \sigma_{n}^{r(\alpha,i)}(x_{\xi_{k}(\alpha)(i)})$. Now define
$$g_{k} = \sum_{i=1}^{I(k)} r(\alpha,i) \in \mathbb{Z} / k\mathbb{Z}.$$
Boyle and Krieger\ai{Krieger, Wolfgang} showed this map is independent of the choices of $x_{i}$'s, and is a homomorphism, so we get homomorphisms
$$g_{k} \colon \aut(\sigma_{n}) \to \mathbb{Z} / k\mathbb{Z}.$$
Now we can define the {\it gyration representation} by
\begin{equation}
\begin{gathered}
g \colon \aut(\sigma_{n}) \to \prod_{k=1}^{\infty}\mathbb{Z} / k\mathbb{Z}\\
g(\alpha) = (g_{1}(\alpha),g_{2}(\alpha),\ldots).
\end{gathered}
\end{equation}
\end{de}

Given $k$, consider $\sgn \xi_{k} \colon \aut(\sigma_{A}|_{P_{k}}) \to \mathbb{Z}/2\mathbb{Z}$, the map $\xi_{k}$ composed with the sign map to $\mathbb{Z}/2\mathbb{Z}$. The gyration map $g_{k}$, together with $\sgn \xi_{k}$, determines the abelianization of $\aut(\sigma_{A}|_{P_{k}})$: any other map from $\aut(\sigma_{A}|_{P_{k}})$ to an abelian group factors through the map
\begin{gather*}
g_{k} \times \sgn \xi_{k} \colon \aut(\sigma_{A}|_{P_{k}}) \to \mathbb{Z}/k\mathbb{Z} \times \mathbb{Z}/2\mathbb{Z}
\end{gather*}
(see \apr{apprem:abelsigngyr})).

\subsection{Inerts\si{inert automorphism} and the sign-gyration compatibility condition\si{sign-gyration compatibility condition}}\label{subsec:sgcc}
A priori, it would seem the dimension representation and the periodic point representation\si{periodic point representation} need not have any relationship. Remarkably, this turns out not to be the case, and there is in fact a connection between them: for inert automorphisms\si{inert automorphism} (recall inert automorphisms\si{inert automorphism} are precisely the kernel of the dimension representation), there are certain conditions which relate the periodic orbit representation and the periodic point representation\si{periodic point representation} of the automorphism. This is formalized in the following way.

\begin{de}
Say $\alpha \in \aut(\sigma_{A})$ {\it satisfies SGCC (sign-gyration compatibility condition)}\si{sign-gyration compatibility condition} if the following holds: for every positive odd integer $m$ and every non-negative integer $i$, if $n=m2^{i}$, then

\begin{alignat*}{2} \\
g_n(\alpha ) &=0 \quad \ \ && \text{if } \quad \prod_{j=0}^{i-1}\sgn \xi_{m2^{j}}(\alpha) = 1\\
g_n(\alpha ) &=\frac n2 && \text{if } \quad \prod_{j=0}^{i-1}\sgn \xi_{m2^{j}}(\alpha) = -1     \ .
\end{alignat*}

The empty product we take to have the value 1.
\end{de}

Thus for $\alpha \in \aut(\sigma_{A})$ satisfying SGCC\si{sign-gyration compatibility condition}, $g(\alpha)$ and $\sgn \xi(\alpha)$ determine each other.\\

An important step is to rephrase the SGCC\si{sign-gyration compatibility condition} condition in terms of certain homomorphisms, which we describe now. Consider now the $\sgn$ homomorphisms as taking values in the group $\mathbb{Z}/2\mathbb{Z}$ (so if $\tau$ is an odd permutation, $\sgn(\tau)=1 \in \mathbb{Z}/2\mathbb{Z}$).
Define for $n \ge 2$ the SGCC\si{sign-gyration compatibility condition} homomorphism
\begin{equation*}
\begin{gathered}
SGCC_{n} \colon \aut(\sigma_{A}) \to \mathbb{Z}/n\mathbb{Z}\\
SGCC_{n}(\alpha) = g_{n}(\alpha) + \left( \frac{n}{2}\right)\sum_{j>0}\sgn \xi_{n/2^{j}}(\alpha)
\end{gathered}
\end{equation*}
where we define $\sgn \xi_{n/2^{j}}(\alpha) = 0$ if $n/2^{j}$ is not an integer. The following is immediate to check, but very useful.

\begin{prop}
Let $(X_{A},\sigma_{A})$ be a mixing shift of finite type, and $\alpha \in \aut(\sigma_{A})$. Then $\alpha$ satisfies SGCC\si{sign-gyration compatibility condition} if and only if for all $n \ge 2$, $SGCC_{n}(\alpha)=0$.
\end{prop}

So which automorphisms satisfy SGCC\si{sign-gyration compatibility condition}? Amazingly enough, any inert automorphism\si{inert automorphism} does. This fact was the culmination of results obtained over several years (see \apr{apprem:sgccremark})), and was finally proved by Kim\ai{Kim, K.H.}
and Roush\ai{Roush, F.W.} in \cite{S23}, using an important cocycle lemma of Wagoner\ai{Wagoner, J.B.}. A more complete picture was subsequently given by Kim-Roush\ai{Roush, F.W.}-Wagoner\ai{Kim, K.H.}\ai{Wagoner, J.B.}
in \cite{S19}; we'll describe this briefly here. The appropriate setting for a deeper understanding is Wagoner's CW complexes, which are the subject of the next lecture.\\
\indent  Suppose $A=RS, B=SR$ is a strong shift equivalence\si{strong shift equivalence} over $\mathbb{Z}_{+}$, and let $\phi_{R,S} \colon (X_{A},\sigma_{A}) \to (X_{B},\sigma_{B})$ be a conjugacy induced by this SSE. In \cite{S19}, Kim-Rough-Wagoner\ai{Kim, K.H.}\ai{Wagoner, J.B.}
showed that, using certain lexicographical orderings on each set of periodic points, one can compute $SGCC_{m}$\si{sign-gyration compatibility condition} values, with respect to this choice of ordering on periodic points, analogous to how the $SGCC_{m}$ homomorphisms are defined for automorphisms. Moreover, they showed these values can be computed in terms of a (complicated) formula defined only using terms from the matrices $R,S$. In fact, this formula makes sense even if we start with a strong shift equivalence\si{strong shift equivalence} $A=RS, B=SR$ over $\mathbb{Z}$, and Kim-Roush\ai{Roush, F.W.}-Wagoner\ai{Wagoner, J.B.}\ai{Kim, K.H.}
showed that these formulas can be used to define homomorphisms $sgcc_{m} \colon \aut(G_{A},\delta_{A}) \to \mathbb{Z}/m\mathbb{Z}$. Note that the domain of this homomorphism is $\aut(G_{A},\delta_{A})$, i.e. automorphisms of the pair $(G_{A},\delta_{A})$ which don't necessarily preserve the positive cone $G_{A}^{+}$. Altogether, Kim-Roush\ai{Roush, F.W.}-Wagoner\ai{Wagoner, J.B.}\ai{Kim, K.H.}
proved the following.

\begin{thm}[\cite{S19}]\label{thm:sgccfactorization}\ai{Kim, K.H.}
Let $(X_{A},\sigma_{A})$ be a mixing shift of finite type. For every $m \ge 2$ there exists a homomorphism $sgcc_{m} \colon \aut(G_{A},\delta_{A}) \to \mathbb{Z}/m\mathbb{Z}$ such that the following diagram commutes
\begin{equation*}
\xymatrix{
\aut(\sigma_{A}) \ar[r]^{\pi_{A}} \ar[dr]_{SGCC_{m}}& \aut(G_{A},\delta_{A}) \ar[d]^{sgcc_{m}}  \\
 & \mathbb{Z}/m\mathbb{Z} \\
}
\end{equation*}
In particular, if $\alpha \in \inert(\sigma_{A})$, then $SGCC_{m}(\alpha)=0$.
\end{thm}

An explicit formula for $sgcc_{2}$ can be found in \cite[Prop. 2.14]{S19}, with a general formula for $sgcc_{m}$ described in \cite[2.31]{S19}.\\

As shown in \cite{S23} and \cite{S19}, that SGCC\si{sign-gyration compatibility condition} vanishes on any inert automorphism\si{inert automorphism} can be used to rule out certain actions on finite subsystems of the shift system. For example, the following was shown in \cite{S19} (based on a suggestion by Ulf Fiebig\ai{Fiebig, Ulf}). Consider an automorphism $\alpha$ of the period 6 points of the full 2 shift $(X_{2},\sigma_{2})$ which acts by the shift on one of the orbits, and the identity on the remaining orbits. It is immediate to compute that $SGCC_{6}(\alpha) = 1 \in \mathbb{Z}/6\mathbb{Z}$. However $\aut(G_{2},\delta_{2}) \cong \mathbb{Z}$ is generated by $\delta_{2}$, the image of the shift $\sigma_{2}$ under the dimension representation $\pi_{2}$, and $sgcc_{6}(\delta_{2}) =  3 \in \mathbb{Z}/6\mathbb{Z}$; by Theorem \ref{thm:sgccfactorization}, this implies the image of $SGCC_{6}$ in $\mathbb{Z}/6\mathbb{Z}$ must be the subgroup $\{0,3\} \subset \mathbb{Z}/6\mathbb{Z}$, which does not contain $1$. Thus $\alpha$ can not be the restriction of an automorphism in $\aut(\sigma_{2})$. This (along with an additional example given in \cite{S19}) resolved a long standing open problem about lifting automorphisms from finite subsystems (see Problem \ref{prob:liftproblem1} in Section \ref{subsec:openproblemsaut}).

\subsection{Actions on finite subsystems}\label{subsec:actionsonfinitesubs}
The SGCC conditions\si{sign-gyration compatibility condition} give necessary conditions for the action of an inert automorphism\si{inert automorphism} on finite subsystems of the shift system. A natural question is whether one can determine precisely what possible actions can be realized: that is, what are sufficient conditions for an automorphism of a finite subsystem to be the restriction of an inert automorphism\si{inert automorphism}? \apr{apprem:actionsfinitesubsrem}) In \cite{MR1125880}, Boyle and Fiebig\ai{Fiebig, Ulf} characterized the possible actions of finite order inert automorphisms\si{inert automorphism} on finite subsystems of the shift. Then, in \cite{KRWForumI,KRWForumII}, Kim-Roush\ai{Roush, F.W.}-Wagoner\ai{Wagoner, J.B.}\ai{Kim, K.H.}
settled this question completely, by showing that the SGCC condition\si{sign-gyration compatibility condition} is also sufficient for lifting an automorphism of a finite subsystem to an automorphism of the shift. Together with the Boyle-Fiebig\ai{Fiebig, Ulf} classification in \cite{MR1125880}, this is used in
\cite{KRWForumI,KRWForumII} to resolve (in the negative) a long standing problem regarding finite order generation\si{Finite Order Generation Problem} of the inert\si{inert automorphism} subgroup $\inert{A}$\si{inert automorphism}; see Section \ref{subsec:openproblemsaut}.

\subsection{Notable problems regarding $\aut(\sigma_{A})$}\label{subsec:openproblemsaut}
There have been a number of questions and conjectures that have been influential in the study of $\aut(\sigma_{A})$, and we'll describe a few of them here. This is
by no means intended to be an exhaustive list; instead, we simply highlight some problems that have been important (both historically, and still), as well as some problems that demonstrate the state of our ignorance regarding the group $\aut(\sigma_{A})$. Some of these have been resolved in some cases, while some are open in all cases.\\

Given a group $G$, let $\textnormal{Fin}(G)$ denote the (normal) subgroup of $G$ generated by elements of finite order.\\

Recall for any shift of finite type $(X_{A},\sigma_{A})$, we have containments of subgroups $\simp{A} \subset \textnormal{Fin}(\inert{A}) \subset \inert{A}$\si{inert automorphism}. One general problem\footnote{What we call the Finite Order Generation Problem\si{Finite Order Generation Problem} here was historically posed as a conjecture. Here we opted instead for the word 'problem', since this conjecture is known to be false in general.} is the following:

\begin{prob}[Finite Order Generation (FOG) Problem\si{Finite Order Generation Problem}]
When is it true that $\inert{A} = \textnormal{Fin}(\inert{A})$\si{inert automorphism}?
\end{prob}

The FOG\si{Finite Order Generation Problem} problem is an outgrowth of a conjecture, originally posed by F. Rhodes to Hedlund in a correspondence, asking whether $\aut(\sigma_{2})$ is generated by $\sigma_{2}$ and elements of finite order.\\

Kim\ai{Kim, K.H.}, Roush\ai{Roush, F.W.} and Wagoner\ai{Wagoner, J.B.} fin \cite{KRWForumI,KRWForumII} showed there exists a shift of finite type $(X_{B},\sigma_{B})$ such that the containment $\textnormal{Fin}(\inert{B}) \subset \inert{B}$\si{inert automorphism} is proper, showing the answer to FOG is `not always' (see the discussion in Section \ref{subsec:actionsonfinitesubs}). Prior to this, in \cite{Wagoner90eventual} Wagoner\ai{Wagoner, J.B.} considered a stronger form of FOG, asking whether it was always true that $\simp{A} = \inert{A}$\si{inert automorphism}; this was sometimes referred to as the Simple Finite Order Generation\si{Finite Order Generation Problem} Conjecture (SFOG). Kim\ai{Kim, K.H.} and Roush\ai{Roush, F.W.} in \cite{S25} showed (prior to their example showing FOG does not always hold) that SFOG does not always hold, giving an example of a shift of finite type $(X_{A},\sigma_{A})$ such that the containment $\simp{A} \subset \inert{A}$\si{inert automorphism} is proper.\\

Expanding on FOG, we have the following more general problem:

\begin{prob}[Index Problem]
Given a shift of finite type $(X_{A},\sigma_{A})$, determine the index of the following subgroup containments:
\begin{enumerate}
\item\label{item:partone}
$\simp{A} \subset \inert{A}$\si{inert automorphism}.
\item\label{item:parttwo}
$\textnormal{Fin}(\inert{A}) \subset \inert{A}$\si{inert automorphism}.
\end{enumerate}
In particular, in each case, must the index be finite?
\end{prob}

When $\aut(G_{A})$ is torsion-free, every element of finite order in $\aut(\sigma_{A})$ lies in $\inert{A}$\si{inert automorphism}. In this case, the FOG problem is equivalent to determining whether the answer to Part (\ref{item:parttwo}) of the Index Problem is one.


In general, it is not known whether, for each part of the Index Problem, the index is finite or infinite. As noted earlier, in \cite{KRWForumI} an example is given of a mixing shift of finite type $(X_{A},\sigma_{A})$ for which the index of $\textnormal{Fin}(\inert{A})$\si{inert automorphism} in $\inert{A}$ is strictly greater than one. This relies on being able to construct an inert automorphism\si{inert automorphism} in $\aut(\sigma_{A})$ which can not be a product of finite order automorphisms; this is carried out using the difficult constructions of Kim-Roush\ai{Roush, F.W.}-Wagoner\ai{Kim, K.H.}\ai{Wagoner, J.B.}
in \cite{KRWForumI,KRWForumII}, in which the polynomial matrix methods (introduced in Lecture 3) play an invaluable role (we do not know how to do such constructions without the polynomial matrix framework).\\

However, whether FOG or even SFOG might hold in the case of a full shift $(X_{n},\sigma_{n})$ is still unknown.\\
\indent Finite order generation\si{Finite Order Generation Problem} of the inerts\si{inert automorphism} for general mixing shifts of finite type is known to hold in the ``eventual'' setting; see \apr{apprem:evensettingfog}).\\

Williams\ai{Williams, R.F.} in \cite{Williams73} asked whether any involution of a pair of fixed points of a shift of finite type can be extended to an automorphism of the whole shift of finite type. More generally, this grew into the following problem (stated in \cite[Question 7.1]{BLR88}) about lifting actions on a finite collection of periodic points of the shift:

\begin{prob}[General Lifting Problem (LIFT)]\label{prob:liftproblem1}
Given a shift of finite type $(X_{A},\sigma_{A})$ and an automorphism $\phi$ of a finite subsystem $F$ of $(X_{A},\sigma_{A})$, does there exist $\tilde{\phi} \in \aut(\sigma_{A})$ such that $\tilde{\phi}|_{F} = \phi$?
\end{prob}

The answer to LIFT is also `not always': Kim\ai{Kim, K.H.} and Roush\ai{Roush, F.W.} showed in \cite{S23}, based on an example of Fiebig\ai{Fiebig, Ulf},
that there exists an automorphism of the set of periodic six points in the full 2-shift which does extend to an automorphism of the full 2-shift.\\

Roughly speaking, the LIFT problem involves two parts: determining the action of inert automorphisms\si{inert automorphism} on finite subsystems, and determining the range of the dimension representation. The first part has been resolved by Kim-Roush\ai{Roush, F.W.}-Wagoner\ai{Kim, K.H.}\ai{Wagoner, J.B.}
in \cite{KRWForumI,KRWForumII}; see Section \ref{subsec:actionsonfinitesubs}. The second part, to determine the range of the dimension representation, is still open in general (this was also stated in Problem \ref{prob:dimreprange} in Section \ref{subsec:dimrep}):

\begin{prob}\label{prob:dimreprange2}
Given a mixing shift of finite type $(X_{A},\sigma_{A})$, what is the image of the dimension representation $\pi_{A} \colon \aut(\sigma_{A}) \to \aut(G_{A},G_{A}^{+},\delta_{A})$? Is the image always finitely generated?
\end{prob}

In \cite{S19}, Kim\ai{Kim, K.H.} and Roush\ai{Roush, F.W.}
constructed  a mixing shift of finite type for which the dimension representation is not surjective.\\

In \cite[Example 6.9]{BLR88}, an example of a primitive matrix $A$ such that $\aut(G_{A},G_{A}^{+},\delta_{A})$ is not finitely generated is given. This does not resolve the second part of Problem \ref{prob:dimreprange2} though, since the range of the dimension representation is not known.\\

Another question concerns the isomorphism type of the groups $\aut(\sigma_{A})$. It is straightforward to check that conjugate shifts of finite type have isomorphic automorphism groups\si{automorphism group}, and that $\aut(\sigma_{A}) = \aut(\sigma_{A}^{-1})$ always holds (note there exists shifts of finite type $(X_{A},\sigma_{A})$ which are not conjugate to their inverse; see for example Proposition \ref{transposeexerciseproof} in Lecture 2). In \cite[Question 4.1]{BLR88} the following was asked:

\begin{prob}[Aut-Isomorphism Problem]\label{prob:autoisoproblem}
If $\aut(\sigma_{A})$ and $\aut(\sigma_{B})$ are isomorphic, must $(X_{A},\sigma_{A})$ be conjugate to either $(X_{B},\sigma_{B})$ or $(X_{B},\sigma_{B}^{-1})$?
\end{prob}

A particular case of this which has been of interest is:

\begin{prob}[Full Shift Aut-Isomorphism Problem]\label{prob:fullshiftautisoproblem}
For which $m,n$ are the groups $\aut(\sigma_{m})$ and $\aut(\sigma_{n})$ isomorphic?
\end{prob}

See Section \ref{subsec:stabilizedgroup} for some results related to Problem \ref{prob:fullshiftautisoproblem}.

\subsection{The stabilized automorphism group\si{stabilized automorphism group}}\label{subsec:stabilizedgroup}
Recently a new approach to the Aut-Isomorphism Problem, and the study of $\aut(\sigma_{A})$ in general, has been undertaken in \cite{HKS2022}. The idea is to consider a certain stabilization of the automorphism group\si{automorphism group}, using the observation that for all $k,m \ge 1$, $\aut(\sigma_{A}^{k})$ is naturally a subgroup of $\aut(\sigma_{A}^{km})$. Define the {\it stabilized automorphism group}\si{stabilized automorphism group} of $(X_{A},\sigma_{A})$ to be
$$\autinf{A} = \bigcup_{k=1}^{\infty}\aut(\sigma_{A}^{k})$$
where the union is taken in the group of all homeomorphisms of $X_{A}$. This is again a countable group. Similar to the definition of $\autinf{A}$, one defines a stabilized group of automorphisms of the dimension group\si{dimension group} by
$$\aut^{(\infty)}(G_{A}) = \bigcup_{k=1}^{\infty} \aut(G_{A},G_{A}^{+},\delta_{A}^{k}).$$
The group $\aut^{(\infty)}(G_{A})$ is precisely the union of the centralizers of $\delta_{A}$ in the group $\aut(G_{A},G_{A}^{+})$ of all order-preserving group automorphisms of $G_{A}$. Recall for a group $G$ we let $G_{\textnormal{ab}}$ denote the abelianization of $G$. In \cite{HKS2022}, the following was proved.
\begin{thm}[\cite{HKS2022}]\label{thm:stabilizedtheorem}
Let $(X_{A},\sigma_{A})$ be a mixing shift of finite type. The dimension representation
$$\pi_{A} \colon \aut(\sigma_{A}) \to \aut(G_{A})$$
extends to a stabilized dimension representation
$$\pi_{A}^{(\infty)} \colon \autinf{A} \to \aut^{(\infty)}(G_{A})$$
and the composition
$$\autinf{A} \stackrel{\pi_{A}^{(\infty)}}\longrightarrow \aut^{(\infty)}(G_{A})\stackrel{\textnormal{ab}}\longrightarrow \aut^{(\infty)}(G_{A})_{\textnormal{ab}}$$
is isomorphic to the abelianization of the stabilized automorphism group\si{stabilized automorphism group} $\autinf{A}$. In particular, if $\aut^{(\infty)}(G_{A})$ is abelian, then the commutator of $\autinf{A}$ coincides with the subgroup of stabilized inert automorphisms\si{inert automorphism}
$$\inertinf{A} = \ker \pi_{A}^{(\infty)} = \bigcup_{k=1}^{\infty}\textnormal{Inert}(\sigma_{A}^{k}).$$
\end{thm}

For example, in the case of a full shift $A=(n)$, it follows
from Theorem \ref{thm:stabilizedtheorem} that $\autinf{n}_{ab}$ is isomorphic to $\mathbb{Z}^{\omega(n)}$, where $\omega(n)$ denotes the number of distinct prime divisors of $n$. As a corollary of this, if $\omega(m) \ne \omega(n)$, then $\autinf{m}$ and $\autinf{n}$ are not isomorphic.

\indent For a mixing shift of finite type, the classical automorphism group\si{automorphism group} $\aut(\sigma_{A})$ is always residually finite. It turns out that in the stabilized case, $\autinf{A}$ is never residually finite \cite[Prop. 4.3]{HKS2022}. In fact, in stark contrast, the following was proved in \cite{HKS2022}:
\begin{thm}[\cite{HKS2022}]
For any $n \ge 2$, the group of stabilized inert automorphisms\si{inert automorphism} $\inertinf{n}$ is simple.
\end{thm}

A significantly more general version of the above theorem was proved by Salo in \cite{SaloGateLattices}. A particular case, Corollary 1 of \cite{SaloGateLattices}, shows that for any (nontrivial) mixing shift of finite type, the group of stabilized inert automorphisms is simple.\\

Subsequent to \cite{HKS2022}, a complete classification, up to isomorphism, of the stabilized automorphism groups\si{stabilized automorphism group} of full shifts was given in \cite{SchmiedingPentropy}. Introduced there is a certain kind of entropy for groups\footnote{More precisely, it is defined for {\it leveled groups}, i.e. pairs $(G,g)$ where $g$ is a distinguished element in the group $G$.} called local $\mathcal{P}$ entropy. Local $\mathcal{P}$ entropy is defined with respect to a chosen class $\mathcal{P}$ of finite groups which is closed under isomorphism. As a rough idea of what local $\mathcal{P}$ entropy measures, fix such a class $\mathcal{P}$, consider some group $G$ with some distinguished element $g \in G$, and consider the conjugation map $C_{g} \colon G \to G$ given by $C_{g}(h) = g^{-1}hg$. One can try to measure the growth rate of the $C_{g}$-periodic point sets $\textrm{Fix}(C_{g^{n}})$, which are precisely the centralizers of $g^{n}$ in $G$; but these sets may be infinite. To proceed, instead one approximates these centralizer sets using groups belonging to the chosen class $\mathcal{P}$ (which are by definition finite), and then considers the doubly exponential\footnote{A related quantity is defined by considering just exponential growth; here we'll consider only the one using doubly exponential.} growth rate of such $\mathcal{P}$-approximations. This (when defined) leads to a nonnegative quantity $h_{\mathcal{P}}(G,g)$ called the local $\mathcal{P}$ entropy of the pair $(G,g)$. A key thing proved in \cite{SchmiedingPentropy} is that the local $\mathcal{P}$ entropy of a pair $(G,g)$ is an invariant of isomorphism of the pair: if there is an isomorphism of groups $G \stackrel{\cong}\to H$ taking $g \in G$ to $h \in H$, then assuming the local $\mathcal{P}$ entropies are defined, we have $h_{\mathcal{P}}(G,g) = h_{\mathcal{P}}(H,h)$\footnote{It is also proved in the same paper that if there is an injective homomorphism $G \to H$ taking $g$ to $h$, then $h_{\mathcal{P}}(G,g) \leq h_{\mathcal{P}}(H,h)$.}.

Using local $\mathcal{P}$ entropy, in \cite{SchmiedingPentropy} the following was proved.

\begin{thm}[\cite{SchmiedingPentropy}]
For a non-trivial mixing shift of finite type $(X_{A},\sigma_{A})$, each of the following hold:
\begin{enumerate}
\item
There exists a class $\mathcal{P}_{A}$ of finite groups such that the local $\mathcal{P}_{A}$ entropy of the pair $(\autinf{A},\sigma_{A})$ is given by $h_{\mathcal{P}_{A}}\left( \autinf{A},\sigma_{A} \right) = h_{top}(\sigma_{A}) = \log \lambda_{A}$.

\item
If $(X_{B},\sigma_{B})$ is any other shift of finite type such that the stabilized automorphism groups\si{stabilized automorphism group} $\autinf{A}$ and $\autinf{B}$ are isomorphic, then $\frac{\log \lambda_{A}}{\log \lambda_{B}}$ is rational.
\end{enumerate}
\end{thm}

As a consequence this gives, as mentioned earlier, a complete classification of the stabilized automorphism groups\si{stabilized automorphism group} of full shifts.

\begin{coro}[\cite{SchmiedingPentropy}]
Given natural numbers $m,n \ge 2$, the stabilized automorphism groups\si{stabilized automorphism group} $\autinf{m}$ and $\autinf{n}$ are isomorphic if and only if there exists natural numbers $k,j$ such that $m^{k} = n^{j}$.
\end{coro}

Finally, we make a few comments about the connection between the stabilized setting for automorphism groups\si{automorphism group} described above and algebraic K-theory. In fact, the idea of the groups $\autinf{A}$ is partly motivated by algebraic K-theory, where the technique of stabilization proves to be fundamental. Recall as outlined in Lecture 5, as a starting point for algebraic K-theory, given a ring $\calR$, one can consider the stabilized general linear group
$$GL(\calR) = \varinjlim GL_{n}(\calR)$$
where $GL_{n}(\calR) \hookrightarrow GL_{n+1}(\calR)$ via $A \mapsto \begin{pmatrix} A & 0 \\ 0 & I \end{pmatrix}$. Inside each $GL_{n}(\calR)$ lies the subgroup $El_{n}(\calR)$ generated by elementary matrices, and one likewise defines the stabilized group of elementary matrices by
$$El(\calR) = \varinjlim El_{n}(\calR).$$
Whitehead showed (see Lecture 5) that, upon stabilizing, the explicitly defined subgroup $El(\calR)$ coincides with the commutator of $GL(\calR)$. From this viewpoint, one may interpret Theorem \ref{thm:stabilizedtheorem} as a Whitehead-type result for shifts of finite type. In particular, in the case of a full shift $(X_{n},\sigma_{n})$ (or more generally a shift of finite type $(X_{A},\sigma_{A})$ where $\aut^{(\infty)}(G_{A})$ is abelian), after stabilizing, the commutator subgroup of $\autinf{A}$ coincides with the subgroup $\inertinf{A}$\footnote{In fact, something stronger is true: the commutator of $\autinf{A}$ coincides with the stabilized group of simple automorphisms\si{simple automorphism}; see
\cite{HKS2022}.}.

\subsection{Mapping class groups\si{mapping class group of a subshift} of subshifts}
Recall from Section \ref{subsec:backgroundstuff} that two homeomorphisms are flow equivalent\si{flow equivalence} if there is a homeomorphism of their mapping tori which takes orbits to orbits and preserve the direction of the suspension flow. For a subshift $(X,\sigma)$, an analog of the automorphism group\si{automorphism group} in the setting of flow equivalence\si{flow equivalence} is given by the mapping class group\si{mapping class group of a subshift} $\mathcal{M}(\sigma)$, which is defined to be the group of isotopy classes of self-flow equivalences of the subshift $(X,\sigma)$.\\
\indent In \cite{BCmcg} a study of the mapping class group\si{mapping class group of a subshift} for shifts of finite type was undertaken. There it was shown that, for a nontrivial irreducible shift of finite type $(X_{A},\sigma_{A})$, the mapping class group\si{mapping class group of a subshift} $\mathcal{M}(\sigma_{A})$ is not residually finite. While the periodic point representations\si{periodic point representation} do not exist for $\mathcal{M}(\sigma_{A})$, a vestige of the dimension representation\si{dimension representation} survives in the form of the
Bowen-Franks\ai{Bowen, Rufus}
representation\si{Bowen-Franks group} of $\mathcal{M}(\sigma_{A})$. It was also shown that $\aut(\sigma_{A}) / \langle \sigma_{A} \rangle$ embeds into $\mathcal{M}(\sigma_{A})$, and there is an analog of block codes, known as flow codes.
In \cite{SaloMCG}, it was shown that Thompson's group $V$ embeds into the mapping class group\si{mapping class group of a subshift} of a particular shift of finite type.\\
\indent See also \cite{SYmcg} for a study of the mapping class group\si{mapping class group of a subshift} in the context of minimal subshifts.

\subsection{Appendix 7}
This appendix contains some proofs, remarks, and solutions of various exercises through Lecture 7.\\

\begin{rem}\label{apprem:lowcomplexity}
Recall from Section \ref{subsec:ctsshiftcommmaps} that for a subshift $(X,\sigma)$, we let $\mathcal{W}_{n}(X)$ denote the set of $X$-words of length $n$. We define the complexity function (of $X$) $P_{X} \colon \mathbb{N} \to \mathbb{N}$ by $P_{X}(n) = |\mathcal{W}_{n}(X)|$. Thus $P_{X}(n)$ simply counts the number of $X$-words of length $n$. For a shift of finite type $(Y,\sigma)$ with positive entropy, the function $P_{Y}(n)$ grows exponentially in $n$; for example, for the full shift $(X_{m},\sigma_{m})$ on $m$ symbols, $P_{X_{m}}(n) = m^{n}$. For a subshift $(X_{\alpha},\sigma_{\alpha})$ of the form given in Example \ref{ex:sturmiansubshift}, the complexity satisfies $P_{X_{\alpha}}(n) = n+1$ (such subshifts are called Sturmian subshifts). This is the slowest possible growth of complexity function for an infinite subshift: a theorem of Morse and Hedlund \cite{MH1938} from 1938 shows that for an infinite subshift $(X,\sigma)$, we must have $P_{X}(n) \ge n+1$.\\
\indent There has been a great deal of interest in studying the automorphism groups\si{automorphism group} of subshifts with slow-growing complexity functions. Numerous results show that such low complexity subshifts often have much more tame automorphism groups\si{automorphism group}, in comparison to subshifts possessing complexity functions of exponential growth (e.g. shifts of finite type). We won't attempt to survey these results, but refer the reader to \cite{DDMP2016,HP1989,Coven1971,Olli2013,SaloTorma2015,CyrKra2016,CyrKra2015,CyrKra2016strexp}.
\end{rem}

\begin{exer}\label{appex:denppresfin}
If $(X,\sigma)$ is a subshift whose periodic points are dense in $X$, then $\aut(\sigma)$ is residually finite.
\end{exer}
\begin{proof}
Given $n \in \mathbb{N}$, let $P_{n}(X)$ denote the set of points of least period $n$ in $X$. Since $X$ is a subshift, $|P_{n}(X)|<\infty$ for every $n$. If $\alpha \in \aut(\sigma)$, then since $\alpha$ commutes with $\sigma$, for any $n$ the set $P_{n}(X)$ is invariant under $\alpha$. It follows there are homomorphisms
\begin{equation*}
\begin{gathered}
\rho_{n} \colon \aut(\sigma) \to \sym(P_{n}(X))\\
\rho_{n} \colon \alpha \mapsto \alpha|_{P_{n}(X)}
\end{gathered}
\end{equation*}
where $\sym(P_{n}(X))$ denotes the group of permutations of the set $P_{n}(X)$. Now suppose $\alpha \in \aut(\sigma)$ and $\rho_{n}(\alpha)=\textnormal{id}$ for all $n$. Then $\alpha$ fixes every periodic point in $X$; since the periodic points are dense in $X$ (by assumption) and $\alpha$ is a homeomorphism, $\alpha$ must be the identity. This shows $\aut(\sigma)$ is residually finite.
\end{proof}

\begin{rem}\label{appnasu}
  Beyond introducing simple automorphisms\si{simple automorphism},
in his memoir \cite{MR1234883}
    Nasu\ai{Nasu, Masakazu} introduced
     the
     powerful machinery of ``textile systems''
     for studying automorphisms and endomorphisms of shifts of finite type;
     he continued to apply and develop this theory
 in subsequent works     (e.g.
     \cite{MR1926865, MR2177043, MR2380306}).
  See \cite[Appendices B,C]{Bopen} for a quick introduction to this theory.
\end{rem}

\begin{rem}\label{apprem:amenradical}
Any discrete group $G$ possesses a maximal normal amenable subgroup $\textnormal{Rad}(G)$ known as the {\it amenable radical} of $G$. By Ryan's Theorem\si{Ryan's Theorem}, the center of $\aut(\sigma_{A})$ is the subgroup generated by $\sigma_{A}$, and hence is contained in $\textnormal{Rad}(G)$. In \cite{FST19} it was shown by Frisch, Schlank and Tamuz that, in the case of a full shift, $\textnormal{Rad}(\aut(\sigma_{n}))$ is precisely the center of $\aut(\sigma_{n})$, i.e. the subgroup generated by $\sigma_{n}$. In \cite{Yang2021} Yang extended this result, proving that for any irreducible shift of finite type $(X_{A},\sigma_{A})$, $\textnormal{Rad}(\aut(\sigma_{A}))$ also coincides with the center of $\aut(\sigma_{A})$ (in fact, Yang also proves the result for any irreducible sofic shift as well).
\end{rem}

\begin{exer}\label{appexer:ryanthmdist}
\begin{enumerate}
\item
Show that if $(X_{A},\sigma_{A})$ and $(X_{B},\sigma_{B})$ are irreducible shifts of finite type such that $\aut(\sigma_{A})$ and $\aut(\sigma_{B})$ are isomorphic, then $\textnormal{root}(\sigma_{A}) = \textnormal{root}(\sigma_{B})$.
\item
Let $(X_{2},\sigma_{2}), (X_{4},\sigma_{4})$ denote the full shift on 2 symbols and on 4 symbols, respectively. Show that $\textnormal{root}(\sigma_{2}) \ne \textnormal{root}(\sigma_{4})$. Use part (1) to conclude that $\aut(\sigma_{2})$ and $\aut(\sigma_{4})$ are not isomorphic as groups.
\end{enumerate}
\end{exer}
\begin{proof}
  For part (1), suppose $\Psi \colon \aut(\sigma_{A}) \to \aut(\sigma_{B})$ is an isomorphism and $k \in \textnormal{root}(\sigma_{A})$. By Ryan's Theorem\si{Ryan's Theorem}, $\Psi(\sigma_{A}) = \sigma_{B}$ or $\Psi(\sigma_{A})=\sigma_{B}^{-1}$. Choose $\alpha \in \aut(\sigma_{A})$ such that $\alpha^{k} = \sigma_{A}$. If $\Psi(\sigma_{A}) = \sigma_{B}$, then we have
  $(\Psi(\alpha))^{k} = \Psi(\alpha^{k}) = \Psi(\sigma_{A}) = \sigma_{B}$, so $k \in \textnormal{root}(\sigma_{B})$. If $\Psi(\sigma_{A}) = \sigma_{B}^{-1}$, then we have $(\Psi(\alpha^{-1}))^{k} = \Psi(\alpha^{-k}) = \Psi(\sigma_{A}^{-1}) = \sigma_{B}$ so again $k \in \textnormal{root}(\sigma_{B})$. Thus $\textnormal{root}(\sigma_{A}) \subset \textnormal{root}(\sigma_{B})$. The proof that $\textnormal{root}(\sigma_{B}) \subset \textnormal{root}(\sigma_{A})$ is analogous.\\
\indent For part (2), choose a topological conjugacy $F \colon (X_{4},\sigma_{4}) \to (X_{2},\sigma_{2}^{2})$. If we let $s = F^{-1}\sigma_{2}F \in \aut(\sigma_{4})$, then $s \in \aut(\sigma_{4})$ and $s^{2} = \sigma_{4}$, so $2 \in \textnormal{root}(\sigma_{4})$. We claim $2 \not \in \textnormal{root}(\sigma_{2})$. To see this, suppose toward a contradiction that $\beta \in \aut(\sigma_{2})$ satisfies $\beta^{2} = \sigma_{2}$. There are precisely two points $x,y$ of least period 2 in $(X_{2},\sigma_{2})$, so $\beta^{2}$ must act by the identity on the points $x,y$. But $\sigma_{2}(x)=y$, a contradiction.
\end{proof}

\begin{exer}\label{appexer:dimgroupfullshiftcomp}
When $A = (n)$ (the case of the full-shift on $n$ symbols), the triple $(G_{n},G_{n}^{+},\delta_{n})$ is isomorphic to the triple $(\mathbb{Z}[\frac{1}{n}],\mathbb{Z}_{+}[\frac{1}{n}],m_{n})$, where $m_{n}$ is the automorphism of $\mathbb{Z}[\frac{1}{n}]$ defined by $m_{n}(x) = x \cdot n$.
\end{exer}
\begin{proof}
The eventual range of $A$ is $\mathbb{Q}$. Given $\frac{p}{q} \in \mathbb{Q}$, $2^{k}\frac{p}{q} \in \mathbb{Z}_{+}$ if and only if $p \in \mathbb{Z}_{+}$ and $q$ is a power of $2$.
\end{proof}

\begin{exer}\label{appexer:indextwoorderautos}
Let $A$ be a primitive matrix and suppose $\Psi$ is an automorphism of $(G_{A},\delta_{A})$. By considering $G_{A}$ as a subgroup of $ER(A)$, show that $\Psi$ extends to a linear automorphism $\tilde{\Psi} \colon ER(A) \to ER(A)$ which multiplies a Perron eigenvector of $A$ by some nonzero real number $\lambda_{\Psi}$. Show that $\Psi$ is also an automorphism of the ordered abelian group $(G_{A},G_{A}^{+},\delta_{A})$ if and only if $\lambda_{\Psi}$ is positive.
\end{exer}
\begin{proof}
That $\Psi$ extends to a linear automorphism $\tilde{\Psi}$ of $ER(A)$ is immediate: given $v \in ER(A)$, write $v = \frac{1}{q}w$ where $w$ is integral, and define $\tilde{\Psi}(v) = \frac{1}{q}\Psi(w)$. The linear map $\Psi$ commutes with $\delta_{A}$ on $G_{A}$, so $\tilde{\Psi}$ commutes with $\delta_{A}$ as a linear automorphism of $ER(A)$. Since $A$ is primitive, a Perron eigenvector $v_{\lambda_{A}}$ for $\lambda_{A}$ spans a one-dimensional eigenspace for $\delta_{A}$, which hence must be preserved by $\tilde{\Psi}$. Thus $v_{\lambda_{A}}$ is also an eigenvector for $\tilde{\Psi}$, and has some corresponding eigenvalue $\lambda_{\Psi}$.\\
\indent For the second part, we'll use the following proposition (a proof of which we include at the end).

\begin{prop} \label{mixingdimgroupprop}
Suppose $A$ is an $N\times N$  primitive matrix over $\R$.
Let the spectral radius be $\lambda$ and let $v$ be a positive eigenvector,
$vA = \lambda v$. Given $x$ in $\R^N$, let
$c_x$ be the real number such that $x= c_xv + u_x$, with
$u_x$ a vector in the  $A$-invariant subspace complementary to $<v>$. Suppose $x$ is not
the zero vector.
Then $xA^n$ is nonnegative for large $n$ iff $c_x>0$.
\end{prop}

To finish the exercise, suppose $0 \ne w \in G_{A}^{+}$, and write $w = c_{w}v_{\lambda_{A}}+u_{w}$ as in the proposition. Since $w \in G_{A}^{+}$, $c_{w} > 0$. Then $\tilde{\Psi}(w) = c_{w}\lambda_{\Psi}v_{\lambda_{A}} + \tilde{\Psi}(u_{w})$. Since $\lambda_{\Psi} > 0$, $c_{w}\lambda_{\Psi} > 0$, so the proposition implies $\Psi(w) \in G_{A}^{+}$ as desired.\\

{\it Proof of Proposition \ref{mixingdimgroupprop}}. The Perron Theorem tells us the positive eigenvector and complementary
invariant subspace exist, with $\varlimsup_n ||u_xA^n||^{1/n} < \lambda $.
Consequently, for large $n$, $xA^n$ is a positive vector if $c_x>0$ and
$xA^n$ is a negative vector if $c_x < 0$. Given $c_x=0 $ and $x\neq 0$,
no vector $w=u_xA^n $ can be nonnegative or nonpositive, because
this would imply
$\lim_n ||xA^n||^{1/n}= \lim_n ||wA^n||^{1/n} = \lambda $, a contradiction.
\end{proof}

\begin{rem}\label{rem:kriegerconstruction}
  Consider an edge shift of finite type $(X_{A},\sigma_{A})$. Here is an outline of
  Krieger's\ai{Krieger, Wolfgang}
  construction of an ordered abelian group which is isomorphic to $(G_{A},G_{A}^{+},\delta_{A})$; our presentation follows the one given in \cite[Sec. 7.5]{LindMarcus2021}.\ai{Lind, Douglas}\ai{Marcus, Brian}
  Recall we are assuming that $A$ is a $k \times k$ irreducible matrix.\\
\indent By an {\it $m$-ray} we mean a subset of $X_{A}$ given by
$$R(x,m) = \{y \in X_{A} \mid y_{(-\infty,m]}=x_{(-\infty,m]}\}$$
for some $x \in X_{A}, m \in \mathbb{Z}$. An {\it $m$-beam} is a (possibly empty) finite union of $m$-rays. By a {\it ray} we mean an $m$-ray for some $m \in \mathbb{Z}$; likewise, by a {\it beam} we mean an $m$-beam for some $m$. It is easy to check that if $U$ is an $m$-beam for some $m$, and $n \ge m$, then $U$ is also an $n$-beam. Given an $m$-beam
$$U = \bigcup_{i=1}^{j}R(x^{(i)},m),$$
define $v_{U,m} \in \mathbb{Z}^{k}$ to be the vector whose $J$th component is given by
$$\#\{x^{(i)} \in U \mid \textnormal{ the edge corresponding to }x_{m}^{(i)} \textnormal{ ends at state } J\}.$$
We define two beams $U$ and $V$ to be equivalent if there exists $m$ such that $v_{U,m} = v_{V,m}$, and let $[U]$ denote the equivalence class of a beam $U$. We will make the collection of equivalence classes of beams into a semi-group as follows. Since $A$ is an irreducible matrix and $0 < h_{top}(\sigma_{A}) = \log \lambda_{A}$, given two beams $U,V$, we may find beams $U^{\prime}, V^{\prime}$ such that
$$[U]=[U^{\prime}], \hspace{.23in} [V] = [V^{\prime}], \hspace{.23in} U^{\prime} \cap V^{\prime} = \emptyset,$$
and we let $D_{A}^{+}$ denote the abelian monoid defined by the operation
$$[U] + [V] = [U^{\prime} \cup V^{\prime}]$$
where the class of the empty set serves as the identity for $D_{A}^{+}$. Now let $D_{A}$ denote the group completion of $D_{A}^{+}$; thus elements of $D_{A}$ are formal differences $[U]-[V]$. Then $D_{A}$ is an ordered abelian group with positive cone $D_{A}^{+}$. The map $d_{A} \colon D_{A} \to D_{A}$ induced by
$$d_{A}([U]) = [\sigma_{A}(U)]$$
is a group automorphism of $D_{A}$ which preserves $D_{A}^{+}$, and the triple $(D_{A},D_{A}^{+},d_{A})$ is Krieger's dimension triple for the SFT $(X_{A},\sigma_{A})$.\\

The connection between Krieger's\ai{Krieger, Wolfgang}
triple $(D_{A},D_{A}^{+},d_{A})$ and the ordered abelian group triple $(G_{A},G_{A}^{+},\delta_{A})$ is given by the following proposition.

\begin{prop}[\cite{LindMarcus2021},\ai{Lind, Douglas}\ai{Marcus, Brian}
    Theorem 7.5.3]\label{prop:kridimgroupiso}
There is a semi-group homomorphism $\theta \colon D_{A}^{+} \to G_{A}^{+}$ induced by the map
$$\theta([U]) = \delta_{A}^{-k-n}(v_{U,n}A^{k}), \hspace{.29in} U \textnormal{ an } n\textnormal{-beam}.$$
The map $\theta$ satisfies $\theta(D_{A}^{+}) = G_{A}^{+}$, and induces an isomorphism $\theta \colon D_{A} \to G_{A}$ such that $\theta \circ d_{A} = \delta_{A} \circ \theta$. Thus $\theta$ induces an isomorphism of triples
$$\theta \colon (D_{A},D_{A}^{+},d_{A}) \to (G_{A},G_{A}^{+},\delta_{A}).$$
\end{prop}
\end{rem}

\begin{rem}\label{apprem:dimreprangeclassification}
  In \cite{S17}, Kim\ai{Kim, K.H.} and Roush\ai{Roush, F.W.} describe how the problem of classifying general (i.e. not necessarily irreducible) shifts of finite type up to topological conjugacy can be broken into two parts: classifying mixing shifts of finite type up to conjugacy, and determining the range of the dimension representation\si{dimension representation} in the mixing shift of finite type case. That the dimension representation\si{dimension representation} need not always be surjective was also instrumental in the Kim-Roush\ai{Roush, F.W.}\ai{Kim, K.H.}
  argument in \cite{S21} that shift equivalence\si{shift equivalence} over $\mathbb{Z}_{+}$ need not imply strong shift equivalence\si{strong shift equivalence} over $\mathbb{Z}_{+}$ in the reducible setting.
\end{rem}

\begin{exer}\label{appexer:simpininert}
For any shift of finite type $(X_{A},\sigma_{A})$, we have $\simp{A} \subset \inert{A}$.
\end{exer}
\begin{proof}
This is easiest seen using Krieger's\ai{Krieger, Wolfgang} presentation \apr{rem:kriegerconstruction}). First suppose $\alpha \in \simp{A}$ is induced by a simple graph symmetry of $\Gamma_{A}$. If $U$ is an $m$-beam in $X_{A}$, then $\alpha(U)$ is an $m$-beam, and $v_{\alpha(U),m} = v_{U,m}$. It follows that $[U] = [\alpha(U)]$, so $\alpha$ acts by the identity on the group $D_{A}$, and hence on $G_{A}$.\\
\indent Now suppose $\beta = \Psi^{-1}\alpha\Psi$ where $\Psi \colon (X_{A},\sigma_{A}) \to (X_{B},\sigma_{B})$ is a topological conjugacy and $\alpha \in \simp{B}$ is induced by a simple graph symmetry of $\Gamma_{B}$. If $U$ is an $m$-beam in $X_{A}$, then by the previous part $\alpha \Psi([U]) = \alpha ([\Psi(U)]) = [\Psi(U)] = \Psi([U])$, so
$$\beta([U]) = \Psi^{-1}\alpha\Psi([U]) =\Psi^{-1}\Psi([U]) = [U].$$
Thus $\beta$ acts by the identity on $G_{A}$. Since $\simp{A}$ is generated by automorphisms in the form of $\beta$, this finishes the proof.
\end{proof}

\begin{rem}\label{apprem:semidirectaut}
Let us write $\aut(P_{k},\sigma_{A})$ for $\aut(\sigma_{A}|_{P_{k}})$. For each orbit $q \in Q_{k}$ choose a point $x_{q} \in q$. There is a surjective homomorphism
$$\aut(P_{k},\sigma_{A}) \stackrel{\pi}\longrightarrow \sym(Q_{k})$$
since any $\alpha \in \aut(P_{k},\sigma_{A})$ must preserve $\sigma_{A}$-orbits, and the map $\pi$ is split by the map $i \colon \sym(Q_{k}) \to \aut(P_{k},\sigma_{A})$ defined by, for $\tau \in \sym(Q_{k})$, setting
$$i(\tau)(\sigma_{A}^{i}(x_{q})) = \sigma_{A}^{i}x_{\tau(q)}, \qquad 0 \le i \le k-1.$$
The kernel of $\pi$ is isomorphic to $\left(\mathbb{Z}/k \mathbb{Z}\right)^{Q_{k}}$ with an isomorphism given by
\begin{equation*}
\begin{gathered}
\left(\mathbb{Z}/k \mathbb{Z}\right)^{Q_{k}} \to \ker \pi\\
g \mapsto \alpha_{g}, \qquad \alpha_{g}(\sigma_{A}^{i}x_{q}) = \sigma_{A}^{i+g(q)}x_{q}, \qquad 0 \le i \le k-1
\end{gathered}
\end{equation*}
and it follows $\aut(P_{k},\sigma_{A})$ is isomorphic to the semidirect product $\left(\mathbb{Z}/k\mathbb{Z}\right)^{Q_{k}} \rtimes \sym(Q_{k})$. The action of $\sym(Q_{k})$ on $\left(\mathbb{Z}/k\mathbb{Z}\right)^{Q_{k}}$ is determined as follows. Let $g \in \left(\mathbb{Z}/k\mathbb{Z}\right)^{Q_{k}}$, so $g \colon Q_{k} \to \mathbb{Z}/k\mathbb{Z}$. Then $\alpha_{g} \in \ker \pi$, and given some $i(\tau)$ for some $\tau \in \sym(Q_{k})$,
$$i(\tau)^{-1}\alpha_{g}i(\tau) = \alpha_{g \circ \tau}.$$
\end{rem}

\begin{rem}\label{apprem:abelsigngyr}
  For a group $G$, let $G_{ab}$ denote the abelianization. Using the notation from \ref{apprem:semidirectaut}, we have an isomorphism $\Phi \colon \aut(\sigma_{A}|_{P_{k}}) \to \left(\mathbb{Z}/k\mathbb{Z}\right)^{Q_{k}} \rtimes \sym(Q_{k})$. The abelianization of $\left(\mathbb{Z}/k\mathbb{Z}\right)^{Q_{k}} \rtimes \sym(Q_{k})$ is isomorphic to $\sym(Q_{k})_{ab} \times (\left(\mathbb{Z}/k\mathbb{Z}\right)^{Q_{k}})_{\sym(Q_{k})}$,
  where $(\left(\mathbb{Z}/k\mathbb{Z}\right)^{Q_{k}})_{\sym(Q_{k})}$ is the quotient of $\left(\mathbb{Z}/k\mathbb{Z}\right)^{Q_{k}}$ by the subgroup generated by elements of the form $\tau^{-1}g\tau - g$, $\tau \in \sym(Q_{k}), g \in \left(\mathbb{Z}/k\mathbb{Z}\right)^{Q_{k}}_{ab}$.

  Now the abelianization of $\sym(Q_{k})$ is given by $\textnormal{sign} \colon \sym(Q_{k}) \to \mathbb{Z}/2$, and the map
\begin{equation*}
\begin{gathered}
\left(\mathbb{Z}/k\mathbb{Z}\right)^{Q_{k}} \to \mathbb{Z}/k\mathbb{Z}\\
g \mapsto \sum_{q \in Q_{k}}g(q)
\end{gathered}
\end{equation*}
maps elements of the form $\tau^{-1}g\tau-g$ to 0, and induces an isomorphism
$$(\left(\mathbb{Z}/k\mathbb{Z}\right)^{Q_{k}})_{\sym(Q_{k})} \stackrel{\cong}\longrightarrow \mathbb{Z}/k\mathbb{Z}.$$
\end{rem}

\begin{rem}\label{apprem:sgccremark}
  SGCC, and the question of which automorphisms satisfy SGCC\si{sign-gyration compatibility condition}, has a history spanning a number of years. The SGCC condition\si{sign-gyration compatibility condition} was introduced by Boyle and Krieger\ai{Krieger, Wolfgang}
  in \cite{MR887501}, where it was also proved that, in the case of many SFT's, it holds for any inert automorphism\si{inert automorphism} which is a product of involutions. This was followed up by a number of more general results, summarized in the following theorem.
\begin{thm}\label{thm:sgcc}
Let $(X_{A},\sigma_{A})$ be a shift of finite type. An automorphism $\alpha \in \aut(\sigma_{A})$ satisfies SGCC\si{sign-gyration compatibility condition} if:
\begin{enumerate}
\item
(Boyle-Krieger\ai{Krieger, Wolfgang} in \cite{MR887501}) $\alpha$ is inert\si{inert automorphism} and a product of involutions (not for all SFT's, but many, including the full shifts).
\item
(Nasu\ai{Nasu, Masakazu} in \cite{Nasu88}) $\alpha$ is a simple automorphism\si{simple automorphism}.
\item
(Fiebig\ai{Fiebig, Ulf} in \cite{FiebigThesis}) $\alpha$ is inert\si{inert automorphism} and finite order.
\item
  (Kim-Roush\ai{Roush, F.W.} in\ai{Kim, K.H.} \cite{S23}, with a key ingredient by
  Wagoner\ai{Wagoner, J.B.}) $\alpha$ is inert\si{inert automorphism}.
\end{enumerate}
\end{thm}
\end{rem}

\begin{rem}\label{apprem:actionsfinitesubsrem}
Williams\ai{Williams, R.F.} first asked (around 1975) whether any permutation of fixed points of a shift of finite type could be lifted to an automorphism. Williams\ai{Williams, R.F.} was motivated in part by the classification problem\si{classification problem}: he was studying an example of two shifts of finite type which were shift equivalent, one of which clearly had an involution of fixed points, while it was not obvious whether the other did. It is interesting to note that, many years later, the automorphism groups\si{automorphism group} proved instrumental in addressing the classification problem\si{classification problem}.
\end{rem}

\begin{rem}\label{apprem:evensettingfog}
In \cite{Wagoner90eventual} Wagoner\ai{Wagoner, J.B.} proved that the inert automorphisms\si{inert automorphism} are generated by simple automorphisms\si{simple automorphism} in the ``eventual'' setting: namely, given a primitive matrix $A$ and inert automorphism\si{inert automorphism} $\alpha \in \inert{A}$, there exists some $m \ge 1$ such that, upon considering $\alpha \in \textnormal{Aut}(\sigma_{A}^{m})$, $\alpha$ lies in $\textnormal{Simp}(\sigma_{A}^{m})$. In \cite{MR964880} Boyle gave an alternative proof of this, and also gave a stronger form of the result.
\end{rem}

\section{Wagoner's strong shift equivalence\si{strong shift equivalence} complex\si{Wagoner's strong shift equivalence complex}, and applications}\ai{Wagoner, J.B.}
\label{sectionWagoner}
In the late 80's, Wagoner introduced certain CW complexes as a tool to study strong shift equivalence\si{strong shift equivalence}. These CW complexes provide an algebraic topological/combinatorial framework for studying strong shift equivalence\si{strong shift equivalence}, and have played a key role in a number of important results in the study of shifts of finite type. Among these, one of the most significant was the construction of a counterexample to Williams'\ai{Williams, R.F.} Conjecture\si{Williams' Conjecture} in the primitive case, which was found by Kim\ai{Kim, K.H.} and Roush\ai{Roush, F.W.} in \cite{S11}\footnote{Earlier counterexamples to Williams' Conjecture in the reducible case were found by Kim\ai{Kim, K.H.}
  and Roush\ai{Roush, F.W.} - see \cite{S21}.}. Wagoner independently developed another framework for finding counterexamples, and in \cite{Wagoner2000} gave a different proof, using matrices generated from Kim\ai{Kim, K.H.} and Roush\ai{Roush, F.W.}'s method in \cite{S11}, of the existence of a counterexample to Williams' Conjecture. Both the Kim\ai{Kim, K.H.} and Roush\ai{Roush, F.W.} strategy, and Wagoner's strategy, take place in the setting of Wagoner's strong shift equivalence\si{strong shift equivalence} complexes\si{Wagoner's strong shift equivalence complex}.\\
\indent The goal in this last lecture is to give a brief introduction to these complexes. After defining and discussing them, we'll give a short introduction into how the Kim-Roush\ai{Roush, F.W.} and Wagoner\ai{Wagoner, J.B.} strategies for producing counterexamples work. This will be very much an overview, and we will not go into details.\\
\indent In summary, our aim here is not to describe the construction of counterexamples to Williams'\ai{Williams, R.F.} Conjecture\si{Williams' Conjecture} in any detail, but instead to give an overview of how Wagoner's\ai{Wagoner, J.B.} spaces are built, how the counterexample strategies make use of them, and where they leave the state of the classification problem\si{classification problem}.\\

\subsection{Wagoner's SSE complexes}\ai{Wagoner, J.B.}
Suppose we have matrices $A, B$ over $\mathbb{Z}_{+}$, and a strong shift equivalence\si{strong shift equivalence} from $A$ to $B$
$$A = A_{0} \begin{tiny} \reallywidesim{\hspace{.02in} $R_{1},S_{1}$} \end{tiny} A_{1} \begin{tiny} \reallywidesim{\hspace{.02in} $R_{2},S_{2}$} \end{tiny} \hspace{.05in}\cdots \begin{tiny} \hspace{.03in} \reallywidesim{\hspace{.02in}$R_{n-1},S_{n-1}$} \hspace{.03in} \end{tiny} A_{n-1} \begin{tiny} \reallywidesim{\hspace{.02in} $R_{n},S_{n}$} \end{tiny} A_{n} = B$$
where for each $i \ge 1$, $A_{i-1} \begin{tiny} \reallywidesim{\hspace{.02in} $R_{i},S_{i}$} \end{tiny} A_{i}$ indicates an elementary strong shift equivalence\si{strong shift equivalence}
$$A_{i-1} = R_{i}S_{i}, \qquad A_{i} = S_{i}R_{i}.$$
We can visualize this as a path (at the moment we use the term path informally; it will be made precise later)


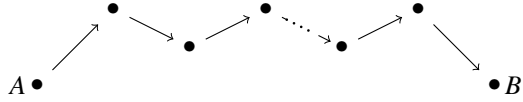
\begin{equation*}
\begin{tikzpicture}
\path
(-3,0) node (A) {$\bullet$}
(-3.25,0) node (Ap) {$A$}
(-2,1) node (a) {$\bullet$}
(-1,0.5) node (b) {$\bullet$}
(0,1) node (c) {$\bullet$}
(1,0.5) node (d) {$\bullet$}
(2,1) node (e) {$\bullet$}
(3,0) node (B) {$\bullet$}
(3.25,0) node (Bp) {$B$};
\path (c) to coordinate[pos=0.11] (aux-1) coordinate[pos=0.75] (aux-2) (d);
\draw[->,shorten >=-2pt, shorten <=1.5pt]   (c) -- (aux-1)
                (aux-2)-- (d);
                \draw[line width=1pt, line cap=round, dash pattern=on 0pt off 3\pgflinewidth]    (aux-1) to (aux-2);
\draw[->] (A)--(a);
\draw[->] (a)--(b);
\draw[->] (b)--(c);
\draw[->] (d)--(e);
\draw[->] (e)--(B);
\end{tikzpicture}
\end{equation*}

where each arrow in this picture represents an elementary strong shift equivalence\si{strong shift equivalence}. From Williams'\ai{Williams, R.F.} Theorem (Theorem \ref{rfwtheorem}), there is a conjugacy $C \colon (X_{A},\sigma_{A}) \to (X_{B},\sigma_{B})$ given by
$$C = \prod_{i=1}^{n}c(R_{i},S_{i})$$
where for each $i$, $c(R_{i},S_{i}) \colon (X_{A_{i-1}},\sigma_{A_{i-1}}) \to (X_{A_{i}},\sigma_{A_{i}})$ is a conjugacy induced by the ESSE $A_{i-1} \begin{tiny} \reallywidesim{\hspace{.02in} $R_{i},S_{i}$} \end{tiny} A_{i}$.

Now suppose, with the matrices $A, B$ over $\mathbb{Z}_{+}$, we have two SSE's from $A$ to $B$. We then have two paths of ESSE's from $A$ to $B$\\

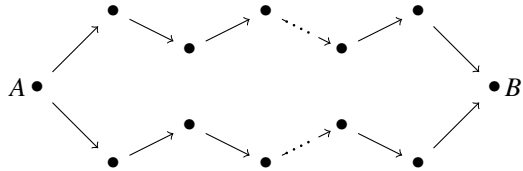
\begin{equation*}
\begin{tikzpicture}
\path
(-3,0) node (A) {$\bullet$}
(-3.25,0) node (Ap) {$A$}
(-2,1) node (a) {$\bullet$}
(-1,0.5) node (b) {$\bullet$}
(0,1) node (c) {$\bullet$}
(1,0.5) node (d) {$\bullet$}
(2,1) node (e) {$\bullet$}
(3,0) node (B) {$\bullet$}
(3.25,0) node (Bp) {$B$}
(-2,-1) node (la) {$\bullet$}
(-1,-0.5) node (lb) {$\bullet$}
(0,-1) node (lc) {$\bullet$}
(1,-0.5) node (ld) {$\bullet$}
(2,-1) node (le) {$\bullet$};
\path (c) to coordinate[pos=0.11] (aux-1) coordinate[pos=0.75] (aux-2) (d);
\draw[->,shorten >=-2pt, shorten <=1.5pt]   (c) -- (aux-1)
                (aux-2)-- (d);
                \draw[line width=1pt, line cap=round, dash pattern=on 0pt off 3\pgflinewidth]    (aux-1) to (aux-2);
\path (lc) to coordinate[pos=0.11] (aux-1) coordinate[pos=0.75] (aux-2) (ld);
\draw[->,shorten >=-2pt, shorten <=1.5pt]   (lc) -- (aux-1)
                (aux-2)-- (ld);
                \draw[line width=1pt, line cap=round, dash pattern=on 0pt off 3\pgflinewidth]    (aux-1) to (aux-2);
\draw[->] (A)--(a);
\draw[->] (a)--(b);
\draw[->] (b)--(c);
\draw[->] (d)--(e);
\draw[->] (e)--(B);
\draw[->] (A)--(la);
\draw[->] (la)--(lb);
\draw[->] (lb)--(lc);
\draw[->] (ld)--(le);
\draw[->] (le)--(B);
\end{tikzpicture}
\end{equation*}


and a pair of conjugacies corresponding to each path
$$C_{1} \colon (X_{A},\sigma_{A}) \to (X_{B},\sigma_{B})$$
$$C_{2} \colon (X_{A},\sigma_{A}) \to (X_{B},\sigma_{B})$$
and one may ask: when do two such paths induce the same conjugacy? Can we determine this from the matrix entries in the paths themselves? Alternatively, is there a space in which we can actually consider these as paths, in which two paths are homotopic if and only if they give rise to the same conjugacy? Wagoner's\ai{Wagoner, J.B.} complexes\si{Wagoner's strong shift equivalence complex} are a way to do this, and one of the key insights in Wagoner's complexes\si{Wagoner's strong shift equivalence complex} is determining the correct relations on matrices to accomplish this. These relations are known as the Triangle Identities. Since the Triangle Identities lead directly to the definition of Wagoner's\ai{Wagoner, J.B.}
Complexes\si{Wagoner's strong shift equivalence complex} \apr{apprem:wagonermarkovspace}), we'll define both simultaneously.

\begin{de}
Let $\calR$ be a semiring. We define a CW-complex $SSE(\calR)$\si{Wagoner's strong shift equivalence complex} as follows:
\begin{enumerate}
\item
The 0-cells of $SSE(\calR)$ are square matrices over $\calR$.
\item
An edge $(R,S)$ from vertex $A$ to vertex $B$ corresponds to an elementary strong shift equivalence\si{strong shift equivalence} over $\calR$ from $A$ to $B$:

\begin{equation*}
\begin{tikzpicture}
\path
(-.25,0) node (a) {$A$}
(2,0) node (b) {$B$}
(0,0) node (A) {$\bullet$}
(1.75,0) node (B) {$\bullet$};
\draw[->] (A) to node[auto] {$\scriptstyle (R,S)$} (B);
\end{tikzpicture}
\end{equation*}

where $A=RS, B = SR$.
\item
2-cells are given by triangles

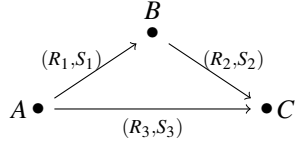
\begin{equation*}
\begin{tikzpicture}
\path
(-.25,0) node (a) {$A$}
(3.25,0) node (b) {$C$}
(1.5,1.3) node (c) {$B$}
(0,0) node (A) {$\bullet$}
(1.5,1) node (B) {$\bullet$}
(3,0) node (C) {$\bullet$};
\draw[->] (A) to node[anchor=east,pos=0.7] {$\scriptstyle (R_{1},S_{1})$} (B);
\draw[->] (A) to node[below] {$\scriptstyle (R_{3},S_{3})$} (C);
\draw[->] (B) to node[anchor=west,pos=0.3] {$\scriptstyle (R_{2},S_{2})$} (C);
\end{tikzpicture}
\end{equation*}

which satisfy the {\it Triangle Identities}: \\
\begin{equation}\label{eqn:triangleidentities}
R_{1}R_{2} = R_{3}, \hspace{.23in} R_{2}S_{3} = S_{1}, \hspace{.23in} S_{3}R_{1} = S_{2}.
\end{equation}
\end{enumerate}
\end{de}

The definition of $SSE(\calR)$ makes sense for any semiring. For this lecture however, we will consider the case where $\calR$ may be one of:
\begin{enumerate}
\item
$ZO = \{0,1\}$
\item
$\mathbb{Z}_{+}$
\item
$\mathbb{Z}$.
\end{enumerate}

Wagoner\ai{Wagoner, J.B.} also defines $n$-cells in $SSE(\calR)$ for $n \ge 3$ in \cite{Wagoner90triangle}, but we won't need these in this lecture.\\

Note that edges have orientations in $SSE(\calR)$. Recall also that, for an edge from $A$ to $B$ given by a SSE $(R,S)$, we may choose an elementary conjugacy $c(R,S)$ (see \ref{rem:crsconj}), and this choice of $c(R,S)$ does not only depend on $R$ and $S$ but also on some choice of simple automorphisms\si{simple automorphism}. By Williams'\ai{Williams, R.F.} Decomposition Theorem\si{Decomposition Theorem} (Theorem \ref{rfwtheorem}; see also \apr{decomp})), if $C \colon (X_{A},\sigma_{A}) \to (X_{B},\sigma_{B})$ is a topological conjugacy, then there is a strong shift equivalence\si{strong shift equivalence}
$$A = A_{0} \begin{tiny} \reallywidesim{\hspace{.02in} $R_{1},S_{1}$} \end{tiny} A_{1} \begin{tiny} \reallywidesim{\hspace{.02in} $R_{2},S_{2}$} \end{tiny} \hspace{.05in}\cdots \begin{tiny} \hspace{.03in} \reallywidesim{\hspace{.02in}$R_{n-1},S_{n-1}$} \hspace{.03in} \end{tiny} A_{n-1} \begin{tiny} \reallywidesim{\hspace{.02in} $R_{n},S_{n}$} \end{tiny} A_{n} = B$$
such that
$$C = \prod_{i=1}^{n}c(R_{i},S_{i})^{s(i)}$$
with each $c(R_{i},S_{i})$ an elementary conjugacy corresponding to the ESSE given by $R_{i}, S_{i}$, and $s(i) = 1$ if $A_{i-1} = R_{i}S_{i}, A_{i} = S_{i}R_{i}$, while $s(i) = -1$ if $A_{i} = R_{i}S_{i}, A_{i-1} = R_{i}S_{i}$. This presentation $C$ of the conjugacy gives us a path in $SSE(\mathbb{Z}_{+})$

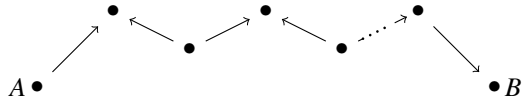
\begin{equation*}
\begin{tikzpicture}
\path
(-3,0) node (A) {$\bullet$}
(-3.25,0) node (Ap) {$A$}
(-2,1) node (a) {$\bullet$}
(-1,0.5) node (b) {$\bullet$}
(0,1) node (c) {$\bullet$}
(1,0.5) node (d) {$\bullet$}
(2,1) node (e) {$\bullet$}
(3,0) node (B) {$\bullet$}
(3.25,0) node (Bp) {$B$};
\path (d) to coordinate[pos=0.11] (aux-1) coordinate[pos=0.75] (aux-2) (e);
\draw[->,shorten >=-2pt, shorten <=1.5pt]   (d) -- (aux-1)
                (aux-2)-- (e);
                \draw[line width=1pt, line cap=round, dash pattern=on 0pt off 3\pgflinewidth]    (aux-1) to (aux-2);
\draw[->] (A)--(a);
\draw[->] (b)--(a);
\draw[->] (b)--(c);
\draw[->] (d)--(c);
\draw[->] (e)--(B);
\end{tikzpicture}
\end{equation*}

Note that some arrows are drawn in reverse, as needed so that the conjugacy $C$ matches the conjugacy given by following the path. Likewise, given a path $\gamma$ in $SSE(\calR)$ between $A$ and $B$
$$\gamma = \prod_{i=1}^{m}\left(R_{i},S_{i}\right)^{s(i)}$$
there is a corresponding conjugacy
$$\tilde{\gamma} = \prod_{i=1}^{m}c(R_{i},S_{i})^{s(i)} \colon (X_{A},\sigma_{A}) \to (X_{B},\sigma_{B}).$$

In particular, vertices of $SSE(\mathbb{Z}_{+})$ correspond to specific presentations of shifts of finite type (edge shift construction), and edges to specific conjugacies (elementary conjugacy coming from an elementary strong shift equivalence\si{strong shift equivalence}). Note that any path between two vertices in these complexes is homotopic to a path following a sequence of edges. \\

Recall from Lecture 1 that a matrix $A$ is degenerate\si{degenerate} if it has a zero row or zero column; otherwise, it is nondegenerate\si{nondegenerate}. Following Wagoner\ai{Wagoner, J.B.}, we only allow nondegenerate\si{nondegenerate} matrices as vertices. It is at times important to work with the larger space $SSE_{deg}(\calR)$ which allows degenerate\si{degenerate} vertices; see for example \cite{BW04}. It turns out that the inclusion $SSE(\mathbb{Z}_{+}) \to SSE_{deg}(\mathbb{Z}_{+})$ induces an isomorphism on $\pi_{0}$ \cite{BW04} and also an isomorphism on $\pi_{1}$ \cite{Epperlein2019} for each path-component.

\subsection{Homotopy groups for Wagoner's complexes\si{Wagoner's strong shift equivalence complex} and $\aut(\sigma_{A})$}\ai{Wagoner, J.B.}

For a semiring $\calR$ and square matrix $A$ over $\calR$, we let $SSE(\calR)_{A}$ denote the path-component of $SSE(\calR)$ containing the vertex $A$. From Williams'\ai{Williams, R.F.} Theorem, the vertices $A,B$ in $SSE(\mathbb{Z}_{+})$ are in the same path-component if and only if the edge shifts $(X_{A},\sigma_{A})$ and $(X_{B},\sigma_{B})$ are topologically conjugate.\\

From the perspective of homotopy theory, the Triangle Identities dictate basic moves for paths in $SSE(\calR)$ to be homotopic. So why the Triangle Identities? The following result of Wagoner\ai{Wagoner, J.B.} explains their importance. In the statement of the theorem, given $A,B$ and two conjugacies $\phi_{1}, \phi_{2} \colon (X_{A},\sigma_{A}) \to (X_{B},\sigma_{B})$, we say $\phi_{1} \sim_{simp} \phi_{2}$ if there exists simple automorphisms\si{simple automorphism} $\gamma_{1} \in \simp{A}, \gamma_{2} \in \simp{B}$ such that $\gamma_{2} \phi_{1} \gamma_{1} = \phi_{2}$ ($\sim_{simp}$ defines an equivalence relation on the set of conjugacies between $(X_{A},\sigma_{A})$ and $(X_{B},\sigma_{B})$).
\begin{thm}[\cite{Wagoner90triangle,MR908217,MR1012941,Wagoner90eventual}]\label{thm:wagonerhomotopyofpaths}\ai{Wagoner, J.B.}
For the spaces $SSE(ZO), SSE(\mathbb{Z}_{+})$ defined above, both of the following hold:
\begin{enumerate}
\item
Given vertices $A,B$ in $SSE(ZO)$, two paths in $SSE(ZO)$ from $A$ to $B$ are homotopic in $SSE(ZO)$ if and only if they induce the same conjugacy from $(X_{A},\sigma_{A})$ to $(X_{B},\sigma_{B})$.
\item
Given vertices $A,B$ in $SSE(\mathbb{Z}_{+})$, two paths in $SSE(\mathbb{Z}_{+})$ from $A$ to $B$ are homotopic in $SSE(\mathbb{Z}_{+})$ if and only if they induce the same conjugacy from $(X_{A},\sigma_{A})$ to $(X_{B},\sigma_{B})$ modulo the relation $\sim_{simp}$.
\end{enumerate}
\end{thm}

Item $(2)$ in the above is perhaps expected; recall the construction given in Section \ref{subsec:sseandclassification} of Lecture 1 for associating conjugacies with SSE's over $\mathbb{Z}_{+}$ requires a choice of labels for certain edges. This choice is where ambiguity up to conjugating by simple automorphisms\si{simple automorphism} may arise.\\

Theorem \ref{thm:wagonerhomotopyofpaths} gives the first two parts of the following theorem of Wagoner\ai{Wagoner, J.B.}. For a space $X$ with point $x \in X$, let $\pi_{k}(X,x)$ denote the $k$th homotopy group based at $x$.
\begin{thm}[\cite{Wagoner90triangle,MR908217,MR1012941,Wagoner90eventual}]\label{thm:wagonerreptheorem}\ai{Wagoner, J.B.}
Let $A$ be a square matrix over $ZO$. Then:
\begin{enumerate}
\item
$\aut(\sigma_{A}) \cong \pi_{1}(SSE(ZO),A)$.
\item
$\aut(\sigma_{A}) / \simp{A} \cong \pi_{1}(SSE(\mathbb{Z}_{+}),A)$.
\item
$\aut(G_{A},\delta_{A}) \cong \pi_{1}(SSE(\mathbb{Z}),A)$.
\end{enumerate}
\end{thm}

It is immediate from the definition of the SSE spaces that the set $\pi_{0}(SSE(\mathbb{Z}_{+}))$ may be identified with the set of conjugacy classes of shifts of finite type. Moreover, $\pi_{0}(SSE(\mathbb{Z}))$ may be identified with the set of strong shift equivalence\si{strong shift equivalence} classes of matrices over $\mathbb{Z}$.\\

Upon using the identifications above, the composition map
$$\pi_{1}(SSE(ZO),A) \to \pi_{1}(SSE(\mathbb{Z}_{+}),A) \to \pi_{1}(SSE(\mathbb{Z}),A)$$
induced by the natural inclusions $SSE(ZO) \hookrightarrow SSE(\mathbb{Z}_{+}) \hookrightarrow SSE(\mathbb{Z})$ is isomorphic to the dimension representation\si{dimension representation} factoring as
$$\aut(\sigma_{A}) \to \aut(\sigma_{A})/\simp{A} \to \aut(G_{A}),$$
i.e. the diagram
\begin{equation*}
\xymatrix{
\pi_{1}(SSE(ZO),A) \ar[d]^{\cong} \ar[r] & \pi_{1}(SSE(\mathbb{Z}_{+}),A) \ar[d]^{\cong} \ar[r] & \pi_{1}(SSE(\mathbb{Z}),A) \ar[d]^{\cong}\\
\aut(\sigma_{A}) \ar[r] & \aut(\sigma_{A})/\simp{A} \ar[r] & \aut(G_{A},\delta_{A})
}
\end{equation*}
commutes.

Wagoner\ai{Wagoner, J.B.} also proves that $\pi_{k}(SSE(ZO),A) = 0$ for $k \ge 2$. This implies $SSE(ZO)_{A}$ is a model for the classifying space of $\aut(\sigma_{A})$, i.e. $SSE(ZO)_{A}$ is homotopy equivalent to $B\aut(\sigma_{A})$ \apr{apprem:classifyingspacerem}). Thus, for example, we have
$$\aut(\sigma_{A})_{ab} \cong H_{1}(\aut(\sigma_{A}),\mathbb{Z}) \cong H_{1}(SSE(ZO)_{A},\mathbb{Z}).$$
It is worth remarking that, at the moment, we do not know what the abelianization $\aut(\sigma_{A})_{ab}$ is for any positive entropy shift of finite type $(X_{A},\sigma_{A})$ (however it is at least known, from \cite[Theorem 7.8]{BLR88}, that $\aut(\sigma_{A})_{ab}$ is not finitely generated).\\

Wagoner\ai{Wagoner, J.B.} also introduced complexes $SE(\calR)$ defined analogously to $SSE(\calR)$ (see \apr{apprem:secomplexdef}) for a definition). Since an ESSE over $\calR$ also gives an SE over $\calR$, there is a continuous inclusion map $i_{\calR} \colon SSE(\calR) \to SE(\calR)$. Wagoner\ai{Wagoner, J.B.} proved in \cite{MR1012941} that, in the case $\calR$ is a principal ideal domain, this map $i_{\calR}$ is a homotopy equivalence, and that $\pi_{n}(SSE(\calR),A) = \pi_{n}(SE(\calR),A) = 0$ for all $n \ge 2$ and any $A$. The map $i_{\calR}$ cannot be a homotopy equivalence for a general ring $\calR$ \apr{apprem:nothomeqingeneral}).\\

Wagoner's\ai{Wagoner, J.B.} complexes, and the results of Theorem \ref{thm:wagonerreptheorem}, have recently been generalized to a groupoid setting in \cite{Epperlein2019}. This setting simplifies some of the proofs and extends Wagoner's construction to shifts of finite type carrying a free action by a finite group, as well as more general shifts of finite type over arbitrary finitely generated groups.

\subsection{Counterexamples to Williams' Conjecture}
A counterexample to Williams' Conjecture\si{Williams' Conjecture}\ai{Williams, R.F.} in the primitive case was given by Kim\ai{Kim, K.H.} and Roush\ai{Roush, F.W.} in \cite{S11}. In \cite{Wagoner2000}\ai{Wagoner, J.B.}, Wagoner also verified the counterexamples using a different framework. Both methods for detecting the counterexamples take place in the setting of Wagoner's SSE complexes, and build on a great deal of work by many authors. We outline the techniques here; one may also see Wagoner's survey article \cite{Wagoner99} for an exposition regarding the counterexamples.\\
\indent Since our goal is only to give a brief introduction to how these counterexamples arise, we won't actually list the explicit matrices involved; they can be found in \cite{S11} or \cite{Wagoner2000}).\ai{Wagoner, J.B.} Instead, we focus on the strategy used to prove that they in fact {\it are} counterexamples.\\

To start, both strategies roughly follow the same initial idea. As mentioned in Section \ref{subsec:comparingseandsse}, to find a counterexample to Williams' Conjecture\si{Williams' Conjecture} it is sufficient to find a pair of primitive matrices which are connected by a path in $SSE(\mathbb{Z})$, and show they cannot be connected by a path through $SSE(\mathbb{Z}_{+})$. We can formalize this approach in terms of homotopy theory (this is an important viewpoint, although not necessary to understand the Kim-Roush\ai{Roush, F.W.}\ai{Kim, K.H.}
counterexample, as we will see). Consider $SSE(\mathbb{Z}_{+})$ as a subcomplex of $SSE(\mathbb{Z})$, and, upon fixing a base point $A$ in $SSE(\mathbb{Z}_{+})$, consider the long exact sequence in homotopy groups based at $A$ for the pair $\left(SSE(\mathbb{Z}),SSE(\mathbb{Z}_{+})\right)$:

$$ \cdots \pi_{1}(SSE(\mathbb{Z}),A) \to \pi_{1}(SSE(\mathbb{Z}),SSE(\mathbb{Z}_{+}),A) \to \pi_{0}(SSE(\mathbb{Z}_{+}),A) \to \pi_{0}(SSE(\mathbb{Z}),A).$$

Here $\pi_{1}(SSE(\mathbb{Z}),SSE(\mathbb{Z}_{+}),A)$ denotes the set of homotopy classes of paths with base point in $SSE(\mathbb{Z}_{+})$ and end point equal to $A$, and the map $\pi_{1}(SSE(\mathbb{Z}),SSE(\mathbb{Z}_{+}),A) \to \pi_{0}(SSE(\mathbb{Z}_{+}),A)$ is defined by sending the homotopy class of a path $\gamma$ to the component containing $\gamma(0)$ (details regarding this sequence can be found in \cite[Ch. 4, Thm. 4.3]{HatcherBook}). The last three terms in this sequence are not actually groups, but just pointed sets. Still, exactness makes sense, by defining the kernel to be the pre-image of the base point. The base point in $\pi_{1}(SSE(\mathbb{Z}),SSE(\mathbb{Z}_{+}),A)$ is given by the homotopy class of a path which lies entirely in $SSE(\mathbb{Z}_{+})$. In particular, the set $\pi_{1}(SSE(\mathbb{Z}),SSE(\mathbb{Z}_{+}),A)$ has only one element if and only if every path beginning in $SSE(\mathbb{Z}_{+})$ and ending at $A$ is homotopic to a path lying entirely in $SSE(\mathbb{Z}_{+})$.\\

In this setup, the goal is to now find a function

\begin{equation*}
F \colon \pi_{1}\left(SSE(\mathbb{Z}),SSE(\mathbb{Z}_{+}),A\right) \to G
\end{equation*}
to some group $G$; for computability, we would like $G$ abelian. Then to find a counterexample, it would be enough to find matrices $A$ and $B$ and a path $\gamma$ in $SSE(\mathbb{Z})$ from $A$ to $B$ such that $F(\gamma) \ne 0$, while $F(\beta) = 0$ for any $\beta \in \pi_{1}(SSE(\mathbb{Z}),A)$. Note that from Theorem \ref{thm:wagonerreptheorem}\ai{Wagoner, J.B.} we know $\pi_{1}(SSE(\mathbb{Z}),A) \cong \aut(G_{A},\delta_{A})$, so in light of the long exact sequence in homotopy written above, being able to compute generators for $\aut(G_{A},\delta_{A})$ plays an important role here.\\

Put another way, we want to find some abelian group $G$, a primitive matrix $A$, and some function $F$ from edges in $SSE(\mathbb{Z})_{A}$ to $G$ which satisfies all of the following, where $\alpha \star \beta$ denotes concatenation of paths:
\begin{gather}\label{eqn:FandG1}
F(\alpha \star \beta) = F(\alpha) + F(\beta)\\
\label{eqn:FandG2}
\textnormal{If } \gamma_{1} \textnormal{ and } \gamma_{2} \textnormal{ are homotopic paths, then } F(\gamma_{1}) = F(\gamma_{2})\\
\label{eqn:FandG3}
F(\gamma) = 0 \textnormal{ if } \gamma \textnormal{ lies in } SSE(\mathbb{Z}_{+})\\
\label{eqn:Fnonvanishing}
F(\gamma_{A,B}) \ne 0 \textnormal{ for some path } \gamma_{A,B} \textnormal{ from } A \textnormal{ to a primitive matrix } B
\end{gather}

Kim\ai{Kim, K.H.} and Roush\ai{Roush, F.W.}, and independently Wagoner\ai{Wagoner, J.B.}, found functions $F_{m}$ each satisfying $\eqref{eqn:FandG1}$, $\eqref{eqn:FandG2}$, $\eqref{eqn:FandG3}$ for $G = \mathbb{Z}/m$ for paths contained in any component of a matrix $A$ satisfying $tr(A^{k}) = 0$ for all $1 \le k \le m$. Finally, for $m=2$, Kim\ai{Kim, K.H.} and Roush\ai{Roush, F.W.} found a pair of matrices $A,B$ and a path $\gamma_{A,B}$ satisfying \eqref{eqn:Fnonvanishing}.

\subsection{Kim-Roush relative sign-gyration
  method }\label{subsec:kimroushrelsg}\ai{Kim, K.H.}
Let $(X_{A},\sigma_{A})$ be a mixing shift of finite type, and recall from Section \ref{subsec:sgcc} the sign-gyration-compatability-condition homomorphisms
\begin{equation*}
\begin{gathered}
SGCC_{m} \colon \aut(\sigma_{A}) \to \mathbb{Z}/m\mathbb{Z}\\
SGCC_{m} = g_{m} + \left(\frac{m}{2}\right)\sum_{j>0}\sgn \xi_{m/2^{j}}.
\end{gathered}
\end{equation*}

Given $\alpha \in \aut(\sigma_{A})$, for any $m$, $SGCC_{m}(\alpha)$ is defined in terms of the action of $\alpha$ on the periodic points up to level $m$.\\

The idea behind the Kim\ai{Kim, K.H.} and Roush\ai{Roush, F.W.} technique is to define, for each $m$, a relative sign-gyration-compatibility-condition map
$$sgc_{m} \colon \pi_{1}(SSE(\mathbb{Z}),SSE(\mathbb{Z}_{+}),A) \to \mathbb{Z}/m\mathbb{Z}.$$

To start, suppose $A \stackrel{(R,S)}\longrightarrow B$ is an edge in $SSE(\mathbb{Z}_{+})$ given by a strong shift equivalence\si{strong shift equivalence} $A=RS, B=SR$ over $\mathbb{Z}_{+}$. Associated to this (by Theorem \ref{rfwtheorem}) is an elementary conjugacy
$$c(R,S) \colon (X_{A},\sigma_{A}) \to (X_{B},\sigma_{B}).$$
Recall this conjugacy $c(R,S)$ is not determined by $(R,S)$, but is only defined up to composition with simple automorphisms\si{simple automorphism} in the domain and range. Given $m$, choose some orderings on the set of orbits whose lengths divide $m$, and a distinguished point in each such orbit, for each of $(X_{A},\sigma_{A})$ and $(X_{B},\sigma_{B})$; in \cite{S19}, these choices are made using certain lexicographic rules on the set of periodic points. The conjugacy $c(R,S)$ induces a bijection between the respective periodic point sets for $\sigma_{A}$ and $\sigma_{B}$, and we may define, with respect to the choices of orderings and distinguished points in each orbit, the sign and gyration maps, and hence define $SGCC_{m}(c(R,S)) \in \mathbb{Z}/m\mathbb{Z}$. If $RS \ne SR$, the value $SGCC_{m}(c(R,S))$ may depend on the choices of orderings and distinguished points.\\
\indent In \cite{S19}, Kim-Roush\ai{Roush, F.W.}-Wagoner\ai{Wagoner, J.B.}\ai{Kim, K.H.} showed that for such a conjugacy $c(R,S)$, there is a formula $sgcc_{m}(R,S)$ for $SGCC_{m}(R,S)$ in terms of the entries from the matrices $R,S$. This was used to prove Theorem \ref{thm:sgccfactorization}, that $SGCC$\si{sign-gyration compatibility condition} factors through the dimension representation\si{dimension representation}. We note that this formula for $sgcc_{m}(R,S)$ in general depends on the choice of orderings on the periodic points. Furthermore, the formulas defined in \cite{S19} are very complicated for large $m$. In \cite{S11}, Kim\ai{Kim, K.H.} and Roush\ai{Roush, F.W.} defined $sgc_{m}$, a slightly different version \apr{apprem:sgcvssgcc}) of $sgcc_{m}$, that also computes $SGCC_{m}$ in terms of entries from $R$ and $S$; for $m=2$, it takes the form
$$sgc_{2}(R,S) = \sum_{\underset{k > l}{i<j}}R_{ik}S_{ki}R_{jl}S_{lj} + \sum_{\underset{k \ge l}{i<j}}R_{ik}S_{kj}R_{jl}S_{li} + \sum_{i,j}\frac{1}{2}R_{ij}(R_{ij}-1)S_{ji}^{2}.$$

In other words, for an elementary strong shift equivalence\si{strong shift equivalence} $A=RS, B=SR$, we have $SGCC_{m}(R,S) = sgc_{m}(R,S)$. The formula given above for $sgc_{2}$ uses orderings on the fixed points and period two points defined by certain lexicographic rules given in \cite{S19}. For the counterexamples to Williams' Conjecture\si{Williams' Conjecture}, only $sgc_{2}$ is needed.\\

We can extend $SGCC_{m}$ from elementary conjugacies $c(R,S)$ to paths in $SSE(\mathbb{Z}_{+})$: given a path
$$\gamma = \prod_{i=1}^{J}\left(R_{i},S_{i}\right)^{s(i)}$$
define
$$SGCC_{m}(\gamma) = \sum_{i}^{J}s(i)SGCC_{m}(R_{i},S_{i}).$$

Note from the above we also know that
$$SGCC_{m}(\gamma) = \sum_{i}^{J}s(i)sgc_{m}(R_{i},S_{i}).$$

Now suppose we have a basic triangle in $SSE(\mathbb{Z}_{+})$ with edges $(R_{1},S_{1}), (R_{2},S_{2}), (R_{3},S_{3})$. If $c(R_{3},S_{3}) = c(R_{1},S_{1})c(R_{2},S_{2})$ then using the fact that $SGCC_{m}$ is defined in terms of dynamical data coming from the corresponding conjugacies, a calculation \cite[Prop. 2.9]{S19} shows that
$$SGCC_{m}(R_{1},S_{1}) + SGCC_{m}(R_{2},S_{2}) = SGCC_{m}(R_{3},S_{3}).$$
But by Theorem \ref{thm:wagonerhomotopyofpaths}\ai{Wagoner, J.B.}, up to conjugating by simple automorphisms\si{simple automorphism}, we do have $c(R_{3},S_{3}) = c(R_{1},S_{1})c(R_{2},S_{2})$; since $SGCC_{m}$ vanishes on simple automorphisms\si{simple automorphism} (Theorem \ref{thm:sgccfactorization}), this gives an addition formula for $SGCC_{m}$ over triangles in $SSE(\mathbb{Z}_{+})$.\\

Now suppose we have an elementary strong shift equivalence\si{strong shift equivalence} $A=RS,B=SR$ over $\mathbb{Z}$ (so not necessarily in $\mathbb{Z}_{+}$). The $sgc_{m}$ formulas still make sense, so we can define $sgc_{m}(\gamma)$ for any path $\gamma$ in $SSE(\mathbb{Z})$. If $sgc_{m}$ also satisfies an addition formula for triangles in $SSE(\mathbb{Z})$, then $sgc_{m}$ will give us an extension of $SGCC_{m}$ to $SSE(\mathbb{Z})$. This turns out to be the case, and is a consequence of the following Cocycle Lemma.
\begin{lem}[\cite{S19,S11}]\label{lemma:cocyclelemma}
If the edges $(R_{1},S_{1}), (R_{2},S_{2}), (R_{3},S_{3})$ form a basic triangle in $SSE(\mathbb{Z})$, then
$$sgc_{m}(R_{1},S_{1}) + sgc_{m}(R_{2},S_{2}) = sgc_{m}(R_{3},S_{3}).$$
\end{lem}

The Cocycle Lemma was first proved in \cite{S19} in the case when the triangle contains a vertex which is strong shift equivalent over $\mathbb{Z}$ to a nonnegative primitive matrix. The version above, which does not require a primitivity assumption, was given in \cite{S11}, with a much shorter proof suggested by Mike Boyle.\\

Putting all of the above together, for a matrix $A$, the map
$$sgc_{m} \colon \pi_{1}(SSE(\mathbb{Z}),SSE(\mathbb{Z}_{+}),A) \to \mathbb{Z}/m\mathbb{Z}$$
satisfies \eqref{eqn:FandG1} and \eqref{eqn:FandG2}.\\

Now suppose that $A$ satisfies $tr(A^{k})=0$ for all $1 \le k \le m$ and $(R,S)$ is an edge in $SSE(\mathbb{Z}_{+})$ from $A$ to $B$. Then both $(X_{A},\sigma_{A})$ and $(X_{B},\sigma_{B})$ have no points of period $k$ for any $1 \le k \le m$, and the dynamically defined $SGCC_{m}(R,S)$ must vanish; since $sgc_{m} = SGCC_{m}$ on edges in $SSE(\mathbb{Z}_{+})$, this implies $sgc_{m}(R,S) = 0$. It follows that on path-components of matrices $A$ with $tr(A^{k})=0$ for all $1 \le k \le m$, the map $sgc_{m}$ also satisfies \eqref{eqn:FandG3}.\\

Finally, using $m=2$, in \cite{S11} Kim\ai{Kim, K.H.} and Roush\ai{Roush, F.W.} found two primitive matrices $A,B$ and a path $\gamma$ in $SSE(\mathbb{Z})$ from $A$ to $B$ such that all of the following hold:
\begin{enumerate}
\item
$tr(A)=tr(A^{2})=0$.
\item
$sgc_{2}(\alpha) = 0$ for any $\alpha \in \pi_{1}(SSE(\mathbb{Z}),A)$.
\item
$sgc_{2}(\gamma) \ne 0$.
\end{enumerate}

It follows these matrices $A$ and $B$ are strong shift equivalent over $\mathbb{Z}$, but not strong shift equivalent over $\mathbb{Z}_{+}$. The matrices $A$ and $B$ given in \cite{S11} are $7 \times 7$.

\subsection{Wagoner's $K_{2}$-valued obstruction map }\label{subsec:wagonerk2obs}\ai{Wagoner, J.B.}
Wagoner\ai{Wagoner, J.B.}, influenced by ideas from pseudo-isotopy theory, constructed a map $F$ satisfying the three conditions \ref{eqn:FandG1} -- \ref{eqn:FandG3} landing in the $K$-theory group $K_{2}(\mathbb{Z}[t]/(t^{m+1}))$. In \cite{Wagoner2000} Wagoner\ai{Wagoner, J.B.} then used this framework to detect counterexamples with matrices found using the technique given by Kim\ai{Kim, K.H.} and Roush\ai{Roush, F.W.} in \cite{S11}. The Kim-Roush\ai{Roush, F.W.}\ai{Kim, K.H.} relative-sign-gyration-compatability method of the previous section enjoys the fact that it is motivated by dynamical data relating directly to the shift systems, being based on ideas from sign-gyration. Wagoner's\ai{Wagoner, J.B.} method is not as easily connected to the dynamics, but offers some alternative benefits, namely:
\begin{enumerate}
\item
Landing in $K_{2}$, it connects directly with algebraic K-theory.
\item
It operates within the polynomial matrix framework.
\item
It is perhaps suggestive of more general strategies for studying the refinement of strong shift equivalence\si{strong shift equivalence} over a ring by strong shift equivalence\si{strong shift equivalence} over the ordered part of a ring, i.e. part (3) in the picture in Lecture 6 describing Williams' Problem.
\end{enumerate}


So how does Wagoner's\ai{Wagoner, J.B.} construction work? We recall two facts about the group $K_{2}(\calR)$ from Section \ref{subsec:k2ofaring}:
\begin{enumerate}
\item
$K_{2}(\calR)$ is an abelian group.
\item
An expression of the form $\prod_{i=1}^{k}E_{i} =1$, where $E_{i}$ are elementary matrices over $\calR$, can be used to construct an element of $K_{2}(\calR)$.
\end{enumerate}

For $m \ge 1$, let $SSE_{2m}(\mathbb{Z}_{+})$ denote the subcomplex of $SSE(\mathbb{Z}_{+})$ consisting of path-components which have a vertex $A$ such that $tr(A^{k})=0$ for all $1 \le k \le 2m$. Wagoner's\ai{Wagoner, J.B.} construction proceeds as follows:
\begin{enumerate}
\item
Consider an edge in $SSE(\mathbb{Z})$ from $A$ to $B$. As shown in Lecture 3, this gives matrices $E_{1},F_{1}$ in $El(\mathbb{Z}[t])$ over $\mathbb{Z}[t]$ such that
$$E_{1}(I-tA)F_{1} = I-tB.$$
\item
Suppose the matrix $A$ satisfies $tr(A^{k}) = 0$ for all $1 \le k \le m$. In \cite[Prop. 4.9]{Wagoner2000}\ai{Wagoner, J.B.} it is shown there exist matrices $E_{2},F_{2}$ in $El(\mathbb{Z}[t])$ and $A^{\prime}$ over $\mathbb{Z}[t]$ such that $E_{2}(I-tA)F_{2} = I-t^{m+1}A^{\prime}$. Doing the same for $B$ yields matrices $E_{3},F_{3}$ in $El(\mathbb{Z}[t])$ and some $B^{\prime}$ over $\mathbb{Z}[t]$ such that
$$E_{2}(I-tA)F_{2} = I-t^{m+1}A^{\prime}$$
$$E_{3}(I-tB)F_{3} = I-t^{m+1}B^{\prime}.$$
\item
Combining steps $(1)$ and $(2)$ we have matrices $X,Y$ in $El(\mathbb{Z}[t])$ such that
$$X(I-t^{m+1}A^{\prime})Y = I-t^{m+1}B^{\prime}.$$
Passing to $\mathbb{Z}[t] / (t^{m+1})$, we get
$$XY = I.$$
We can now use this expression to produce an element of $K_{2}(\mathbb{Z}[t]/(t^{m+1}))$.
\end{enumerate}

Wagoner\ai{Wagoner, J.B.} shows this assignment defined above is additive with respect to concatenation of paths given by two subsequent edges, so one can extend it to arbitrary paths. Thus, given an edge $\gamma$ in $SSE(\mathbb{Z})$, applying the above gives an element $F(\gamma) \in K_{2}(\mathbb{Z}[t]/(t^{m+1}))$. Then, given a path $\gamma$ between two vertices $A$ and $B$ in $SSE_{2m}(\mathbb{Z})$, Wagoner\ai{Wagoner, J.B.} shows:
\begin{enumerate}
\item[(a)] The element $F(\gamma)$ in $K_{2}(\mathbb{Z}[t]/(t^{m+1}))$ produced by the above construction is independent of the choices of elementary matrices made in the construction.
\item[(b)]
If $A,B$ are nonnegative and $\gamma^{\prime}$ is another path in $SSE(\mathbb{Z})$ from $A$ to $B$ such that $\gamma$ and $\gamma^{\prime}$ are homotopic (with endpoints fixed), then $F(\gamma) = F(\gamma^{\prime})$ in $K_{2}(\mathbb{Z}[t]/(t^{m+1}))$.
\item[(c)]
If the path $\gamma$ lies entirely in $SSE_{2m}(\mathbb{Z}_{+})$, then the corresponding element $F(\gamma)$ in $K_{2}(\mathbb{Z}[t]/(t^{m+1}))$ vanishes.
\end{enumerate}

Altogether this defines a function
$$\Phi_{2m} \colon \pi_{1}(SSE(\mathbb{Z}),SSE_{2m}(\mathbb{Z}_{+}),A) \to K_{2}(\mathbb{Z}[t]/(t^{m+1}))$$
satisfying the properties \ref{eqn:FandG1} -- \ref{eqn:FandG3} for $A$ in $SSE_{2m}(\mathbb{Z}_{+})$.\\

Let $K_{2}(\mathbb{Z}[t]/(t^{m+1}),(t))$ denote the kernel of the split surjection $K_{2}(\mathbb{Z}[t]/(t^{m+1})) \to K_{2}(\mathbb{Z})$ induced by the ring map $\mathbb{Z}[t]/(t^{m+1}) \to \mathbb{Z}$ induced by $t \to 0$. Wagoner\ai{Wagoner, J.B.} proved that the maps $\Phi_{2m}$ defined above actually lands in $K_{2}(\mathbb{Z}[t]/(t^{m+1}),(t))$. This is a significant fact, since van der Kallen proved in \cite{vdK} that $K_{2}(\mathbb{Z}[t]/(t^{2}),(t)) \cong \mathbb{Z}/2$. This calculation by van der Kallen was used by Wagoner to explicitly compute \cite[Eq. 1.21]{Wagoner2000} $\Phi_{2}$, and to detect some explicit counterexamples in \cite{Wagoner2000}.\ai{Wagoner, J.B.}



\subsection{Some remarks and open problems} \label{WagonerSubSecRemarks}
At the $m=2$ level, each method outlined above gives a map
$$sgc_{2} \colon \pi_{1}(SSE(\mathbb{Z}),SSE_{2}(\mathbb{Z}_{+}),A) \to \mathbb{Z}/2\mathbb{Z}$$
$$\Phi_{2} \colon \pi_{1}(SSE(\mathbb{Z}),SSE_{2}(\mathbb{Z}_{+}),A) \to K_{2}(\mathbb{Z}[t]/(t^{2}),(t)) \cong \mathbb{Z}/2\mathbb{Z}.$$
While these were developed independently, remarkably, it was shown by Kim\ai{Kim, K.H.} and Roush\ai{Roush, F.W.} in the Appendix of \cite{Wagoner2000}\ai{Wagoner, J.B.} that $\Phi_{2} = sgc_{2}$. Wagoner explicitly poses the problem in \cite[Number 6]{Wagoner99}\ai{Wagoner, J.B.} to determine, for larger $m$, the relationship between $\Phi_{2m}$ and $sgc_{m}$.\\

Finally, let us note that both the Kim-Roush\ai{Roush, F.W.}\ai{Kim, K.H.}
method and Wagoner's method rely on the non-existence of periodic points at certain low levels. In Wagoner's\ai{Wagoner, J.B.} case, without vanishing trace conditions, step $(2)$ above can not be carried out. Moreover, step $(3)$ also relies on the vanishing trace conditions. As a result, Wagoner's\ai{Wagoner, J.B.} construction is {\it only} defined in the case of shifts of finite type lacking periodic points of certain low order levels. For the Kim-Roush\ai{Roush, F.W.}\ai{Kim, K.H.} technique, the non-existence of low-order periodic points comes in when one wants to conclude that the assignment from edges to some element of $\mathbb{Z}/m$ vanishes along any path through $SSE(\mathbb{Z}_{+})$: for an edge in $SSE(\mathbb{Z}_{+})$, the assignment coincides with the relative sign-gyration numbers associated to a conjugacy, which, in the absence of any periodic points of the given levels, must vanish.\\
\indent In light of this, neither method is able to produce more than a finite index refinement of the strong shift equivalence\si{strong shift equivalence} class of a given primitive matrix $A$ over $\mathbb{Z}_{+}$, since $(X_{A},\sigma_{A})$ will, above some level $k$ depending on $A$, eventually contain periodic points at all levels larger than $k$.

To finish, we highlight two open problems (Problem \ref{prob:finiterefinementproblem1} below was mentioned informally in the discussion following Conjecture \ref{conj:williams} in
Section~\ref{sec:basics}):\\
\\

\begin{prob}
  If $A$ is shift equivalent over $\mathbb{Z}_{+}$ to the $1 \times 1$ matrix $(n)$, must $A$ be strong shift equivalent over $\mathbb{Z}_{+}$ to $(n)$?
  In other words, does Williams'\ai{Williams, R.F.} Conjecture\si{Williams' Conjecture} hold in the case of full shifts?
\end{prob}

\begin{prob}\label{prob:finiterefinementproblem1}
For a primitive matrix $A$, is the refinement of the SE-$\mathbb{Z}_{+}$-equivalence class of $A$ by SSE-$\mathbb{Z}_{+}$ finite?
\end{prob}

Finally, we think the complexes $SSE(ZO), SSE(\mathbb{Z}_{+})$ and $SSE(\mathbb{Z})$ probably have much more to offer, and obtaining a deeper understanding of them would be valuable for studying both strong shift equivalence\si{strong shift equivalence} and the conjugacy problem for shifts of finite type.

\subsection{Appendix 8}
This appendix contains some proofs, remarks, and solutions of various exercises throughout Lecture 8.\\

\begin{rem}\label{apprem:wagonermarkovspace}\ai{Wagoner, J.B.}
  Prior to considering the strong shift equivalence\si{strong shift equivalence}
  spaces $SSE(\calR)$, Wagoner also introduced a related `space of Markov partitions' for a shift of finite type; we won't describe these here, and instead refer the reader to
  \cite{Wagoner90triangle,MR1748178,Epperlein2019}\ai{Wagoner, J.B.}.
\end{rem}

\begin{rem}\label{apprem:classifyingspacerem}
For a discrete group $G$, a classifying space is a path-connected space $BG$ such that $\pi_{1}(BG) \cong G$ and $\pi_{k}(BG) = 0$ for all $k \ge 2$. The space $BG$ has the property that $H_{k}(G,\mathbb{Z})$, the integral group homology of the group $G$, is isomorphic to $H_{k}(BG,\mathbb{Z})$, the integral singular homology of the space $BG$. See \cite[6.10.4]{WeibelHomAlgBook} for details.
\end{rem}

\begin{rem}\label{apprem:secomplexdef}
For a semiring $\calR$, the shift equivalence\si{shift equivalence} space $SE(\calR)$ is the CW complex defined as follows.
\begin{enumerate}
\item
The 0-cells of $SE(\calR)$ are square matrices over $\calR$.
\item
An edge from vertex $A$ to vertex $B$ corresponds to a shift equivalence\si{shift equivalence} over $\calR$ from $A$ to $B$, i.e. matrices $R,S$ over $\calR$ and $k \ge 1$ such that
$$A^{k}=RS, \qquad B^{k} = SR, \qquad AR = RB, \qquad SA = BS.$$


\item
2-cells are given by triangles

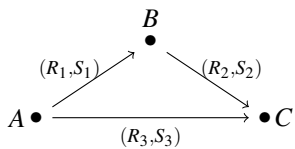
\begin{equation*}
\begin{tikzpicture}
\path
(-.25,0) node (a) {$A$}
(3.25,0) node (b) {$C$}
(1.5,1.3) node (c) {$B$}
(0,0) node (A) {$\bullet$}
(1.5,1) node (B) {$\bullet$}
(3,0) node (C) {$\bullet$};
\draw[->] (A) to node[anchor=east,pos=0.7] {$\scriptstyle (R_{1},S_{1})$} (B);
\draw[->] (A) to node[below] {$\scriptstyle (R_{3},S_{3})$} (C);
\draw[->] (B) to node[anchor=west,pos=0.3] {$\scriptstyle (R_{2},S_{2})$} (C);
\end{tikzpicture}
\end{equation*}

such that
\begin{equation}
R_{1}R_{2} = R_{3}.
\end{equation}
\end{enumerate}
Higher cells are defined in the same way as for the SSE spaces. It is immediate from the definition that $\pi_{0}(SE(\calR))$ is in bijective correspondence with the set of shift equivalence\si{shift equivalence} classes of matrices over $\calR$.
\end{rem}

\begin{rem}\label{apprem:nothomeqingeneral}
We'll show here that $i_{\calR}$ cannot in general be a homotopy equivalence. The map $i_{\calR}$ induces a map of sets $i_{\calR,*} \colon \pi_{0}(SSE(\calR)) \to \pi_{0}(SE(\calR))$. We can identify $\pi_{0}(SSE(\calR))$ with the set of SSE-classes of matrices over $\calR$ and $\pi_{0}(SE(\calR)$ with the set of SE-classes of matrices over $\calR$, and upon making these identifications, the map $i_{\calR,*}$ agrees with the map $\pi$ given in \eqref{eqn:ssetosemap}. Theorem \ref{thm:ssesefibers} from Lecture 6 gives a description of the fibers of this map in terms of some K-theoretic data. In particular, from Corollary \ref{cor:nk1vanishandsse} we know that the map $i_{\calR,*} \colon \pi_{0}(SSE(\calR)) \to \pi_{0}(SE(\calR))$ is not always an injection. Thus Wagoner's\ai{Wagoner, J.B.} result that $i_{\calR}$ is a homotopy equivalence when $\calR$ is a principal ideal domain can not hold in the case $NK_{1}(\calR) \ne 0$; as we see, it need not even induce an injection on the level of $\pi_{0}$.
\end{rem}

\begin{rem}\label{apprem:sgcvssgcc}
As pointed out in \cite[Section 8]{S11}, the maps $sgc_{m}$ and $sgcc_{m}$ are not the same in general. However, they do yield the same value on path-components containing a primitive matrix whose trace is zero. The definition for $sgcc_{m}$ requires a component with a matrix which is shift equivalent to a primitive matrix, whereas the map $sgc_{m}$ does not. See \cite[Section 8]{S11} for more details regarding the difference between $sgc_{m}$ and $sgcc_{m}$.
\end{rem}

\bibliographystyle{plain}
\bibliography{BS}


\printindex 

\printindex[authors]

(The index of authors supplements, and does not repeat, the
  page citings listed in the bibliography.)


\end{document}